\documentclass[a4paper, 12pt]{amsart}
\usepackage{amsfonts, amsthm, amssymb, amsmath, stmaryrd}
\usepackage{mathrsfs,array}
\usepackage{xy}
\usepackage{verbatim}
\usepackage{amscd} 
\usepackage{tikz}
\usepackage{tikz-cd}
\usetikzlibrary{matrix,arrows,decorations.pathmorphing}
\usepackage{textcomp}
\input xy
\xyoption{all}
\usepackage[unicode]{hyperref}

\numberwithin{equation}{subsection}

\newtheorem{thm}{Theorem}[subsection]

\newtheorem{cor}[thm]{Corollary}

\newtheorem{lem}[thm]{Lemma} 
\newtheorem{prop}[thm]{Proposition}
 \theoremstyle{definition}
 \theoremstyle{definition}
\newtheorem{defn}[thm]{Definition} \theoremstyle{remark}
\newtheorem{rem}[thm]{\bf Remark}

\newtheorem{para}[thm]{\bf} 
 
\newtheorem{exa}[thm]{\bf Example}

\DeclareMathOperator{\Hom}{Hom}
\DeclareMathOperator{\End}{End}
\DeclareMathOperator{\Ext}{Ext}

\DeclareMathOperator{\im}{im}
\DeclareMathOperator{\coker}{coker}

\DeclareMathOperator{\Spa}{Spa}
\DeclareMathOperator{\Spec}{Spec}

\DeclareMathOperator{\Fil}{Fil}
\DeclareMathOperator{\gr}{gr}
\DeclareMathOperator{\Gal}{Gal}

\DeclareMathOperator{\Lie}{Lie}
\DeclareMathOperator{\Frob}{Frob}

\def\cont{\mathrm{cont}}

\def\dR{\mathrm{dR}}

\def\HT{\mathrm{HT}}

\def\ps{\mathrm{ps}}
\def\sm{\mathrm{sm}}

\def\an{\mathrm{an}}
\def\la{\mathrm{la}}
\def\proet{\mathrm{pro\acute{e}t}}
\def\proket{\mathrm{prok\acute{e}t}}
\def\et{\mathrm{\acute{e}t}}

\newcommand{\N}{\mathbb{N}}
\newcommand{\Z}{\mathbb{Z}}
\newcommand{\F}{\mathbb{F}}
\newcommand{\Q}{\mathbb{Q}}
\newcommand{\R}{\mathbb{R}}
\newcommand{\T}{\mathbb{T}}
\newcommand{\A}{\mathbb{A}}
\newcommand{\bC}{\mathbb{C}}

\newcommand{\tr}{\mathrm{tr}}
\newcommand{\GL}{\mathrm{GL}}

\newcommand{\Fl}{{\mathscr{F}\!\ell}}

\newcommand{\cF}{\mathcal{F}}

\newcommand{\cO}{\mathcal{O}}

\newcommand{\kq}{\mathfrak{q}}
\newcommand{\kp}{\mathfrak{p}}

\newcommand{\kh}{\mathfrak{h}}
\newcommand{\km}{\mathfrak{m}}
\newcommand{\kn}{\mathfrak{n}}

\newcommand{\sC}{\mathscr{C}}

\newcommand{\RNum}[1]{\uppercase\expandafter{\romannumeral #1\relax}}

\newcommand{\overbar}[1]{\mkern 1.5mu\overline{\mkern-1.5mu#1\mkern-1.5mu}\mkern 1.5mu}

\begin{document}

\title[On locally analytic vectors of completed cohomology]{On locally analytic vectors of the completed cohomology of modular curves}
\author{Lue Pan}
\address{Department of Mathematics, University of Chicago, 5734 S. University Ave., Chicago, IL 60637}
\email{lpan@math.uchicago.edu}
\begin{abstract} 
We study the locally analytic vectors in the completed cohomology of modular curves and determine the eigenvectors of a rational Borel subalgebra of $\mathfrak{gl}_2(\Q_p)$. As applications, we prove a classicality result for overconvergent eigenforms of weight one and give a new proof of the Fontaine-Mazur conjecture in the irregular case under some mild hypotheses. For an overconvergent eigenform of weight $k$, we show its corresponding Galois representation has Hodge-Tate-Sen weights $0,k-1$ and prove a converse result.
\end{abstract}

\maketitle

\tableofcontents

\section{Introduction}
Let $p$ be a rational prime. In his pioneering work \cite{Eme06}, Emerton introduced completed cohomology, which $p$-adically interpolates automorphic representations and explained how techniques of locally analytic $p$-adic representation theory may be applied to study it. In this paper, we will focus on the simplest (non-abelian) case: completed cohomology of the modular curves. More precisely, let $K$ be an open subgroup $\GL_2(\A_f)$, where $\A_f$ denotes the ring of finite ad\`eles of $\Q$. We have the modular curve of level $K$
\[Y_K(\bC)=\GL_2(\Q)\backslash (\mathbb{H}^{\pm1}\times\GL_2(\A_f)/K),\]
where $\mathbb{H}^{\pm1}=\bC-\R$ is the union of the usual upper and lower half planes. Denote by $\A^p_f$ the prime-to-$p$ part of $\A_f$. For an open subgroup $K^p$ of $\GL_2(\A^p_f)$, the completed cohomology of tame level $K^p$ is defined as
\[\tilde{H}^i(K^p,\Z_p):=\varprojlim_n \varinjlim_{K_p\subset\GL_2(\Q_p)}H^i(Y_{K^pK_p}(\bC),\Z/p^n).\]
It is $p$-adically complete and equipped with a natural continuous action of $\GL_2(\Q_p)$. Our main tool to study this will be ($p$-adic) Hodge theory. Hence we would like to extend the coefficients $\Q_p$ to a complete, algebraically closed field. So let $C=\bC_p$, the completion of an algebraic closure $\overbar\Q_p$ of $\Q_p$, and $\cO_C$ be its ring of integers. Consider
\[\tilde{H}^i(K^p,\cO_C):=\varprojlim_n \varinjlim_{K_p\subset\GL_2(\Q_p)}H^i(Y_{K^pK_p}(\bC),\cO_C/p^n).\]
Then $\tilde{H}^i(K^p,C):=\tilde{H}^i(K^p,\cO_C)\otimes_{\cO_C} C$ is a $\Q_p$-Banach space representation of $\GL_2(\Q_p)$. As in \cite{ST03}, we can consider its locally analytic vectors $\tilde{H}^i(K^p,C)^{\la}$. Let
\begin{itemize}
\item $\mathfrak{g}=C\otimes_{\Q_p}\mathfrak{gl}_2(\Q_p)$: the ``complexified'' Lie algebra of $\GL_2(\Q_p)$;
\item $B\subset\GL_2(\Q_p)$: the upper-triangular Borel subgroup;
\item $\mathfrak{b}\subset\mathfrak{g}$: the ``complexified'' Lie algebra of $B$.
\end{itemize}
The Lie algebra $\mathfrak{g}$ (in particular $\mathfrak{b}$) acts naturally on  $\tilde{H}^i(K^p,C)^{\la}$ by the infinitesimal action of $\GL_2(\Q_p)$. Given a character $\mu:\mathfrak{b}\to C$, one main goal of this paper is to study the $\mu$-isotypic part $\tilde{H}^i(K^p,C)^{\la}_\mu$. We can pass to the direct limit over all tame levels $K^p$:
\[\tilde{H}^i(C)^{\la}_\mu:=\varinjlim_{K^p\subset\GL_2(\A^p_f)} \tilde{H}^i(K^p,C)^{\la}_\mu.\]
Clearly there is a natural action of $\GL_2(\A^p_f)\times B$ on it. 

To state our result, we need more notation. The open modular curve $Y_K(\bC)$ has a natural compactification $X_K(\bC)$ by adding cusps. This complete curve $X_K(\bC)$ has a natural model $X_{K}$ over $\Q$ and we denote by $X_{K,C}$ its base change to $C$. Let $k$ be an integer. We have the usual automorphic line bundle $\omega^k$ on $X_{K,C}$ whose global sections correspond to level-$K$ modular forms of weight $k$ (with coefficients in $C$). Denote by
\[M_k:=\varinjlim_{K\subset\GL_2(\A_f)}H^0(X_{K,C},\omega^k),\]
\[H^1(\omega^k):=\varinjlim_{K\subset\GL_2(\A_f)}H^1(X_{K,C},\omega^k),\]
where both limits are taken over all open compact subgroups of $\GL_2(\A_f)$. Both spaces have natural actions of $\GL_2(\A_f)$ and relate to automorphic representations of $\GL_2(\A)$ (after choosing an isomorphism $C\cong\bC$). Here $\A$ denotes the ring of ad\`eles of $\Q$.

We also need players which are special in this $p$-adic story: overconvergent modular forms. This will be slightly different from the usual one in the literature as we want an action of the Borel $B$. We denote by $\mathcal{X}_K$ the rigid analytic space associated to $X_{K,C}$. Let $\Gamma(p^n)=1+p^nM_2(\Z_p)\subset\GL_2(\Z_p)$ be the principal congruence subgroup of level $n\geq 2$ and let $K^p$ be a tame level. As explained by Katz-Mazur \cite[13.7.6]{KM85}, $X_{K^p\Gamma(p^n),C}$ has a natural integral model over $\cO_C$ and the irreducible components of its special fiber can be indexed by surjective homomorphisms $(\Z/p^n)^{2}\to\Z/p^n$. We denote by $\mathcal{X}_{K^p\Gamma(p^n),c}\subset\mathcal{X}_{K^p\Gamma(p^n)}$ the tubular neighborhood of the non-supersingular points of irreducible components of indices sending $(1,0)\in(\Z/p^n)^2$ to $0$. Let $M^\dagger_k(K^p\Gamma(p^n))$ be the space of sections of $\omega^k$ defined on a strict neighborhood of $\mathcal{X}_{K^p\Gamma(p^n),c}$. Overconvergent modular forms of weight $k$ are defined as:
\[M^\dagger_k:=\varinjlim_{K^p\subset\GL_2(\A^p_f)}\varinjlim_{n} M^\dagger_k(K^p\Gamma(p^n)).\]
One can check that there is a natural action of $\GL_2(\A^p_f)\times B$ on it.

Since modular curves can be defined over $\Q_p$, the absolute Galois group $G_{\Q_p}=\Gal(\overbar\Q_p/\Q_p)$ acts on $C$ and everything we have defined and commutes with the action of $\GL_2(\A^p_f)\times B$. Our main results describe $\tilde{H}^i(C)^{\la}_\mu$ as a representation of $\GL_2(\A^p_f)\times B\times G_{\Q_p}$.

We need the following characters of $\GL_2(\A^p_f)\times B$ for stating our result. Denote by $\varepsilon:\A_f^\times/\Q^\times_{>0}\to\Z_p^\times$ the $p$-adic cyclotomic character (via class field theory which sends a uniformizer $l\neq p$ to a geometric Frobenius). We define $t:\GL_2(\A_f)\to\Z_p^\times$ as $\varepsilon\circ\det$
and view it as a character of $\GL_2(\A^p_f)\times B$ by restriction. For $i=1,2$, we define $e_i:B\to \Q_p^\times$ by sending $\begin{pmatrix} a_1 & b \\ 0 & a_2 \end{pmatrix}\in B$ to $a_i$
and view them as characters of  $\GL_2(\A^p_f)\times B$ by projecting onto $B$. We write $\cdot$ rather than $\otimes$ when twisting by such a character. One main result of this paper is as  follows. (This is obtained from Theorem \ref{hwvg} by taking direct limit over all $K^p$. There is a small difference as we don't put a Galois action on $e_1,e_2,t$ here.)

\begin{thm} \label{thmkneq1}
Let $k\in\Z$ and let $\mu=\mu_k:\mathfrak{b}\to C$ be the character sending $\begin{pmatrix} a & b \\ 0 & d \end{pmatrix}$ to $kd$. Assume $k\neq 1$. There is a natural decomposition 
\[\tilde{H}^1(C)^{\la}_{\mu_k}=N_{k,1}\cdot e_1^{-k}t^{k}\oplus N_{k,w}(k-1)\cdot e_1^{-1}e_2^{k-1}t,\]
where $(k-1)$ denotes a Tate twist by $k-1$. Moreover, we have the following description of $N_{k,1}$ and $N_{k,w}$. 
\begin{enumerate}
\item $N_{k,w}\cong M^{\dagger}_{2-k}$ for $k\neq 2$. When $k=2$, we have  
\[0\to  M^{\dagger}_0/M_0  \to  N_{k,w} \to M_0\to 0.\]
\item If  $k\leq -1$, then there is an exact sequence
\[0\to H^1(\omega^{k})\to N_{k,1}\to M^\dagger_{k}\to 0.\]
\item If $k=0$, then  there is an exact sequence
\[0\to H^1(\omega^0)\to N_{k,1}\to  M^{\dagger}_0/M_0\to 0.\]
\item If $k\geq 2$, then $N_{k,1}\cong M^\dagger_k/ M_k$.
\end{enumerate}
All the maps and isomorphisms here are  $\GL_2(\A^p_f)\times B\times G_{\Q_p}$-equivariant.
\end{thm}

One may view this as a Hodge-Tate decomposition into weight $0$ and $1-k$ components. Hence this is a $p$-adic analogue of the Shimura isomorphism. Our convention is that the cyclotomic character is of Hodge-Tate weight $-1$. We have a similar (Hodge-Tate-Sen) decomposition if $k\notin\Z$ (Theorem \ref{hwvg}). Unfortunately, we can only determine half of the decomposition, which roughly is given by overconvergent modular forms of weight $2-k$. We will give a definition of this space for non-integral weights in \ref{ovc}.

When $k=1$, we don't get a decomposition because of the existence of non-Hodge-Tate representations with equal Hodge-Tate-Sen weights. Here is our result (which follows fromTheorem \ref{weight1} by taking the direct limit over all $K^p$).

\begin{thm} \label{w1i}
Suppose $\mu=\mu_1:\mathfrak{b}\to C$ sends $\begin{pmatrix} a & b \\ 0 & d \end{pmatrix}$ to $d$.  There is a  natural exact sequence 
\[0\to N_{1}\cdot e_1^{-1}t \to \tilde{H}^1(C)^{\la}_{\mu_1} \to M^{\dagger}_{1}\cdot e_1^{-1}t\to 0,\]
where $N_{1}$ sits inside an exact sequence:
\[0\to  M^{\dagger}_1/M_1\to N_{1}\to H^1(\omega)\to 0.\]
Moreover, let $\tilde{H}^1(C)^{\la,(1,0)}_{\mu_1}\subset  \tilde{H}^1(C)^{\la}_{\mu_1}$ be the pull-back of $M_{1}\cdot e_1^{-1}t\subset M^{\dagger}_{1}\cdot e_1^{-1}t$, so that there is an exact sequence 
\[0\to N_{1}\cdot e_1^{-1}t \to \tilde{H}^1(C)^{\la,(1,0)}_{\mu_1} \to M_{1}\cdot e_1^{-1}t\to 0.\]
Then $(\tilde{H}^1(C)^{\la}_{\mu_1})^{G_{\Q_p}}\subset\tilde{H}^1(C)^{\la,(1,0)}_{\mu_1}$. Again, all the maps here are  $\GL_2(\A^p_f)\times B\times G_{\Q_p}$-equivariant.
\end{thm}

In fact, we will generalize the usual Sen operator to this situation in Definition \ref{Senop}, then show it is nilpotent on $\tilde{H}^1(C)^{\la}_\mu$ and factors through the quotient $M^{\dagger}_{1}\cdot e_1^{-1}t$ and maps it to $M^{\dagger}_1/M_1\cdot e_1^{-1}t\subset N_{1}\cdot e_1^{-1}t$.

Both theorems give a concrete relation between completed cohomology and overconvergent modular forms. As a first application, we are able to prove a classicality result for weight one overconvergent eigenforms. Fix a tame level $K^p=K^SK_{S\setminus\{p\}}$ for some finite set of primes $S$ containing $p$ and suppose $K^S$ is a maximal open compact subgroup of $\GL_2(\A_f^S)$, where $\A_f^S$ is the prime-to-$S$ part of $\A_f$. Denote by $\T_{K^p}:=\Z_p[\GL_2(\A^S_f)//K^S]$ the abstract Hecke algebra of $K^S$-biinvariant compactly supported functions on $\GL_2(\A^S_f)$, where the Haar measure gives $K^S$ measure $1$. Then $\T_{K^p}$ acts on $H^1(Y_{K^pK_p}(\bC),\Z_p)$ and we denote its image in $\End_{\Z_p}(H^1(Y_{K^pK_p}(\bC),\Z_p))$ by $\T(K^pK_p)$ and let $\T:=\varprojlim_{K_p}\T(K^pK_p)$. It can be shown that $\T$ is independent of $S$. Let $\lambda:\T\to\overbar\Q_p$ be a $\Z_p$-algebra homomorphism.  We can associate to $\lambda$ a two-dimensional  semi-simple Galois representation 
\[\rho_\lambda:\Gal(\overbar\Q/\Q)\to \GL_2(\overbar\Q_p)\]
satisfying $\tr(\rho_\lambda(\Frob_l))=\lambda(T_l)l^{-1}$ for $l\notin S$. Here $T_l=[K^S\begin{pmatrix} l & 0 \\ 0 & 1 \end{pmatrix}K^S]$ and $\Frob_l$ is a lift of geometric Frobenius at $l$. See \ref{HA} for more details.

Let $M^\dagger_1(K^p)=\varinjlim_{n}M^\dagger_1(K^p\Gamma(p^n))$ and $N_0=\begin{pmatrix} 1 & \Z_p \\ 0 & 1 \end{pmatrix}\subset\GL_2(\Q_p)$. We have the usual $U_p$-operator acting on $N_0$-invariants by 
\[U_p:=\sum_{i=0}^{p-1}\begin{pmatrix} p & i \\ 0 & 1 \end{pmatrix}.\]
Since $\T$ acts on $\tilde H^1(K^p,\Z_p)$, it follows from Theorem \ref{w1i} that $\T$ acts on $M^\dagger_1(K^p)$. We denote by $\kp_\lambda$ the kernel of $\lambda$ and by  $M^\dagger_1(K^p)[\kp_\lambda]$ the $\lambda$-isotypic part of $M^\dagger_1(K^p)$.

\begin{thm}[Theorem \ref{cl1}]
Suppose 
\begin{itemize}
\item $M^\dagger_1(K^p)[\kp_\lambda]^{N_0}\neq0$ has a non-zero $U_p$-eigenvector;
\item $\rho_\lambda|_{G_{\Q_p}}$ is Hodge-Tate of weights $0,0$.
\end{itemize}
Then $\lambda$ comes from a classical weight one eigenform, i.e. $M_1(K^p)[\kp_\lambda]\neq 0$.
\end{thm}

In the ordinary case, this reproves a result of Buzzard-Taylor \cite{BT99}. Note that we don't assume the $U_p$-eigenvalue is non-zero. This will allow us to attack the Fontaine-Mazur conjecture in the non-ordinary, irregular case.

\begin{proof}
We give a sketch of the  proof here. We may assume $\rho_\lambda$ is irreducible. Let $\mu_1$ be as in Theorem \ref{w1i}  and denote by $\lambda\cdot t:\T\to\overbar\Q_p$ the map which sends $T_l$ to $\lambda(T_l)l^{-1}$. The key point is that by the Eichler-Shimura relation,  $\tilde H^1(K^p,C)^{\la}_{\mu_1}[\kp_{\lambda\cdot t}]=\rho_\lambda\otimes_{\overbar\Q_p} W$ for some $W$. Since $\rho_{\lambda}$ is Hodge-Tate of weights $0,0$, we have $\tilde H^1(K^p,C)^{\la}_{\mu_1}[\kp_{\lambda\cdot t}]=\tilde H^1(C)^{\la,(1,0)}_{\mu_1}[\kp_{\lambda\cdot t}]$. Now assume $\lambda$ is not classical. Then it follows from the second part of Theorem \ref{w1i} that 
\[(M^\dagger_1(K^p)[\kp_\lambda]\cdot e_1^{-1}t)^{N_0}=\tilde H^1(C)^{\la,(1,0)}_{\mu_1}[\kp_{\lambda\cdot t}]^{N_0}=\tilde H^1(K^p,C)^{\la}_{\mu_1}[\kp_{\lambda\cdot t}]^{N_0}=\rho_\lambda\otimes W^{N_0}.\]
By choosing suitable $K^p$ and taking suitable eigenvectors of Hecke operators $U_l$ and diamond operators at ramified places, the non-zero $U_p$-eigenspace of the first term will be one-dimensional by the $q$-expansion principle. However the corresponding subspace of the last term will always have dimension at least $\dim \rho_\lambda=2$ if it is non-zero! Contradiction. This shows that $\lambda$ has to be classical.
\end{proof}

Thanks to Emerton's work on local-global compatibility, we have a good understanding of $\tilde{H}^1(K^p,C)[\kp_\lambda]$ in terms of the $p$-adic local Langlands correspondence. Colmez's Kirillov model then shows that  a $U_p$-eigenvector always exists (See Corollary \ref{U_pev} below). 

\begin{thm}[Theorem \ref{classicality1}] \label{ovclassicality}
Suppose
\begin{itemize}
\item $\rho_\lambda$ is irreducible;
\item $\rho_\lambda|_{G_{\Q_p}}$ is Hodge-Tate of weight $0,0$. 
\end{itemize}
Then $\lambda$ is classical, i.e. $M_1(K^p)[\kp_\lambda]\neq 0$ for some tame level $K^p$.
\end{thm}

We say a Galois representation $\rho$ is pro-modular if $\rho=\rho_\lambda$ for some $K^p$ and $\lambda$. Combining Theorem \ref{ovclassicality} with  work on promodularity of a Galois representation (\cite[Theorem 1.2.3]{Eme1}, which is based on previous work of B\"ockle, Diamond-Flach-Guo, Khare-Wintenberger and Kisin), we give a new proof of the Fontaine-Mazur conjecture in the irregular case under some mild hypotheses. 

\begin{thm}[Theorem \ref{promodularity} and Corollary \ref{modwt1}]
Let $p>2$ be a prime number and $\rho:\Gal(\overbar\Q/\Q)\to\GL_2(\overbar\Q_p)$  a continuous irreducible two-dimensional representation. Assume 
\begin{enumerate}
\item $\rho$ is unramified outside of finitely many primes;
\item $\det\rho(c)=-1$ for a complex conjugation $c\in \Gal(\overbar\Q/\Q)$;
\item $\rho|_{G_{\Q_p}}$ is Hodge-Tate of weights $0,0$;
\item the residual representation $\bar\rho|_{\Gal(\overbar\Q/\Q(\mu_p))}$ is irreducible;
\item $(\bar\rho|_{G_{\Q_p}})^{ss}$ is either irreducible or of the form $\eta_1\oplus \eta_2$ for some characters $\eta_1,\eta_2$ satisfying $\eta_1/\eta_2\neq 1,\omega^{\pm1}$.
\end{enumerate}
Then $\rho$ is modular in the sense that it is isomorphic to the representation attached to a classical weight one cuspidal eigenform by \cite{DS74}. In particular, $\rho$ has finite image.
\end{thm}

\begin{rem}
This result was already known by the work of Pilloni-Stroh \cite{PS16}. Before their work, much work on this problem was done in the ordinary case: Buzzard-Taylor \cite{BT99}, Calegari-Geraghty \cite{CG18}... and has been generalized to the case of Hilbert modular forms. I hope this paper gives a new perspective on this problem.

If the results of this paper can also be generalized to Hilbert case,  one should be able to remove the last two conditions by invoking our previous work on promodularity \cite{Pan19}. 
\end{rem}

For overconvergent eigenforms of weight not necessarily $1$, we have the following result.
\begin{thm}[Theorem \ref{ocmfwt}] \label{ocmwn1}
Suppose $\rho_\lambda$ is irreducible and $M^\dagger_k(K^p)[\kp_\lambda]\neq 0$, then $\rho_\lambda|_{G_{\Q_p}}$ has Hodge-Tate-Sen weights $0,k-1$. Conversely, suppose
\begin{itemize}
\item $\rho_\lambda|_{G_{\Q_p}}$ is irreducible and has Hodge-Tate-Sen weights $0,k-1$;
\end{itemize}
Then $M^\dagger_k(K^p)[\kp_\lambda]\neq 0$.
\end{thm}

\begin{rem}
As pointed out by a referee, Gouv\^{e}a in \cite[Conjecture 4]{Gou88} conjectured that the Hodge-Tate weights of the Galois representation associated to an overconvergent eigenform of weight $k$ are $0,k-1$ and guessed that this should be an equivalence. Our result confirms his conjecture and his guess under the given hypotheses.

We note that Sean Howe also proved Gouv\^{e}a's conjecture in \cite{Ho21} independently. Recently, we found a more direct way to  prove Gouv\^{e}a's conjecture  in \cite{Pan20} using Scholze's fake-Hasse invariants. But the method in  \cite{Pan20} does not seem to be enough to give a converse result. Indeed,  one key ingredient in our proof of Theorem \ref{ocmwn1} here is Colmez's Kirillov model \cite[Chap. VI]{Col10}, which is a deep result in $p$-adic local Langlands correspondence for $\GL_2(\Q_p)$. 
\end{rem}

Now we explain our strategy to compute $\tilde{H}^1(C)^{\la}_\mu$. As mentioned before, our main tool is $p$-adic Hodge theory. Fix a tame level $K^p$. By the work of Scholze \cite{Sch15}, there is an adic space $\mathcal{X}_{K^p}\sim\varprojlim_{K_p}\mathcal{X}_{K^pK_p}$ (the modular curve at infinite level) with Hodge-Tate period map  $\pi_{\HT}:\mathcal{X}_{K^p}\to\Fl=\mathbb{P}^1$, where $\Fl$ denotes the adic space associated to the flag variety of $\GL_2/C$. Let $\cO_{K^p}$ be the push-forward of the structure sheaf of $\mathcal{X}_{K^p}$ along $\pi_{\HT}$. Then Scholze shows that there is a natural isomorphism
\[\tilde{H}^i(K^p,C)\cong H^i(\Fl,\cO_{K^p}),\]
where the right hand side is computed on the analytic site of $\Fl$.
In order to study the ($\GL_2(\Q_p)$-)locally analytic vectors, we consider the subsheaf of locally analytic sections $\cO_{K^p}^{\la}\subset\cO_{K^p}$ (\ref{OKpla}). 

\begin{thm}[Theorem \ref{comcccc}] \label{comccccI}
There is a natural isomorphism
\[\tilde{H}^i(K^p,C)^{\la}\cong H^i(\Fl,\cO^{\la}_{K^p}).\]
\end{thm}

Next we need a careful local study of the sheaf $\cO^{\la}_{K^p}$. This part is largely inspired by the work of Berger-Colmez \cite{BC16}. There are two key ingredients:
\begin{enumerate}
\item $\cO^{\la}_{K^p}$ satisfies a differential equation on $\Fl$;
\item an explicit description of sections of $\cO^{\la}_{K^p}$ (Theorem \ref{cnijk}).
\end{enumerate} 
We will only explain the differential equation in detail here. It turns out to be nothing but the horizontal nilpotent subalgebra. To be more precise, we follow some constructions on the flag variety from  \cite{BB83}. Recall $\mathfrak{g}=\mathfrak{gl}_2(C)$. For a $C$-point $x$ of the flag variety $\mathrm{Fl}$ of $\GL_2/C$, let $\mathfrak{b}_x,\mathfrak{n}_x\subset \mathfrak{g}$ be its corresponding Borel subalgebra and nilpotent subalgebra. Let
\begin{eqnarray*}
\mathfrak{g}^0&:=&\cO_{\mathrm{Fl}}\otimes_{C}\mathfrak{g},\\
\mathfrak{b}^0&:=&\{f\in \mathfrak{g}^0\,| \, f_x\in \mathfrak{b}_x,\mbox{ for all }x\in\mathrm{Fl}(C)\},\\
\mathfrak{n}^0&:=&\{f\in \mathfrak{g}^0\,| \, f_x\in \mathfrak{n}_x,\mbox{ for all }x\in\mathrm{Fl}(C).\}
\end{eqnarray*}
By abuse of notation, we also view these as vector bundles on $\Fl$. Then $\mathfrak{g}^0$ acts on $\cO^{\la}_{K^p}$ in an obvious way.
\begin{thm}[Theorem \ref{n0triv}] \label{n0trivI}
$\mathfrak{n}^0$ acts trivially on $\cO^{\la}_{K^p}$.
\end{thm}

In fact, we will show the existence of such a differential equation in a very general setting.

\begin{thm}[Theorem \ref{pCR}] \label{pCRI}
Suppose
\begin{itemize}
\item $X=\Spa(A,A^+)$: a one-dimensional smooth affinoid adic space (over $\Spa(C,\cO_C)$);
\item $G$: a finite-dimensional compact $p$-adic Lie group;
\item $\Spa(B,B^+)$: a Galois pro-\'etale perfectoid covering  of $X$ with Galois group $G$.
\end{itemize}
Then under some smallness assumption on $X$ (cf. \ref{setup}), there is an element of $B\otimes_{\Q_p}\Lie(G)$ that annihilates the $G$-locally analytic vectors in $B$. It is non-zero if $\Spa(B,B^+)$ is a locally analytic covering in the sense of \ref{laedefn} below.
\end{thm}

\begin{rem}
This can be viewed as a $p$-adic analogue of the Cauchy-Riemann equation in classical complex analysis. See also Remarks \ref{theSen}, \ref{FAGTLZ} below for the relation with classical and relative Sen theory (also called the $p$-adic Simpson correspondence in the literature). We can also allow some ramification, cf. Theorem \ref{pCR}. 
\end{rem}

\begin{exa}
One simple example is 
\begin{itemize}
\item $X=\Spa(C\langle T^{\pm 1}\rangle, \cO_C\langle T^{\pm1}\rangle)$, the one-dimensional torus;
\item $\widetilde X=\Spa(C\langle T^{\pm \frac{1}{p^{\infty}}}\rangle, \cO_C\langle T^{\pm\frac{1}{p^{\infty}}}\rangle)$, the perfectoid inverse limit of $\Spa(C\langle T^{\pm \frac{1}{p^{m}}}\rangle, \cO_C\langle T^{\pm\frac{1}{p^{m}}}\rangle)$;
\item $G\cong \Z_p$.
\end{itemize}
In this case, the $G$-locally analytic vectors in $C\langle T^{\pm \frac{1}{p^{\infty}}}\rangle$ are the smooth vectors $\bigcup_{m\in\N} C\langle T^{\pm \frac{1}{p^{m}}}\rangle$. Our differential operator is simply a generator of $\Lie(G)$.
\end{exa}

Back to the flag variety. Since $\mathfrak{n}^0$ acts trivially on $\cO^{\la}_{K^p}$, there is an induced action of $\mathfrak{b}^0/\mathfrak{n}^0\cong \kh\otimes\cO_{\Fl}$, where $\kh$ is a Cartan subalgebra of $\mathfrak{g}$. We choose it as the Levi quotient of $\mathfrak{b}$, i.e. the diagonal matrices. Hence we have  an (horizontal) action $\theta_{\kh}$ of $\kh$ on $\cO^{\la}_{K^p}$. By Harish-Chandra's theory, this action is closely related to the action of the centre $Z(U(\mathfrak{g}))$ (\ref{khact}). In the study of cohomology of locally symmetric spaces, the infinitesimal group action is related to Hodge theory. Our next result gives a $p$-adic Hodge-theoretic interpretation of $\theta_\kh$.

\begin{thm}[Theorem \ref{Senkh}] \label{SenkhI}
$\theta_\kh(\begin{pmatrix}0 & 0\\ 0 & 1 \end{pmatrix})$ is the unique Sen operator on the cohomology $H^i(\Fl,\cO_{K^p}^{\la})\cong\tilde{H}^i(K^p,C)^{\la}$ for any $i$.
\end{thm}

\begin{rem}
Since the usual Sen operator is defined for a finite-dimensional $C$-semilinear representation of $G_{\Q_p}$, we will generalize this notion to this situation \ref{Senop}. For example, for any finite-dimensional $C$-semilinear subrepresentation $V$ of $G_{\Q_p}$ in $\tilde{H}^i(K^p,C)^{\la}$, this result implies that $\theta_\kh(\begin{pmatrix}0 & 0\\ 0 & 1 \end{pmatrix})$ agrees with usual Sen operator when restricted to $V$. We also sketch a more direct construction of the Sen operator on $\tilde{H}^i(K^p,C)^{\la}$ in Remark \ref{dconSen}.
\end{rem}

\begin{rem}
In Beilinson-Bernstein's theory of localization (see \S C of \cite{Bei84}), $\cO^{\la}_{K^p}$ is a $\tilde{\mathscr{D}}$-module. Note that everything here is also $\GL_2(\Q_p)$-equivariant. Hence we obtain a $(\tilde{\mathscr{D}},\GL_2(\Q_p))$-module on $\Fl$, which is very similar to the picture in the complex analytic setting (for example the work of Kashiwara-Schmid \cite{KS94}).
\end{rem}

As a corollary, this implies that on the locally analytic vectors in the completed cohomology, the infinitesimal character of $\GL_2(\Q_p)$ (i.e. action of $Z(U(\mathfrak{g}))$) and the infinitesimal character of $G_{\Q_p}$ (Sen operator as an element in the centre of``$C\otimes_{\Q_p}\Lie(G_{\Q_p})$'') are closely related. This directly implies the (Hodge-Tate) decomposition in Theorem \ref{thmkneq1}. 

\begin{rem}
The same strategy of this paper should give Theorem \ref{pCRI} for higher dimensional rigid analytic varieties, and Theorem  \ref{n0trivI} for general Hodge-Tate period maps, and Theorem \ref{SenkhI} for Shimura varieties of Hodge type. I plan to report about this in a future paper. We also note that the relation between the action of the centre of the universal enveloping algebra and Sen operator on the locally analytic vectors in the completed cohomology (not necessarily of the modular curves) was already vastly studied by Dospinescu-Pa\v{s}k\={u}nas-Schraen in \cite{DPS20}.
\end{rem}

After writing up this paper, I was informed by Sean Howe that very recently he also obtained some part of Theorem \ref{hwvg} and the easier direction of \ref{ocmwn1} in \cite{Ho21}. 

This paper is organized as follows. We first collect some basic facts about locally analytic vectors in \S\ref{LAV}. One important notion is a derived functor $R^i\mathfrak{LA}$ of taking locally analytic vectors introduced in \ref{LAacyclic}.  

In \S\ref{LARS}, we explain how to derive the differential equation claimed in Theorem \ref{pCRI} from relative Sen theory. Along the way, we also prove a vanishing result for $R^i\mathfrak{LA}$ which later will be used to prove Theorem \ref{comccccI}. Nothing is new in the sense that all of the techniques were probably well-known in Sen's theory.

In \S\ref{lafopmc}, we apply our theory developed in the previous section to modular curves at infinite level. The computation of the differential equation in this case is a simple reinterpretation of an old result of Faltings. We also give an explicit description of the sections of $\cO^{\la}_{K^p}$ and use it to prove Theorem \ref{comccccI}.

In \S\ref{muipocc}, we first give a proof of Theorem \ref{SenkhI}. Using our explicit description of $\cO^{\la}_{K^p}$, it is easy to verify that $\theta_\kh(\begin{pmatrix}0 & 0\\ 0 & 1 \end{pmatrix})$ behaves as a Sen operator on $\cO^{\la}_{K^p}$ and its cohomology. Next we compute $\tilde{H}^1(K^p,C)^{\la}_{\mu}$. Fix a character $\chi\in\kh^*$ and denote by $\cO^{\la,\chi}_{K^p}$  the $\theta_\kh=\chi$-isotypic part. Let $\kn$ be the nilpotent subalgebra of $\mathfrak{b}$. Basically we are reduced to computing the $\kn$-cohomology of the sheaf $\cO^{\la,\chi}_{K^p}$ and the cohomology of the $\kn$-invariants and $\kn$-coinvariants of $\cO^{\la,\chi}_{K^p}$ on $\Fl$ (at least for non-singular weight).

In \S\ref{app}, we give the global applications mentioned earlier. To do this, we need Emerton's work on local-global compatibility and Colmez's work on Kirillov models. It should be mentioned that we slightly generalize Emerton's local-global compatibility result so that it is now valid when localized at a Eisenstein maximal ideal.

\subsection*{Acknowledgement} 
I would like to thank Matthew Emerton for useful discussions and suggesting the possibility of applying Colmez's Kirillov model to the problem of modularity of weight one forms. I would like to thank Richard Taylor for 
a key conversation on possible structure in Theorem \ref{w1i}.  I would also like to thank Tsao-Hsien Chen, Pierre Colmez, Weibo Fu, Kai-Wen Lan, Ruochuan Liu, Jun Su, Xinwen Zhu for answering my questions and helpful discussions. Further, I would like to thank Michael Harris and Kai-Wen Lan for several comments and corrections on an earlier version of this paper.  Finally, I am very grateful to the anonymous referees for their detailed comments on this paper.

\subsection*{Notation} 
Fix an algebraic closure $\overbar\Q$ of $\Q$. Denote by $G_{\Q}$ the absolute Galois group $\Gal(\overbar\Q/\Q)$. For each rational prime $l$, fix an algebraic closure $\overbar\Q_l$ of $\Q_l$ with ring of integers $\overbar\Z_l$, an embedding $\overbar\Q\to\overbar\Q_l$ which determines a decomposition group $G_{\Q_l}\subset G_{\Q}$ at $l$, and a lift of geometric Frobenius $\Frob_l\in G_{\Q_l}$. Our convention for local class field theory sends a uniformizer to a lift of geometric Frobenius. We denote by $\varepsilon$ the $p$-adic cyclotomic character $G_{\Q}\to\Z_p^\times$ and by abuse of notation, also the corresponding character $\A_f^\times/\Q_{>0}^{\times}\to\Z_p^\times$  via global class field theory. Similarly, we denote by $\varepsilon_p$ the $p$-adic cyclotomic character $G_{\Q_p}\to\Z_p^\times$ and also the character  $\Q_p^\times\to\Z_p^\times$ sending $x$ to $x|x|$. Suppose $S$ is a finite set of rational primes. We denote by $G_{\Q,S}$ the the Galois group of the maximal extension of $\Q$ unramified outside $S$ and $\infty$. 

All completed tensor products mean ``$p$-adically completed tensor product''. For example, if $V$ and $W$ are two $\Q_p$-Banach space, then $V\widehat{\otimes}_{\Q_p}W=(\varprojlim_n (V^o\otimes_{\Z_p}W^o/p^n))\otimes_{\Z_p}\Q_p$, where $V^o$ and $W^o$ denote the open unit balls of $V$ and $W$. 

If $\rho:\Gamma\to\GL_n(\overbar\Q_p)$ denotes a continuous representation of a profinite group $\Gamma$, we denote by $\bar\rho^{ss}:\Gamma\to\GL_2(\overbar\F_p)$ the semi-simplification of its residual representation, which is well-defined up to conjugation. If $\bar\rho^{ss}$ is irreducible, we simply write $\bar\rho$. Here $\overbar\F_p$ denotes the residue field of $\overbar\Z_p$.

For an adic space $X$, an open set $V\subset X$ and a sheaf $\mathcal{F}$ on $X$, we sometimes write $\mathcal{F}(V)$ instead of $H^0(\mathcal{V},\mathcal{F})$. For example, $\cO_{X}(V)$ denotes analytic functions on $V$.

For a tuple $\mathbf{k}=(k_1,\cdots,k_d)\in \N^d$, we denote the sum $\sum_{i=1}^d k_i$ by $|\mathbf{k}|$.

Suppose $G$ is a group and $f$ is a function on $G$ (valued in some group). Let $g\in G$. The left translation action of $g$ on $f$ is denoted by ${}^{l}g\cdot f$, which sends $x\in G$ to $f(g^{-1}x)$. Similarly, the right translation action of $g$ on $f$ is denoted by ${}^{r}g\cdot f$ and sends $x\in G$ to $f(xg)$. Note that both actions are left actions.

\section{Locally analytic vectors} \label{LAV}
In this section, we collect some basic results in the theory of locally analytic representations studied by Schneider and Teitelbaum \cite{ST02,ST03}. We will first recall the (naive) definition of locally analytic functions on a compact $p$-adic Lie group, following \S2.1 of \cite{BC16}. Then we will introduce a derived functor of  taking locally analytic vectors.

\subsection{Definition and basic properties}

\begin{para} \label{2ndcoord}
Let $G$ be a  $p$-adic Lie group of dimension $d$. By Theorem 27.1 of \cite{Sch11}, there exists a compact open subgroup $G_0$ of $G$ equipped with an integer valued, saturated $p$-valuation. As a consequence, there exist $g_1,\cdots,g_d\in G_0$ (an ordered basis) such that the map $c:\Z_p^d\to G_0$ defined by $(x_1,\cdots,x_d)\mapsto g_1^{x_1}\cdots g_d^{x_d}$ is a homeomorphism. Let $G_n=G_0^{p^n}=\{g^{p^n},g\in G_0\}$ for $n\geq 1$. This is in fact a subgroup and $c^{-1}(G_n)=(p^n\Z_p)^d$. So $\{G_n\}_n$ forms a fundamental system of open neighborhoods of the identity element $1$ in $G$. We will fix such a $c$ in the following discussion, but all of the definitions we make below won't depend on the choice of $c$. See \S23, \S26 of \cite{Sch11} for more details.

If $B$ is a $\Q_p$-Banach space with norm $\|\cdot\|$, we denote by $\sC(G_n,B)$ the space of $B$-valued continuous functions on $G_n$ and by $\sC^{\an}(G_n,B)$ the subspace of $B$-valued analytic functions. More precisely, $f:G_n\to B$ is analytic if its pull-back under $c$ can be written as 
\[\mathbf{x}=(x_1,\cdots,x_d)\in (p^n\Z_p)^d\mapsto \sum_{\mathbf{k}=(k_1,\cdots,k_d)\in \N^d} b_{\mathbf{k}} x_1^{k_1}\cdots x_d^{k_d}\]
for some $b_{\mathbf{k}}\in B$ such that $p^{n|\mathbf{k}|}b_{\mathbf{k}}\to 0$ as $|\mathbf{k}|\to\infty$. We define $\|f\|_{G_n}=\sup_{\mathbf{k}\in\N^d}\|p^{n|\mathbf{k}|}b_{\mathbf{k}}\|$ and $\sC^{\an}(G_n,B)$ is a $\Q_p$-Banach space under this norm. There is a natural norm-preserving isomorphism $\sC^{\an}(G_n,\Q_p) \widehat\otimes_{\Q_p} B\cong \sC^{\an}(G_n,B)$ induced by the $\Q_p$-linear structure of $B$. When $B=\Q_p$, the usual point-wise multiplication of two functions defines a natural $\Q_p$-algebra structure on $\sC^{\an}(G_n,\Q_p)$

To see this definition is independent of the choice of $g_1,\cdots,g_d$, we introduce another set of coordinates of $G_n$ given by the the exponential map. These coordinates are called coordinates of the first kind in \S 34 of \cite{Sch11}. More precisely, let $\Lie(G)$ be the Lie algebra of $G$ and consider the logarithm map
\[\log:G_0\to\Lie(G),\]
which is a homeomorphism onto its image. In fact, its image $\log(G_0)$ is a free $\Z_p$-module with a basis $\log(g_1),\cdots,\log(g_d)$. Hence 
\[e:\Z_p^{d}\to G_0,\, (y_1,\cdots,y_d)\mapsto \exp(y_1\log(g_1)+\cdots+y_d\log(g_d))\]
defines a homeomorphism. Here $\exp$  is the inverse of $\log$ (exponential map). Under this parametrization, $G_n$ is identified with the image of $(p^n\Z_p)^d$. As before, a function on $G_n$ is called analytic if its pull-back under $e$ can be written as a power series in $y_i$ with certain growth conditions on its coefficients. This definition of analyticity agrees with the previous definition as the transition functions between the coordinates $(x_1,\cdots,x_d)$ and $(y_1,\cdots,y_d)$ are given by convergent power series with coefficients in $\Z_p$, cf. Proposition 34.1 of \cite{Sch11}.

Different choices of ordered bases $g_1,\cdots,g_d$ give rise to different bases of $\log(G_0)$ over $\Z_p$ (as a $\Z_p$-module). It is easy to see now that $\sC^{\an}(G_n,B)$ and $\|\cdot\|_{G_n}$ are independent of choice of an ordered basis. As a corollary, $\sC^{\an}(G_n,B)$ and $\|\cdot\|_{G_n}$ are invariant under both the left and right translation actions of $G_0$. We sketch a proof for the left translation by $G_0$ here. The case for right translation and general $n$ will be exactly the same. Let $x$ be a non-trivial element of $G_0$. It suffices to prove the claim under the assumption that $x$ is not a $p$-th power. Choose an ordered basis $g_1,\cdots,g_d$ with $g_1=x$. This is possible by our assumption on $G_0$, see \S26 of \cite{Sch11}. To compute the left translation action of $x$, we are essentially reduced to the case $G_0=\Z_p$, which can be verified directly. In fact, this argument shows a bit more:
\end{para}

\begin{lem} \label{modpn+1trivial}
For any integer $n\geq0$, let $\sC^{\an}(G_n,\Q_p)^o\subset \sC^{\an}(G_n,\Q_p)$ be the unit ball with respect to the norm $\|\cdot\|_{G_n}$. Then both left and right translation actions of $G_{n+1}$ are trivial on $\sC^{\an}(G_n,\Q_p)^o/p$. 
\end{lem}

\begin{proof}
One can reduce to the trivial case $G_n=\Z_p$ by the same argument.
\end{proof}

The following density result will be used later. 

\begin{prop} \label{algdense}
For $n$ large enough, there is a dense subspace $\varinjlim_{k\in\N} V_k\subset\sC^{\an}(G_n,\Q_p)$, where each $V_k$ is a finite-dimensional subspace of $\sC^{\an}(G_n,\Q_p)$ stable under both the left and right translation actions of $G_n$ and such that for any $k,l\in\N, f_k\in V_k, f_l\in V_l$, we have $f_k f_l\in V_{k+l}$.

\end{prop}

\begin{proof}
This is essentially proved in the proof of Theorem 6.1 of \cite{BC16}. We recall their argument here. The rough idea is to reduce to the case $G=\GL_N(\Z_p)$, in which case the algebraic functions on $\GL_N$ are dense in the space of analytic functions.

For $n$ large enough, we may find an embedding $G_n$ into $1+p^2 M_N(\Z_p)\subset\GL_N(\Z_p)$ for some $N$, where $M_N(\Z_p)$ denotes the space of $N\times N$ matrices over $\Z_p$. We sketch a proof of this well-known result here because it seems to lack a suitable reference. By Ado's theorem (\S7,  \textnumero 3, Theorem 3 of \cite{Bou60}), there is a faithful representation of $\Lie(G)$ over $\Q_p^{\oplus N}$ for some $N$. Choose $n$ large enough so that $\log(G_n)$ maps to $p^2 M_N(\Z_p)$.  Then this Lie algebra morphism can be integrated to a group homomorphism $G_n\to 1+p^2 M_N(\Z_p)$, which is necessarily a closed embedding. This can be deduced from results in \S31, notably Theorem 31.5 of \cite{Sch11}.

Fix such an embedding. $G_n$ is now viewed as a subset of $M_N(\Q_p)$. For $k\in N$, let $V_k$ be the space of $\Q_p$-valued functions on $G_n$ that are restrictions of polynomials of degree $\leq k$ on $M_N(\Q_p)$. Here we identify $M_N(\Q_p)$ with an affine space of dimension $N^2$ (over $\Q_p$). We claim that the union of all $V_k$ is dense in $\sC^{\an}(G_n,\Q_p)$. To see this, note that the space of polynomial functions on $M_N(\Q_p)$ is dense in $\sC^{\an}(1+p^2M_N(\Z_p),\Q_p)$ when restricted to $1+p^2M_N(\Z_p)$. Hence its pull-back under the exponential map is dense in $\sC^{\an}(p^2M_N(\Z_p),\Q_p)$. As the image of $\sC^{\an}(p^2M_N(\Z_p),\Q_p)\to\sC^{\an}(\log(G_n),\Q_p)$ is dense, this implies our claim.

It's clear from the construction that all $V_k$ are stable under left and right translation actions and closed under multiplications.
\end{proof}

\begin{para} \label{Gnan}
Now if $B$ is a $\Q_p$-Banach space equipped with a continuous action of $G_n$, then we denote by $B^{G_n-\an}\subset B$ the subset of $G_n$-analytic vectors. One can define this as follows: there is a left action on $\sC^{\an}(G_n,B)$ by $(g\cdot f)(h)=g(f(g^{-1}h)),g,h\in G_n,f\in\sC^{\an}(G_n,B)$. We define $B^{G_n-\an}$ as the image of $\sC^{\an}(G_n,B)^{G_n}\to B$ under the evaluation map at the identity element in $G_n$.  Equivalently, $v\in B^{G_n-\an}$ if $f_v:G_n\to B,~f_v(g)=g\cdot v$ lies in $\sC^{\an}(G_n,B)$. It is clear that $\|\cdot\|_{G_n}$ induces a norm on $B^{G_n-\an}$, which we denote by the same notation. Now $G_n$ acts on $B^{G_n-\an}=\sC^{\an}(G_n,B)^{G_n}$ by right translation. One checks easily that this action is unitary and equivariant with respect to the inclusion $B^{G_n-\an}\subset B$.

We say $B$ is an analytic representation of $G_n$ if $B=B^{G_n-\an}$.

One has the following lemma (cf. Lemme 2.4. of \cite{BC16}):
\end{para}

\begin{lem} \label{Gnnorm}
Let $B$ be a Banach representation of $G_n$. If $b\in B^{G_n-\an}$, then
\begin{enumerate}
\item $b\in B^{G_{n+1}-\an}$,
\item $\|b\|_{G_{n+1}}\leq\|b\|_{G_n}$.
\item $\|b\|_{G_m}$ agrees with the norm of $b$ in $B$ for $m$ sufficiently large.
\end{enumerate}
\end{lem}

\begin{defn}
Suppose $B$ is a Banach representation of $G_0$. The space of locally analytic vectors in $B$ is $B^{\la}=\bigcup_{n\geq 0} B^{G_n-\an}$ with the direct limit topology.  
\end{defn}

\begin{rem}
If $B$ is a representation of $G$, then $B^{\la}$ can be defined by restricting to the action of $G_0$. It is easy to see that $B^{\la}$ is independent of the choice of $G_0$ and $B^{\la}$ is $G$-stable.
\end{rem}



The Lie algebra $\Lie(G)$ acts on the locally analytic vectors. One can check this in the universal case $\sC^{\an}(G_n,\Q_p)$. Indeed,  one writes $B^{G_n-\an}=(\sC^{\an}(G_n,\Q_p) \widehat\otimes_{\Q_p} B)^{G_n}$. The Lie algebra action of $\Lie(G)$ on $B^{G_n-\an}$ comes from the action on $\sC^{\an}(G_n,\Q_p)$ induced from the right translation action. The following lemma is Lemme 2.6 of \cite{BC16}.

\begin{lem} \label{LieBound}
Suppose $B$ is a Banach representation of $G_0$. If $D\in\Lie(G)$ and $n\geq 1$, then there exists a constant $C_{D,n}$ such that $\|D(x)\|_{G_n}\leq C_{D,n}\|x\|_{G_n}$ for any $x\in B^{G_n-\an}$.
\end{lem}
The following example will appear naturally later.

\begin{exa} \label{exaZp}
Suppose $G=G_0=\Z_p$ and $B$ is a $\Q_p$-Banach space representation of $\Z_p$. Let $B^o$ be the open unit ball of $B$. Suppose there exists an integer $m>0$ such that 
\[(\gamma-1)^m\cdot (B^o)\subset p B^o\]
for any $\gamma\in\Z_p$. Then $B$ is an analytic representation of $p^{h}\Z_p$ for $h$ sufficiently large. To see this, it suffices to prove that $B$ is an analytic representation of $\Z_p$ if $(\gamma-1)\cdot (B^o)\subset p^2 B^o$ for any $\gamma\in\Z_p$. Let $\gamma=1\in\Z_p$. For any $f\in B$, consider $g:\Z_p\to B$ defined by
 $g(x)=\sum_{n=0}^{+\infty}{x\choose{n}}(\gamma-1)^n\cdot f$. One checks easily that $g$ is analytic and $g(x)=x\cdot f$. Hence $f\in B^{\Z_p-\an}$ and $B=B^{\Z_p-\an}$.
\end{exa}

\subsection{Derived functor} \label{LAacyclic}
Let $G,G_0,G_n$ be as in the first paragraph of \ref{2ndcoord} and let $B$ be a $\Q_p$-Banach representation of $G$. Recall 
\[B^{\la}=\varinjlim_{n}B^{G_n-\an}=\varinjlim_{n}(B\widehat{\otimes}_{\Q_p}\sC^{\an}(G_n,\Q_p))^{G_n}=\varinjlim_{n}H^0_{\cont}(G_n,B\widehat{\otimes}_{\Q_p}\sC^{\an}(G_n,\Q_p)).\]
Write $\mathfrak{LA}(B)=B^{\la}$. We can also consider the following ``right derived functor'' of $\mathfrak{LA}$:

\begin{defn}
For a $\Q_p$-Banach representation $B$ of $G$, let
\[R^i\mathfrak{LA}(B):=\varinjlim_{n}H^i_{\cont}(G_n,B\widehat{\otimes}_{\Q_p}\sC^{\an}(G_n,\Q_p)).\]
We say $B$ is 
\begin{itemize}
\item $\mathfrak{LA}$-\textit{acyclic} if $R^i\mathfrak{LA}(B)=0,i\geq 1$; 
\item \textit{strongly} $\mathfrak{LA}$-\textit{acyclic} if the direct systems $\{H^i_{\cont}(G_n,B\widehat{\otimes}_{\Q_p}\sC^{\an}(G_n,\Q_p)) \}_n,i\geq 1$ are essentially zero, i.e. for any $n$, the image of $H^i_{\cont}(G_n,B\widehat{\otimes}_{\Q_p}\sC^{\an}(G_n,\Q_p)) \to H^i_{\cont}(G_m,B\widehat{\otimes}_{\Q_p}\sC^{\an}(G_m,\Q_p))$ is zero for sufficiently large $m\geq n$.
\end{itemize}
\end{defn}

We put a double quotation mark here as we don't plan to discuss the abelian category we are working with. Clearly $R^i\mathfrak{LA}$ measures the failure of exactness when taking the locally analytic vectors and does not depend on choices of $G_0$. This \textit{ad hoc} definition will be enough for our purpose in view of the following simple lemma.

\begin{lem} \label{LAlongexa}
\hspace{2em}
\begin{enumerate}
\item Suppose $0\to M^0\to M^1\to M^2\to 0$ is a short exact sequence of $\Q_p$-Banach representations of $G$ with $G$-equivariant continuous homomorphisms (which are necessarily strict by the open mapping theorem). Then there is a long exact sequence:
\[0\to (M^0)^{\la}\to (M^1)^{\la}\to (M^2)^{\la}\to R^1\mathfrak{LA}(M^0)\to R^1\mathfrak{LA}(M^1) \to R^1\mathfrak{LA}(M^2)\to ...\]
\item Let $M^\bullet$ be a bounded chain complex of $\Q_p$-Banach representations of $G$ with $G$-equivariant strict homomorphisms. If $M^q$ and $H^q(M^\bullet)$ are $\mathfrak{LA}$-acyclic for any $q$, then $(H^q(M^\bullet))^{\la}=H^q((M^\bullet)^\la)$.
\item  Let $0\to B\to M^0\to M^1\to\cdots$ be an exact chain complex of $\Q_p$-Banach representations of $G$ with $G$-equivariant strict homomorphisms. Assume $M^q$ is $\mathfrak{LA}$-acyclic for any $q$. Then  $R^i\mathfrak{LA}(B)=H^i((M^\bullet)^\la)$. Moreover if all $M^q$ are strongly $\mathfrak{LA}$-acyclic and the direct systems $\{H^i((M^\bullet)^{G_n-\an})\}_n,i\geq 1$ are essentially zero, then $B$ is strongly $\mathfrak{LA}$-acyclic.
\end{enumerate}
\end{lem}

\begin{proof}
For the first part, since all the homomorphisms are strict,  $H^i(M^{\bullet,o})$ can be killed by a bounded power of $p$, where $M^{\bullet,o}$ denotes the open unit ball of $M^{\bullet}$. Therefore, $M^{\bullet}\widehat{\otimes}_{\Q_p}\sC^{\an}(G_n,\Q_p)$ remains exact for any $n$ and we obtain the desired long exact sequence by taking $G_n$-cohomology and passing to the direct limit over $n$.

The second and third parts follow from the first part by a simple induction argument.  For example, one can  prove that the kernel and image of $M^q\to M^{q+1}$ are $\mathfrak{LA}$-acyclic by induction on $q$. This certainly implies that $(H^q(M^\bullet))^{\la}=H^q((M^\bullet)^\la)$. We leave the third part as an exercise.
\end{proof}

A result of Schneider-Teitelbaum says that admissible representations are $\mathfrak{LA}$-acyclic (Theorem 7.1  of \cite{ST03}). In fact, their argument also essentially proves strong $\mathfrak{LA}$-acyclicity. 

\begin{thm}
Any admissible $\Q_p$-Banach representation $B$ of $G$ is strongly $\mathfrak{LA}$-acyclic. 
\end{thm}

\begin{proof}
Since $B$ is admissible, there is a $G_0$-equivariant embedding $B \hookrightarrow \sC(G_0,\Q_p)^{\oplus d}$ for some $d$. Then the quotient $C$ is an admissible $\Q_p$-Banach representation of $G_0$. Note that $ \sC(G_0,\Q_p)\cong  \sC(G_n,\Q_p)^{\oplus d_n}$ as a representation of $G_n$ for some $d_n$. It follows from Shapiro's lemma for continuous cohomology that 
$H_\cont^1(G_n,\sC(G_0,\Q_p)\widehat{\otimes}_{\Q_p}\sC^{\an}(G_n,\Q_p))=0$. 
Hence 
\[(\sC(G_0,\Q_p)^{\oplus d})^{G_n-\an}\to C^{G_n-\an} \to H^1_{\cont}(G_n,B\widehat{\otimes}_{\Q_p}\sC^{\an}(G_n,\Q_p))\to 0.\]
By Corollaire IV.14. of \cite{CD14}, we have an exact sequence 
\[0\to B^{(m)}\to (\sC(G_0,\Q_p)^{\oplus d})^{(m)}\to C^{(m)}\to 0.\] 
See \cite[\S IV]{CD14} for the notation here. One can check $C^{G_n-\an}\subset C^{(n+2)}\subset C^{G_{n+1}-\an}$. This means that the image of $(\sC(G_0,\Q_p)^{\oplus d})^{G_{n+1}-\an}$ in $C^{G_{n+1}-\an}$ contains $C^{G_n-\an}$. Hence $H^1_{\cont}(G_n,B\widehat{\otimes}_{\Q_p}\sC^{\an}(G_n,\Q_p)) \to H^1_{\cont}(G_{n+1},B\widehat{\otimes}_{\Q_p}\sC^{\an}(G_{n+1},\Q_p))$ has zero image. A simple induction on cohomology degree $i\geq 1$ proves the theorem.
\end{proof}

We need the following variant for later applications.
\begin{cor} \label{admLAacyc}
Let $B$ be an admissible $\Q_p$-Banach representation of $G$ and $M$ be a $\Q_p$-Banach space with trivial $G$-action. Then $B\widehat{\otimes}_{\Q_p}M$ is strongly $\mathfrak{LA}$-acyclic.
\end{cor}

\begin{proof}
Let $B^o$ be the unit open ball of $B$. We claim that $H^\bullet_\cont(G_n,B^o)$ is a finitely generated $\Z_p$-module hence has bounded $p$-torsion. Indeed, since $B$ is admissible, there exists a $\Z_p[G_n]$-equivariant injection $B^o \hookrightarrow \sC(G_0,\Z_p)^{\oplus d}$ for some $d$, whose quotient is $p$-torsion free. An induction argument on $i$ implies the finiteness here. It is well-known that  $H^\bullet_\cont(G_n,B)$ can be computed by a cochain complex $(\sC(G_n,B)^{\oplus i},d^i)$. This is a strict complex because $(\sC(G_n,B^o)^{\oplus i},d^i)$ computes $H^\bullet_\cont(G_n,B^o)$. Since $\sC(G_n,B\widehat{\otimes}_{\Q_p}M)\cong \sC(G_n,B)\widehat{\otimes}_{\Q_p}M$, we know that $(\sC(G_n,B)^{\oplus i}\widehat{\otimes}_{\Q_p}M,d^i\otimes \mathbf{1})$ computes $H^\bullet_\cont(G_n,B\widehat{\otimes}_{\Q_p}M)$. Therefore 
\[H^\bullet_\cont(G_n,B\widehat{\otimes}_{\Q_p}M)\cong H^\bullet_\cont(G_n,B)\widehat{\otimes}_{\Q_p}M.\] 
Our claim now follows from the previous theorem.
\end{proof}

\section{Locally analytic vectors and relative Sen theory} \label{LARS}
The main goal of this section is to generalize results of Berger and Colmez \cite{BC16} and Sen \cite{Sen80} in the one-dimensional geometric setting. The main result roughly says that the locally analytic vectors satisfy a differential equation given by a (relative) Sen operator. We claim no originality for most results here. One strong tool in these works is the theory of decompletions. We will review the Tate-Sen formalism of Berger-Colmez \cite{BC08} which is flexible enough to handle our situation. As an application, we will show that taking locally analytic vectors is exact under certain conditions.

The results of this section are closely related to the work of Faltings \cite{Fa05}, Abbes-Gros-Tsuji \cite{AGT}, Liu-Zhu \cite{LZ17} on $p$-adic Simpson correspondences. See Remark \ref{FAGTLZ} below.

\subsection{Statement of the main result}
\begin{para} \label{setup}
Fix a complete algebraically closed non-archimedean field $C$ of characteristic zero. Recall that a non-archimedean field is a topological field whose topology is induced  by a non-archimedean norm $|\cdot|:C\to \R_{\geq 0}$. It is naturally an extension of $\Q_p$ for some $p$ and we assume its norm agrees with the usual $p$-adic norm $|\cdot|_p=p^{-\mathrm{val}_p(\cdot)}$ on $\Q_p$. Let $\cO_C$ be the ring of integers of $C$. Our setup is as follows:
\begin{itemize}
\item $X=\Spa(A,A^+)$: a one-dimensional smooth affinoid adic space over $\Spa(C,\cO_C)$;
\item $G$: a finite-dimensional compact $p$-adic Lie group;
\item $\widetilde X=\Spa(B,B^+)$: an affinoid perfectoid algebra over $\Spa(C,\cO_C)$, which is a ``log $G$-Galois pro-\'etale perfectoid covering'' of $X$. More precisely, this means that there is a finite set $S$ of classical points in $X$ and $\widetilde X\sim\varprojlim_{i\in I} X_i$ in the sense of Definition 7.14 of \cite{Sch12} for some index set $I$, where each $X_i=\Spa(B_i,B_i^+)$ is a finite Galois covering of $X$ unramified outside of $S$, and $B^{+}$ is the $p$-adic completion of $\varinjlim_i B_i^{+}$. Moreover the inverse limit of the Galois group of $X_i$ over $X$ is identified with $G$. When $S$ is non-empty, we further assume for each point $s$ in $S$, the ramification-index $e_i$ of $X_i\to X$ at $s$ is a $p$-power for any $i$ and $\{e_i\}_{i\in I}$ is unbounded;
\item assume $X$ is \textit{small} in the sense that $S$ contains at most one element and there is an \'etale map $X\to \mathbb{T}^1=\Spa(C\langle T^{\pm1}\rangle, \cO_C\langle T^{\pm1}\rangle)$ (resp. $X\to \mathbb{B}^1=\Spa(C\langle T\rangle, \cO_C\langle T\rangle)$ mapping $S$ to $0$) when $S$ is empty (resp. otherwise) which factors as a composite of rational embeddings and finite \'etale maps.
\end{itemize}

Note that $G$ acts continuously on the Banach algebra $B$ and we can consider the locally analytic vectors $B^{\la}\subset B$. Let $\Lie(G)$ be  the Lie algebra of $G$. It acts on $B^{\la}$ and this action can be extended $B$-linearly to a map $B\otimes_{\Q_p}\Lie(G) \times B^{\la} \to B$.
\end{para}

\begin{thm} \label{pCR}
Fix an  \'etale map $X\to \mathbb{T}^1$ (resp. $X\to \mathbb{B}^1$) if $S$ is empty  (resp. non-empty) as above.
For each ``log $G$-Galois pro-\'etale perfectoid covering'' $\widetilde{X}$ of $X$, we can assign an element $\theta=\theta_{\widetilde{X}}\in B\otimes_{\Q_p}\Lie(G)$, satisfying
\begin{enumerate}
\item $\theta$ annihilates $B^{\la}$. In other words, locally analytic vectors satisfy certain (first-order) differential equation.
\item $\theta$ is functorial in $(G,\tilde{X})$: if $H$ is a closed normal  subgroup of $G$ so that $\widetilde{X}'=\Spa(B^H,(B^+)^H)$ is a ``log $G/H$-Galois pro-\'etale perfectoid covering'' of $X$, then
\[\theta_{\widetilde{X}}\equiv \theta_{\widetilde{X}'}\mod B\otimes_{\Q_p}\Lie(H),\] 
where $\theta_{\widetilde{X}'}\in B^H\otimes_{\Q_p}\Lie(G/H)$ is viewed as an element in $B\otimes_{\Q_p}\Lie(G/H)$. 
\item $\theta\neq 0$ if $\widetilde{X}$ is a locally analytic covering in the sense of \ref{laedefn} below.
\end{enumerate}
Moreover, if we start with another  \'etale map $X\to \mathbb{T}^1$ (or  \'etale map $X\to \mathbb{B}^1$), then the element $\theta'\in B\otimes_{\Q_p}\Lie(G)$ obtained using this \'etale map will differ $\theta$ by a unit of $A$. In other words, $A^\times\theta\subset B\otimes_{\Q_p}\Lie(G)$ does not depend on the choice of the   \'etale map.
\end{thm}

\begin{rem}
From the point of view of differential operators, one may regard  $\theta$ as some $p$-adic analogue of Cauchy-Riemann operator in the classical complex analysis.
\end{rem}

\begin{rem} \label{theSen}
In the classical $p$-adic Hodge theory of $p$-adic Galois representations, this $\theta$ is nothing but the Sen operator, cf. Theorem 12 of \cite{Sen80}, Th\'eor\`eme 1.9 of \cite{BC16}. So it is reasonable to call it relative Sen operator.
\end{rem}

\begin{rem} \label{canonicaltheta}
It is natural to ask whether there is a canonical representative of $\theta$ so that one can glue them in some global situation. First, it turns out that the natural place for $\theta$ to live is $B\otimes_A\Omega_{A/C}^{1}(S)\otimes_{\Q_p}\Lie(G)(-1)$, where $\Omega_{A/C}^{1}(S)$ denotes the continuous $1$-forms of $A$ over $C$ with simple poles at $S$ and $(1)$ denotes the usual Tate-twist. So $\theta$ is viewed as a (log) Higgs field. 

Secondly, if $S$ is empty, i.e. $\widetilde X$ is a $G$-Galois pro-\'etale perfectoid covering of $X$, then we can make this element canonical using the functorial isomorphism $\Omega_{A/C}^{1}(-1)\cong H^1_{\cont}(G,B)$ (Proposition 3.23 of \cite{PerfSur}). Here $H^1_{\cont}(G,B)$ denotes the continuous group cohomology group and is identified with $H^1(X_{\mathrm{pro\acute{e}t}},\hat\cO_X)$ by an argument similar to the proof of Lemma 5.6 of \cite{Sch13}. By the functorial property, it is enough to pin down this element $\theta$ for a particular $\widetilde X$, provided that $\theta_{\tilde{X}}\neq 0$. Indeed, if $\widetilde X'$ is a $H$-Galois pro-\'etale perfectoid covering of $X$, we can consider $\widetilde X'':=\widetilde X\times_X \widetilde{X}'$, the fiber product in the category of pro-\'etale coverings of $X$. Then $\widetilde X''$ is a $(G\times H)$-Galois pro-\'etale perfectoid covering of $X$. Once $\theta_{\widetilde{X}}$ is determined, so is $\theta_{\widetilde X''}$ and hence $\theta_{\widetilde X'}$ by the functorial property.

Suppose $G\cong\Z_p$.  Then there is a canonical isomorphism:
\[\Omega_{A/C}^{1}(-1)\otimes_{\Q_p}\Lie(G)\stackrel{\sim}{\longrightarrow} H^1_{\cont}(G,B)\otimes_{\Z_p} G\stackrel{\sim}{\longrightarrow} B_G,\]
where the first map is induced by the exponential map of $G$ and $B_G$ denotes the $G$-coinvariants of $B$. Under this isomorphism, the canonical $\theta$ is given by the image of $1\in B$ in $B_G$. It can be checked that this is independent of the choice of such a $\Z_p$-covering. 

The proof of Proposition 3.23 of \cite{PerfSur} crucially uses the fact that when $X$ is defined over some finite extension of $\Q_p$, then there is the Faltings's extension (Corollary 6.14 of \cite{Sch13}) which produces this canonical isomorphism. Faltings's extension has also been generalized to the log case, cf. Corollary 2.4.5 of \cite{DLLZ2}. Hence it is conceivable that there is a canonical $\theta$ in general.
\end{rem}

\begin{exa} \label{toricexa}
\hspace{2em}
\begin{itemize}
\item $X=\Spa(C\langle T^{\pm 1}\rangle, \cO_C\langle T^{\pm1}\rangle)$, the one-dimensional torus;
\item $\widetilde X=\Spa(C\langle T^{\pm \frac{1}{p^{\infty}}}\rangle, \cO_C\langle T^{\pm\frac{1}{p^{\infty}}}\rangle)$, the inverse limit of $\Spa(C\langle T^{\pm \frac{1}{p^{m}}}\rangle, \cO_C\langle T^{\pm\frac{1}{p^{m}}}\rangle)$;
\item $G=\Z_p$ which acts by $k\cdot T^{\frac{1}{p^m}}=\zeta^{\frac{k}{p^m}}T^{\frac{1}{p^m}},~k\in\Z_p$ where $\{\zeta_{p^m}\}_m$ is a choice of compatible system of $p^m$-th roots of unity.
\end{itemize}
The locally analytic vectors in $C\langle T^{\pm \frac{1}{p^{\infty}}}\rangle$ are just $\bigcup_m C\langle T^{\pm \frac{1}{p^{m}}}\rangle$, the smooth vectors. This can be deduced easily from the existence of Tate's normalized trace in this situation, i.e. there is a continuous left inverse of the inclusion $C\langle T^{\pm \frac{1}{p^{m}}}\rangle\hookrightarrow C\langle T^{\pm \frac{1}{p^{\infty}}}\rangle$. See Lemma \ref{Ainftyla} below or Th\'eor\`eme 3.2 of \cite{BC16}. The differential operator $\theta$ in this case is $\frac{dT}{T}\otimes \mathbf{1}$, where $\mathbf{1}$ is the image of $1\in\Z_p\stackrel{\log}{\longrightarrow} \Lie(\Z_p)$ under the logarithm.  The same result holds for $X=\Spa(C\langle T\rangle, \cO_C\langle T\rangle)$, the one-dimensional unit ball and $\widetilde X=\Spa(C\langle T^{ \frac{1}{p^{\infty}}}\rangle, \cO_C\langle T^{\frac{1}{p^{\infty}}}\rangle)$.
\end{exa}

\begin{rem} \label{FAGTLZ}
As mentioned in Remark \ref{canonicaltheta}, $\theta$ should be considered as a log Higgs field. This is closely related to previous work of Faltings \cite{Fa05}, Abbes-Gros-Tsuji \cite{AGT} and Liu-Zhu \cite{LZ17} on $p$-adic Simpson correspondences.  More precisely, for a finite-dimensional continuous representation $V$ of $G$,  the $p$-adic Simpson correspondence associates a log Higgs field (after extending the coefficients to $B$):
\[\phi_V: B\otimes_{\Q_p}V \to \Omega_{A/C}^{1}(S)\otimes_A B\otimes_{\Q_p}V(-1),\]
which is monoidal in $V$. Hence from the Tannakian point of view, this gives rise to a log Higgs field  which is nothing but our $\theta$. Basically, we will construct $\theta$ by taking $V$ as the space of analytic functions on $G$ (after taking a certain limit). 

To see that $\theta$ annihilates $B^{\la}$, we remark that,  as a $B$-module, $\ker(\phi_V)$ are generated by the $G$-smooth vectors in $B\otimes_{\Q_p} V$. Hence when $V$ is the space of analytic functions on $G$, it is tautological that $B^{\la}$ is in the kernel of $\theta$. I hope this provides some intuition for  the constructions below.
\end{rem}

\subsection{Relative Sen theory}
In this subsection, we will first define a bigger perfectoid algebra $B_\infty$ containing $B$ and construct (functorially) an operator in the endomorphism group $\End_{B_\infty}(B_\infty\otimes_{\Q_p} V)$ for any continuous finite-dimensional $\Q_p$-representation $V$ of $G$. 

\begin{para} \label{smallY}
Fix an \'etale map $f_1:X\to Y$ which factors as a composite of rational embeddings and finite \'etale maps, where $Y$ is either $\mathbb{T}^1=\Spa(C\langle T^{\pm1}\rangle, \cO_C\langle T^{\pm1}\rangle)$ or $\mathbb{B}^1=\Spa(C\langle T\rangle, \cO_C\langle T\rangle)$. In the latter case, we assume the image of $S$ in $Y$ is $0$. We will fix such a choice of $f_1:X\to Y$ from now on.

For any  $n\in\N$, when $Y=\mathbb{T}^1$ (resp. $\mathbb{B}^1$), let 
\[Y_n=\Spa(R_n,R_n^+):=\Spa(C\langle T^{\pm \frac{1}{p^{n}}}\rangle, \cO_C\langle T^{\pm\frac{1}{p^{n}}}\rangle)~ (\mbox{resp. } \Spa(C\langle T^{ \frac{1}{p^{n}}}\rangle, \cO_C\langle T^{\frac{1}{p^{n}}}\rangle)),\] 
\[ Y_\infty=\Spa(R,R^+):=\Spa(C\langle T^{\pm \frac{1}{p^{\infty}}}\rangle, \cO_C\langle T^{\pm\frac{1}{p^{\infty}}}\rangle)~ (\mbox{resp. } \Spa(C\langle T^{ \frac{1}{p^{\infty}}}\rangle, \cO_C\langle T^{\frac{1}{p^{\infty}}}\rangle)).\]
This is the example considered in \ref{toricexa}. Note that $Y_\infty\sim\varprojlim_{i\in\N} Y_i$ is a Galois covering of $Y$ and $R^+$ is the $p$-adic completion of $\varinjlim R^+_n$. We denote the Galois group by $\Gamma$, which can be identified with $\Z_p$ non-canonically. For $n\geq m$, there is the usual trace map
\[\tr_{Y,n,m}:R_n^+\to R_m^+.\]
Concretely, it sends $T^{\frac{l}{p^k}},(l,p)=1$ to $p^{n-m}T^{\frac{l}{p^k}}$ if $k\leq m$ and $0$ otherwise. Clearly, $\frac{1}{p^{n-m}}\tr_{Y,n,m}$ are compatible when $n$ varies and extends to a $R_m^+$-linear map :
\[\overline\tr_{Y,m}^+:R^+\to R^+_m,\]
which is a left inverse of the inclusion $R^+_m\to R^+$. It commutes with the action of $\Gamma$ as taking traces commutes with the Galois action. Moreover, for any $x\in R^+$,
\[\lim_{m\to +\infty}\overline\tr_{Y,m}^+(x)=x.\]
After inverting $p$, we get a $R_m$-linear map $\overline\tr_{Y,m} :R\to R_m$  (Tate's normalized trace). Let $\gamma$ be any topological generator of $p^m\Gamma$. It is easy to see that $\gamma-1$ is invertible on $\ker(\overline\tr_{Y,m})$. Moreover the norm $\|(\gamma-1)^{-1}\|$ on $\ker(\overline\tr_{Y,m})$ equals $|(\zeta_{p^{m+1}}-1)^{-1}|_p=p^{\frac{1}{p^m(p-1)}}$ which converges to $1$ as  $m\to \infty$.
\end{para}

\begin{para} \label{TNT}
The material here should be a consequence of  some general results of Diao-Lan-Liu-Zhu on log affinoid perfectoid spaces. See 5.3 of \cite{DLLZ1}. For our later applications, we give a more explicit presentation.

Now we base change everything along the map $X\to Y$:
\[X_n=\Spa(A_n,A_n^+):=X\times_Y Y_n,\] 
\[X_\infty=\Spa(A_\infty,A_\infty^+):=X\times_Y Y_\infty,\]
where all the fiber products are taken in the category of adic spaces over $C$. We remark that $X\to Y$ is locally of finite type in the sense of \cite[Definition 1.2.1]{Hu96} and the fiber product exists by  \cite[Proposition 1.2.2]{Hu96}. All the fiber products we consider below will exist for the same reason.

\[\begin{tikzcd}
X_\infty \arrow[d] \arrow[r] & Y_\infty  \arrow[d]  \\
X_n  \arrow[d] \arrow[r] & Y_n \arrow[d]  \\
X  \arrow[r] & Y \\
\end{tikzcd}.\]

Recall that $C$ is equipped with an non-archimedean norm $|\cdot|$. For any $|p|^{c}\in |C^\times|\subset\R^\times_{\geq 0}$, we fix a choice of element in $C$, formally written as $p^c$, with norm $|p|^{c}$. 
\end{para}

\begin{lem} \label{AnRn}
Let $(A_n^+\otimes_{R_n^+} R^+)^{\mathrm{tf},\wedge}$ be the $p$-adic completion of the $p$-torsion free quotient of $A_n^+\otimes_{R_n^+} R^+$. 
\begin{enumerate}
\item The natural map $(A_n^+\otimes_{R_n^+} R^+)^{\mathrm{tf},\wedge}\to A_\infty^+$ is injective and its cokernel is killed by $p^{C_n}$ for some constant $C_n$ with $C_n\to 0$ as $n\to 0$. 
\item $A_\infty^+$ is the $p$-adic completion of $\varinjlim_n A^+_n$.
\item $(A_\infty,A_\infty^+)$ is a perfectoid affinoid $(C,\cO_C)$-algebra.
\end{enumerate}
\end{lem}

\begin{proof}
As $X\to Y$ can be written as a composite of rational embeddings and finite \'etale maps, we may apply Lemma 4.5 of \cite{Sch13} here. Strictly speaking, Lemma 4.5 assumes $Y_n$ is  \'etale over $Y$ but the same argument works here using that $R$ is a perfectoid $C$-algebra and $R^+$ is the $p$-adic completion of $\varinjlim R^+_n$.
\end{proof}

Hence we may extend ($A_m$-linearly) Tate's normalized trace $\overline\tr_{Y,m}$ to $X_\infty$: 
\[\overline\tr_{X,m}: A_\infty\to A_m,\]
which is a continuous $A_m$-linear left inverse of the inclusion $A_m\to A_\infty$. Moreover, the image of $A_\infty^+$
is contained in $p^{-C_m}A_m^+$ for some constant $C_m$ with $C_m\to 0$ as $m\to 0$. For any $x\in A_\infty$, we still have
\[\lim_{m\to +\infty}\overline\tr_{X,m}(x)=x.\]
Note that $\Gamma$ acts on $X_\infty$ and commutes with $\overline\tr_{X,m}$. For any topological generator $\gamma$ of $p^m\Gamma$, the action of $\gamma-1$ on $\ker(\overline\tr_{X,m})$ is invertible. The norm of its inverse $\|(\gamma-1)^{-1}\|$ is bounded by some $p^{c_m}$ for some constant $c_m>0$, and $c_m\to 0$ as $m\to \infty$. 

\begin{rem}
These properties of Tate's normalized trace appear as (TS2), (TS3) in the Tate-Sen conditions formulated by Berger-Colmez, cf. \cite{BC08} D\'efinition 3.1.3.
\end{rem}

$A_\infty$ is a Banach space representation of $\Gamma$. We can consider its $\Gamma$-locally analytic vectors, cf. \ref{Gnan}. The following lemma is a direct consequence of the existence of Tate's normalized traces. See also Th\'eor\`eme 3.2 of \cite{BC16}.

\begin{lem} \label{Ainftyla}
$(A_\infty)^{p^n\Gamma-\an}=A_n$ for any integer $n\geq 0$. In particular, the subspace of $\Gamma$-locally analytic vectors in $A_\infty$ is $\bigcup_{n\geq 0}A_n$.  
\end{lem}

\begin{proof}
For $m\geq n$, as $\overline\tr_{X,m}$ is continuous and $\Gamma$-equivariant, it can be restricted to the $p^n\Gamma$-analytic vectors
\[(A_\infty)^{p^n\Gamma-\an}\xrightarrow{\overline\tr_{X,m}} (A_m)^{p^n\Gamma-\an}=(A_m)^{p^n\Gamma}=A_n.\]
Note that the action of $p^n\Gamma$ on $A_m$ is trivial on $p^m\Gamma$, hence the analyticity implies the first equality. Now the lemma follows from $\lim_{m\to +\infty}\overline\tr_{X,m}(x)=x$.
\end{proof}

\begin{para} \label{BCtotildeX}
Next we base change the tower $\{X_n\}_n$ to $\widetilde{X}$. Suppose $G_0$ is an open subgroup of $G$. Let
\[X_{G_0}:=\Spa(B^{G_0},(B^+)^{G_0})\]
be the corresponding covering of $X$. We can take the fiber product in the category of adic spaces over $C$
\[X'_{G_0,n}:=X_{G_0}\times_X X_n=X_{G_0}\times_Y Y_n,\]
and take its normalization 
\[X_{G_0,n}=\Spa(B_{G_0,n},B_{G_0,n}^+).\]
Note that by Abhyankar's lemma for rigid analytic spaces (cf. Lemma 4.2.2, 4.2.3 of \cite{DLLZ1}, which is based on earlier work of L\"utkebohmert), $X_{G_0,n}\to Y_n$ is unramified when $n$ is sufficiently large, 
hence $X_{G_0,n}\to X_n$ is finite \'etale and
\[X_{G_0,n}=X_{G_0,m}\times_{X_m} X_n= X_{G_0,m}\times_{Y_m} Y_n\]
for sufficiently large $m<n$. This implies that when $n$ is sufficiently large,
\[X_{G_0,\infty}=\Spa(B_{G_0,\infty},B_{G_0,\infty}^+):=X_{G_0,n}\times_{X_n} X_\infty=X_{G_0,n}\times_{Y_n} Y_\infty\]
is independent of $n$. Since $X_{G_0,n}\to Y_n$ can be written as a composite of rational embeddings and finite \'etale maps, Lemma \ref{AnRn} still holds in this setting. For example, $B_{G_0,\infty}^+$ is the $p$-adic completion of $\varinjlim_n B_{G_0,n}^+$. For sufficiently large $n$, there exist Tate's normalized traces $\overline\tr_{X_{G_0},n}:B_{G_0,\infty}\to B_{G_0,n}$ with same properties as in the case $\{X_n\}_n$. In particular, $(B_{G_0,\infty})^{p^n\Gamma-\an}=B_{G_0,n}$ for $n$ large enough.

\[\begin{tikzcd}
\widetilde{X}_{\infty}\arrow[d] \arrow[r] &X_{G_0,\infty}  \arrow[d] \arrow[r] &X_\infty \arrow[d] \arrow[r] & Y_\infty  \arrow[d]  \\
\widetilde{X}_{n}\arrow[d] \arrow[r] &X_{G_0,n}  \arrow[d] \arrow[r] &X_n  \arrow[d] \arrow[r] & Y_n \arrow[d]  \\
\widetilde{X} \arrow[r] &X_{G_0} \arrow[r] &X  \arrow[r] & Y \\
\end{tikzcd}.\]

Suppose $G'_0$ is another open subgroup of $G$ containing $G_0$. We have compatible maps $X_{G_0,n}\to X_{G'_0,n}$ which are finite \'etale when $n$ is sufficiently large. Hence $X_{G_0,\infty}\to X_{G'_0,\infty}$ is a finite \'etale covering of affinoid perfectoid spaces. Let $B_\infty^+$ be the $p$-adic completion of $\varinjlim_{G_0}B_{G_0,\infty}^+$ over all open subgroups $G_0$ of $G$ and $B_\infty=B_\infty^+[\frac{1}{p}]$ equipped with the norm induced from $B_{G_0,\infty}^+$. Then $(B_\infty, B_\infty^+)$ is again an affinoid perfectoid $(C,\cO_C)$-algebra as it is the completion of a direct limit of perfectoid affinoid $(C,\cO_C)$-algebras.  We denote $\Spa(B_\infty, B_\infty^+)$ by $\widetilde X_\infty$. Note that $\widetilde{X}_\infty$ also agrees with the fiber product of $\tilde{X}=\varprojlim X_{G_0}$ and $X_\infty=\varprojlim X_n$ in  the pro-Kummer \'etale site of $X$  equipped with the natural log structure defined by $S$, cf. \cite{DLLZ2}.

Clearly $G\times \Gamma$ acts on $B_\infty^+$ and $B_\infty$. Then it follows from Faltings's almost purity theorem, cf. \cite[Theorem 7.9 (iii)]{Sch12},  that $(B_\infty^+/(p))^{G_0}$ almost equals $B^{+}_{G_0,\infty}/(p)$. Hence $(B_\infty)^{G_0}=B_{G_0,\infty}$ and $(B_\infty^+)^{G_0}$ almost equals $B_{G_0,\infty}^+$.
\end{para}

\begin{prop}
$B_\infty$ satisfies the Tate-Sen conditions formulated by Berger-Colmez in \cite[D\'efinition 3.1.3]{BC08}  with 
\begin{itemize}
\item $S=\Q_p$;
\item $\tilde\Lambda=B_\infty$ equipped with the valuation induced by its norm;
\item $G_0=G\times\Gamma$, $H_0=G\times\{1\}$ as a subgroup of  $G\times\Gamma$;
\item $\Lambda_{H,n}=B_{H,n}, R_{H,n}=\overline\tr_{X_H,n}$ for any open subgroup $H$ of $H_0=G$.
\end{itemize}
\end{prop}

\begin{proof}
To see (TS1), for any open subgroups $H_1\subset H_2$ of $G$, we know that $(B_\infty)^{H_2}=B_{H_2,\infty}\to (B_\infty)^{H_1}=B_{H_1,\infty}$ is a finite \'etale map between perfectoid algebras, hence by Faltings's almost purity theorem, $B_{H_2,\infty}^+\to B_{H_1,\infty}^+$ is almost finite \'etale. All the claims in (TS2), (TS3) are essentially verified in \ref{TNT} and above discussion.
\end{proof}

\begin{rem} \label{TSG0}
It follows from the definition directly that $B_\infty$ satisfies the Tate-Sen conditions with respect to the action of $G_0\times\Gamma$ for any open subgroup $G_0$ of $G$.
\end{rem}

\begin{lem} \label{Gammainv}
$(B_\infty)^\Gamma=B$.
\end{lem}

\begin{proof}
Recall that $\widetilde{X}=\Spa(B,B^+)$ is perfectoid and $B^+$ is the $p$-adic completion of the direct limit $\varinjlim_{G_0} B^+_{G_0,0}$ over all open subgroups $G_0$ of $G$. Let $n$ be a positive integer. Note that by our assumption on the ramification-index in \ref{setup} and Abhyankar's lemma (cf. Lemma 4.2.2, 4.2.3 of \cite{DLLZ1}), $X_{G_0,n}\to X_{G_0,0}$ is finite \'etale for sufficiently small $G_0$. Hence for a sufficiently small subgroup $G_0$, the fiber product 
\[\widetilde{X}_n=\Spa(B_n,B_n^+):=\widetilde{X}\times_{X_{G_0,0}} X_{G_0,n}\]
is independent of the choice of $G_0$ and $\widetilde{X}_n\to \widetilde{X}$ is a finite \'etale map between affinoid perfectoid spaces. By the same argument as in the proof of Lemma \ref{AnRn}, $B^+_n$ is the $p$-adic completion of $\varinjlim_{G_0} B^+_{G_0,n}$. Hence it follows from our discussion in \ref{BCtotildeX} that $B_\infty^+$ is the $p$-adic completion of $\varinjlim_n B^+_n$. By Faltings's almost purity theorem, we see that $(B^+_\infty)^\Gamma$ is almost $B^+$ and $(B_\infty)^\Gamma=B$.
\end{proof}

\begin{para}
Now suppose $V$ is a finite-dimensional continuous representation of $G$ over $\Q_p$. We are going to construct a morphism: $\Lie(\Gamma)\to \End_{B_\infty}(B_\infty\otimes_{\Q_p}V)$, i.e. a Lie algebra representation of $\Lie(\Gamma)$ on $B_\infty\otimes_{\Q_p}V$.

As $G$ is compact, there exists a $G$-stable $\Z_p$-lattice $T\subset V$. Let $(B_\infty)^{\circ}\subset B_\infty$ be the  subset of powerbounded elements, equivalently elements with spectral norm at most $1$. Note that  $(B_\infty)^{\circ}\otimes_{\Z_p}T$ carries a diagonal action of $G$ and an action of $\Gamma$ on the first factor. This defines an action of $G\times\Gamma$ on $(B_\infty)^{\circ}\otimes_{\Z_p}T$.
\end{para}

\begin{prop} \label{decompletion}
Fix a constant $c<\frac{1}{2}$ inside $|C^{\times}|\subset \R^{\times}_{\geq 0}$. Suppose $G_0$ is an open subgroup of $G$ acting trivially on $T/pT$. Then there exists a constant $n(G_0)$, which is independent of $T$ and only depends on $c$, such that for any integer $n\geq n(G_0)$, the tensor product $(B_\infty)^{\circ}\otimes_{\Z_p}T$ has a unique free $B_{G_0,n}^{+}$-submodule $D_{G_0,n}^{+}(T)$ of rank $\dim_{\Q_p}V$ with the following properties:
\begin{enumerate}
\item $D_{G_0,n}^{+}(T)$ is fixed by $G_0$ and $G\times\Gamma$-stable;
\item the natural morphism $(B_\infty)^{\circ}\otimes_{B_{G_0,n}^{+}} D_{G_0,n}^{+}(T)\to (B_\infty)^{\circ}\otimes_{\Z_p}T$ is an isomorphism;
\item there exists a basis $\mathfrak{B}$ of $D_{G_0,n}^{+}(T)$ over $B_{G_0,n}^{+}$ such that $(\gamma-1)(\mathfrak{B})\subset p^c D_{G_0,n}^{+}(T)$ for any $\gamma\in \Gamma$;
\item $(\gamma-1)^m (D_{G_0,n}^{+}(T))\subset p D_{G_0,n}^{+}(T)$ for any $\gamma\in\Gamma$ and $m\geq m(c,n)$, a  constant only depending on $c,n$.
\end{enumerate}
\end{prop}

\begin{proof}
This follows from Proposition 3.3.1 of \cite{BC08} by choosing $c_3=c$ and $c_1,c_2$ sufficiently small such that $c_1+2c_2+2c_3< 1$. Note that in our setup, we may choose the constants $c_1,c_2,c_3$ in the Tate-Sen conditions to be arbitrarily small. The first three parts are the same as \cite[3.3.1]{BC08}. The last part is a consequence of the third one.

\end{proof}

Now we construct an action of $\Lie(\Gamma)$ on $B_\infty\otimes_{\Q_p}V$. Choose a constant $c$, a sufficiently small open subgroup $G_0$ and $n\geq n(G_0)$ as in the proposition. By Amice's result (see \ref{exaZp}), the last part of the proposition implies that the action of $p^m\Gamma$ on $D_{G_0,n}^{+}(T)\otimes_{\Z_p} \Q_p$ is analytic for sufficiently large $m$. Thus $\Lie(\Gamma)$ acts on $D_{G_0,n}^{+}(T)\otimes_{\Z_p} \Q_p$. Since $B_{G_0,n}$ is fixed by $p^n\Gamma$, this action of $\Lie(\Gamma)$ is $B_{G_0,n}$-linear. By the second part of the proposition, we may extend it $B_\infty$-linearly to an action of $\Lie(\Gamma)$ on $(B_\infty)^{\circ}\otimes_{\Z_p}T\otimes\Q_p=B_\infty\otimes_{\Q_p}V$. This action commutes with the action of $G\times\Gamma$ as the two actions commute when restricted to $D_{G_0,n}^{+}(T)\otimes \Q_p$.

Note that the second part of Proposition \ref{decompletion} implies that 
\[(B_\infty\otimes_{\Q_p} V)^{G_0}=(B_{\infty})^{G_0}\otimes_{B_{G_0,n}^{+}} D_{G_0,n}^{+}(T)=B_{G_0,\infty}\otimes_{B_{G_0,n}^{+}} D_{G_0,n}^{+}(T).\]

\begin{lem} \label{LAm}
For any $m\geq m(c,n)$, the subspace of $p^m\Gamma$-analytic vectors in $(B_\infty\otimes V)^{G_0}$ is $B_{G_0,m}\otimes_{B_{G_0,n}^{+}} D_{G_0,n}^{+}(T)$. In particular, the second part of Proposition \ref{decompletion} implies that there is a natural isomorphism
\[(B_\infty\otimes_{\Q_p}V)^{G_0,p^m\Gamma-\an}\otimes_{B_{G_0,m}}B_\infty\cong B_\infty\otimes_{\Q_p}V. \]
\end{lem}

\begin{proof} \label{Dla}
For any $k\geq m$, the normalized trace $\overline\tr_{X_{G_0},k}:B_{G_0,\infty}\to B_{G_0,k}$ induces a map 
\[(B_\infty\otimes_{\Q_p} V)^{G_0,p^m\Gamma-\an}=(B_{G_0,\infty}\otimes_{B_{G_0,n}^{+}} D_{G_0,n}^{+}(T))^{p^m\Gamma-\an}\to (B_{G_0,k}\otimes_{B_{G_0,n}^{+}} D_{G_0,n}^{+}(T))^{p^m\Gamma-\an}.\]
As the action of $p^m\Gamma$ on $B_{G_0,k}$ is trivial on $p^k\Gamma$, we have a natural decomposition 
\[B_{G_0,k}=\bigoplus_{\chi: p^m\Gamma/p^k\Gamma\to C^\times}B_{G_0,k}[\chi]\]
into the direct sum of $\chi$-isotypic components, where $\chi$ runs through all characters of $p^m\Gamma/p^k\Gamma$. Note that $\chi$ is an analytic function on $p^m\Gamma$ only when $\chi$ is trivial. Hence
\[(B_{G_0,k}\otimes_{B_{G_0,n}^{+}} D_{G_0,n}^{+}(T))^{p^m\Gamma-\an}=(B_{G_0,k})^{p^m\Gamma}\otimes_{B_{G_0,n}^{+}} D_{G_0,n}^{+}(T)=B_{G_0,m}\otimes_{B_{G_0,n}^{+}} D_{G_0,n}^{+}(T).\] 
The rest of the proof is the same as the one of Lemma \ref{Ainftyla}.
\end{proof}

If $G_0$ is moreover a normal subgroup of $G$, we may choose $n$ large enough so that $B_{G,n}\to B_{G_0,n}$ is finite \'etale. Then by Galois descent, there is a $\Gamma$-equivariant isomorphism 
\[(D_{G_0,n}^{+}(T)\otimes_{\Z_p} \Q_p)^G\otimes_{B_{G,n}}B_{G_0,n}=D_{G_0,n}^{+}(T)\otimes_{\Z_p} \Q_p.\]
Let $D_{G,n}(V)= (D_{G_0,n}^{+}(T)\otimes_{\Z_p} \Q_p)^G$.  By the second part of Proposition \ref{decompletion}, there is a natural isomorphism
\[B_\infty\otimes_{B_{G,n}} D_{G,n}(V)\cong B_\infty\otimes_{\Q_p} V\]
and one can repeat all the above discussion with $G_0$ replaced by $G_0$. 
Hence we may reformulate the above construction into the following form, which clearly is independent of the choice of $T,c,G_0,n$.

\begin{prop} \label{relSenop}
For each finite-dimensional continuous representation $V$ of $G$ over $\Q_p$, there exists a (necessarily unique) $B_\infty$-linear action of $\Lie(\Gamma)$ on $B_\infty\otimes_{\Q_p}V$ 
\[\phi_V:\Lie(\Gamma)\to \End_{B_\infty}(B_\infty\otimes_{\Q_p}V)\] 
extending the natural action of $\Lie(\Gamma)$ on the $\Gamma$-locally analytic vectors in $(B_\infty\otimes_{\Q_p}V)^G$. Moreover, it satisfies the following properties:
\begin{enumerate}
\item $\phi_V$ commutes with the action of $G\times\Gamma$;
\item $\phi_V$ is functorial in $V$, i.e. suppose $\psi:V\to W$ is a $G$-equivariant map between $G$-representations, then $\mathbf{1}\otimes \psi: B_\infty\otimes V\to B_\infty\otimes W$ intertwines $\phi_V$ and $\phi_W$.
\item $\phi_V$ commutes with tensor products, i.e. suppose $V_1,V_2$ are two finite-dimensional representations of $G$, then $\phi_{V_1}\otimes \mathbf{1}+\mathbf{1}\otimes\phi_{V_2}=\phi_{V_1\otimes_{\Q_p}V_2}$ on $(B_\infty\otimes_{\Q_p} V_1)\otimes_{B_\infty}(B_\infty\otimes_{\Q_p} V_2)=B_\infty\otimes_{\Q_p}(V_1\otimes_{\Q_p} V_2)$.
\end{enumerate}
\end{prop}

\begin{proof}
It is easy to check all these properties on the $\Gamma$-locally analytic vectors in $(B_\infty\otimes_{\Q_p}V)^G$. We omit the details here.
\end{proof}


If we fix a generator of $\Lie(\Gamma)$, then this proposition becomes the form claimed in the beginning of this subsection.

Suppose $G_0$ is an open subgroup of $G$. As $B_\infty$ satisfies the Tate-Sen conditions with respect to $G_0\times\Gamma$ (see Remark \ref{TSG0}), Proposition \ref{relSenop} still holds with all $G$ replaced by $G_0$. Thus Proposition \ref{relSenop} can be generalized as follows. 

\begin{prop} 
For each finite-dimensional continuous representation $V$ of $G$ over $\Q_p$, there exists a (necessarily unique) $B_\infty$-linear action of $\Lie(\Gamma)$ on $B_\infty\otimes_{\Q_p}V$, extending its natural action on the $G$-smooth, $\Gamma$-locally analytic vectors in $B_\infty\otimes_{\Q_p}V$. Here an element in  $B_\infty\otimes_{\Q_p}V$ is called $G$-smooth if it is fixed by some open subgroup of $G$.  Moreover, this action satisfies all three properties as in the previous proposition.
\end{prop}

Hence this action of $\Lie(\Gamma)$ on $B_\infty\otimes_{\Q_p}V$ only depends on the restriction of the representation to any open subgroup of $G$. As it commutes with $\Gamma$, it induces a $B$-linear action of $\Lie(\Gamma)$ on $(B_\infty\otimes_{\Q_p} V)^\Gamma=B\otimes_{\Q_p} V$ by Lemma \ref{Gammainv}.

\subsection{Proof of the main result I: construction}  \label{p1construction}
\begin{para} \label{VsCan}
We will first show that the action in Proposition \ref{relSenop} factors through $B\otimes\Lie(G)$. As a byproduct, this will imply Theorem \ref{pCR}. The strategy is to apply Proposition \ref{relSenop} to $V=\sC^{\an}(G,\Q_p)$, the space of analytic functions on $G$. However, as this is an infinite-dimensional vector space over $\Q_p$, it requires a limiting argument plus some extra work.

Let $G_0\subset G$ be a compact open subgroup equipped with an integer valued, saturated $p$-valuation as in \ref{2ndcoord}. By Proposition \ref{algdense}, we may replace $G_0$ by a smaller subgroup and assume that there is a dense sub-algebra $\varinjlim_{k\in\N} V_k\subset\sC^{\an}(G_0,\Q_p)$, where each $V_k$ is a finite-dimensional subspace of $\sC^{\an}(G_0,\Q_p)$ stable under both the left and right translation actions of $G_0$. We will always view $V_k$ as a representation of $G_0$ using the \textit{left} translation action.

Let $\sC^{\an}(G_0,\Q_p)^o$ be the unit ball of $\sC^{\an}(G_0,\Q_p)$ with respect to the norm $\|\cdot\|_{G_0}$  (see \ref{2ndcoord} for the notation here). Then 
\[V_k^o:=V_k\cap \sC^{\an}(G_0,\Q_p)^o\]
is a $G_0$-stable lattice. Moreover, it follows from Lemma \ref{modpn+1trivial} that $V_k^o/p$ is fixed by an open subgroup $G_1$ of $G_0$ for all $k$. Now we apply proposition \ref{decompletion} to $V_k$, with $G=G_0$: fix a constant $c$ as in the proposition and $n\geq n(G_1)$. There is $D^+_{G_1,n}(V_k^o)\subset (B_\infty)^o\otimes V_k^o$ such that 
\[(B_\infty)^{\circ}\otimes_{B_{G_1,n}^{+}} D_{G_1,n}^{+}(V_k^o)=(B_\infty)^{\circ}\otimes_{\Z_p}V_k^o.\] 
By the uniqueness of $D_{G_1,n}^{+}(V_k^o)$, these $\{D_{G_1,n}^{+}(V_k^o)\}_k$ form a direct system. Hence we may take the direct limit of this equality over $k$:
\[(B_\infty)^{\circ}\otimes_{B_{G_1,n}^{+}} \varinjlim_k D_{G_1,n}^{+}(V_k^o)=(B_\infty)^{\circ}\otimes_{\Z_p}\varinjlim_k V_k^o.\] 
By taking the $p$-adic completion and inverting $p$, we obtain
\[B_\infty\widehat{\otimes}_{B_{G_1,n}} D_{G_1,n}=B_\infty\widehat{\otimes}_{\Q_p}\sC^{\an}(G_0,\Q_p).\]
Here $D^+_{G_1,n}$ is the $p$-adic completion of $\varinjlim_k D_{G_1,n}^{+}(V_k^o)$ and $D_{G_1,n}=D^+_{G_1,n}\otimes_{\Z_p}\Q_p$ equipped with the $p$-adic topology. The right hand side becomes $\sC^{\an}(G_0,\Q_p)$ as $\varinjlim_k V_k$ is dense inside of it.

Now we can construct an action of $\Lie(\Gamma)$ as before: there is an integer $m\geq 0$ such that $(\gamma-1)^mD^+_{G_1,n}(V_k^o)\subset pD^+_{G_1,n}(V_k^o)$ for all $k\geq0$ and $\gamma\in\Gamma$. Hence
\[(\gamma-1)^mD^+_{G_1,n}\subset pD^+_{G_1,n}\]
for any $\gamma\in\Gamma$. By \ref{exaZp}, the action of $\Gamma$ on $D_{G_1,n}$ is locally analytic. Its Lie algebra action can be extended $B_\infty$-linearly to an action on $B_\infty\widehat{\otimes}_{B_{G_1,n}} D_{G_1,n}=B_\infty\widehat{\otimes}_{\Q_p}\sC^{\an}(G_0,\Q_p)$:
\[\phi_{G_0}:\Lie(\Gamma)\to\End_{B_\infty}(B_\infty\widehat{\otimes}_{\Q_p}\sC^{\an}(G_0,\Q_p)).\]
Again this action uniquely extends the natural action of $\Lie(\Gamma)$ on the $\Gamma$-locally analytic vectors in $(B_\infty\widehat{\otimes}_{\Q_p}\sC^{\an}(G_0,\Q_p))^{G_0}$, where $G_0$ acts diagonally on $B_\infty$ and $\sC^{\an}(G_0,\Q_p)$. In fact, let $D_{G_0,n}^{+}=(D_{G_1,n}^{+})^{G_0}$ and $D_{G_0,n}=(D_{G_1,n})^{G_0}=D_{G_0,n}^{+}\otimes_{\Z_p}{\Q_p}$. Then by arguments similar to Lemma \ref{LAm} and the paragraph below it, we have the following lemma.
\end{para}

\begin{lem} \label{DG0n}
For $n$ sufficiently large, $D_{G_0,n}$ is the subspace of $p^n\Gamma$-analytic vectors in $(B_\infty\widehat{\otimes}_{\Q_p}\sC^{\an}(G_0,\Q_p))^{G_0}=(B_\infty)^{G_0-\an}$, and $B_\infty\widehat{\otimes}_{B_{G_0,n}} D_{G_0,n}=B_\infty\widehat{\otimes}_{\Q_p}\sC^{\an}(G_0,\Q_p)$.
 \end{lem}

The action of $\Lie(\Gamma)$ via $\phi_{G_0}$ commutes with $\Gamma$ as both actions commute on $D_{G_1,n}$. Moreover, it follows from the functorial property of $\phi_{V_k}$ in proposition \ref{relSenop} that the action of $\Lie(\Gamma)$ on $D_{G_1,n}^{+}(V_k^o)\otimes\Q_p$ commutes with the right translation action of $G_0$. By passing to the limit, we see that $\phi_{G_0}$ also commutes with the right translation action of $G_0$. Recall that the multiplication structure on $\sC^{\an}(G_0,\Q_p)$ induces maps $V_k\otimes V_l\to V_{k+l}$, cf. \ref{algdense}. Hence the last part of proposition \ref{relSenop} implies that for any $\theta\in \phi_{G_0}(\Lie(\Gamma)), f_1,f_2\in\sC^{\an}(G_0,\Q_p)$,
\[\theta(f_1)f_2+f_1\theta(f_2)=\theta(f_1f_2),\]
i.e. $\theta$ is a derivation. Hence we have proved:

\begin{prop}
$\phi_{G_0}(x)$ acts as a $G_0$-right-invariant derivation on $B_\infty\widehat{\otimes}_{\Q_p}\sC^{\an}(G_0,\Q_p)$ for any $x\in\Lie(\Gamma)$. Moreover it commutes with $\Gamma$.
\end{prop}

\begin{cor}
$\phi_{G_0}$ factors through $B\otimes_{\Q_p}\Lie(G_0)\subset\End_{B_\infty}(B_\infty\widehat{\otimes}_{\Q_p}\sC^{\an}(G_0,\Q_p))$. Here $\Lie(G_0)$ acts on $\sC^{\an}(G_0,\Q_p)$ by the infinitesimal action of the left translation of $G_0$, and we extend it $B_\infty$-linearly to an action of $B\otimes_{\Q_p}\Lie(G_0)$ on $B_\infty\widehat{\otimes}_{\Q_p}\sC^{\an}(G_0,\Q_p)$.
\end{cor}

\begin{proof}
Clearly $\phi_{G_0}$ factors through $B_\infty\otimes_{\Q_p}\Lie(G_0)$. As it commutes with $\Gamma$, it also factors through $(B_\infty\otimes_{\Q_p}\Lie(G_0))^{\Gamma}=B\otimes_{\Q_p}\Lie(G_0)$ by Lemma \ref{Gammainv}. 
\end{proof}

By abuse of notation, we also denote by $\phi_{G_0}:\Lie(\Gamma)\to B\otimes_{\Q_p}\Lie(G_0)$. Suppose we replace $G_0$ by a smaller subgroup $G_0'$ and consider the restriction of $\sC^{\an}(G_0,\Q_p)\to\sC^{\an}(G_0',\Q_p)$. It is easy to see that $\phi_{G_0}=\phi_{G_0'}$ once we identify $\Lie(G_0)=\Lie(G)=\Lie(G_0')$. Hence $\phi_{G_0}$ is independent of the choice of $G_0$ and we denote it by 
\[\phi_{\widetilde{X}}:\Lie(\Gamma)\to B\otimes_{\Q_p}\Lie(G).\]

\begin{cor} \label{pCRB}
The image of $\phi_{\widetilde{X}}$ acts trivially on the $G$-locally analytic vectors in $B$.
\end{cor}

\begin{proof}
Recall that (see \ref{2ndcoord})
\[B^{G_0-\an}=(B\widehat{\otimes}_{\Q_p}\sC^{\an}(G_0,\Q_p))^{G_0}=(B_\infty\widehat{\otimes}_{\Q_p}\sC^{\an}(G_0,\Q_p))^{G_0\times\Gamma},\]
where the second equality follows from Lemma \ref{Gammainv}. Hence by our construction of $\phi_{G_0}$, the action of $\Lie(\Gamma)$ is trivial on $B^{G_0-\an}$. An easy computation shows that this action is nothing but $\phi_{\widetilde{X}}$. Note that this argument works for all sufficiently small $G_0$, which clearly implies the claim in the corollary.
\end{proof}

Now if we fix a generator of $\Lie(\Gamma)$ and denote by $\theta_{\widetilde{X}}$ its image in $B\otimes_{\Q_p}\Lie(G)$ under $\phi_{\widetilde{X}}$, then we obtain the form claimed in the first part of Theorem \ref{pCR}.

Our next result implies that $\phi_{\widetilde{X}}$ is universal. Let $V$ be a finite dimensional representation of $G$ over $\Q_p$. The action of $G$ on $V$ is locally analytic, hence there is a natural $B_\infty$-linear action of $B_\infty\otimes_{\Q_p}\Lie(G)$ on $B_\infty\otimes_{\Q_p}V$.  Therefore $\phi_{\widetilde{X}}$ gives an action of $\Lie(\Gamma)$ on $B_\infty\otimes_{\Q_p}V$. On the other hand, we defined another action of $\Lie(\Gamma)$ called $\phi_V$ in Proposition \ref{relSenop}.

\begin{cor} \label{phiV=}
$\phi_V$ agrees with $\phi_{\widetilde{X}}$ for any $V$.
\end{cor}

\begin{proof}
Let $G_0$ be a sufficiently small open subgroup of $G$ as in the beginning of \ref{VsCan}. Moreover we may assume $V$ is a $G_0$-analytic representation. Let $V^*$ be the dual vector space of $V$. Then taking the matrix coefficients of $V$ induces a map:
\[m_V:V\otimes_{\Q_p}V^*\to\sC^{\an}(G_0,\Q_p),\]
\[m_V(v\otimes l)(g)=l(g^{-1}\cdot v),~v\in V,l\in V^*,g\in G_0.\]
This map is $G_0$-equivariant where $G_0$ acts only on the first factor of $V\otimes V^*$ and acts via the left translation on $\sC^{\an}(G_0,\Q_p)$. The induced map 
\[\mathbf{1}_{B_\infty}\otimes m_V:B_\infty\otimes_{\Q_p}V\otimes_{\Q_p}V^*\to B_\infty\widehat{\otimes}_{\Q_p}\sC^{\an}(G_0,\Q_p)\]
intertwines $\phi_{V\otimes V^*}=\phi_V\otimes \mathbf{1}_{V^*}$ and $\phi_{G_0}=\phi_{\widetilde{X}}$. One can check this on the $G_0$-fixed, $\Gamma$-locally analytic vectors. Trivially, this map also intertwines $\phi_{\widetilde{X}}\otimes \mathbf{1}_{V^*}$ and $\phi_{\widetilde{X}}$. Now for any non-zero $v\in V$, there exists $l\in V^*$ such that $m_V(v\otimes l)\neq 0$. From this, it's easy to see that $\phi_V$ has to factor through $\phi_{\widetilde{X}}$.
\end{proof}

\begin{rem} \label{computephi}
This corollary gives a ``practical way'' to compute $\phi_{\widetilde{X}}$: choose a faithful representation $V$ of an open subgroup of $G$. Then $\phi_V$ completely determines $\phi_{\widetilde{X}}$.
\end{rem}

\begin{rem} \label{Bla}
$B\otimes_{\Q_p}\Lie(G)$ acts naturally on the $G$-locally analytic vectors $(B\otimes_{\Q_p} V)^{\la}$ of $B\otimes_{\Q_p} V$. Hence $\phi_{\widetilde{X}}:\Lie(\Gamma)\to B\otimes_{\Q_p}\Lie(G)$ induces an action of $\Lie(\Gamma)$ on $(B\otimes_{\Q_p} V)^{\la}$. Combining Corollary \ref{pCRB} and \ref{phiV=}, it is easy to see that this action is nothing but $\phi_V$. 
\end{rem}

\begin{cor} \label{functoriality}
$\phi_{\widetilde{X}}$ is functorial in the pair $(B,G)$, i.e. suppose that $H$ is a closed normal  subgroup of $G$ such that $\widetilde{X}'=\Spa(B^H,(B^+)^H)$ is a ``log $G/H$-Galois pro-\'etale perfectoid covering'' of $X$ as in \ref{setup}, then $\phi_{\widetilde{X}'}:\Lie(\Gamma)\to B^H\otimes_{\Q_p}\Lie(G/H)\subset B\otimes_{\Q_p}\Lie(G/H)$ can be identified with the composite 
\[\Lie(\Gamma)\xrightarrow{\phi_{\widetilde{X}}} B\otimes_{\Q_p}\Lie(G)\xrightarrow{\mod B\otimes_{\Q_p}\Lie(H)}B\otimes_{\Q_p}\Lie(G/H).\]
\end{cor}

\begin{proof}
It is enough to check that the formulation of $\phi_V$ in Proposition \ref{relSenop}  is functorial in a similar sense. But this basically follows from the construction. We omit the details here.
\end{proof}

We mention the following results concerning the uniqueness of our differential equation. Note that $\widetilde{X}_\infty=\Spa(B_\infty,B_\infty^+)$ is a  ``log $G\times\Gamma$-Galois pro-\'etale perfectoid covering'' of $X$. Let $B_\infty^\la$ be the $G\times\Gamma$-locally analytic vectors in $B_\infty$. 

\begin{prop} \label{nvsh}
If $D\in B_\infty\otimes_{\Q_p}\Lie(G)$ and $D(v)=0$ for any $v\in B_\infty^\la$, then $D=0$.
\end{prop}

\begin{proof} 
We will freely use notation introduced in \ref{VsCan}. Let $G_0$ be a compact open subgroup of $G$ considered there. By Lemma \ref{DG0n}, we have an isomorphism for sufficiently large $n$
\[B_\infty\widehat{\otimes}_{B_{G_0,n}} (B_\infty)^{G_0-\an,p^n\Gamma-\an}=B_\infty\widehat{\otimes}_{\Q_p}\sC^{\an}(G_0,\Q_p).\]
It follows from  \ref{Gnan}  that this is equivariant with respect to the following actions of $G_0$: the natural action on $(B_\infty)^{G_0-\an,p^n\Gamma-\an}$, the right translation action on $\sC^{\an}(G_0,\Q_p)$ and trivial actions on both $B_\infty$. Using this action of $G_0$, we get $B_\infty$-linear actions of $B_\infty\otimes_{\Q_p}\Lie(G)$ on both sides. In particular, $D$ annihilates both sides. However if $D\neq 0$, we can always find a $B_\infty$-valued analytic function on $G_0$ which is not annihilated by $D$. Contradiction. This proves our claim.
\end{proof}

\begin{cor} \label{ud}
Suppose $D\in B_\infty\otimes_{\Q_p}\Lie(G\times\Gamma)$ annihilates $B_\infty^\la$. Then $D$ can be written as
\[D=bd,\]
where $b\in B_\infty$ and $d\in\phi_{\widetilde{X}\infty}(\Lie(\Gamma))$.
\end{cor}

\begin{proof}
We note that by the functorial property \ref{functoriality}, $\phi_{\widetilde{X}_\infty}\equiv \phi_{{X}_\infty}\mod B_\infty\otimes_{\Q_p}\Lie(G)$, hence the composite of $\phi_{\widetilde{X}_\infty}$ and the projection map
\[\Lie(\Gamma)\xrightarrow{\phi_{\widetilde{X}_\infty}}B_\infty\otimes_{\Q_p}\Lie(G\times\Gamma)\to B_\infty\otimes_{\Q_p}\Lie(\Gamma)\]
is simply the identity map. In particular, we can find  $b\in B_\infty$ and $d\in\phi_{\widetilde{X}\infty}(\Lie(\Gamma))$ such that $D-bd\in B_\infty\otimes_{\Q_p}\Lie(G)$. The rest follows from  Corollary \ref{pCRB} and Proposition \ref{nvsh}.
\end{proof}

\subsection{Proof of the main result II: uniqueness}
\begin{para}
In the previous section, we proved the existence of $\theta$ in Theorem \ref{pCR}. Note that the construction depends on a choice of an \'etale map $f_1:X\to Y=\mathbb{T}^1$ or $\mathbb{B}^1$ in \ref{smallY}. In this section, we will show that in fact $\theta$ is well-defined up to $A^\times$.  People who are interested in global applications can skip reading this part as this part will not play any role later. 
\end{para}

\begin{para} \label{f1f2}
Consider another \'etale  map  $f_2:X\to Y$ which factors as a composite of rational embeddings and finite \'etale maps. Let
\[\widetilde{X}':=X\times_{f_2,Y}Y_\infty\]
be the pull-back of the profinite covering $Y_\infty$ along $f_2$. See \ref{smallY} for the definition of $Y_\infty$. In the notation below, we put a superscript $'$ for everything constructed using $f_2$ instead of $f_1$.
For example, $\widetilde X'=\Spa(B',B'^+)$ is an affinoid  ``log $G'$-Galois pro-\'etale perfectoid covering'' of $X$ with $G'=\Gamma$. We would like to compute $\phi_{\widetilde{X}'}:\Lie(\Gamma)\to B'\otimes_{\Q_p}\Lie(G')=B'\otimes_{\Q_p}\Lie(\Gamma)$ first.  In general, we will use the functorial property to reduce to this computation.

We will freely use the notation introduced in the previous subsections. In particular, $\widetilde{X}_\infty'=\Spa(B_\infty',B_\infty'^+)$ is a $G'\times\Gamma$-covering of $X$. It is also a $G'$-covering of $X_\infty=\Spa(A_\infty,A_\infty^+)$ and a $\Gamma$-covering of $\widetilde{X}'=\Spa(B',B'^+)$. Note that by Abhyankar's lemma, both coverings are profinite \'etale of perfectoid algebras. By the Hochschild-Serre spectral sequence, we have

\[H^1_{\cont}(G'\times\Gamma,B_\infty')\cong H^1_{\cont}(\Gamma,(B_\infty')^{G'\times{\{\mathbf{1}\}}})=H^1_{\cont}(\Gamma,A_\infty)\cong H^1_{\cont}(\Gamma,A).\]
The first isomorphism follows from $H^1_{\cont}(G'\times\{\mathbf{1}\},B_\infty')=0$, which is a consequence of the almost purity theorem. The last isomorphism follows from the existence and properties of the Tate normalized trace $\overline\tr_{X,0}:A_\infty\to A$. We denote by $\tau_1$ the composite isomorphism $H^1_{\cont}(G'\times\Gamma,B_\infty')\xrightarrow{\sim}H^1_{\cont}(\Gamma,A)$.

Symmetrically, we can get another isomorphism $\tau_2:H^1_{\cont}(G'\times\Gamma,B_\infty')\xrightarrow{\sim}H^1_{\cont}(\Gamma,A)$ by 
\[H^1_{\cont}(G'\times\Gamma,B_\infty')\cong H^1_{\cont}(G',(B_\infty')^{{\{\mathbf{1}\}}\times\Gamma})=H^1_{\cont}(G',B')=H^1_{\cont}(\Gamma,B')\cong H^1_{\cont}(\Gamma,A).\]
As $H^1_{\cont}(\Gamma,A)$ is a free $A$-module of rank one, the composite $\tau_1\circ\tau_2^{-1}\in\End_A(H^1_{\cont}(\Gamma,A))=A$ is given by an element $a\in A^\times$. 
\end{para}

\begin{lem}
$\phi_{\widetilde{X}'}:\Lie(\Gamma)\to B'\otimes_{\Q _p}\Lie(\Gamma)$ is multiplication by $a$.
\end{lem}

\begin{proof}
Fix a topological generator $\gamma$ of $\Gamma$. Unravelling all the isomorphisms above, we can find an element $b\in B_\infty'$ such that
\begin{eqnarray} 
(\gamma,\mathbf{1})\cdot b-b&=&-1\\ \label{abaction}
(\mathbf{1},\gamma)\cdot b-b&=&a.
\end{eqnarray}
Here $(\gamma,\mathbf{1}), (\mathbf{1},\gamma)$ are elements in $G'\times\Gamma=\Gamma\times\Gamma$. In particular, both actions of $\Gamma$ on $b$ are analytic.

As pointed out in Remark \ref{computephi}, in order to compute $\phi_{\widetilde{X}'}$, we choose a faithful representation of $G'=\Gamma$ on $V=\Q_p^{\oplus 2}$:
\[\Gamma\to\GL_2(\Q_p),\, \gamma\mapsto \begin{pmatrix} 1 & 1 \\ 0 &1 \end{pmatrix}.\]
Then $B_\infty'\otimes_{\Q_p} V=(B_\infty')^{\oplus 2}$ has a $G'$-fixed basis over $B_\infty'$: 
\[e_1=(1,0),e_2=(b,1).\]
Moreover it is easy to see that the action of $\Gamma$ on this basis is (locally) analytic. Therefore $\phi_V$ is obtained by $B_\infty'$-linearly extending the action of $\Lie(\Gamma)$ on $e_1,e_2$. Let $t\in\Lie(\Gamma)$ be the logarithm of $\gamma$. Then $t\cdot e_1=0$ and it follows from \eqref{abaction} that $t\cdot e_2=ae_1$.  Hence,
\[\phi_V:\Lie(\Gamma)\to B_\infty'\otimes_{\Q_p}\mathfrak{gl}_2(\Q_p)\]
\[t\mapsto \begin{pmatrix} 0 & a \\ 0 &0 \end{pmatrix}.\]
Comparing this with the definition of $V$, we see that $\phi_{\widetilde{X}'}$ is just multiplication by $a$.
\end{proof}

\begin{para}
Now we are ready to prove $\phi_{\widetilde{X}}$ is well-defined up to $A^\times$. We will keep using the notation introduced in \ref{f1f2}. Suppose $\widetilde{X}$ is ``log $G$-Galois pro-\'etale perfectoid covering'' of $X$. Now using $f_2:X\to Y$ instead of $f_1:X\to Y$ in the setup \ref{smallY} and redoing everything before, then rather than $\phi_{\widetilde{X}}$, we get
\[\phi'_{\widetilde{X}}:\Lie(\Gamma)\to B\otimes_{\Q_p}\Lie(G).\]
\end{para}

\begin{prop}
$a\phi'_{\widetilde{X}}(x)=\phi_{\widetilde{X}}(x)$, for any $x\in\Lie(\Gamma)$. In particular, $\phi'_{\widetilde{X}}$ and $\phi_{\widetilde{X}}$ differ by a unit of $A$.
\end{prop}

\begin{proof}
First note that the special case $\widetilde{X}=X_\infty=\Spa(A_\infty,A_\infty^+)$ is essentially proved by the same computation as above (after switching the role of $f_1$ and $f_2$). In general, let $V$ be a finite-dimensional continuous representation of $G$ over $\Q_p$. It is enough to show that $a^{-1}\phi_V$ on $B_\infty\otimes_{\Q_p} V$ agrees with the action of $\Lie(\Gamma)$ induced from $\phi'_{\widetilde{X}}$. Consider $\widetilde{X}_{\infty}:=\Spa(B_\infty,B_\infty^+)$. This is a ``log $G\times\Gamma$-Galois pro-\'etale perfectoid covering'' of $X$. Then by the functorial property \ref{functoriality}, 
\[\phi'_{\widetilde{X}}\equiv \phi'_{\widetilde{X}_\infty} \mod B_\infty\otimes_{\Q_p}\Lie(\Gamma).\]
Hence the actions of $\Lie(\Gamma)$ on $B_\infty\otimes_{\Q_p} V$ induced from $\phi'_{\widetilde{X}}$ and $\phi'_{\widetilde{X}_\infty}$ are the same because $\Gamma$ acts trivially on $V$. So it suffices to compare $a^{-1}\phi_V$ and $\phi'_{\widetilde{X}_\infty}$.

Since both actions are $B_\infty$-linear, it follows from Proposition \ref{decompletion} that we only need to compare two actions on the $\Gamma$-locally analytic vectors in $(B_\infty\otimes_{\Q_p} V)^{G}$. Now on this $G$-fixed subspace, $\phi_V$ acts by the natural Lie algebra action of $\Lie(\Gamma)$, while by Remark \ref{Bla} and the functorial property \ref{functoriality}
\[\phi'_{\widetilde{X}_\infty}\equiv\phi'_{X_\infty}\mod B_\infty\otimes_{\Q_p}\Lie(G),\]
the action on $(B_\infty\otimes_{\Q_p} V)^{G,\Gamma-\la}$ induced  by $\phi'_{\widetilde{X}_\infty}$ is just $\phi'_{X_\infty}$. Hence the desired equality is a direct consequence of the special case $\widetilde{X}=X_\infty$.
\end{proof}

\begin{rem}
One can also use Corollary \ref{ud} to reduce the general case to the special case $\widetilde{X}=X_\infty$.
\end{rem}

\subsection{Locally analytic covering}
\begin{para} \label{adhocFE}
The goal of this subsection is to give a sufficient condition for $\theta$ to be non-zero. We will continue using the notation introduced before. In particular, there is a fixed \'etale map $f_1:X\to Y$ as in \ref{smallY}. First, let's recall Faltings's extension in the rigid analytic variety setting, cf. Corollary 6.14 of \cite{Sch13}, Corollary 2.4.5 of \cite{DLLZ2}. However as both references assume $X$ is defined over a discretely valued complete non-archimedean extension of $\Q_p$, we give a rather \textit{ad hoc} definition here which will be sufficient for our purpose.

Fix a generator $\gamma$ of $\Gamma$ from now on and consider the following unipotent representation of $\Gamma$ on $V=\Q_p^{\oplus 2}$.
\[\Gamma\to\GL_2(\Q_p):~\gamma\mapsto \begin{pmatrix} 1 & 1 \\ 0 &1 \end{pmatrix}.\]
Clearly $V$ is an extension of the trivial representation by itself:
\[0\to\Q_p\to V\to \Q_p\to0.\]
Tensor this exact sequence with $B_\infty$ over $\Q_p$ and take $\Gamma$-invariants with respect to the diagonal actions:
\[
0\to B\to (B_\infty\otimes_{\Q_p} V)^{\Gamma}\to B\to 0. \label{FE} \tag{FE}
\]
Note that this $G$-equivariant sequence is exact by the almost purity theorem. It follows from the discussion in \ref{f1f2} that its extension class is independent of the choice of $f_1$ up to multiplication by a unit of $A$.

The norm on $B_\infty$ induces norms on $B_\infty\otimes V=(B_\infty)^{\oplus 2}$ and its subspaces, and \eqref{FE} is continuous with respect to this topology. One important property of \eqref{FE} is that if we take the continuous $G$-cohomology, the following connecting homomorphism of the long exact sequence is an isomorphism:
\[A=B^G\xrightarrow{\sim} H^1_{\cont}(G,B).\] 
This can be seen by identifying 
\[H^i_{\cont}(G,B)\cong H^i_{\cont}(G\times\Gamma,B_\infty)\cong H^i_{\cont}(\Gamma, A_\infty)\cong H^i_{\cont}(\Gamma, A),\]
where the first two isomorphisms follow from the almost purity theorem  and the last isomorphism is a consequence of the Tate's normalized trace. Then it is easy to compute the connecting homomorphism.

In fact, this holds for any open subgroup $G_0$ of $G$, i.e. if we take the continuous $G_0$-cohomology of \eqref{FE}, the connecting homomorphism $B^{G_0}\to H^1_{\cont}(G_0,B)$ is an isomorphism. To see this, note that $\Spa(B,B^+)$ is a ``log $G_0$-Galois pro-\'etale perfectoid covering'' of $\Spa(B^{G_0},(B^+)^{G_0})$. Hence $H^1_{\cont}(G_0,B)$ is also a free rank-one $B^{G_0}$-module. So in view of what we have proved, the claim follows from the fact that the natural map $B^{G_0}\otimes_{A} H^1_{\cont}(G,B)\to H^1_{\cont}(G_0,B)$ is an isomorphism.
\end{para}

\begin{rem} \label{relFE}
If $X$ is defined over a finite extension of $\Q_p$, then \eqref{FE} recovers (log) Faltings's extension on $\tilde{X}$ defined in \cite{Sch13,DLLZ2} by identifying the quotient $B$ with $B\otimes_{A}\Omega_{A/C}^{1}(S)$ sending $1\in B$ to $f_1^*(\frac{dT}{T})$, where $\Omega_{A/C}^{1}(S)$ denotes the continuous $1$-forms of $A$ over $C$ with a simple pole at $S$. See Proposition 2.3.15 and proof of Corollary 2.4.2 of \cite{DLLZ2}. Here $\tilde{X}$ is viewed as an open set in the pro-Kummer \'etale site of $X$ equipped with the natural log structure defined by $S$ (cf. Example 2.1.2 of \cite{DLLZ2}).
\end{rem}

\begin{prop} \label{LAETFAE}
The following conditions are equivalent:
\begin{enumerate}
\item \eqref{FE} remains exact after taking $G_0$-analytic vectors for some open subgroup $G_0$ of $G$ equipped with an integer valued, saturated $p$-valuation as in \ref{2ndcoord}.\label{G0anaex}
\item \eqref{FE} remains exact after taking $G$-locally analytic vectors of each term. \label{Glaex}
\item There exists a $G$-locally analytic vector $z\in B_\infty$ such that $\gamma(z)=z-1$.  \label{exz}
\end{enumerate}
\end{prop}

\begin{proof}
The first two parts are clearly equivalent with existence of a $G$-locally analytic vector $x\in (B_\infty\otimes_{\Q_p} V)^{\Gamma}$ mapping to $1\in B$ in \eqref{FE}. Write $B_\infty\otimes_{\Q_p} V=(B_{\infty})^{\oplus 2}$. Then $x=(z,1)$ for some $z\in B_\infty$ which is $G$-locally analytic and satisfies $\gamma(z)=z-1$. This proves the equivalence between these three conditions.
\end{proof}

\begin{defn} \label{laedefn}
We say $\widetilde{X}$ is a locally analytic covering of $X$ if one of the conditions in Proposition \ref{LAETFAE} holds. 
\end{defn}

\begin{rem}
An analysis similar to  \ref{f1f2} shows that this definition is independent of the choice of $f_1:X\to Y$.  There is another intrinsic definition:  $\widetilde{X}$ is a locally analytic covering of $X$ if and only if  the natural map $H^1_{\cont}(G,B^{\la})\to H_\cont^1(G,B)$ is an isomorphism. We sketch a proof here. It follows from Tate's normalized trace that $H^1_\cont(\Gamma,(A_\infty)^\la)\cong H^1(\Gamma,A_\infty)$.
Note that $A_\infty=(B_\infty)^{G}$. We claim that there are natural isomorphisms
\[H^1_\cont(G\times\Gamma,(B_\infty)^{\la})\cong H^1_\cont(\Gamma,(A_\infty)^\la)\cong H^1_\cont(\Gamma,A_\infty)\cong H^1_\cont(G\times\Gamma, B_\infty).\]
The first and the third isomorphisms follow from the Hochschild-Serre spectral sequence and $H^1_\cont(G,B_\infty)=H^1_\cont(G,(B_\infty)^{\la})=0$. The vanishing of $H^1_\cont(G,B_\infty)$ is clear by the almost purity theorem. For the vanishing of $H^1_\cont(G,(B_\infty)^{\la})$, using Lemma \ref{DG0n}, it suffices to show that $\varinjlim_nH^1(G_n,\sC^{\an}(G_n,\Q_p))=0$, which is true. Now we apply  the Hochschild-Serre spectral sequence to the subgroup $\{\mathbf{1}\}\times\Gamma\subset G\times\Gamma$ in the computation of  $H^1_\cont(G\times\Gamma,(B_\infty)^{\la})$ and $H^1_\cont(G\times\Gamma, B_\infty)$. Since $H^1_\cont(\Gamma,B_\infty)=0$  with , it is easy to see that $H^1_{\cont}(G,B^{\la})\cong H_\cont^1(G,B)$ is equivalent with 
\[H^1(\Gamma,(B_\infty)^\la)^G=0,\]
i.e. the $\Gamma$-coinvariants of $(B_\infty)^\la$ has no $G$-invariants. Note that $1\in (B_\infty)^\la$ is  fixed by $G$. Hence the vanishing of $H^1(\Gamma,(B_\infty)^\la)^G$ certainly implies the existence of element $z$ in the third part of Proposition \ref{LAETFAE}. On the other hand, if there exists such an element $z$, the argument in \ref{elex} below will imply that $H^1(\Gamma,(B_\infty)^\la)=0$. This proves the proof.
\end{rem}

\begin{prop} \label{nonvantheta}
$\phi_{\widetilde{X}}:\Lie(\Gamma)\to B\otimes_{\Q_p}\Lie(G)$ is non-zero if $\widetilde{X}$ is a locally analytic covering of $X$.
\end{prop}

\begin{proof}
We need results established in \ref{VsCan}. Let $G_0$ an open subgroup as in \ref{VsCan}. Suppose $\phi_{\widetilde{X}}=0$. It follows from our construction there and Lemma \ref{DG0n} that $\Lie(\Gamma)$ is acting trivially on the $\Gamma$-locally analytic vectors in $(B_\infty)^{G_0-\an}$. Hence there does not exist an element $z\in (B_\infty)^{G_0-\an}$ such that $\gamma(z)=z-1$. As $G_0$ can be taken arbitrarily small, part \eqref{exz} of Proposition \ref{LAETFAE} implies that $\widetilde{X}$ is not a locally analytic covering of $X$.
\end{proof}


\subsection{Application: acyclicity of taking locally analytic vectors of  \texorpdfstring{$B$}{Lg}}
We keep the notation and setup from \ref{setup}. Our main result of this subsection gives an equivalent condition for $B$ being $\mathfrak{LA}$-acyclic. See \ref{LAacyclic} for more details about this notion. 

\begin{thm} \label{BLAacyc}
Suppose $\widetilde X=\Spa(B,B^+)$ is a ``log $G$-Galois pro-\'etale perfectoid covering'' of $X$ and $X$ is small. Then $R^i\mathfrak{LA}(B)=0$ for all $i\geq 1$ if and only if $\widetilde X$ is a locally analytic covering of $X$ (see definition \ref{laedefn}). When this happens, $B$ is strongly $\mathfrak{LA}$-acyclic.
\end{thm}

\begin{para}
One direction is clear. Suppose $R^i\mathfrak{LA}(B)=0$, $i\geq 1$. Take completed tensor products of \eqref{FE} with $\sC^{\an}(G_n,\Q_p)$ over $\Q_p$ and take the continuous $G_n$-cohomology:
\[0\to B^{G_n-\an}\to (B_\infty\otimes_{\Q_p} V)^{\Gamma,G_n-\an}\to B^{G_n-\an}\to H^1_{\cont}(G_n, \sC^{\an}(G_n,\Q_p)\widehat\otimes_{\Q_p} B).\]
Then by passing to the direct limit with $n\to\infty$, we see that the last term vanishes. Hence \eqref{FE} remains exact when taking $G$-locally analytic vectors.
\end{para}

\begin{para}
The proof of the other direction goes as follows: first  $H^i_{\cont}(G_n,B\widehat{\otimes}_{\Q_p}\sC^{\an}(G_n,\Q_p))$ will be identified with a certain group cohomology of $\Gamma$, then we construct explicit elements to kill these cohomology groups. In particular, it implies strongly $\mathfrak{LA}$-acyclicity. We remark that this is a standard technique in the theory of $(\varphi,\Gamma)$-modules for studying Galois cohomology of $p$-adic Galois representations.

Keep the notation introduced in \ref{smallY}. Let $G_0$ be an open subgroup considered in \ref{VsCan}. We can consider $H^i_{\cont}(G_0\times \Gamma,B_\infty\widehat{\otimes}_{\Q_p}\sC^{\an}(G_0,\Q_p))$. On the one hand, by the almost purity theorem, $H^i_{\cont}(\Gamma,B_\infty^+\otimes_{\Z_p} \sC^{\an}(G_0,\Q_p)^\circ/p)$ almost vanishes for $i\geq1$, hence 
\[H^i_{\cont}(\Gamma,B_\infty\widehat{\otimes}_{\Q_p}\sC^{\an}(G_0,\Q_p))=0.\]
 Applying the Hochschild-Serre spectral sequence to $\{\mathbf{1}\}\times\Gamma \subset G_0\times \Gamma$ and using Lemma \ref{Gammainv}, we get
\[H^i_{\cont}(G_0,B\widehat{\otimes}_{\Q_p}\sC^{\an}(G_0,\Q_p))\xrightarrow{\sim}H^i_{\cont}(G_0\times \Gamma,B_\infty\widehat{\otimes}_{\Q_p}\sC^{\an}(G_0,\Q_p)).\]

On the other hand, if we apply Hochschild-Serre spectral sequence to $G_0\times\{\mathbf{1}\} \subset G_0\times \Gamma$ and note that $(B_\infty\widehat{\otimes}_{\Q_p}\sC^{\an}(G_0,\Q_p))^{G_0}=(B_\infty)^{G_0-\an}$, we get
\[H^i_{\cont}(\Gamma,(B_\infty)^{G_0-\an})\xrightarrow{\sim}H^i_{\cont}(G_0\times \Gamma,B_\infty\widehat{\otimes}_{\Q_p}\sC^{\an}(G_0,\Q_p)).\]
All these isomorphism are functorial in $G_0$. Therefore,
\end{para}

\begin{lem}
There is a natural isomorphism 
\[H^i_{\cont}(G_0,B\widehat{\otimes}_{\Q_p}\sC^{\an}(G_0,\Q_p))\cong H^i_{\cont}(\Gamma,(B_\infty)^{G_0-\an})\]
functorial in $G_0$.
\end{lem}

\begin{para}
From this, it is clear that $H^i_{\cont}(G_0,B\widehat{\otimes}_{\Q_p}\sC^{\an}(G_0,\Q_p))=0,i\geq 2$ as $\Gamma$ is one-dimensional. So we assume $i=1$ from now on. Our next step is to replace $(B_\infty)^{G_0-\an}$ on the right hand side by a smaller space. Let $D_{G_0,n}$ be the subspace of $p^n\Gamma$-analytic vectors in $(B_\infty)^{G_0-\an}$, i.e. $(B_\infty)^{G_0\times p^n\Gamma-\an}$. It follows from Lemma \ref{DG0n} that for sufficiently large $n$, 
\[(B_\infty)^{G_0-\an}=B_{G_0,\infty}\widehat{\otimes}_{B_{G_0,n}}D_{G_0,n}.\]
Moreover, let $D_{G_0,n}^{+}=B_\infty^{+}\cap D_{G_0,n}$, then by our construction in \ref{VsCan}
\[(\gamma-1)^m(D_{G_0,n}^{+})\subset pD_{G_0,n}^{+}\]
for some $m$ depending on $n$. Fix such $n$ and $m$ for the moment. Recall that from results in \ref{BCtotildeX}, for $k$ large enough, there is Tate's normalized trace $\overline{\tr}_{X_{G_0,k}}:B_{G_0,\infty}\to B_{G_0,k}$ and $\gamma-1$ is invertible on $\ker(\overline{\tr}_{X_{G_0,k}})$ with the norm of its inverse $\|(\gamma-1)^{-1}\|<p^{\frac{1}{2m}}$. We claim that $H^1_{\cont}(\Gamma,\ker(\overline{\tr}_{X_{G_0,k}})\widehat{\otimes}_{B_{G_0,n}}D_{G_0,n})=0$ for such $k$. Assuming this, as $D_{G_0,k}=B_{G_n,k}\otimes_{B_{G_0,n}}D_{G_0,n}$, we get
\[H^1_{\cont}(\Gamma,(B_\infty)^{G_0-\an})=H^1_{\cont}(\Gamma,D_{G_0,k})=H^1_{\cont}(\Gamma,(B_\infty)^{(G_0\times p^k\Gamma)-\an})\]
for $k$ sufficiently large. Therefore,
\[R^1\mathfrak{LA}(B)=\varinjlim_{n,k}H^1_{\cont}(\Gamma,(B_\infty)^{(G_n\times p^k\Gamma)-\an})\]

The vanishing of $H^1_{\cont}(\Gamma,\ker(\overline{\tr}_{X_{G_0,k}})\widehat{\otimes}_{B_{G_0,n}}D_{G_0,n})$ is a consequence of the following easy lemma.
\end{para}

\begin{lem}
For any $a\in \ker(\overline{\tr}_{X_{G_0,k}})^{+}:=\ker(\overline{\tr}_{X_{G_0,k}})\cap B_\infty^{+}$ and $b\in D_{G_0,n}^{+}$, we have $pa\otimes b\in (\gamma-1)(\ker(\overline{\tr}_{X_{G_0,k}})^+\widehat{\otimes}_{B_{G_0,k}^{+}}D_{G_0,n}^{+})$.
\end{lem}

\begin{proof}
Let $c=(\gamma-1)^{-1}(pa)\in \ker(\overline{\tr}_{X_{G_0,k}})^{+}$. Consider the series
\[\sum^{+\infty}_{l=0} (\gamma^{-1}-1)^{-l}(c)\otimes (\gamma-1)^{l}(b)=\sum^{+\infty}_{l=0} \gamma^l(1-\gamma)^{-l}(c)\otimes (\gamma-1)^{l}(b)\]
which by our assumption converges to an element $x\in \ker(\overline{\tr}_{X_{G_0,k}})^+\widehat{\otimes}_{B_{G_0,k}^{+}}D_{G_0,n}^{+}$. A direct computation gives $(\gamma-1)(x)=pa\otimes b$. Indeed, let 
\[y_l=(\gamma^{-1}-1)^{-l}(c), z_l=(\gamma-1)^l(b).\]
Then $(\gamma-1)(y_l)=-\gamma(y_{l-1}), (\gamma-1)(z_l)=z_{l+1}$. Hence
\[(\gamma-1)(y_l\otimes z_l)=(\gamma-1)(y_l)\otimes z_l+\gamma(y_l)\otimes(\gamma-1)(z_l)=
-\gamma(y_{l-1})\otimes z_l+\gamma(y_l)\otimes z_{l+1}.\]
When taking the summation of $l$ from $0$ to $+\infty$, only $-\gamma(y_{-1})\otimes z_0$ is not cancelled out. But this is just $(\gamma-1)(c)\otimes b=pa\otimes b$.
\end{proof}

\begin{para} \label{elex}
Now it suffices to show that for any element $x\in (B_\infty)^{(G_0\times p^k\Gamma)-\an}$, we can find $n>0$ and $l\geq k$ such that 
\[x\in(\gamma-1)((B_\infty)^{(G_n\times p^l\Gamma)-\an}).\]
Fix such $k,x$. Let $\phi_{\gamma}\in \Lie(\Gamma)$ be the logarithm of $\gamma$. Then by Lemma \ref{LieBound}, there exists a constant $C>0$ such that
\[\|\phi_\gamma(x)\|_{G_0\times p^k\Gamma}\leq C \|x\|_{G_0\times p^k\Gamma}.\]
Recall that $\|\cdot\|_{G_0\times p^k\Gamma}$ is a norm on the $G_0\times p^k\Gamma$-analytic vectors introduced in \ref{2ndcoord}.

Recall that we assume $\widetilde{X}$ is a locally analytic covering of $X$. Hence by \eqref{exz} of Proposition \ref{LAETFAE}, we may find an element $z\in (B_\infty)^{\la}$ such that $\gamma(z)=z-1$. By our construction of $B_\infty$ in \ref{BCtotildeX}, $\bigcup_{n,l} B_{G_n,l}$ is dense in $B_\infty$. Hence there exists an element $z_0\in B_{G_n,l}$ for some $n,l$ such that $z\in (B_\infty)^{(G_n\times p^l\Gamma)-\an}$ and the norm of $z-z_0$ (as an element in $B_\infty$) satisfies
\[\|z-z_0\|\leq \frac{1}{2Cp^{1/(p-1)}}.\]
Enlarging $n,l$ if necessary, we may assume $\|z-z_0\|_{G_n\times p^l\Gamma}=\|z-z_0\|$ by Lemma \ref{LieBound}. Now consider the series
\[-\sum_{m=0}^{+\infty}\frac{\phi_{\gamma}^{(m)}(x)}{(m+1)!}(z-z_0)^{m+1}.\]
It is easy to see that $\|\frac{\phi_\gamma^{(m)}(x)}{(m+1)!}(z-z_0)^{m+1}\|_{G_n\times p^l\Gamma}\leq \frac{\|x\|_{G_0\times p^k\Gamma}}{2^{m+1}C}$. So this series converges (with respect to $\|\cdot\|_{G_n\times p^l\Gamma}$) to an element $y\in (B_\infty)^{(G_n\times p^l\Gamma)-\an}$. Note that $\phi_{\gamma}(z-z_0)=\phi_{\gamma}(z)=-1$. A direct computation shows
\[\phi_{\gamma}(y)=x.\]
After replacing $l$ by a larger integer, we may assume $(\gamma-1)^m(D_{G_n,l}^+)\subset pD_{G_n,l}^+$ for some $m>0$ by our construction in \ref{VsCan}. Recall that $D_{G_n,l}^+=(B_\infty)^{(G_n\times p^l\Gamma)-\an}\cap B_\infty^+$. From this and $\phi_{\gamma}(y)=x$, we conclude $x\in (\gamma-1)((B_\infty)^{(G_n\times p^l\Gamma)-\an})$ by the following simple lemma.
\end{para}

\begin{lem}
Let $M$ be a unitary $\Q_p$-Banach representation of $\Gamma$ and $M^o$ be its unit ball. Suppose $(\gamma-1)^mM^o \subset pM^o$ for some $m\geq 0$. Then $\phi_{\gamma}(M)\subset (\gamma-1)(M)$.
\end{lem}

\begin{proof}
The argument here was pointed out by a referee and is much simpler than my previous argument. One only needs to observe that $\phi_\gamma$ is the logarithm of $\gamma$, hence
\[\phi_\gamma=(\gamma-1)\sum_{m=0}(-1)^m\frac{(\gamma-1)^{m}}{m+1}\]
and the second series converges to an operator on $M$ by our assumption.


\end{proof}

\begin{rem}
The argument of this subsection is basically the same as the proof of Th\'eor\`eme 3.4 of \cite{BC16}. Note that in their setting, the assumption  ``locally analytic'' always holds by Lemme 3.6 of \cite{BC16}.
\end{rem}

For later applications, it will be useful to remove the smallness assumption on $X$ in Theorem \ref{BLAacyc}. Unfortunately, we have to make some assumption in order to do this. Let $\widetilde X=\Spa(B,B^+)$ be a ``log $G$-Galois pro-\'etale perfectoid covering'' of $X=\Spa(A,A^+)$ as in  \ref{setup} except that we don't require $X$ to be small anymore. Denote by $\pi:\widetilde X \to X$ and $\widetilde{\cO}=\pi_*\cO_{\widetilde X}$. Then we can consider subsheaf $\widetilde{\cO}^{\la}\subset \widetilde{\cO}$ of $G$-locally analytic sections and subsheaves $\widetilde{\cO}^{n}\subset \widetilde{\cO}$ of $G_n$-analytic sections. Clearly $\varinjlim_n \widetilde{\cO}^{n}=\cO^{\la}$.

\begin{cor} \label{LAacyw/osm}
Suppose $\widetilde X=\Spa(B,B^+)$ is a ``log $G$-Galois pro-\'etale perfectoid covering'' of $X=\Spa(A,A^+)$ and $X$ can be covered by small rational open subsets $X_i,i=1,\cdots,k$, whose preimage $\widetilde{X}_i$ in $\tilde X$ is a locally analytic covering of $X_i$. Then 
\begin{enumerate}
\item $R^i\mathfrak{LA}(B)=H^i(X,\widetilde{\cO}^{\la})=\check{H}^i(X,\widetilde{\cO}^{\la})$ for any $i$, where $\check{H}^i(X,\widetilde{\cO}^{\la})$ denotes the \v{C}ech cohomology. In particular, $B$ is $\mathfrak{LA}$-acyclic if and only if 
\[\check{H}^i(X,\widetilde{\cO}^{\la})=0,\,i\geq 1.\]
\item Write $\mathfrak{U}_0=\{X_1,\cdots,X_k\}$. For any sheaf $F$ on $X$, denote by $\check{H}^i(\mathfrak{U}_0,F)$ the \v{C}ech cohomology of $F$ with respect to the cover $\mathfrak{U}_0$.  Suppose the direct system $\{\check{H}^i(\mathfrak{U}_0,\widetilde{\cO}^{n})\}_{n}$ is essentially zero for any $i>0$. Then $B$ is strongly $\mathfrak{LA}$-acyclic.
\end{enumerate}
\end{cor}

\begin{proof}
The argument here is a standard application of \v{C}ech cohomology. Let $\mathfrak{B}$ be the set of small rational open subsets of $X$ contained in some $X_i$. It is easy to see that $\mathfrak{B}$ is closed under finite intersections and forms a basis of open subsets of $X$. Moreover, for any $U\in\mathfrak{B}$, $\pi^{-1}(U)$ is a locally analytic covering of $U$. We claim that for any $i\geq 1, U\in\mathfrak{B}$, the \v{C}ech cohomology
\[\check{H}^i(U,\widetilde{\cO}^{\la})=0.\]
To see this, since $U$ is quasi-compact, any open cover of $U$ can be refined to a finite cover $\mathfrak{U}\subset \mathfrak{B}$. Let $C^{\bullet}(\mathfrak{U},\widetilde{\cO}),C^{\bullet}(\mathfrak{U},\widetilde{\cO}^{\la})$ be the \v{C}ech complex for $\widetilde{\cO}$ and $\widetilde{\cO}^{\la}$ with respect to this cover. Then $B\to C^{\bullet}(\mathfrak{U},\widetilde{\cO})$ is strictly exact by the almost vanishing of higher cohomology (Theorem 1.8.(iv) of \cite{Sch12}). Note that each term in this complex is $\mathfrak{LA}$-acyclic by Theorem \ref{BLAacyc}. Passing to locally analytic vectors, we see that $B^{\la}\to C^{\bullet}(\mathfrak{U},\widetilde{\cO}^{\la})$ is also exact by the second part of Lemma \ref{LAlongexa}. Hence $\check{H}^i(U,\widetilde{\cO}^{\la})=0,i\geq 1$.

Now by Corollaire 4., p.176 of \cite{Gro57}, the vanishing of higher \v{C}ech cohomology for any $U\in\mathfrak{B}$ implies that $H^i(U,\widetilde{\cO}^{\la})=0,i\geq 1$ and ${H}^i(X,\widetilde{\cO}^{\la})=\check{H}^i(X,\widetilde{\cO}^{\la})$. Hence we can compute ${H}^i(X,\widetilde{\cO}^{\la})$ using a \v{C}ech complex with respect to a finite cover in $\mathfrak{B}$. We can use $\mathfrak{U}_0$ here. All the claims in the corollary follow on applying the third part of Lemma \ref{LAlongexa} to $B\to C^{\bullet}(\mathfrak{U}_0,\widetilde{\cO})$.
\end{proof}

\section{Locally analytic functions on perfectoid modular curves} \label{lafopmc}
Now we apply the previous general theory to the case of modular curves of infinite level at $p$. It turns out that in this case, the differential operator in Theorem \ref{pCR} is very classical (see Theorem \ref{n0triv} below) and (twisted) $\mathscr{D}$-modules on the flag variety appear naturally in this setup. Following Berger-Colmez \cite{BC16}, we also give explicit descriptions of the $\GL_2(\Q_p)$-locally analytic functions on the infinite level modular curve. This will be important for our calculations in the next section.

In the first subsection, we will collect some basic facts about modular curves of infinite level at $p$ and the Hodge-Tate period map. This is the simplest case in the theory of perfectoid Shimura varieties developed by Scholze in \cite{Sch15}.

\subsection{Modular curves and the Hodge-Tate period map}
\begin{para}
We define modular curves ad\`elically. Fix a neat open compact subgroup $K\subset\GL_2(\A_f)$ and let $Y_{K}/\Q$ be the moduli space of elliptic curves with level-$K$-structure. Let $\mathbb{H}^{\pm1}$ be the union of upper and lower half-planes. The complex points of $Y_{K}$ are given by the usual double quotient
\[Y_{K}(\bC)=\GL_2(\Q)\backslash (\mathbb{H}^{\pm1}\times\GL_2(\A_f)/K).\]
$Y_{K}$ admits a natural compactification $X_K/\Q$ by adding finitely many cusps. The universal elliptic curve over $Y_{K}$ extends to a semi-abelian variety over $X_K$ and we denote by $\omega$ the sheaf of its invariant differentials. On the complex points, $\omega$ is  the canonical extension of $\omega|_{Y_K(\bC)}$ as defined in Main Theorem 3.1 of \cite{Mum77}.

Fix a complete algebraically closed non-archimedean field extension $C$ of $\Q_p$ as in \ref{setup}. Denote by $\mathcal{X}_{K}$ (resp. $\mathcal{Y}_K$) the adic space associated to $X_{K}\times_{\Q}C$ (resp. $Y_{K}\times_{\Q}C$).
\end{para}

\begin{thm} \label{infperf}
For any tame level $K^p\subset\GL_2(\A^p_f)$ contained in the level-$N$-subgroup $\{g\in\GL_2(\hat{\Z}^p)=\prod_{l\neq p}\GL_2(\Z_l)\,\vert\, g\equiv1\mod N\}$ for some $N\geq 3$ prime to $p$, there exists a unique perfectoid space $\mathcal{X}_{K^p}$ over $C$ such that 
\[\mathcal{X}_{K^p}\sim\varprojlim_{K_p\subset\GL_2(\Q_p)}\mathcal{X}_{K^pK_p},\]
where $K_p$ runs through all open compact subgroups of $\GL_2(\Q_p)$. Therefore there is a natural right action of $\GL_2(\Q_p)$ on $\mathcal{X}_{K^p}$. Moreover, for any open compact subgroup $K_p$ of $\GL_2(\Q_p)$, there is a basis consisting of open affinoid subsets $U$ of $\mathcal{X}_{K^pK_p}$ with affinoid perfectoid preimage $U_{\infty}$ in $\mathcal{X}_{K^p}$, and the map
\[\varinjlim_{K_p'\subset K_p}H^0(U_{K_p'},\cO_{\mathcal{X}_{K^pK_p'}})\to H^0(U_{\infty},\cO_{\mathcal{X}_{K^p}})\]
has dense image. Here $U_{K_p'}$ denotes the preimage of $U$ in $\mathcal{X}_{K^pK_p'}$.
\end{thm}

\begin{proof}
The existence of $\mathcal{X}_{K^p}$ basically follows from Theorem III.1.2 of \cite{Sch15} by taking connected components into account. For the existence of a basis of open subsets of $\mathcal{X}_{K^pK_p}$, by part (iii) of  Theorem III.1.2 of \cite{Sch15}, for sufficiently small $K_p$, we can find an open cover of $\mathcal{X}_{K^pK_p}$ with the desired density property. As taking rational subsets preserves this density property by Lemma 4.5 of \cite{Sch13}, we may find such a basis $\mathfrak{B}_{K_p}$ of open affinoid subsets for sufficiently small $K_p$. In general, we can descend this property to any $K_p$. To see this, it is enough to find an affinoid cover of $\mathcal{X}_{K^pK_p}$ with the desired properties. We may assume $K_p$ is a subgroup of $\GL_2(\Z_p)$. Take an $\GL_2(\Z_p)$-invariant affinoid cover of $\Fl$. Then Scholze's result implies that this cover comes from an affinoid cover of $\mathcal{X}_{K^pK_p'}$ for some $K'_p$ sufficiently small. We may assume $K'_p$ is a normal subgroup of $K_p$. Hence $\mathcal{X}_{K^pK_p}$ can be viewed as the quotient of $\mathcal{X}_{K^pK_p'}$ by $K_p/K'_p$ and  we can descend this cover to $\mathcal{X}_{K^pK_p}$ by \cite[Theorem 1.3]{Han16}.
\end{proof}

\begin{para}
One powerful tool to study the geometry of $\mathcal{X}_{K^p}$ is the Hodge-Tate period map introduced in Theorem III.1.2 of \cite{Sch15}. We'll give an equivalent definition in our setup. 

Fix a $K^p$ as in the theorem and an open subgroup $K_p$ of $\GL_2(\Z_p)$. Let $\mathcal{X}=\mathcal{X}_{K},\mathcal{Y}=\mathcal{Y}_{K}$ with $K=K^pK_p$ and let $f:\mathcal{E}\to\mathcal{Y}$ be the universal elliptic curve. Then $R^1f_*\Z_p$ defines a rank two \'etale $\Z_p$-local system $\underline{V}$.  Our normalization is that $\underline{V}$ corresponds to the standard representation of $K_p\subset\GL_2(\Z_p)$. See for example \S2.4 of \cite{Eme06C}. This $\Z_p$-local system $\underline{V}$ induces a $\hat\Z_p$-local system $\hat{\underline{V}}$ on $\mathcal{Y}_{\proet}$, the pro-\'etale site of $\mathcal{Y}$, see 8.1, 8.2 of \cite{Sch13} for the notation here. On the other hand, we denote by $D_{\dR}$ the relative de Rham cohomology of $\mathcal{E}$ over $\mathcal{Y}$. This is a rank two vector bundle on $\mathcal{Y}$ equipped with a (decreasing) Hodge filtration $\Fil^\bullet$ and Gauss-Manin connection $\nabla_{\underline{V}}$. Its graded components are given by
\[\gr^n(D_{\dR})\cong\left\{
	\begin{array}{lll}
		\omega^{-1}|_{\mathcal{Y}},~n=0\\
		\omega|_{\mathcal{Y}},~n=1\\
		0,~n\neq 0,1
	\end{array}.\right.\]
Recall that there is the structural de Rham sheaf $\cO\mathbb{B}_{\dR}$ on $\mathcal{Y}_{\proet}$ also equipped with a decreasing filtration and a $\mathbb{B}_{\dR}$-linear connection, cf. \S 6 of \cite{Sch13} and \cite{Sch13e}. By the relative de Rham comparison theorem (Theorem 8.8 of \cite{Sch13}), there is a natural isomorphism
\[\hat{\underline{V}}\otimes_{\hat\Z_p}\cO\mathbb{B}_{\dR}\cong D_{\dR}\otimes_{\cO_{\mathcal{Y}}}\cO\mathbb{B}_{\dR}\]
compatible with the filtrations and connections. Here by abuse of notation, $D_{\dR}$ is viewed as a sheaf on $\mathcal{Y}_{\proet}$ by Lemma 7.3 of \cite{Sch13}. 

Next, we extend these results to $\mathcal{X}$ using the theory of log adic spaces. See \cite{DLLZ1,DLLZ2} for more details here. We will view $\mathcal{X}$ as a log adic space by equipping it with the natural log structure defined by the cusps $\mathcal{C}:=\mathcal{X}-\mathcal{Y}$, cf. Example 2.1.2 of \cite{DLLZ2}. Then $\hat{\underline{V}}$ defines a rank two $\hat\Z_p$-local system $\hat{\underline{V}}_{\log}$ on $\mathcal{X}_{\proket}$, the pro-Kummer \'etale site of $\mathcal{X}$. There is also the structural de Rham sheaf $\cO\mathbb{B}_{\dR,\log}$ on  $\mathcal{X}_{\proket}$ equipped with a decreasing filtration and a logarithmic connection:
\[\nabla:\cO\mathbb{B}_{\dR,\log}\to\cO\mathbb{B}_{\dR,\log}\otimes_{\cO_{\mathcal{X}}}\Omega^1_{\mathcal{X}}(\mathcal{C}).\]
Here $\Omega^1_{\mathcal{X}}(\mathcal{C})$ as usual denotes the sheaf of differentials on $\mathcal{X}$ with simple poles at $\mathcal{C}$. As the monodromy of ${\underline{V}}$ along each cusp is \textit{unipotent}, we have the following relative log de Rham comparison isomorphism (Theorem 3.2.12  and its proof of \cite{DLLZ2})
\begin{eqnarray} \label{logdR}
\hat{\underline{V}}_{\log}\otimes_{\hat\Z_p}\cO\mathbb{B}_{\dR,\log}\cong D_{\dR,\log}\otimes_{\cO_{\mathcal{X}}}\cO\mathbb{B}_{\dR,\log}
\end{eqnarray}
compatible with the filtrations and logarithmic connections, where $D_{\dR,\log}$ is a filtered vector bundle equipped with a logarithmic connection $\nabla_{\underline{V}}$.  In fact, by Theorem 1.7 of \cite{DLLZ2}, $(D_{\dR,\log},\nabla_{\underline{V}})$ is the canonical extension of $(D_{\dR},\nabla_{\underline{V}})$. Hence in particular,
\[\gr^n(D_{\dR,\log})\cong\left\{
	\begin{array}{lll}
		\omega^{-1},~n=0\\
		\omega,~n=1\\
		0,~n\neq 0,1
	\end{array}.\right.\]

Now we have a decreasing filtration (the relative Hodge-Tate filtration) on $\hat{\underline{V}}_{\log}\otimes_{\hat\Z_p}\hat\cO_{\mathcal{X}}$ as in \S2.2. of \cite{CS17}. More precisely, let $\cO\mathbb{B}_{\dR,\log}^+$ be the positive structural de Rham sheaf on $\mathcal{X}_{\proket}$ (the geometric de Rham period sheaf in definition 2.2.10 of \cite{DLLZ2}) which also admits a logarithmic connection. Consider
\[\mathbb{M}_0=(D_{\dR,\log}\otimes_{\cO_{\mathcal{X}}}\cO\mathbb{B}_{\dR,\log}^+)^{\nabla=0}.\]
Recall that there is the positive de Rham sheaf $\mathbb{B}_{\dR}^+$ on $\mathcal{X}_{\proket}$ (definition 2.2.3.(2) of \cite{DLLZ2}) and the usual surjective map: $\theta:\mathbb{B}_{\dR}^+\to \hat\cO_{\mathcal{X}}$. Then Proposition 7.9 of \cite{Sch13} implies
\begin{eqnarray} \label{MM0}
\hat{\underline{V}}_{\log}\otimes_{\hat\Z_p}\mathbb{B}_{\dR,\log}^+\supseteq \mathbb{M}_0 \supseteq \hat{\underline{V}}_{\log}\otimes_{\hat\Z_p}\ker(\theta).
\end{eqnarray}
Hence we obtain a filtration on $\hat{\underline{V}}_{\log}\otimes_{\hat\Z_p}\mathbb{B}_{\dR}^+/(\ker\theta)=\hat{\underline{V}}_{\log}\otimes_{\hat\Z_p}\hat\cO_{\mathcal{X}}$, which again by Proposition 7.9 of \cite{Sch13} can be identified with
\[0\to \gr^0(D_{\dR,\log})\otimes_{\cO_{\mathcal{X}}}\hat\cO_{\mathcal{X}}\to \hat{\underline{V}}_{\log}\otimes_{\hat\Z_p}\hat\cO_{\mathcal{X}} \to \gr^1(D_{\dR,\log})\otimes_{\cO_{\mathcal{X}}}\hat\cO_{\mathcal{X}}(-1)\to 0.\]
In other words, we get the following exact sequence of locally free $\hat\cO_{\mathcal{X}}$-modules on $\mathcal{X}_{\proket}$ (\textit{relative Hodge-Tate filtration} of $ \hat{\underline{V}}_{\log}\otimes_{\hat\Z_p}\hat\cO_{\mathcal{X}}$): 
\begin{eqnarray} \label{relHT}
0\to \omega^{-1}\otimes_{\cO_{\mathcal{X}}}\hat\cO_{\mathcal{X}}(1) \to \hat{\underline{V}}_{\log}(1)\otimes_{\hat\Z_p}\hat\cO_{\mathcal{X}} \to \omega\otimes_{\cO_{\mathcal{X}}}\hat\cO_{\mathcal{X}}\to 0.
\end{eqnarray}
\end{para}

\begin{rem}
Strictly speaking, Proposition 7.9 of \cite{Sch13} is only proved when there is no log structure. However, all the arguments work here with the input replaced by the corresponding results in \S2 of \cite{DLLZ2}. For example, the local structure of $\cO\mathbb{B}_{\dR,\log}^+$ in Proposition 6.10 is now replaced by Proposition 2.3.15 of \cite{DLLZ2}. 

There is one minor subtlety here: compared with the definition of $\cO\mathbb{B}_{\dR}$ in \cite{Sch13}, there is an extra completion of $\cO\mathbb{B}_{\dR,\log}^+[t^{-1}]$ in Definition 2.2.10 of \cite{DLLZ2}, where $t\in\mathbb{B}_{\dR}^+$ is a generator of $\ker(\theta)$. I claim that such a completion is unnecessary in our case, i.e. the isomorphism \eqref{logdR} can be restricted to
\[\hat{\underline{V}}_{\log}\otimes_{\hat\Z_p}\cO\mathbb{B}_{\dR,\log}^+[t^{-1}]\cong D_{\dR,\log}\otimes_{\cO_{\mathcal{X}}}\cO\mathbb{B}_{\dR,\log}^+[t^{-1}].\]
In fact, using the local structure of $\cO\mathbb{B}_{\dR,\log}^+$, we may argue as in Theorem 7.2 of \cite{Sch13} that $\mathbb{M}_0$ is a $\mathbb{B}_{\dR}^+$-lattice of $(\hat{\underline{V}}_{\log}\otimes_{\hat\Z_p}\cO\mathbb{B}_{\dR,\log})^{\nabla=0}=\hat{\underline{V}}_{\log}\otimes_{\hat\Z_p}\mathbb{B}_{\dR}$ by the Poincar\'e Lemma (Corollary 2.4.2 of \cite{DLLZ2}). This is enough to deduce our claim.
\end{rem}

\begin{para}
Now we take $K_p=\GL_2(\Z_p)$. Hence $\mathcal{X}=\mathcal{X}_{K^p\GL_2(\Z_p)}$. As $\varprojlim_{K_p\subset\GL_2(\Q_p)}\mathcal{X}_{K^pK_p}$ (equipped with log structures defined by cusps) can be viewed as an open covering of $\mathcal{X}$ in $\mathcal{X}_{\proket}$, we can evaluate the exact sequence \eqref{relHT} on $\varprojlim_{K_p}\mathcal{X}_{K^pK_p}$. Note that $\mathcal{X}_{K^p}$ trivializes the universal Tate module, hence by Theorem \ref{infperf}, we obtain the following exact sequence of vector bundles on $\mathcal{X}_{K^p}$:
\end{para}

\begin{thm} \label{RHT}
\[0\to\omega_{K^p}^{-1}(1)\to V(1)\otimes_{\Q_p} \cO_{\mathcal{X}_{K^p}} \to \omega_{K^p}\to 0,\]
where $\omega_{K^p}$ is the pull-back of $\omega$ as a coherent sheaf from $\mathcal{X}$ to $\mathcal{X}_{K^p}$ and $V=\Q_p^{\oplus 2}$ is the standard representation of $\GL_2(\Q_p)$ and viewed as a constant sheaf on $\mathcal{X}_{K^p}$. This exact sequence is $\GL_2(\Z_p)$-equivariant.
\end{thm}


Clearly, the position of $\omega_{K^p}^{-1}$ in $V\otimes_{\Q_p} \cO_{\mathcal{X}_{K^p}}$ induces a map (Hodge-Tate period map)
\[\pi_{\HT}:\mathcal{X}_{K^p}\to\Fl,\]
where $\Fl$ is the adic space over $C$ associated to the usual flag variety for $\GL_2$. One can check that our definition of $\pi_{\HT}$ agrees with the one defined in III.3 of \cite{Sch15},
\footnote{there is a slight difference: in our setup, we trivialize the first relative \'{e}tale cohomology the universal elliptic curve, while  \cite{Sch15} trivializes the universal Tate module. As a result, there is a Tate twist in the middle term of the exact sequence in Theorem \ref{RHT}.}
 using Lemma III.3.4 and Corollary III.3.17 in the reference.

There is a right action of $\GL_2(\Q_p)$ on $\Fl$ by acting on the total space (rather than the position of the flag). We will always use this right action.

Theorem III.3.18 of \cite{Sch15} provides an affinoid cover $\{U_1,U_2\}$ of $\Fl$. In our case, $\Fl=\mathbb{P}^1$ and $U_1=\{[x_1:x_2],\|x_1\|\geq\|x_2\|\},U_2=\{[x_1:x_2],\|x_2\|\geq\|x_1\|\}$. 

\begin{thm} \label{piHT}
$\pi_{\HT}$ is $\GL_2(\Q_p)$-equivariant and commutes with Hecke operators away from $p$ (when changing $K^p$), for the trivial action of these Hecke operators on $\Fl$. Moreover, let $\mathfrak{B}$ be the set of  finite intersections of rational subsets of $U_1,U_2$. Then $\mathfrak{B}$ is a basis of open affinoid subsets of $\Fl$ stable under finite intersections and  each $U\in\mathfrak{B}$ has the following properties:
\begin{itemize}
\item its preimage $V_\infty=\pi_{\HT}^{-1}(U)$ is affinoid perfectoid;
\item $V_\infty$ is the preimage of an affinoid subset $V_{K_p}\subset\mathcal{X}_{K^pK_p}$ for sufficiently small open subgroup $K_p$ of $\GL_2(\Q_p)$;
\item the map $\varinjlim_{K_p}H^0(V_{K_p},\cO_{\mathcal{X}_{K^pK_p}})\to H^0(V_\infty,\cO_{\mathcal{X}_{K^p}})$
has dense image;
\item $U$ does not contain all $\Q_p$-rational points of $\Fl$.
\end{itemize}
\end{thm}

\begin{proof}
See Theorem III.1.2 of \cite{Sch15} and Theorem III.3.18 of \cite{Sch15}.
\end{proof}

\subsection{Faltings's extension and computation of  \texorpdfstring{$\theta$}{Lg}}
Let $K_p\subset\GL_2(\Q_p)$ be an open subgroup and $X$ an affinoid subset of $\mathcal{X}=\mathcal{X}_{K^pK_p}$ containing at most one cusp. Suppose its preimage $\tilde{X}$ in $\mathcal{X}_{K^p}$ is a ``log $K_p$-Galois pro-\'etale perfectoid covering'' of $X$ and $X$ is small in the sense of \ref{setup}. Note that by Theorem \ref{infperf} and Lemma 5.2 (and its proof) of \cite{Sch15}, such $X$ form a basis of open subsets of $\mathcal{X}$.

The goal of this subsection is to compute the differential operator $\theta$ in Theorem \ref{pCR} for this Galois covering. It turns out that this follows from a classical result of Faltings which (up to a twist) identifies the relative Hodge-Tate filtration sequence in \ref{RHT} with (log) Faltings's extension, cf. Theorem 5 of \cite{Fa87}. We will give a proof of this result (Theorem \ref{FERHT}) in our setup below. A more conceptual computation of $\theta$ in terms of the $p$-adic Simpson correspondence is given in Remark \ref{pSimtheta}.

\begin{para}
We use notation introduced in the previous subsection. First we recall the log Faltings's extension as defined in Corollary 6.14 of \cite{Sch13} and Corollary 2.4.5 of \cite{DLLZ2}. Taking the first graded piece of the Poincar\'e lemma sequence (Corollary 2.4.2 of \cite{DLLZ2}):
\begin{eqnarray} \label{PL}
0\to \mathbb{B}^+_{\dR}\to \cO\mathbb{B}_{\dR,\log}^+ \xrightarrow{\nabla} \cO\mathbb{B}_{\dR,\log}^+\otimes_{\cO_{\mathcal{X}}}\Omega^1_{\mathcal{X}}(\mathcal{C})\to 0,
\end{eqnarray}
we obtain an exact sequence of $\hat\cO_{\mathcal{X}}$-modules on $\mathcal{X}_{\proket}$ (the log Faltings's extension):
\[0\to \hat\cO_{\mathcal{X}}(1)\to \gr^1\cO\mathbb{B}_{\dR,\log}^+\xrightarrow{\nabla} \hat\cO_{\mathcal{X}}\otimes_{\cO_{\mathcal{X}}}\Omega^1_{\mathcal{X}}(\mathcal{C})\to 0.\]
Its tensor product with $\omega^{-1}$ becomes
\[0\to \omega^{-1}\otimes_{\cO_{\mathcal{X}}}\hat\cO_{\mathcal{X}}(1)\to \omega^{-1}\otimes_{\cO_{\mathcal{X}}}\gr^1\cO\mathbb{B}_{\dR,\log}^+\to \omega^{-1}\otimes_{\cO_{\mathcal{X}}}\Omega^1_{\mathcal{X}}(\mathcal{C})\otimes_{\cO_{\mathcal{X}}}\hat\cO_{\mathcal{X}}\to 0.\]
Recall that there is a Kodaira-Spencer map $\omega\to \omega^{-1}\otimes_{\cO_{\mathcal{X}}}\Omega^1_{\mathcal{X}}(\mathcal{C})$ of coherent sheaves on $\mathcal{X}$ defined as the composite
\[\omega=\Fil^1 D_{\dR,\log}\subset D_{\dR,\log}\xrightarrow{\nabla}D_{\dR,\log}\otimes_{\cO_{\mathcal{X}}}\Omega^1_{\mathcal{X}}(\mathcal{C})\to \gr^0 D_{\dR,\log}\otimes_{\cO_{\mathcal{X}}}\Omega^1_{\mathcal{X}}(\mathcal{C})=\omega^{-1}\otimes_{\cO_{\mathcal{X}}}\Omega^1_{\mathcal{X}}(\mathcal{C}).\]
It is well-known that this is an isomorphism: $\omega\xrightarrow{\sim} \omega^{-1}\otimes_{\cO_{\mathcal{X}}}\Omega^1_{\mathcal{X}}(\mathcal{C})$. Under this isomorphism, $\omega^{-1}\otimes_{\cO_{\mathcal{X}}}\gr^1\cO\mathbb{B}_{\dR,\log}^+$ can be viewed as an element of 
\[\Ext^1_{\mathcal{X}_{\proket}}(\omega\otimes_{\cO_{\mathcal{X}}}\hat\cO_{\mathcal{X}},\omega^{-1}\otimes_{\cO_{\mathcal{X}}}\hat\cO_{\mathcal{X}}(1)).\]

On the other hand,  recall that $\hat{\underline{V}}_{\log}\otimes_{\hat\Z_p}\hat\cO_{\mathcal{X}}$ also defines an element in this extension group by the relative Hodge-Tate filtration \eqref{relHT}.

\end{para}

\begin{thm} \label{FERHT}
There is an isomorphism between $\omega^{-1}\otimes_{\cO_{\mathcal{X}}}\gr^1\cO\mathbb{B}_{\dR,\log}^+$ and $\hat{\underline{V}}_{\log}(1)\otimes_{\hat\Z_p}\hat\cO_{\mathcal{X}}$ as $\hat\cO_{\mathcal{X}}$-modules such that as extension classes, they differ by $-1$.
\end{thm}

\begin{proof}
Consider the first graded piece of the tensor product of $\hat{\underline{V}}_{\log}$ and the Poincar\'e lemma sequence \eqref{PL} over $\hat\Z_p$, i.e. the tensor product of $\hat{\underline{V}}_{\log}$ with the log Faltings's extension:
\[0\to \hat{\underline{V}}_{\log}(1)\otimes_{\hat\Z_p} \hat\cO_{\mathcal{X}}\to \hat{\underline{V}}_{\log}\otimes_{\hat\Z_p} \gr^1\cO\mathbb{B}_{\dR,\log}^+ \xrightarrow{\nabla} \hat{\underline{V}}_{\log}\otimes_{\hat\Z_p}\hat\cO_{\mathcal{X}}\otimes_{\cO_{\mathcal{X}}}\Omega^1_{\mathcal{X}}(\mathcal{C})\to 0.\]
Then the inclusion $\gr^1(D_{\dR,\log}\otimes_{\cO_{\mathcal{X}}} \cO\mathbb{B}_{\dR,\log}^+)\subset \gr^1(\hat{\underline{V}}_{\log}\otimes_{\hat\Z_p}\cO\mathbb{B}_{\dR,\log}^+)$ induces
\[0\to \hat{\underline{V}}_{\log}(1)\otimes_{\hat\Z_p} \hat\cO_{\mathcal{X}}\to \gr^1(D_{\dR,\log}\otimes_{\cO_{\mathcal{X}}} \cO\mathbb{B}_{\dR,\log}^+) \xrightarrow{\nabla} \gr^0 D_{\dR,\log}\otimes_{\cO_{\mathcal{X}}}\hat\cO_{\mathcal{X}}\otimes_{\cO_{\mathcal{X}}}\Omega^1_{\mathcal{X}}(\mathcal{C}),   \]
where the first inclusion follows from \eqref{MM0} and the image of $\nabla$ lies in $\gr^0 D_{\dR,\log}\otimes_{\cO_{\mathcal{X}}}\hat\cO_{\mathcal{X}}\otimes_{\cO_{\mathcal{X}}}\Omega^1_{\mathcal{X}}(\mathcal{C})$ by Griffiths transversality. I claim that $\nabla$ is in fact surjective as it has a left inverse: the composite map
\[\gr^1(D_{\dR,\log})\otimes_{\cO_{\mathcal{X}}} \hat\cO_{\mathcal{X}}\subset \gr^1(D_{\dR,\log}\otimes_{\cO_{\mathcal{X}}} \cO\mathbb{B}_{\dR,\log}^+) \xrightarrow{\nabla} \gr^0 D_{\dR,\log}\otimes_{\cO_{\mathcal{X}}}\hat\cO_{\mathcal{X}}\otimes_{\cO_{\mathcal{X}}}\Omega^1_{\mathcal{X}}(\mathcal{C})\]
is nothing but the tensor product of Kodaira-Spencer map with $\hat\cO_{\mathcal{X}}$, hence an isomorphism. This implies that 
\[\hat{\underline{V}}_{\log}(1)\otimes_{\hat\Z_p} \hat\cO_{\mathcal{X}} \xrightarrow{\sim} \gr^1(D_{\dR,\log}\otimes_{\cO_{\mathcal{X}}} \cO\mathbb{B}_{\dR,\log}^+) / \gr^1(D_{\dR,\log})\otimes_{\cO_{\mathcal{X}}} \hat\cO_{\mathcal{X}}.\]
Note that this right hand side can be identified with 
$\gr^0(D_{\dR,\log})\otimes_{\cO_{\mathcal{X}}}\gr^1\cO\mathbb{B}_{\dR,\log}^+=\omega^{-1}\otimes_{\cO_{\mathcal{X}}}\gr^1\cO\mathbb{B}_{\dR,\log}^+$ as there is a canonical decomposition
\[\gr^1(D_{\dR,\log}\otimes_{\cO_{\mathcal{X}}} \cO\mathbb{B}_{\dR,\log}^+) = \gr^1(D_{\dR,\log})\otimes_{\cO_{\mathcal{X}}} \hat\cO_{\mathcal{X}} \oplus \gr^0(D_{\dR,\log})\otimes_{\cO_{\mathcal{X}}}\gr^1\cO\mathbb{B}_{\dR,\log}^+.\]
Hence we obtain an isomorphism $\hat{\underline{V}}_{\log}(1)\otimes_{\hat\Z_p} \hat\cO_{\mathcal{X}} \cong \omega^{-1}\otimes_{\cO_{\mathcal{X}}}\gr^1\cO\mathbb{B}_{\dR,\log}^+$ as claimed in the Theorem. To check their relation as extension classes, one useful observation is that the surjection $\hat{\underline{V}}_{\log}(1)\otimes_{\hat\Z_p} \hat\cO_{\mathcal{X}}\to \omega^{1}\otimes_{\cO_{\mathcal{X}}}\hat\cO_{\mathcal{X}}$ in \eqref{relHT} agrees with the composite map
\[ \hat{\underline{V}}_{\log}(1)\otimes_{\hat\Z_p} \hat\cO_{\mathcal{X}}\subset \gr^1(D_{\dR,\log}\otimes_{\cO_{\mathcal{X}}} \cO\mathbb{B}_{\dR,\log}^+) \to  \gr^1(D_{\dR,\log})\otimes_{\cO_{\mathcal{X}}} \hat\cO_{\mathcal{X}}, \]
where the second map is the projection using the above decomposition. Hence it differs by $-1$ with the composite
\[ \hat{\underline{V}}_{\log}(1)\otimes_{\hat\Z_p} \hat\cO_{\mathcal{X}}\to \gr^1(D_{\dR,\log}\otimes_{\cO_{\mathcal{X}}} \cO\mathbb{B}_{\dR,\log}^+) \to  \gr^0(D_{\dR,\log})\otimes_{\cO_{\mathcal{X}}}\gr^1\cO\mathbb{B}_{\dR,\log}^+\]
\[\xrightarrow{\nabla} \gr^0 D_{\dR,\log}\otimes_{\cO_{\mathcal{X}}}\hat\cO_{\mathcal{X}}\otimes_{\cO_{\mathcal{X}}}\Omega^1_{\mathcal{X}}(\mathcal{C}) \xleftarrow{\sim} \gr^1(D_{\dR,\log})\otimes_{\cO_{\mathcal{X}}} \hat\cO_{\mathcal{X}}, \]
where the second map is the other projection and  the last isomorphism is given by Kodaira-Spencer map. Note that the composite of first row is the isomorphism constructed before. This finishes the proof of theorem.
\end{proof}

\begin{cor} \label{FEana}
Let $X,\tilde{X}$ be as in the beginning of this subsection. Moreover, we assume $\pi_{\HT}(\tilde{X})$ is contained in an affinoid open subset of $\Fl$. Then $\widetilde{X}$ is a locally analytic covering of $X$ in the sense of \ref{laedefn}.
\end{cor}

\begin{proof}
By Proposition \ref{LAETFAE}, one equivalent definition of locally analytic covering is that if we take the global sections of Faltings's extension on $\tilde{X}$ (viewed as open set in $\mathcal{X}_{\proket}$ by abuse of notation), it remains exact after taking $K_p$-locally analytic vectors. First consider taking the global sections of relative Hodge-Tate filtration \eqref{relHT} on $\tilde{X}$:
\begin{eqnarray} \label{RHTsection}
\;\;\;\;\;\;\;0\to H^0(\tilde{X},\omega^{-1}\otimes\hat\cO_{\mathcal{X}}(1)) \to H^0(\tilde{X},\hat{\underline{V}}_{\log}(1)\otimes\hat\cO_{\mathcal{X}}) \to H^0(\tilde{X},\omega\otimes\hat\cO_{\mathcal{X}})\to 0.
\end{eqnarray}
This is exact as $H^1(\tilde{X},\omega^{-1}\otimes_{\cO_{\mathcal{X}}}\hat\cO_{\mathcal{X}}(1))=0$ by Theorem 5.4.3 of \cite{DLLZ1}. Note that $\hat{\underline{V}}_{\log}$ becomes a trivial local system on $\tilde{X}$, hence the middle term is naturally isomorphic to $V(1)\otimes_{\Q_p} H^0(\tilde{X},\hat\cO_{\mathcal{X}})$ where $V=\Q_p^{\oplus 2}$ is the standard representation of $K_p\subset\GL_2(\Q_p)$. By our assumption on $\pi_{\HT}(\tilde{X})$, there exists an element $\mathbf{e}\in V(1)\otimes_{\Q_p} C\subset V(1)\otimes_{\Q_p} H^0(\tilde{X},\hat\cO_{\mathcal{X}})$ whose image  $e\in H^0(\tilde{X},\omega\otimes_{\cO_{\mathcal{X}}}\hat\cO_{\mathcal{X}})$ is a basis of $H^0(\tilde{X},\omega\otimes_{\cO_{\mathcal{X}}}\hat\cO_{\mathcal{X}})$ as a $H^0(\tilde{X},\hat\cO_{\mathcal{X}})$-module.

As the action of $K_p$ on $V$ is analytic (even algebraic), we conclude that \eqref{RHTsection} remains exact after passing to the $K_p$-locally analytic vectors. The same holds when taking global sections of the tensor product of \eqref{relHT} and $\omega$, which is equivalent with multiplying \eqref{RHTsection} by ${e}$. Now the corollary follows by identifying this exact sequence with Faltings's extension (up to $-1$).
\end{proof}

Now we are ready to calculate $\theta$. We will give an explicit description first and rewrite in a coordinate-free way later. We introduce some notation. Let $\mathbf{e}_1=(1,0),\mathbf{e}_2=(0,1)$ be the standard basis of $V=\Q_p^{\oplus2}$ and denote their images by $e_1,e_2\in H^0(\mathcal{X}_{K^p}, \omega_{K^p})$ (fake-Hasse invariants) using the surjective map in the exact sequence in Theorem \ref{RHT}. We remark that $\mathrm{Aut}(C/\Q_p)$ acts on $e_1,e_2$ via the cyclotomic character because of the Tate-twist in $V(1)$.

Let $X,\tilde{X}$ be as in the previous corollary. In particular, $\omega_{K^p}|_{\tilde X}$ is trivial and generated by some section $e$. Hence we may view $\frac{e_1}{e},\frac{e_2}{e}$ as elements in $\cO_{\mathcal{X}_{K^p}}(\tilde X)$.

\begin{thm} \label{exptheta}
The differential operator $\theta$ associated to $(K_p,\tilde{X})$ in Theorem \ref{pCR} is, up to a unit of $\cO_{\mathcal{X}_{K^p}}(\tilde X)$,
\[\frac{1}{e^2}\begin{pmatrix}e_1e_2 & e_2^2\\ -e_1^2 & -e_1e_2\end{pmatrix}\in \mathfrak{gl}_2(\cO_{\mathcal{X}_{K^p}}(\tilde X))=\cO_{\mathcal{X}_{K^p}}(\tilde X)\otimes_{\Q_p} \Lie(K_p).\]
\end{thm}

\begin{rem} \label{pSimtheta}
As mentioned in Remark \ref{FAGTLZ}, it is better to view $\theta$ as a log Higgs field. The computation of $\theta$ here is just computing the Higgs bundle associated to $\underline{\hat{V}}_{\log}$ under the $p$-adic Simpson correspondence, which can be deduced from $(D_{\dR,\log},\nabla)$, the other side of the de Rham comparison. This was pointed out to me by Michael Harris.
\end{rem}

\begin{proof}
We will freely use the notation and constructions introduced in Section \ref{LARS}. Hence $X=\Spa(A,A^+),\tilde{X}=\Spa(B,B^+)$. Fix an \'etale map $X\to Y$ as in \ref{smallY}. Let $X_\infty=\Spa(A_\infty,A_\infty^+)$ (resp. $\tilde{X}_\infty=\Spa(B_\infty,B_\infty^+)$) be the pull-back of the perfectoid covering $Y_\infty$ to $X$ (resp. $\tilde X$), cf. \ref{TNT} (resp. \ref{BCtotildeX}).

To compute $\theta$, we follow Remark \ref{computephi}. Recall that $V$ is the standard two-dimensional $\Q_p$-representation of $K_p\subset\GL_2(\Q_p)$. We need to compute $\phi_{V}:\Lie(\Gamma)\to B_\infty\otimes_{\Q_p}V$  in Proposition \ref{relSenop}, which agrees with the snatural action of $\Lie(\Gamma)$ when restricted to $(B_\infty\otimes_{\Q_p}V)^{K_p,\Gamma-\la}$. By Theorem \ref{RHT}, ignoring the Tate-twist, we have
\[0\to B\cdot e^{-1}\to B\otimes_{\Q_p} V\to B\cdot e \to 0,\]
where $e^{-1}\in H^0(\tilde{X},\omega_{K^p}^{-1})$ is dual to $e\in H^0(\tilde{X},\omega_{K^p})$. Take the tensor product with $B_\infty$ over $B$ and then take the $K_p$-invariants:
\begin{eqnarray} \label{FEXinfty}
0\to (B_\infty\cdot e^{-1})^{K_p}\to (B_\infty\otimes V)^{K_p}\to (B_\infty\cdot e)^{K_p} \to 0.
\end{eqnarray}
To compute the action of $\Lie(\Gamma)$ on $(B_\infty\otimes V)^{K_p,\Gamma-\la}$, by Theorem \ref{FERHT}, this sequence can be identified with the global sections of Faltings's extension on $X_\infty$ (viewed as an open set in $\mathcal{X}_{\proket}$), up to $-1$. Hence by results in \S2 of \cite{DLLZ2} (in particular Proposition 2.3.15 and proof of Corollary 2.4.2), \eqref{FEXinfty} is $A_\infty$-linearly and $\Gamma$-equivariant isomorphic to the tensor product over $\Q_p$ of $A_\infty$ with a non-trivial self extension of trivial representation of $\Gamma$:
\[0\to \Q_p \to V_1 \to \Q_p\to0,\]
which we essentially write down in \ref{adhocFE}. Clearly $\Lie(\Gamma)$ acts trivially on $(B_\infty\cdot e^{-1})^{K_p,\Gamma-\la}\cong (A_\infty)^{\Gamma-\la}$ and if we fix a generator $x$ of $\Lie(\Gamma)$, it maps the quotient $(B_\infty\cdot e)^{K_p,\Gamma-\la}$ isomorphically to $(B_\infty\cdot e^{-1})^{K_p,\Gamma-\la}$. Therefore $\phi_V(x)$ is zero on $B\cdot e^{-1}$ and $\phi_V(x)(B\otimes V)=B\cdot e^{-1}$. 

The rest is to write this as an element of $B\otimes_{\Q_p}\Lie(K_p)$ up to $B^\times$. It is easy to see that $\mathbf{f}=\frac{e_2}{e}\mathbf{e_1}-\frac{e_1}{e}\mathbf{e_2}\in B\otimes_{\Q_p} V$ generates $B\cdot e^{-1}$. As $\phi_V(x)|_{B\otimes V}$ factors through $B\cdot e$, we may assume it sends $e$ to $\mathbf{f}$. A direct computation gives the matrix in the theorem.
\end{proof}

\begin{para} \label{OKpla}
To rewrite this theorem in a coordinate-free way, it is better to work with sheaves. Consider the push-forward of $\cO_{\mathcal{X}_{K^p}}$ along $\pi_{\HT}:\mathcal{X}_{K^p}\to\Fl$
\[\cO_{K^p}:={\pi_{\HT}}_{*} \cO_{\mathcal{X}_{K^p}}.\]
We define a subsheaf 
\[\cO_{K^p}^{\la}\subset\cO_{K^p}\] 
as follows: for any quasi-compact open set $U$ of $\Fl$, then $H^0(U,\cO_{K^p}^{\la})$ is the subspace of $K_p$-locally analytic vectors in $H^0(U,\cO_{K^p})$, where $K_p$ is an open subgroup of $\GL_2(\Q_p)$ stabilizing $U$. Note that such $K_p$ always exists by the compactness of $U$ and $H^0(U,\cO_{K^p}^{\la})$ is independent of choice of $K_p$. Since taking locally analytic vectors is a left-exact process (see the construction in section \ref{LAV}) and ${\pi_{\HT}}_*\cO^+_{\mathcal{X}_{K^p}}$ is a sheaf, $\cO_{K^p}^{\la}$ indeed defines a sheaf on $\Fl$. 

Clearly $\cO^{\la}_{K^p}$ is $\GL_2(\Q_p)$-equivariant and there is a natural action of the Lie algebra of $\GL_2(\Q_p)$ on $\cO^{\la}_{K^p}$.  Hecke operators away from $p$ act on it naturally as we only use the action at $p$ in the construction.

We claim that the natural map $\cO_{\Fl}\to\cO_{K^p}$ has image in $\cO_{K^p}^{\la}$. It is enough to check this for a basis of affinoid open subsets $U$ of $\Fl$. Since $\Fl$ is an adic space associated to a variety over $C$ of finite type, we may consider all affinoid subsets $U$ such that $\cO_{\Fl}(U)$ can be written as a continuous quotient of $C\langle T_1\cdots,T_m\rangle$ for some $m$ and the quotient topology on $\cO_{\Fl}(U)$ agrees with its natural topology. Fix such a $U$, then there exists an open subgroup $K_p\subset\GL_2(\Q_p)$ stabilizing $U$ and acting trivially on $\cO^+_{\Fl}(U)/p$. From this, we see that the action of $K_p$ on $\cO_{\Fl}(U)$ is locally analytic (using arguments in \ref{exaZp} for example).

We follow some constructions on the flag variety in  \cite{BB83}. Denote by $\mathfrak{g}$ the tensor product of the Lie algebra of $\GL_2(\Q_p)$ with $C$, i.e. $\mathfrak{gl}_2(C)$. For a $C$-point $x$ of the flag variety $\mathrm{Fl}$ of $\GL_2/C$, let $\mathfrak{b}_x,\mathfrak{n}_x\subset \mathfrak{g}$ be its corresponding Borel subalgebra and nilpotent subalgebra. Let
\begin{eqnarray*}
\mathfrak{g}^0&:=&\cO_{\mathrm{Fl}}\otimes_{C}\mathfrak{g},\\
\mathfrak{b}^0&:=&\{f\in \mathfrak{g}^0\,| \, f_x\in \mathfrak{b}_x,\mbox{ for all }x\in\mathrm{Fl}(C)\},\\
\mathfrak{n}^0&:=&\{f\in \mathfrak{g}^0\,| \, f_x\in \mathfrak{n}_x,\mbox{ for all }x\in\mathrm{Fl}(C)\}
\end{eqnarray*}
be $\GL_2$-equivariant vector bundles on $\mathrm{Fl}$, where $\GL_2$ acts by conjugation on $\mathfrak{g}$. By abuse of notation, we also view these as vector bundles on $\Fl$, the associated adic space of $\mathrm{Fl}$. Then we have a natural action of $\mathfrak{g}^0$ on $\cO^{\la}_{K^p}$: for $f\in \cO_{\Fl},z\in\mathfrak{g},s\in \cO^{\la}_{K^p}$,
\[(f\otimes z)\cdot s=f(z\cdot s).\]
\end{para}

\begin{thm} \label{n0triv}
$\mathfrak{n}^0$ acts trivially on $\cO^{\la}_{K^p}$.
\end{thm}

\begin{proof}
Its enough to check this on open subsets as in Theorem \ref{piHT}, for which we may invoke Theorem \ref{exptheta} and our claim follows immediately. In fact, it follows from the proof of Theorem \ref{exptheta} that the action of $\phi_V$ on $V(1)\otimes_{\Q_p} \cO_{\mathcal{X}_{K^p}}$ is trivial on $\omega^{-1}_{K^p}(1)$ and induces an isomorphism $\omega_{K^p}\xrightarrow{\sim} \omega^{-1}_{K^p}(1)$. Since $\pi_{\HT}$ is defined by the position of $\omega^{-1}_{K^p}(1)$ in $V(1)\otimes_{\Q_p} \cO_{\mathcal{X}_{K^p}}$, it is tautological that  $\cO_{\mathcal{X}_{K^p}}\theta\subset\cO_{\mathcal{X}_{K^p}}\otimes_C \mathfrak{g}$ contains $\mathfrak{n}^0$.
\end{proof}

Fix a rational Borel subalgebra $\mathfrak{b}:=\{\begin{pmatrix} * & * \\ 0 & * \end{pmatrix}\}$ and a Cartan subalgebra $\mathfrak{h}:=\{\begin{pmatrix} * & 0 \\ 0 & * \end{pmatrix}\}$ of $\mathfrak{g}$ and let $W$ be the Weyl group. Recall the Harish-Chandra isomorphism $Z(U(\mathfrak{g}))\cong S(\mathfrak{h})^W$, where  $Z(U(\mathfrak{g}))$ is the center of the universal enveloping algebra $U(\mathfrak{g})$ of $\mathfrak{g}$, and $S(\mathfrak{h})$ is the symmetric algebra of $\mathfrak{h}$ equipped with the dot action of $W$:  $w\cdot\mu=w(\mu+\rho)-\rho,\mu\in\mathfrak{h}^*$. Here $\rho$ is the half sum of positive roots as usual. Let $\mathfrak{g}_0:=\mathfrak{sl}_2(C)\subset \mathfrak{g}$ and $\kh_0=\kh\cap \mathfrak{g}_0$. Then similarly $Z(U(\mathfrak{g_0}))\cong S(\mathfrak{h_0})^W$. We also denote by $\mathfrak{z}:=\{\begin{pmatrix} a & 0 \\ 0 & a \end{pmatrix}\}\subset \mathfrak{g}$ the centre of $\mathfrak{g}$. 

It is clear that $\mathfrak{b}^0/\mathfrak{n}^0=\cO_{\mathrm{\Fl}}\otimes_{C}\mathfrak{h}$ by identifying global sections of $\mathfrak{b}^0/\mathfrak{n}^0$ with $\mathfrak{h}$. So we have a a natural embedding $\mathfrak{h}\to\mathfrak{b}^0/\mathfrak{n}^0$. 

\begin{cor} \label{khact}
There is a natural (horizontal) action $\theta_{\kh}$ of $\mathfrak{h}$ on $\cO^{\la}_{K^p}$ commuting with $\mathfrak{g}$. In fact, $-\theta_{\kh}$ induces an action of $S(\mathfrak{h_0})$ extending the one of $Z(U(\mathfrak{g_0}))$ and $\theta_\kh|_{\mathfrak{z}}$ agrees with the action of $\mathfrak{z}\subset\mathfrak{g}$.
\end{cor}

\begin{proof}
The first part follows from the fact that $H^0(\Fl,\mathfrak{b}^0/\mathfrak{n}^0)$ is $\GL_2$-invariant. The second part essentially follows from Harish-Chandra's theorem. Note that there is a  sign here since this construction studies the $\mathfrak{n}_x$-coinvariants rather than invariants. See \S2, \S3 of \cite{BB83}.
\end{proof}

\begin{rem}
See Theorem \ref{Sensf}, \ref{Senkh} below for an interpretation of this action in terms of the classical Sen operator arising from the action of $G_{\Q_p}$.
\end{rem}

\begin{rem} \label{relBB}
In Beilinson-Bernstein's theory of localization (see for example \S C of \cite{Bei84}), $\cO^{\la}_{K^p}$ is a $\tilde{\mathscr{D}}$-module in this analytic setting. Similarly, fix a weight $\chi$ of $\mathfrak{h}$. We denote by $\cO^{\la,\chi}_{K^p}$ the $\chi$-component of $\cO^{\la}_{K^p}$, i.e. the subsheaf of sections of weight $\chi$. Then $\cO^{\la,\chi}_{K^p}$ is a $\tilde{\mathscr{D}}_{\chi}$-module. Recall that $\tilde{\mathscr{D}}_{\chi}$ is a ring of twisted differential operators on the flag variety. Everything here is also $\GL_2(\Q_p)$-equivariant. Hence we obtain a $(\tilde{\mathscr{D}}_{\chi},\GL_2(\Q_p))$-module on $\Fl$, which is very similar to the picture in the complex analytic setting.

In a series of works (see for example \cite{Ard17}), Ardakov proved a $p$-adic analogue of Beilinson-Bernstein's localization theorem. Roughly speaking, he showed that one can realize the dual of an admissible locally analytic representation of $\GL_2(\Q_p)$ as global sections of certain $\GL_2(\Q_p)$-equivariant $\mathscr{D}$-modules on $\Fl$. It is  very natural to ask whether there is  some form of duality  relating his construction with $\cO^{\la,\chi}_{K^p}$ here.
\end{rem}

\subsection{Local structure of \texorpdfstring{$\cO^{\la}_{K^p}$}{Lg}} \label{LS}
So far $\cO^{\la}_{K^p}$ is defined in an abstract way. We give an explicit description of its sections in this subsection. This is largely inspired by the calculations of Berger-Colmez in \S4, \S5 of \cite{BC16}. We will follow their arguments and combine with our differential equation $\theta=0$. To do this, one needs enough sections of $\cO^{\la}_{K^p}$. 

\begin{para} \label{ttn}
Let $\Gamma(p^n)=1+p^n M_2(\Z_p)$. Recall that $\mathcal{Y}_{K^p\Gamma(p^n)}$ parametrizes elliptic curves with certain level structures away from $p$ and a trivialization of its $p^n$-torsion points. In particular, the Weil pairing induces a trivialization of $\mu_{p^n}(C)$. This defines a map $\mathcal{Y}_{K^p\Gamma(p^n)}\to \mathrm{Isom}(\Z/p^n,\mu_{p^n}(C))$, which extends to its compactification $\mathcal{X}_{K^p\Gamma(p^n)}\to \mathrm{Isom}(\Z/p^n,\mu_{p^n}(C))$. Classically, this map can be obtained by looking at the connected components. Taking inverse limits, we obtain $\mathcal{X}_{K^p}\to \mathrm{Isom}(\Z_p,\Z_p(1))$, where $\mathrm{Isom}(\Z_p,\Z_p(1))$ is identified as $\Spa(\sC(\mathrm{Isom}(\Z_p,\Z_p(1)),C),\sC(\mathrm{Isom}(\Z_p,\Z_p(1)),\cO_C))$. Now if we choose a basis of $\Z_p(1)$, i.e. an isomorphism $\mathrm{Isom}(\Z_p,\Z_p(1))\cong\Z_p^\times$, then $\mathcal{X}_{K^p}\to \mathrm{Isom}(\Z_p,\Z_p(1))\cong\Z_p^\times$ defines a function
\[t\in H^0(\mathcal{X}_{K^p},\cO_{\mathcal{X}_{K^p}}),\]
which is well-defined up to $\Z_p^\times$. Fix such a choice from now on. See also (4.2) of \cite{Eme06} for another description. We remark that $G_{\Q_p}$ acts on $t$ via the cyclotomic character because it is essentially an element of $\Z_p(1)$. The group $\GL_2(\Q_p)$ acts on $t$ via $\varepsilon_p\circ\det$, where $\varepsilon_p:\Q_p^\times\to\Z_p^\times$ sends $x$ to $x|x|$, hence the action of $\mathfrak{g}=\mathfrak{gl}_2(\Q_p)$ on $t$ is the trace map:
\[\begin{pmatrix} a & b \\ c & d \end{pmatrix}\cdot t=(a+d)t,\, \begin{pmatrix} a & b \\ c & d \end{pmatrix}\in \mathfrak{g}.\]

If we take the direct limit over all tame levels $K^p$: $\tilde{H}^0(C):=\varinjlim_{K^p}H^0(\mathcal{X}_{K^p},\cO_{\mathcal{X}_{K^p}})$, then $\GL_2(\A_f)$ acts naturally on $\tilde{H}^0(C)$ and acts on $t$ via $\varepsilon\circ \det$. See Notation in the beginning of this paper for the definition of $\varepsilon$. Using this observation, it's easy to figure out how Hecke operators away from $p$ change when multiplying by a power of $t$.

For $n\geq 1$, we fix $t_n\in H^0(\mathcal{X}_{K^p\Gamma(p^n)},\cO_{\mathcal{X}_{K^p\Gamma(p^n)}})$ so that $\|t-t_n\|\leq p^{-n}$.
\end{para}

\begin{para}[A consequence of $\theta=0$] \label{acot0}
We introduced a basis $\mathfrak{B}$ of affinoid open subsets of $\Fl$ in Theorem \ref{piHT}. Fix $U\in\mathfrak{B}$ from now on. Then by our construction, $V_\infty:=\pi_{\HT}^{-1}(U)$ is affinoid perfectoid and is the preimage of an affinoid subset $V_{K_p}\subset\mathcal{X}_{K^pK_p}$. Recall that we have fake-Hasse invariants $e_1,e_2\in H^0(\mathcal{X}_{K^p}, \omega_{K^p})$. Throughout this subsection, we assume 
\begin{itemize}
\item $e_1$ generates $H^0(V_\infty, \omega_{K^p})$.
\end{itemize}
In our later computations, we will always assume one of $e_1,e_2$ generates $H^0(V_\infty, \omega_{K^p})$. When  $e_2$ generates $H^0(V_\infty, \omega_{K^p})$, one can use the action of $\begin{pmatrix} 0 & 1 \\ 1 & 0\end{pmatrix}$ to reduce to the above case. When $e_1$ is a basis, we can take $e=e_1$ in Theorem \ref{exptheta}. Hence up to a unit, $\theta$ on $V_\infty$ is 
\[\begin{pmatrix}
x & x^2\\
-1 & -x 
\end{pmatrix}\in\mathfrak{gl}_2(\cO_{K^p}(U)),\]
where $x:=\frac{e_2}{e_1}\in \cO_{\Fl}(U)$, a standard coordinate on $\Fl=\mathbb{P}^1$. In particular, we have the following important lemma:
\end{para}

\begin{lem} \label{annbg}
Suppose $f\in \cO^{\la}_{K^p}(U)$ is annihilated by $\mathfrak{b}=\{\begin{pmatrix} * & * \\ 0 & * \end{pmatrix}\}\subset \mathfrak{gl}_2(\Q_p)$. Then $f$ is annihilated by $\mathfrak{gl}_2(\Q_p)$, i.e. $f$ is fixed by an open subgroup of $\GL_2(\Q_p)$.
\end{lem}

\begin{proof}
By Theorem \ref{pCR}, $\theta(f)=0$. But $\mathfrak{b}\cdot f=0$. Hence $\theta(f)=\begin{pmatrix} 0 & 0 \\ -1 & 0\end{pmatrix}\cdot f=0$. Thus $\mathfrak{gl}_2(\Q_p)\cdot f=0$.
\end{proof}

\begin{rem}
Here is a sketch of another proof without using $\theta$. By using the action of $\begin{pmatrix} p & 0 \\ 0 & 1\end{pmatrix}$, one may assume $V_\infty$ is a rational open subset of some $\epsilon$-neighborhood of the anticanonical locus $\mathcal{X}^*_{\Gamma (p^\infty)}(\epsilon)_a$ for some $\epsilon\in(0,1/2)$. See Chapter III of \cite{Sch15} for notation here.  Since $f$ is annihilated by $\mathfrak{b}$, it is fixed by an open subgroup $\Gamma'_0$ of the upper-triangular Borel in $\GL_2(\Q_p)$. We may assume $\Gamma'_0\subset K_p$ and denote by $V_{\Gamma'_0}$ the preimage of $V_{K_p}$ in $\mathcal{X}^*_{\Gamma'_0}(\epsilon)_a$. Then $f$ comes from a section in $\cO_{\mathcal{X}^*_{\Gamma'_0}(\epsilon)_a}(V_{\Gamma'_0})$. The key point here is that by Corollary III.2.23 of \cite{Sch15}, there exist Tate's normalized traces from $\cO_{\mathcal{X}^*_{\Gamma'_0}(\epsilon)_a}$ to sections on finite levels. Hence a standard argument shows that $f$ comes from some finite level by analyticity, cf. the proof of Lemma \ref{Ainftyla} or  \cite[Th\'eor\`eme 3.2]{BC16}. Thus $f$ is fixed by some open subgroup of $\GL_2(\Q_p)$.
\end{rem}

\begin{para} \label{setupexplicit}
Now we can give an explicit description of $\cO^{\la}_{K^p}(U)$, following \S4.2, 5.2 of \cite{BC16}. The basic idea is to find ``power series expansions along $e_1,e_2,t$''.  Note that the actions of $K_p$ on these elements are analytic. Our setup is as follows. Fix a compact open subgroup $G_0$ of $K_p$ equipped with an integer valued, saturated $p$-valuation. For example, one can take $1+p^m M_2(\Z_p)$ for sufficiently large $m$.  Then we have $G_n=G_0^{p^n}$ as in \ref{2ndcoord}. Let $V_{G_n}\subset\mathcal{X}_{K^pG_n}$ be the preimage of $V_{K_p}$. By our choice of $U$, cf. Theorem \ref{piHT}, the direct limit $\varinjlim_n H^0(V_{G_n},\cO_{\mathcal{X}_{K^pG_n}})\to H^0(V_\infty,\cO_{\mathcal{X}_{K^p}})$ has dense image. Hence for any $n\geq0$, we can find (cf. \ref{Gnan} and Lemma \ref{Gnnorm})
\begin{itemize}
\item an integer $r(n)>r(n-1)>0$;
\item $x_n\in H^0(V_{G_{r(n)}},\cO_{\mathcal{X}_{K^pG_{r(n)}}})$ such that $\|x-x_n\|_{G_{r(n)}}=\|x-x_n\|\leq p^{-n}$ in $H^0(V_\infty,\cO_{\mathcal{X}_{K^p}})$.
\end{itemize}
As in \ref{ttn}, we can find (after replacing $r(n)$ by a larger number if necessary)
\begin{itemize}
\item $t_n\in \cO_{\mathcal{X}_{K^pG_{r(n)}}}(\mathcal{X}_{K^pG_{r(n)}})$ such that $\|t-t_n\|_{G_{r(n)}}=\|t-t_n\|\leq p^{-n}$ in $H^0(\mathcal{X}_{K^p},\cO_{\mathcal{X}_{K^p}})$.
\end{itemize}
We can define a norm on $H^0(V_\infty,\omega_{K^p})$ by identifying it with $H^0(V_\infty,\cO_{\mathcal{X}_{K^p}})$ using $e_1$. It is easy to see $\varinjlim_n H^0(V_{G_n},\omega)\to H^0(V_\infty,\omega_{K^p})$ also has dense image. Here by abuse of notation, $\omega$ denotes  ``the $\omega$ on $\mathcal{X}_{K^pK_p}$''.  Hence we can also find (after enlarging $r(n)$ if necessary)
\begin{itemize}
\item $e_{1,n}\in H^0(V_{G_{r(n)}},\omega)$ such that $\|e_1-e_{1,n}\|_{G_{r(n)}}=\|e_1-e_{1,n}\|\leq p^{-n}$ in $H^0(V_\infty,\omega_{K^p})$.
\end{itemize}
Note that $e_{1,n}$ is also a basis of $H^0(V_\infty,\omega_{K^p})$, if $n>0$.

As before, we denote by $\mathfrak{b}=\{\begin{pmatrix} * & * \\ 0 & * \end{pmatrix}\}\subset \mathfrak{gl}_2(\Q_p)$ the Lie algebra of the upper-triangular Borel subgroup and by $\mathfrak{n}$ its nilpotent subalgebra. Let
\[u^+=\begin{pmatrix} 0 & 1 \\ 0 & 0 \end{pmatrix},\,h=\begin{pmatrix} 1 & 0 \\ 0 & -1 \end{pmatrix},\,z=\begin{pmatrix} 1 & 0 \\ 0 & 1 \end{pmatrix}\]
be a basis of $\mathfrak{b}$.  Then $u^+\cdot x= u^+\cdot \frac{e_2}{e_1}=1$ and $h\cdot e_1=e_1$ and $z\cdot t=2t$.
\end{para}

\begin{para}[Expansion along $\mathfrak{n}$]
Fix $f\in \cO^{\la}_{K^p}(U)$. We are going to write $f$ as a power series. Suppose $f\in \cO_{K^p}(U)^{G_m-\an}$. By Lemma \ref{LieBound}, there is a constant $C_1$ such that $\|u^+ \cdot s\|_{G_m}\leq p^{C_1}\|s\|_{G_m}$, for any $s\in \cO_{K^p}(U)^{G_m-\an}$. Choose $n''\geq\max(C_1+1/(p-1)+1,m)$. Then 
\[\|(x-x_{n''})^l\frac{(u^+)^l\cdot s}{l!}\|_{G_{r(n'')}}\leq p^{-l}\|s\|_{G_{r(m)}}\]
for any $l\geq 0$ and $s\in \cO_{K^p}(U)^{G_m-\an}$ by a simple calculation. Hence 
\[D_x(s):=\sum_{l= 0}^{+\infty}(-1)^l (x-x_{n''})^l\frac{(u^+)^{l}\cdot s}{l!}\]
defines a bounded map $D_x:\cO_{K^p}(U)^{G_{m}-\an}\to\cO_{K^p}(U)^{G_{r(n'')}-\an}$ with norm $\leq 1$. Moreover, since $u^+\cdot (x-x_{n''})=1$, one checks easily $u^+\circ D_x=0$. Let
\[a_i:=D_x\left(\frac{(u^+)^i\cdot f}{i!}\right)=\sum_{l= 0}^{+\infty}(-1)^l (x-x_{n''})^l\frac{(u^+)^{l+i}\cdot f}{l!i!}\in \cO_{K^p}(U)^{G_{r(n'')}-\an}.\]
Its norm $\|a_i\|_{G_{r(n'')}}\leq \|\frac{(u^+)^i\cdot f}{i!}\|_{G_{m}}\leq p^{(n''-1)i}\|f\|_{G_{m}}$. Thus $\sum_{i\geq0}a_i(x-x_{n''})^i$ is convergent in $\cO_{K^p}(U)^{G_{r(n'')}-\an}$ and a direct computation shows (using $\sum_{i=0}^l (-1)^i \frac{1}{i!(l-i)!}=0$ if $l\geq 1$) that in fact
\[f=\sum_{i= 0}^{+\infty}a_i(x-x_{n''})^i\]
in $\cO_{K^p}(U)^{G_{r(n'')}-\an}$. Note that $u^+\cdot a_i=0$ for any $i$ by our construction.
\end{para}

\begin{para}[Expansion along $h$] \label{Exph}
Next we write $a_i$ as a power series. Note that $\|e_1-e_{1,l}\|_{G_{r(l)}}=\|e_1-e_{1,l}\|\leq p^{-l}$. Hence $\|e_{1,l}\|_{G_{r(l)}}=\|e_1\|_{G_{r(l)}}=1$. For $j\geq 1$, the series
\[\log(\frac{e_1}{e_{1,l}})=\log\left(1+\frac{e_1-e_{1,l}}{e_{1,l}}\right):=-\sum_{j=1}^{+\infty} (-1)^{j}\frac{1}{j}(\frac{e_1}{e_{1,l}}-1)^j\]
converges in $\cO_{K^p}(U)^{G_{r(l)}-\an}$ and has norm $\|\log(\frac{e_1}{e_{1,l}})\|_{G_{r(l)}}\leq p^{-l}$. Moreover, since $u^+\cdot e_1=0,h\cdot e_1=e_1$, it follows that $u^+\cdot\log(\frac{e_1}{e_{1,l}})=0, h\cdot \log(\frac{e_1}{e_{1,l}})=1$. Hence we can repeat the previous process with $u^+,x-x_{n''}$ replaced by $h,
\log(\frac{e_1}{e_{1,n'}})$: by Lemma \ref{LieBound}, there is a constant $C_2$ such that $\|h \cdot s\|_{G_{r(n'')}}\leq p^{C_2}\|s\|_{G_{r(n'')}},s\in \cO_{K^p}(U)^{G_{r(n'')}-\an}$. Choose $n'\geq\max\left(C_2+1/(p-1)+1,r(n'')\right)$. Then we have
\[a_i=\sum_{j=0}^{+\infty}b_{i,j}\left(\log(\frac{e_1}{e_{1,n'}})\right)^j\]
in $\cO_{K^p}(U)^{G_{r(n')}-\an}$, where 
\[b_{i,j}=\sum_{l= 0}^{+\infty}(-1)^l \left(\log(\frac{e_1}{e_{1,n'}})\right)^l\frac{h^{l+j}\cdot a_i}{l!j!}\]
is convergent in $\cO_{K^p}(U)^{G_{r(n')}-\an}$ and has norm $\|b_{i,j}\|_{G_{r(n')}}\leq p^{(n'-1)j+(n''-1)i}\|f\|_{G_{m}}$. Also since $[h,u^+]=2u^+$ and $u^+\cdot a_i=0$, we have
\[u^+\cdot b_{i,j}=h\cdot b_{i,j}=0.\]
\end{para}

\begin{para}[Expansion along $z$]
Finally we can expand $b_{i,j}$ using $z$. This is almost the same as in \ref{Exph}.  Note that $t$ is invertible with  norm $1$, hence $t_l$ is also invertible. Thus $\log(\frac{t}{t_l})$ converges in $\cO_{K^p}(U)^{G_{r(l)}-\an}$ with norm $\leq p^{-l}$. It satisfies $u^+\cdot \log(\frac{t}{t_l})=h\cdot \log(\frac{t}{t_l})=0$ and $z\cdot \log(\frac{t}{t_l})=2$. Then for any sufficiently large integer $n\geq r(n')$, we have
\[b_{i,j}=\sum_{k=0}^{+\infty}c_{i,j,k}\left(\log(\frac{t}{t_{n}})\right)^k\]
converges in $\cO_{K^p}(U)^{G_{r(n)}-\an}$, where 
\[c_{i,j,k}=\frac{1}{2^k}\sum_{l= 0}^{+\infty}(-1)^l (\frac{1}{2}\log(\frac{t}{t_{n}}))^l\frac{z^{l+k}\cdot b_{i,j}}{l!k!}\] 
and has norm 
\[\|c_{i,j,k}\|_{G_{r(n)}}\leq p^{(n-1)k+(n'-1)j+(n''-1)i}\|f\|_{G_{m}}.\]
Since $z$ commutes with $u^+,h$, it follows that 
\[u^+\cdot c_{i,j,k}=h\cdot c_{i,j,k}=z\cdot c_{i,j,k}=0,\]
i.e. $c_{i,j,k}$ is annihilated by $\mathfrak{b}$. By Lemma \ref{annbg}, this implies $c_{i,j,k}$ is fixed $G_{r(n)}$. Therefore
\[c_{i,j,k}\in H^0(V_{G_{r(n)}},\cO_{\mathcal{X}_{K^pG_{r(n)}}}).\] 
In particular $\|c_{i,j,k}\|_{G_{r(n)}}=\|c_{i,j,k}\|$.
\end{para}

\begin{thm} \label{cnijk}
Keep the notation in \ref{setupexplicit}. For any $f\in \cO^{\la}_{K^p}(U)$ which is $G_m$-analytic, there exist an integer $N=N(f)$, bounded above by some constant only depending on $m$, and unique elements $c^{(n)}_{i,j,k}(f)\in H^0(V_{G_{r(n)}},\cO_{\mathcal{X}_{K^pG_{r(n)}}})$ for $i,j,k\geq 0,n\geq N$ with norm $\|c^{(n)}_{i,j,k}(f)\|\leq p^{(n-1)(i+j+k)}\|f\|_{G_m}$ for which
\[f=\sum_{i,j,k\geq0}c^{(n)}_{i,j,k}(f)(x-x_n)^i\left(\log(\frac{e_1}{e_{1,n}})\right)^j\left(\log(\frac{t}{t_{n}})\right)^k\]
holds in $\cO_{K^p}(U)^{G_{r(n)}-\an}$. Conversely, given $n>0$ and $\mathrm{c}_{i,j,k}\in H^0(V_{G_{r(n)}},\cO_{\mathcal{X}_{K^pG_{r(n)}}}),i,j,k\geq 0$ such that $\|\mathrm{c}_{i,j,k}\|\leq p^{(n-1)(i+j+k)}C'$ holds for a uniform constant $C'$, then 
\[f':=\sum_{i,j,k\geq0}\mathrm{c}_{i,j,k}(x-x_n)^i\left(\log(\frac{e_1}{e_{1,n}})\right)^j\left(\log(\frac{t}{t_{n}})\right)^k\]
defines an element in $\cO_{K^p}(U)^{G_{r(n)}-\an}$. We can take $N(f')=n$, hence $c^{(n)}_{i,j,k}(f')=\mathrm{c}_{i,j,k}$.
\end{thm}

\begin{proof}
We have written $f=\sum_{i,j,k\geq 0} c_{i,j,k} (x-x_{n''})^i\left(\log(\frac{e_1}{e_{1,n'}})\right)^j\left(\log(\frac{t}{t_{n}})\right)^k$ with $\|c_{i,j,k}\|_{G_{r(n)}}\leq p^{(n-1)k+(n'-1)j+(n''-1)i}\|f\|_{G_m}$. The rest is to change the coordinates from $x-x_{n''}$ to $x-x_{n}$ and from $\log(\frac{e_1}{e_{1,n'}})$ to $\log(\frac{e_1}{e_{1,n}})$. By our construction, both $(x-x_n)-(x-x_{n''}), \log(\frac{e_1}{e_{1,n}})-\log(\frac{e_1}{e_{1,n'}})\in H^0(V_{G_{r(n)}},\cO_{\mathcal{X}_{K^pG_{r(n)}}})$ and 
\[\|(x-x_n)-(x-x_{n''})\|\leq p^{-n''},\, \|\log(\frac{e_1}{e_{1,n}})-\log(\frac{e_1}{e_{1,n'}})\|\leq p^{-n'}.\]
Hence if we write 
\[\sum_{i,j,k\geq 0} c_{i,j,k} (x-x_{n''})^i\left(\log(\frac{e_1}{e_{1,n'}})\right)^j\left(\log(\frac{t}{t_{n}})\right)^k=\sum_{i,j,k\geq 0} c^{(n)}_{i,j,k}(f) (x-x_{n})^i\left(\log(\frac{e_1}{e_{1,n}})\right)^j\left(\log(\frac{t}{t_{n}})\right)^k\]
 for some $c^{(n)}_{i,j,k}(f)\in H^0(V_{G_{r(n)}},\cO_{\mathcal{X}_{K^pG_{r(n)}}})$, a simple computation gives
\[\|c^{(n)}_{i,j,k}(f)\|_{G_{r(n)}}\leq p^{(n-1)k+(n'-1)j+(n''-1)i}\|f\|_{G_m}\leq  p^{(n-1)(i+j+k)}\|f\|_{G_m}\] 
as $n\geq n'\geq n''$. 
For the uniqueness and the converse part, using the bounds on $c^{(n)}_{i,j,k}(f)$, we can repeat our construction (with $n''=n'=n$ this time) and recover $c^{(n)}_{i,j,k}(f)$ from $f$. We omit the details here.
\end{proof}

\begin{defn} \label{OUnijk}
Keep the notation in \ref{setupexplicit}. For any $n>0$, we define $\cO^{n}(U)\{x,e_1,t\}\subset \cO^{\la}_{K^p}(U)$ to be the subset of $f$ which can be written as 
\[f=\sum_{i,j,k\geq0}{c}_{i,j,k}(x-x_n)^i\left(\log(\frac{e_1}{e_{1,n}})\right)^j\left(\log(\frac{t}{t_{n}})\right)^k\]
with ${c}_{i,j,k}\in H^0(V_{G_{r(n)}},\cO_{\mathcal{X}_{K^pG_{r(n)}}}),i,j,k\geq 0$ such that $\|{c}_{i,j,k}\|\leq p^{(n-1)(i+j+k)}C'$ holds for some uniform constant $C'$. It is a Banach algebra over $C$ with norm 
\[\|f\|_n:=\sup_{i,j,k\geq 0}\|c^{(n)}_{i,j,k}(f)p^{(n-1)(i+j+k)}\|.\]
Let $\cO^{n}(U)^+\{x,e_1,t\}$ be its open unit ball. By Theorem \ref{cnijk}, sending $f$ to $c^{(n)}_{i,j,k}(f)p^{(n-1)(i+j+k)}$ induces an isomorphism of topological $\Z_p$-modules
\[\cO^{n}(U)^+\{x,e_1,t\}\cong\prod_{i,j,k\geq 0}H^0(V_{G_{r(n)}},\cO^+_{\mathcal{X}_{K^pG_{r(n)}}}).\]
\end{defn}

\begin{rem}
It is clear from the proof of Theorem \ref{cnijk} that $\cO^{n}(U)\{x,e_1,t\}$ is independent of the choice of $x_n,e_{1,n},t_n$.
\end{rem}

\begin{rem} 
For any $m>0$, by Theorem \ref{cnijk}, we can find $n$ such that there are continuous embeddings (of Banach spaces)
\[\cO_{K^p}(U)^{G_m-\an}\to \cO^{n}(U)\{x,e_1,t\}\to \cO_{K^p}(U)^{G_{r(n)}-\an}.\]
Hence $\varinjlim_n \cO_{K^p}(U)^{G_n-\an}\cong \varinjlim_n\cO^{n}(U)\{x,e_1,t\}$ as topological spaces. It will be clear later that  $\cO^{n}(U)\{x,e_1,t\}$ behaves a lot better in applications.
\end{rem}

\begin{para} \label{stdesc}
Theorem \ref{cnijk} can be rephrased sheaf-theoretically. Recall some construction in the paragraph above Corollary \ref{LAacyw/osm}. We denote by $\cO^+_{r(n)}$ the push-forward of $\cO^+_{V_{G_{r(n)}}}$ from $V_{G_{r(n)}}$ to $V_{K_p}$, by $\widetilde{\cO}$ the push-forward of $\cO_{V_\infty}$ from $V_\infty$ to $V_{K_p}$, by $\widetilde{\cO}^{n}\subset \widetilde{\cO}$ the subsheaf of $G_n$-analytic sections and by $\widetilde{\cO}^{\la}\subset \widetilde{\cO}$ the subsheaf of $K_p$-locally analytic sections. For each $n>0$, we can define a map 
\[\phi^+_n: \prod_{(i,j,k)\in\N^3} \cO^+_{r(n)}\to \widetilde{\cO}^{\la}\]
sending
$(a_{i,j,k})\in \prod_{i,j,k\geq 0} \cO^+_{r(n)}(W)$ to
\[\sum_{i,j,k\geq0}p^{-(n-1)(i+j+k)}a_{i,j,k}(x-x_n)^i\left(\log(\frac{e_1}{e_{1,n}})\right)^j\left(\log(\frac{t}{t_{n}})\right)^k\in \widetilde{\cO}^{\la}(W),\]
for any open set $W$ of $V_{K_p}$. Let $\phi_n:(\prod_{(i,j,k)\in\N^3} \cO^+_{r(n)})\otimes_{\Z_p}\Q_p\to \widetilde{\cO}^{\la}$ be $\phi^+_n\otimes\Q_p$. Then Theorem \ref{cnijk} implies that $\phi_n$ is an isomorphism onto its image $\im(\phi_n)$ and for any $m>0$, we have $\widetilde{\cO}^{m}\subset\im(\phi_n)\subset \widetilde{\cO}^{r(n)}$ for some $n$.  A direct consequence is
\end{para}

\begin{lem} \label{lemdirlim}
Let $\mathfrak{U}$ be a finite cover of $V_{K^p}$ by rational open subsets. We have following assertions for \v{C}ech cohomology with respect to $\mathfrak{U}$.
\begin{enumerate}
\item $\check{H}^i(\mathfrak{U},\prod_{(i,j,k)\in\N^3}\cO^+_{r(n)})\otimes_{\Z_p}\Q_p=0, i\geq 1$.
\item The direct system $\{\check{H}^i(\mathfrak{U},\widetilde\cO^{n})\}_n$ is essentially zero for any $i\geq 1$.
\end{enumerate}
\end{lem}

\begin{proof}
Let $C^\bullet(\mathfrak{U},\cO^+_{r(n)})$ be the \v{C}ech complex for $\cO^+_{r(n)}$ with respect to $\mathfrak{U}$. Note that if we denote by $\mathfrak{U'}$ the pull-back of $\mathfrak{U}$ to $V_{G_{r(n)}}$, then $C^\bullet(\mathfrak{U},\cO^+_{r(n)})$ is nothing but 
the \v{C}ech complex for $\cO^+_{V_{G_{r(n)}}}$ with respect to $\mathfrak{U}'$. Hence Tate's acyclicity result implies that $H^i(C^\bullet(\mathfrak{U},\cO^+_{r(n)}))\otimes_{\Z_p}\Q_p=0,i\geq 1$. Therefore $H^i(C^\bullet(\mathfrak{U},\cO^+_{r(n)})),i\geq 1$ is annihilated by some $p^k$ by open mapping theorem. From this, we get $H^i(C^\bullet(\mathfrak{U},\prod_{(i,j,k)\in\N^3}\cO^+_{r(n)}))\otimes_{\Z_p}\Q_p=0,i\geq 1$, which is exactly what we want.

The second part is a direct consequence of the first one and the inclusion $\widetilde{\cO}^{m}\subset\im(\phi_n)\subset \widetilde{\cO}^{r(n)}$.
\end{proof}

Since $\cO_{K^p}(U)$  is a Banach space representation of $K_p$, we can talk about its (strongly) $\mathfrak{LA}$-acyclicity (with respect to $K_p$). See \ref{LAacyclic} for more details. 

\begin{prop} \label{cOULAacyc}
$\cO_{K^p}(U)$ is strongly $\mathfrak{LA}$-acyclic for any $U\in\mathfrak{B}$.
\end{prop}

\begin{proof}
By Lemma 5.2 and its proof of \cite{Sch13}, we can find a cover of $V_{K_p}$ by finitely many rational subsets $U_1,\cdots,U_m$ such that each $U_i$ is small in the sense of \ref{setup} with $S=\{\mbox{cusp in }U_i\}$. By Corollary \ref{FEana}, the preimage of $U_i$ in $V_\infty$ is a locally analytic covering of $U_i$. The proposition now follows from Corollary \ref{LAacyw/osm} and  Lemma \ref{lemdirlim}.
\end{proof}

\subsection{Cohomology of \texorpdfstring{$\cO^{\la}_{K^p}$}{Lg} and completed cohomology}

In this subsection, we compare the coherent cohomology of $\cO^{\la}_{K^p}$ (on $\Fl$) and $\cO_{\mathcal{X}_{K^p}}$ (on $\mathcal{X}_{K^p}$). By Scholze's result \cite{Sch15}, the latter one is closely related to the completed cohomology of modular curves introduced earlier by Emerton \cite{Eme06}. The main result is that the cohomology of $\cO^{\la}_{K^p}$ is more or less the subspace of locally analytic vectors in completed cohomology.

\begin{para}
First recall the construction of completed cohomology. See for example \cite{Eme06,CE12}. For a tame level $K^p\subset\GL_2(\A^p_f)$, let
\[\tilde{H}^i(K^p,\Z/p^n):=\varinjlim_{K_p\subset\GL_2(\Q_p)}H^i(Y_{K^pK_p}(\bC),\Z/p^n).\]
Since $Y_{K^pK_p}(\bC)$ is affine, $\tilde{H}^i=0,i\geq 2$. Note that $\tilde{H}^i(K^p,\Z/p^n)$ can also be defined using the compactified modular curves, i.e. the natural  restriction map
\[\varinjlim_{K_p}H^i(X_{K^pK_p}(\bC),\Z/p^n)\to\varinjlim_{K_p}H^i(Y_{K^pK_p}(\bC),\Z/p^n)\]
is an isomorphism. This is clear when $i=0$. When $i=2$, then both sides are zero. When $i=1$, the cokernel of above comes from cohomology of top degree around each cusp, which vanishes as the ramification degree of each cusp is divisible by arbitrary power of $p$.

The completed cohomology of tame level $K^p$ is defined as
\[\tilde{H}^i(K^p,\Z_p):=\varprojlim_{n}\tilde{H}^i(K^p,\Z/p^n).\]
It has a natural admissible continuous action of $\GL_2(\Q_p)$, i.e. $\tilde{H}^i(K^p,\Z_p)/p$ is a smooth admissible representation of $\GL_2(\Q_p)$ over $\F_p$,  cf. \cite[Theorem 1.16]{CE12}. As a consequence, $\tilde{H}^i(K^p,\Z_p)$ has bounded $p$-torsion, i.e. $p$-power torsion classes in $\tilde{H}^i(K^p,\Z_p)$ have bounded exponent.

The following result (essentially due to Scholze) relates completed cohomology and the cohomology of $\cO_{\mathcal{X}_{K^p}}^+$. Here we say a map is Hecke-equivariant if it commutes with Hecke operators away from $p$.
\end{para}

\begin{thm} 
There is a natural $\GL_2(\Q_p)$ and Hecke-equivariant isomorphism of almost $\cO_{C}$-modules
\[\tilde{H}^i(K^p,\Z/p^n)\otimes_{\Z_p/p^n}\cO_{C}/p^n\cong H^i(\mathcal{X}_{K^p},\cO_{\mathcal{X}_{K^p}}^+/p^n),\]
where the right hand side is computed using the analytic topology of $\mathcal{X}_{K^p}$.
\end{thm}

\begin{proof}
Basically the same proof of Theorem IV.2.1 of \cite{Sch15} works here: first we may identify $H^i(X_{K^pK_p}(\bC),\Z/p^n)$ with $H^i_{\mathrm{\acute{e}t}}(\mathcal{X}_{K^pK_p},\Z/p^n)$ by the comparison theorem; then using the primitive comparison theorem (Theorem 1.3 of \cite{Sch13}) and taking  the direct limit over all $K_p$, we obtain the desired almost isomorphism.
\end{proof}

\begin{cor} \label{OCcomp}
There is a natural $\GL_2(\Q_p)$ and Hecke-equivariant isomorphism of almost $\cO_{C}$-modules
\[\tilde{H}^i(K^p,\Z_p)\widehat\otimes_{\Z_p}\cO_{C}\cong H^i(\mathcal{X}_{K^p},\cO_{\mathcal{X}_{K^p}}^+).\]
\end{cor}

\begin{proof}
Since the higher cohomology of $\cO_{\mathcal{X}_{K^p}}^+$ almost vanishes on any affinoid perfectoid open subset (Theorem 1.8.(iv) of \cite{Sch12}), we can compute $H^i(\mathcal{X}_{K^p},\cO_{\mathcal{X}_{K^p}}^+)$ and $H^i(\mathcal{X}_{K^p},\cO_{\mathcal{X}_{K^p}}^+/p^n)$ by \v{C}ech cohomology. Take a finite affinoid perfectoid cover of $\mathcal{X}_{K^p}$ and let $M^\bullet$ be the \v{C}ech complex for $\cO_{\mathcal{X}_{K^p}}^+$ with respect to this cover. Then as almost $\cO_C$-modules, $H^i(M^\bullet)=H^i(\mathcal{X}_{K^p},\cO_{\mathcal{X}_{K^p}}^+)$ and $H^i(M^\bullet/p^n)=H^i(\mathcal{X}_{K^p},\cO_{\mathcal{X}_{K^p}}^+/p^n)$. In view of the previous theorem, the corollary above is reduced to the following lemma.
\end{proof}

\begin{lem} \label{ptorfin}
Let $M^\bullet$ be a bounded above chain complex of $p$-adically complete, $p$-torsion free $\Z_p$-modules. Assume that $\varprojlim_{n}H^i(M^\bullet/p^n)$ has  bounded $p$-torsion for any $i$. Then we have natural isomorphisms
\[H^i(M^\bullet)\xrightarrow{\sim}\varprojlim_{n}H^i(M^\bullet)/p^n\xrightarrow{\sim}\varprojlim_{n}H^i(M^\bullet/p^n).\]
\end{lem}

\begin{proof}
Since $M^\bullet$ is $p$-adically complete and $p$-torsion free, we know that $H^i(M^\bullet)$ is derived $p$-adically complete. See  \cite[\href{https://stacks.math.columbia.edu/tag/091N}{Tag 091N}]{SP} for more details. In particular, 
\[\Hom_{\Z_p}(\Q_p,H^i(M^\bullet))=0.\] 
For the reader's convenience, we recall the argument here. Suppose $\Hom_{\Z_p}(\Q_p,H^i(M^\bullet))\neq0$. We can find $x_n\in H^i(M^\bullet),n=0,1,2,\cdots$ satisfying $px_n=x_{n-1}$ and $x_0\neq 0$. Let $\tilde{x}_n\in M^i$ be a lift of $x_n$. Then there exist $y_n\in M^{i-1}$ such that $dy_n=\tilde{x}_{n-1}-p\tilde{x}_n$. Define $y=y_1+py_2+p^2y_3+\cdots\in M^{i-1}$. One checks easily $dy=\tilde{x}_0$. Hence $x_0=0$, contradiction.

Thus $\Hom_{\Z_p}(\Q_p/\Z_p,H^i(M^\bullet))=0$. Now by the universal coefficient theorem, we have 
\[0\to H^i(M^\bullet)/p^n\to H^i(M^\bullet/p^n)\to H^{i+1}(M^\bullet)[p^n]\to 0. \]
When $n$ varies, the transition map $H^{i+1}(M^\bullet)[p^{n+1}]\to H^{i+1}(M^\bullet)[p^n]$ is multiplication by $p$. Hence $\varprojlim_n H^{i+1}(M^\bullet)[p^n]=\Hom_{\Z_p}(\Q_p/\Z_p,H^i(M^\bullet))=0$. We get $\varprojlim_{n}H^i(M^\bullet)/p^n\xrightarrow{\sim}\varprojlim_{n}H^i(M^\bullet/p^n)$ by passing to the limit over $n$ of the above exact sequence.

It remains to show that $H^i(M^\bullet)\to\varprojlim_{n}H^i(M^\bullet)/p^n$ is an isomorphism. This is clearly surjective as $H^i(M^\bullet)$ is a quotient of $\ker(M^i\xrightarrow{d} M^{i+1})$, which is $p$-adically complete. Let $K$ be the kernel of this map. By our assumption, all of the torsion in $\varprojlim_{n}H^i(M^\bullet)/p^n$ can be annihilated by $p^k$ for some $k$. For any $x\in K$, we can find $x'\in H^i(M^\bullet)$ satisfying $p^{k+1}x'=x$. Then $x'$ maps to a torsion element in $\varprojlim_{n}H^i(M^\bullet)/p^n$. Hence $y=p^kx'\in K$ and $py=x$. Therefore $pK=K$, which implies $K=0$ as $\Hom_{\Z_p}(\Q_p,H^i(M^\bullet))=0$. 
\end{proof}

\begin{rem}
In fact, it is well-known that $\tilde{H}^i(K^p,\Z_p)$ is $p$-torsion free because the $p$-adic \'etale cohomology of curves has no torsion. Hence the proof of Corollary \ref{OCcomp} can be greatly simplified in this case. We decide to present this complicated proof here because it works in more general settings.
\end{rem}

We write $\tilde{H}^i(K^p,\cO_C)=\tilde{H}^i(K^p,\Z_p)\widehat\otimes_{\Z_p}\cO_{C}$ and $\tilde{H}^i(K^p,C)=\tilde{H}^i(K^p,\Z_p)\widehat\otimes_{\Z_p}C$. Then 
\[\tilde{H}^i(K^p,C)\cong H^i(\mathcal{X}_{K^p},\cO_{\mathcal{X}_{K^p}})\] 
is a $\Q_p$-Banach representation of $\GL_2(\Q_p)$. Our main result here identifies its subspace of $\GL_2(\Q_p)$-locally analytic vectors.

\begin{thm} \label{comcccc}
For any $i\geq 0$, there are natural $\GL_2(\Q_p)$ and Hecke-equivariant isomorphisms
\begin{itemize}
\item $\tilde{H}^i(K^p,C)\cong H^i(\Fl,\cO_{K^p})$,
\item $\tilde{H}^i(K^p,C)^{\la}\cong H^i(\Fl,\cO^{\la}_{K^p})$.
\end{itemize}
\end{thm}

\begin{proof}
First note that all higher direct images $R^j{\pi_{\HT}}_*\cO_{\mathcal{X}_{K^p}}=0$, $j>0$. One can check this on a basis $\mathfrak{B}$ of open subsets of $\Fl$ in Theorem \ref{piHT} and invoke the acyclicity result on affinoid perfectoid spaces, cf. Theorem 1.8.(iv) of \cite{Sch12}. Hence
\[H^i(\mathcal{X}_{K^p},\cO_{\mathcal{X}_{K^p}})=H^i(\Fl,{\pi_{\HT}}_*\cO_{\mathcal{X}_{K^p}})=H^i(\Fl,\cO_{K^p}).\]
This proves the first isomorphism in the theorem and shows that  $H^i(\Fl,\cO_{K^p})$ can be computed by the \v{C}ech complex of a finite cover of $\Fl$ of open subsets in $\mathfrak{B}$. We claim the same is true for $\cO_{K^p}^{\la}$, i.e.  
\[H^j(U,\cO_{K^p}^{\la})=0,\]
for any $U\in\mathfrak{B}$ and $j>0$. Therefore $H^i(\Fl,\cO_{K^p}^{\la})$ can also be computed using \v{C}ech cohomology. Recall that $\mathfrak{B}$ is stable under finite intersections. By  Corollaire 4., p.176 of \cite{Gro57}, it suffices to show the \v{C}ech cohomology $\check{H}^j(U,\cO_{K^p}^{\la})=0$ for any $U\in\mathfrak{B}$ and $j>0$. This can be proved in exactly the same way as in the first paragraph of the proof of Corollary \ref{LAacyw/osm} using the acyclicity result in Proposition \ref{cOULAacyc}.

Now let $\mathfrak{U}\subset\mathfrak{B}$ be a finite cover of $\Fl$ and $C^{\bullet}(\mathfrak{U},\cO_{K^p}),C^{\bullet}(\mathfrak{U},\cO_{K^p}^{\la})$ be the \v{C}ech complexes for $\cO_{K^p},\cO_{K^p}^{\la}$ using this cover. Then $C^{\bullet}(\mathfrak{U},\cO_{K^p})$ is a strict complex because $H^i(C^{\bullet}(\mathfrak{U},\cO^+_{K^p}))\cong H^i(\mathcal{X}_{K^p},\cO^+_{\mathcal{X}_{K^p}})\cong  \tilde{H}^i(K^p,\cO_C)$ (as almost $\cO_C$-modules) has bounded $p$-power torsion, where $\cO^+_{K^p}={\pi_{\HT}}_*\cO^+_{\mathcal{X}_{K^p}}$. Moreover, each
 $C^{i}(\mathfrak{U},\cO_{K^p})$ is $\mathfrak{LA}$-acyclic by Proposition \ref{cOULAacyc} and $H^i(C^{\bullet}(\mathfrak{U},\cO_{K^p}))=\tilde{H}^i(K^p,C)$ is $\mathfrak{LA}$-acyclic because $\tilde{H}^i(K^p,\Z_p)$ is an admissible representation of $\GL_2(\Q_p)$ and we can apply the result of Schneider-Teitelbaum, cf. Corollary \ref{admLAacyc}. Hence the theorem follows from the second part of Lemma \ref{LAlongexa} as $(C^{\bullet}(\mathfrak{U},\cO_{K^p}))^{\la}=C^{\bullet}(\mathfrak{U},\cO_{K^p}^{\la})$.
\end{proof}

\section{\texorpdfstring{$\mu$}{Lg}-isotypic part of completed cohomology} \label{muipocc}
The goal of this section is to determine the $\mu$-isotypic part of $\tilde{H}^i(K^p,C)^{\la}$. We will give a complete answer for integral weights as described in the introduction. Roughly speaking, the answer is a mixture of coherent cohomology groups of modular curves at finite level and overconvergent modular forms. Also we will give a $p$-adic Hodge-theoretic interpretation of the horizontal action $\theta_\kh$.

From now on, we assume $C=\bC_p$ is the completion of $\overline\Q_p$.  Then $G_{\Q_p}$ acts continuously on $C,\tilde{H}^i(K^p,\Z_p),\tilde{H}^i(K^p,C)$ and commutes with the action of $\GL_2(\Q_p)$ and Hecke operators away from $p$.

\subsection{A \texorpdfstring{$p$}{Lg}-adic Hodge-theoretic interpretation of \texorpdfstring{$\theta_\kh$}{Lg}}
\begin{para} \label{expkh}
We would like to write down the action $\theta_\kh$ introduced in Corollary \ref{khact} on $\cO^{\la}_{K^p}$ using the explicit description in Theorem \ref{cnijk}. So keep the notation in \S\ref{LS}. In particular, $e_1$ generates $H^0(V, \omega_{K^p})$. Then for any $f\in \cO^{\la}_{K^p}(U)$, we can write 
\[f=\sum_{i,j,k\geq0}c^{(n)}_{i,j,k}(f)(x-x_n)^i\left(\log(\frac{e_1}{e_{1,n}})\right)^j\left(\log(\frac{t}{t_{n}})\right)^k\]
for sufficiently large $n$ (as in Theorem \ref{cnijk}). By our construction in Corollary \ref{khact}, a direct computation shows that $\theta_\kh(\begin{pmatrix}a & 0\\ 0 & d \end{pmatrix})$ acts on $\cO^{\la}_{K^p}(U)$ as $\begin{pmatrix} d & (d-a)x \\ 0 & a \end{pmatrix}\in \cO_{\Fl}(U)\otimes_C\mathfrak{g}$. (To see this, recall that $\begin{pmatrix}
x & x^2\\
-1 & -x 
\end{pmatrix}$ is a generator of $H^0(U,\kn^0)$, cf. \ref{acot0}. One computes directly that $[\begin{pmatrix}d & (d-a)x\\ 0 & a \end{pmatrix},\begin{pmatrix}
x & x^2\\
-1 & -x 
\end{pmatrix}]=(a-d)\begin{pmatrix}
x & x^2\\
-1 & -x 
\end{pmatrix}$ and this $(a-d)$ agrees with $[\begin{pmatrix}a & 0\\ 0 & d \end{pmatrix},\begin{pmatrix} 0 & 1\\ 0 & 0\end{pmatrix}]=(a-d)\begin{pmatrix} 0 & 1\\ 0 & 0\end{pmatrix}$.) Hence
\[c^{(n)}_{i,j,k}(\theta_\kh(\begin{pmatrix}a & 0\\ 0 & d \end{pmatrix})\cdot f)=d(j+1)c^{(n)}_{i,j+1,k}(f)+(a+d)(k+1)c^{(n)}_{i,j,k+1}(f).\]
Let $\chi$ be a weight of $\kh$, i.e. a $C$-linear map $\chi:\kh\to C$. We can write $\chi(\begin{pmatrix}a & 0\\ 0 & d \end{pmatrix})=an_1+dn_2$ for some $n_1,n_2\in C$. Fix $N$ sufficiently large so that 
\[(\frac{t}{t_{N}})^{n_1}:=\sum_{l\geq0}{n_1\choose{l}}(\frac{t}{t_N}-1)^l,\]
\[(\frac{e_1}{e_{1,N}})^{n_2-n_1}:=\sum_{l\geq0}{n_2-n_1\choose{l}}(\frac{e_1}{e_{1,N}}-1)^l\]
converge in $\cO_{K^p}^{G_{r(N)}-\an}$. One checks easily that $\theta_\kh(\begin{pmatrix}a & 0\\ 0 & d \end{pmatrix})$ acts as $\chi$ on $(\frac{t}{t_{N}})^{n_1}(\frac{e_1}{e_{1,N}})^{n_2-n_1}$. Denote by $\cO^{\la,\chi}_{K^p}\subset \cO^{\la}_{K^p}$ the subsheaf of sections of weight $\chi$.
\end{para}

\begin{lem} \label{expchi}
For any weight $\chi$ and $U\in\mathfrak{B}$,
\begin{enumerate}
\item $H^i(\kh,\cO^{\la}_{K^p}(U)\otimes \chi)=0,i\geq1$.
\item Suppose $e_1$ is a generator on $V=\pi_{\HT}^{-1}(U)$, then any $f\in \cO^{\la,\chi}_{K^p}(U)$ can be written as 
\[f=(\frac{t}{t_{N}})^{n_1}(\frac{e_1}{e_{1,N}})^{n_2-n_1}\sum_{i\geq0} c^{(n)}_i(f) (x-x_n)^i \]
for some $n>N$ sufficiently large and $c^{(n)}_i(f)\in H^0(V_{G_{r(n)}},\cO_{\mathcal{X}_{K^pG_{r(n)}}})$ with  bound $\|c^{(n)}_i(f)\|\leq C'p^{(n-1)i}$ for a uniform $C'$.
\end{enumerate}
\end{lem}

\begin{proof}
Using the action of $\GL_2(\Q_p)$, we can reduce to the case considered above, i.e. that $e_1$ generates $H^0(V, \omega_{K^p})$. Note that $(\frac{t}{t_{N}})^{n_1}(\frac{e_1}{e_{1,N}})^{n_2-n_1}$ is invertible. Hence multiplication by it induces an $\kh$-equivariant isomorphism $\cO^{\la}_{K^p}(U)\otimes \chi\xrightarrow{\sim}\cO^{\la}_{K^p}(U)$. Therefore it is enough to prove the case $\chi=0$. The second part is clear in view of the explicit formula for $c^{(n)}_{i,j,k}(\theta_\kh(\begin{pmatrix}a & 0\\ 0 & d \end{pmatrix})\cdot f)$ above. To see the first part, write $\mathfrak{a}=\{\begin{pmatrix} * & 0 \\ 0 & 0 \end{pmatrix}\}\subset \kh$. We claim
\begin{enumerate}
\item $H^1(\mathfrak{a},\cO^{\la}_{K^p}(U))=0$
\item $H^0(\mathfrak{a},\cO^{\la}_{K^p}(U))\subset \cO^{\la}_{K^p}(U)$ is the subset of $f$ such that $c^{(n)}_{i,j,k}(f)=0,k\geq 1$.
\end{enumerate}
Again the second claim is clear by our explicit formula. For the first claim, suppose $f\in \cO^{\la}_{K^p}(U)$ has an expansion as in Theorem \ref{cnijk}. For any $i,j,k\geq0$, let $c'_{i,j,k+1}=\frac{1}{k+1}c^{(n)}_{i,j,k}(f)$.
Then $\|c'_{i,j,k}\|\leq C''p^{(n-0.5)(i+j+k)}$ for some uniform $C''$. Hence
\[f'=\sum_{i,j\geq 0,k\geq1}c'_{i,j,k}(x-x_n)^i\left(\log(\frac{e_1}{e_{1,n}})\right)^j\left(\log(\frac{t}{t_{n}})\right)^k\]
converges in $\cO_{K^p}^{G_{r(n)}-\an}$. One checks  easily $\theta_\kh(\begin{pmatrix}1 & 0\\ 0 & 0 \end{pmatrix})\cdot f'=f$. This proves the vanishing of $H^1(\mathfrak{a},\cO^{\la}_{K^p}(U))$. The same argument also gives
$H^1(\kh/\mathfrak{a},H^0(\mathfrak{a},\cO^{\la}_{K^p}(U)))=0$. By the Hochschild-Serre spectral sequence, we deduce our claim in the lemma.
\end{proof}

We denote by $H^i(\Fl,\cO^{\la}_{K^p})^{\chi}$, the subspace of $H^i(\Fl,\cO^{\la}_{K^p})$ where $\theta_\kh$ acts by $\chi$. Then $H^0(\Fl,\cO^{\la,\chi}_{K^p})=H^0(\Fl,\cO^{\la}_{K^p})^{\chi}$.

\begin{cor} \label{chicomp}
\hspace{2em}
\begin{enumerate}
\item $\theta_\kh(h)\cdot f=0,\,\theta_\kh(z)\cdot f=z\cdot f,\,f\in H^0(\Fl,\cO^{\la}_{K^p})$. In particular, $H^0(\Fl,\cO^{\la,\chi}_{K^p})=0$ if $\chi(h)\neq 0$.
\item If $\chi(h)\neq 0$, then $H^1(\Fl,\cO^{\la,\chi}_{K^p})=H^1(\Fl,\cO^{\la}_{K^p})^{\chi}$. If $\chi(h)=0$, there is a $\mathfrak{g}$-equivariant exact sequence
\[0\to \varinjlim_{K_p\subset \GL_2(\Q_p)}H^0(\mathcal{X}_{K^pK_p},\cO_{\mathcal{X}_{K^pK_p}})\cdot(\frac{t}{t_N})^{n_1}\to H^1(\Fl,\cO^{\la,\chi}_{K^p}) \to H^1(\Fl,\cO^{\la}_{K^p})^{\chi}\to 0,\]
where $n_1=\chi(\begin{pmatrix} 1 & 0 \\ 0 & 0 \end{pmatrix})$ and $t_N\in H^0(\mathcal{X}_{K^pK_p},\cO_{\mathcal{X}_{K^pK_p}})$ sufficiently close to $t$ for some $K_p$.
\end{enumerate}
\end{cor}

\begin{proof}
The action of $\GL_2(\Q_p)$ on $\tilde{H}^0(K^p,C)^{\la}$ factors through the determinant map, hence for any global section $f\in H^0(\Fl,\cO_{K^p}^{\la})\cong \tilde{H}^0(K^p,C)^{\la}$, we have $\theta_\kh(h)\cdot f=0$ and $\theta_\kh(z)\cdot f={z}\cdot f$ by the explicit expression of $\theta_\kh$ in \ref{expkh}. This also shows that the horizontal action $\theta_\kh$ of $\kh$  agrees with the constant action of $\kh\subset\mathfrak{g}$ on $\tilde{H}^0(K^p,C)^{\la}$.

For the second part, it follows from the first part of  Lemma \ref{expchi} that there is a spectral sequence
\[E_2^{pq}= \Ext^p_{C[\kh]}(\chi,H^q(\Fl,\cO^{\la}_{K^p}))\Rightarrow H^{p+q}(\Fl,\cO^{\la,\chi}_{K^p}).\]
The exact sequence of low degrees reads \footnote{In fact, one can avoid the machinery of spectral sequences in this simple case. It was shown in the proof of Theorem \ref{comcccc} that $H^i(\Fl,\cO^{\la}_{K^p})$ can be computed by the \v{C}ech cohomology.  Hence we can use the cover $\{U_1,U_2\}$ of $\Fl$ introduced  in \ref{xy12} below. Therefore $\cO^{\la}_{K^p}(U_1)\oplus \cO^{\la}_{K^p}(U_2)\to \cO^{\la}_{K^p}(U_1\cap U_2)$ computes $H^i(\Fl,\cO^{\la}_{K^p})$. Now this exact sequence comes from applying $\Ext^{\bullet}_{C[\kh]}(\chi,\cdot)$ to this \v{C}ech complex.}
\[0\to \Ext^1_{C[\kh]}(\chi,H^0(\Fl,\cO^{\la}_{K^p}))\to H^1(\Fl,\cO^{\la,\chi}_{K^p}) \to H^1(\Fl,\cO^{\la}_{K^p})^{\chi}\to \Ext^2_{C[\kh]}(\chi,H^0(\Fl,\cO^{\la}_{K^p})).\]
When $\chi(h)\neq 0$, both $\Ext^1$ and $\Ext^2$ vanish because $\theta_\kh(h)$ acts via zero on $H^0(\Fl,\cO^{\la}_{K^p})$. Now assume $\chi(h)=0$. After multiplying by $(\frac{t}{t_N})^{-n_1}\in H^0(\Fl,\cO^{\la}_{K^p})^{-\chi}$, we may assume $\chi=0$. It suffices to show
\begin{itemize}
\item $H^1(\kh,H^0(\Fl,\cO^{\la}_{K^p}))\cong \varinjlim_{K_p\subset \GL_2(\Q_p)}H^0(\mathcal{X}_{K^pK_p},\cO_{\mathcal{X}_{K^pK_p}})$.
\item $H^2(\kh,H^0(\Fl,\cO^{\la}_{K^p}))=0$.
\end{itemize}
Both claims follow from the Hochschild-Serre spectral sequence and 
\begin{itemize}
\item $H^1(\mathfrak{a},\tilde{H}^0(K^p,C)^{\la})=0$;
\item $H^0(\mathfrak{a},\tilde{H}^0(K^p,C)^{\la})=H^0(\mathfrak{g},\tilde{H}^0(K^p,C)^{\la})=\varinjlim_{K_p\subset \GL_2(\Q_p)}H^0(\mathcal{X}_{K^pK_p},\cO_{\mathcal{X}_{K^pK_p}})$
\end{itemize}
by the explicit description of $\tilde{H}^0(K^p,C)$ in \cite[(4.2)]{Eme06}.  
\end{proof}

\begin{para}
It is interesting to investigate the $p$-adic Hodge-theoretic meaning of $\theta_\kh$. First, we generalize the classical notion of the Sen operator.
\end{para}

\begin{defn} \label{Senop}
Suppose  $W$ is a $C$-Banach space equipped with a semi-linear continuous action of an open subgroup of $G_{\Q_p}$, say $G_K$. We say a continuous $C$-linear endomorphism $\theta_{\mathrm{Sen}}\in\End_C(W)$ is a \textit{Sen operator} if it extends the natural action of $1\in \Q_p\cong \Lie(\Gal(\Q_p(\mu_{p^\infty})/\Q_p))$ on the $G_{K(\mu_{p^\infty})}$-smooth, $G_K$-locally analytic vectors of $W$ (viewed as a $\Q_p$-Banach space). 

If $W=\varinjlim_n W_n$ is an increasing union of $C$-Banach spaces $W_n$ equipped with a semi-linear continuous action of an open subgroup of $G_{\Q_p}$, then we say $\theta\in\End_C(W)$ is a Sen operator if $\theta$ preserves each $W_n$ and acts as a Sen operator on it. We also say $W$ has pure Hodge-Tate-Sen weight $k\in C$ if multiplication by $-k$ is a Sen operator on $W$.
\end{defn}

\begin{rem}
The first part of the definition makes sense as for any $G_{K(\mu_{p^\infty})}$-smooth vector $v$, the action of $G_K$ on $v$ factors through a finite-dimensional $p$-adic Lie group which has an open subgroup naturally isomorphic to an open subgroup of $\Gal(\Q_p(\mu_{p^\infty})/\Q_p)$. Also it is clear that this definition is independent of the choice of $K$.
\end{rem}

\begin{rem}
If $W$ is a finite-dimensional $C$-vector space, then in \cite{Sen80} Sen proves that $\theta_{\mathrm{Sen}}$ exists and is unique. However for a general $W$, to what extent $\theta_{\mathrm{Sen}}$ exists uniquely is not known to the author.
\end{rem}

In our case, we will take $W=\cO_{K^p}(U)^{\la}$ and $\tilde{H}^i(K^p,C)^{\la}$. Note that $V_{K_p}$ can be defined over a finite extension $K$ of $\Q_p$ so $\cO_{K^p}(U)^{\la}$ has a natural action of $G_K$.

\begin{thm} \label{Sensf}
$\theta_\kh(\begin{pmatrix}0 & 0\\ 0 & 1 \end{pmatrix})$ is the unique Sen operator on $\cO_{K^p}(U)^{\la}$.
\end{thm}

\begin{rem}
Roughly speaking, this result relates an operator in $p$-adic Hodge theory on (the infinite level) modular curves with some group-theoretic operator ($\theta_\kh$). This is very classical  in the study of the cohomology of locally symmetric spaces using complex Hodge theory. See for example Chapt. II \S4 of \cite{BW00}.
\end{rem}

\begin{rem}
This result and Theorem \ref{Senkh} below are all obtained by explicit calculations. It should be possible to avoid these calculations by further decompleting $\cO_{K^p}(U)^{\la}$ with respect to the action of $G_K$ (i.e. usual Sen theory). I plan to come back to this in a future work.
\end{rem}

\begin{proof}
By Theorem \ref{cnijk}, for each $m$, we have continuous embeddings 
\[\cO_{K^p}(U)^{G_m-\an}\to\cO^{n}(U)\{x,e_1,t\} \to \cO_{K^p}(U)^{G_{r(n)}-\an}\]
for some $n$. See Definition \ref{OUnijk}. Since elements defined over a finite extension of $K$ are dense in $H^0(V_{G_{r(n)}},\cO_{\mathcal{X}_{K^pG_{r(n)}}})$, we may assume $x_n,e_{1,n},t_n$ defined over $K$ after enlarging $K$ if necessary. Hence $G_K$ preserves $\cO^{n}(U)\{x,e_1,t\}$ and  it is enough to show that the action of $\theta_\kh(\begin{pmatrix}0 & 0\\ 0 & 1 \end{pmatrix})$ is a Sen operator on $\cO^{n}(U)\{x,e_1,t\}$ and is the unique one. 

For simplicity, we write $M$ for $\cO^{n}(U)\{x,e_1,t\}$ and for any finite extension $K'$ of $K$, we denote by $M_{K'}\subset M$ the subspace of $f$ with all $c^{(n)}_{i,j,k}(f)$ defined over $K'$. It is clear that 
\[M\cong M_{K'}\widehat{\otimes}_{K'} C.\] 
One useful fact is  
\begin{itemize}
\item $G_{\Q_p}$ acts trivially on $x$  and acts via cyclotomic character on $e_1,t$.
\end{itemize}
From this, one can check that $M_{K'}$ is $G_{K'(\mu_{p^\infty})}$-fixed and $G_{K'}$ acts analytically on it. Conversely, any such element $f$ of $M$ is contained in $M_{K'}$ because $c^{(n)}_{i,j,k}(f)$ can be computed from $f$ using the action of $\mathfrak{g}$ as in \ref{LS}, hence is $G_{K'}$-analytic and an argument using Tate's normalized trace implies that it is in fact fixed by $G_{K'}$. Now a direct computation (using results in \ref{expkh}) shows that $\theta_\kh(\begin{pmatrix}0 & 0\\ 0 & 1 \end{pmatrix})$  agrees with the natural action of $1\in \Q_p\cong \Lie(\Gal(\Q_p(\mu_{p^\infty})/\Q_p))$ on  $M_{K'}$. The uniqueness follows from $M\cong M_{K'}\widehat{\otimes}_{K'} C$.
\end{proof}

\begin{thm} \label{Senkh}
$\theta_\kh(\begin{pmatrix}0 & 0\\ 0 & 1 \end{pmatrix})$ is the unique Sen operator on $H^i(\Fl,\cO_{K^p}^{\la})=\tilde{H}^i(K^p,C)^{\la}$ for any $i$.
\end{thm}

\begin{para} \label{xy12}
The case $i=0$ follows from Theorem \ref{Sensf}. So it suffices to prove the case $i=1$. We introduce some notation first.

We have shown in the proof of Theorem \ref{comcccc} that $H^i(\Fl,\cO^{\la}_{K^p})$ can be computed by \v{C}ech cohomology using a finite cover of $\Fl$ in $\mathfrak{B}$. In particular, we can use the cover $\{U_1,U_2\}$, where $U_1$ (resp. $U_2$) is the subset $|x|\leq 1$ (resp. $|x|\geq 1$). See Theorem \ref{piHT}. Denote by $U_{12}=U_1\cap U_2$. Then
\[\cO_{K^p}^{\la}(U_1)\oplus \cO_{K^p}^{\la}(U_2)\to \cO_{K^p}^{\la}(U_{12})\]
computes $H^i(\Fl,\cO_{K^p}^{\la})$. Denote by $V_?=\pi_{\HT}^{-1}(U_?)$ with $?=1,2,12$. Fix an open subgroup $G_0=1+p^lM_2(\Z_p)$ for some sufficiently large $l\geq 2$ so that  $V_?$ is the preimage of some affinoid subset $V_{?,G_0}$ of  $\mathcal{X}_{K^pG_0}$. As before, we write $G_n=G_0^{p^n}$. Since $e_1$ is a basis on $U_1$, we can find $x_n,e_{1,n},t_n$ as in \ref{setupexplicit} and define $\cO^{n}(U_1)\{x,e_1,t\}$ as in \ref{OUnijk}. Note that by restricting $x_n,e_{1,n},t_n$ on $U_{12}$, we can define $\cO^{n}(U_{12})\{x,e_1,t\}$ similarly with unit ball $\cO^{n}(U_{12})^+\{x,e_1,t\}$.  Now $e_1$ is not a basis on $U_2$. So we work with $\frac{1}{x},e_2,t$ instead of $x,e_1,t$. Write $y=\frac{1}{x}=\frac{e_1}{e_2}$. Since the action of $w=\begin{pmatrix} 0 & 1 \\ 1 & 0\end{pmatrix}$ interchanges $U_1,U_2$, we obtain $y_n:=w^*x_n,e_{2,n}:=w^*e_{1,n}$ on $U_2$. Using $y_n,e_{2,n},t_n$, we can define $\cO^{n}(U_2)\{y,e_2,t\}\subset\cO^{\la}_{K^p}(U_2)$ on $U_2$.
\end{para}

\begin{lem} \label{ye2t}
For any $n>2$, the restriction from $U_2$ to $U_{12}$ induces a map 
\[\cO^{n}(U_2)\{y,e_2,t\}\to \cO^{n}(U_{12})\{x,e_1,t\}\]
preserving norms of $y-y_n$ and $\log(\frac{e_2}{e_{2,n}})$. See \ref{OUnijk} for the definition of norm $\|\cdot\|_n$. 
\end{lem}

\begin{proof}
On $U_{12}$, we have 
\[y-y_n=\frac{1}{x}-y_n=\frac{1}{x_n}\cdot\frac{1}{1+\frac{x-x_n}{x_n}}-y_n=(\frac{1}{x_n}-y_n)-\sum_{i\geq 1}\frac{1}{(-x_n)^{i+1}}(x-x_n)^i.\]
Since $\|x_n\|=\|x\|=1$ and $\frac{1}{x_n}-y_n=(\frac{1}{x_n}-\frac{1}{x})+(y-y_n)$ has norm $\leq p^{-n}$ on $U_{12}$, 
hence 
\[p^{-(n-1)}(y-y_n)\in -\frac{1}{x_n^2}p^{-(n-1)}(x-x_n)+ p\cO^{n}(U_{12})^+\{x,e_1,t\}.\]
The claim for $\log(\frac{e_2}{e_{2,n}})$ can be proved in a similar way.
\end{proof}

\begin{proof}[Proof of Theorem \ref{Senkh}]
Fix an integer $m>2$. Recall that $\tilde{H}^1(K^p,C)= \coker(\cO_{K^p}(U_1)\oplus \cO_{K^p}(U_2)\to \cO_{K^p}(U_{12}))$. By Proposition \ref{cOULAacyc} and Corollary \ref{admLAacyc}, all the terms in the following exact sequence
\[0\to \tilde{H}^0(K^p,C)\to \cO_{K^p}(U_1)\oplus \cO_{K^p}(U_2)\to \cO_{K^p}(U_{12})\]
are strongly $\mathfrak{LA}$-acyclic with respect to the action of $G_0$. From this, it is easy to see that there exists an integer $m'\geq m$ such that $\tilde{H}^1(K^p,C)^{G_{m}-\an}$ is contained in the image of $\cO_{K^p}(U_{12})^{G_{m'}-\an}$. Then by Theorem \ref{cnijk}, we can find $n\geq m'$ so that $\cO_{K^p}(U_{12})^{G_{m'}-\an}\subset \cO^n(U_{12})\{x,e_1,t\}\subset \cO_{K^p}(U_{12})^{G_{r(n)}-\an}$. As a consequence, the inclusion $\tilde{H}^1(K^p,C)^{G_m-\an}\subset \tilde{H}^1(K^p,C)^{G_{r(n)}-\an}$ factors through the largest separated quotient $M^n$ of
\[  \coker( \cO^n(U_{1})\{x,e_1,t\}\oplus \cO^n(U_{2})\{y,e_2,t\}\to \cO^n(U_{12})\{x,e_1,t\})    \]
(using the quotient topology on the cokernel). It suffices to show that $\theta_\kh(\begin{pmatrix} 0 & 0 \\ 0 & 1 \end{pmatrix})$ acts as the unique Sen operator on $M^n$.

Now let $K$ be a finite extension of $\Q_p$ so that $V_{1,G_0},x_n,e_{1,n},t_n$ are all defined over $K$. Denote by $M_1\subset \cO^{n}(U_{1})\{x,e_1,t\}$ the subspace of $G_{K(\mu_{p^\infty})}$-fixed, $G_{K}$-analytic vectors. This is a $K$-Banach algebra with norm $\|\cdot\|_n$ and we denote its unit ball by $M_1^o$. Similarly we can define $M_2\subset \cO^{n}(U_{2})\{y,e_2,t\}$ and $M_{12}\subset\cO^{n}(U_{12})\{x,e_1,t\}$ and their unit balls $M^o_2,M^o_{12}$. Then as in the proof of Theorem \ref{Sensf}, we have $M^o_1\widehat\otimes_{\cO_K} C\cong \cO^{n}(U_{1})\{x,e_1,t\}$, and similar results hold for $M_{12}$ and $M_2$. The previous lemma implies that $\cO_{K^p}^{\la}(U_1)\oplus \cO_{K^p}^{\la}(U_2)\to \cO_{K^p}^{\la}(U_{12})$ has a subcomplex $M_1^o\oplus M_2^o\to M_{12}^o$. Moreover let $M^o$ be the quotient of $\coker(M_1^o\oplus M_2^o\to M_{12}^o)$ by its torsion subgroup. Then
\[M^n=M^o\widehat\otimes_{\cO_K} C\]
and  $M=M^o\otimes_{\Z_p} \Q_p$ is the subspace of $G_{K(\mu_{p^\infty})}$-fixed $G_K$-analytic vectors in $M^n$. As in the proof of Theorem \ref{Sensf}, we know that $\theta_\kh(\begin{pmatrix} 0 & 0 \\ 0 & 1 \end{pmatrix})$ agrees with $1\in \Lie(\Gal(\Q_p(\mu_{p^\infty})/\Q_p))$ on $M_{12}$, hence also agrees on its quotient $M$. This verifies our definition of Sen operator.
\end{proof}

\begin{rem}
It can be proved by explicit calculations that $\coker(M_1^o\oplus M_2^o\to M_{12}^o)$ is already torsion-free, so taking the largest separated quotient is in fact unnecessary.
\end{rem}

\begin{rem}
It is clear from the proof that there is a close relation between two analytic aspects of the completed cohomology: one comes from the group action of $\GL_2(\Q_p)$, and one comes from the Galois action of $G_{\Q_p}$. Also Theorem \ref{Senkh} implies that the two infinitesimal characters are closely related. Both are actually deep theorems in the $p$-adic local Langlands for $\GL_2(\Q_p)$. See Th\'eor\`{e}me V.3 of \cite{CD14} and  Th\'eor\`{e}me 1.2 of \cite{Dos12}.
\end{rem}

\begin{rem} \label{dconSen}
It is natural to ask whether we can show the existence of Sen operator on  $\tilde{H}^i(K^p,C)^{\la}$ in a more direct way. The answer is affirmative. We sketch a construction here based on the Tate-Sen formalism of Berger-Colmez \cite{BC08}. Let $G=G_n$ for some $n$. The key point here is that 
\begin{itemize}
\item the action of $G_{\Q_p}$ on $\tilde{H}^i(K^p,\Q_p)^{G-\an,o}/p$ is trivial on an open subgroup, i.e. the image of $G_{\Q_p}\to \End\left((\tilde{H}^i(K^p,\Q_p)^{G-\an})^{o}/p\right)$ is finite, where $\tilde{H}^i(K^p,\Q_p)^{G-\an,o}$ denotes the unit open ball of $\tilde{H}^i(K^p,\Q_p)^{G-\an}$.
\end{itemize}
The argument is almost the same as the proof of \cite[Theorem 6.1]{Pan20}. Let $W^o$ be the unit open ball of $\tilde{H}^i(K^p,\Q_p)$. Then
\begin{eqnarray*}
\tilde{H}^i(K^p,\Q_p)^{G-\an,o}/p&=&\left(W^o\widehat{\otimes}_{\Z_p}\sC^{\an}(G_n,\Q_p)^o\right)^{G_n}/p\\
&\subseteq&  \left(W^o\otimes_{\Z_p}\sC^{\an}(G_n,\Q_p)^o/p\right)^{G_n}\\
&\subseteq&  \left(W^o/pW^o\otimes_{\F_p}\sC^{\an}(G_n,\Q_p)^o/p\right)^{G_{n+1}}\\
&=& (W^o/pW^o)^{G_{n+1}} \otimes_{\F_p}\sC^{\an}(G_n,\Q_p)^o/p,
\end{eqnarray*}
where the last equality follows from Lemma \ref{modpn+1trivial}. All maps are $G_{\Q_p}$-equivariant. Note that $(W^o/pW^o)^{G_{n+1}}$ is a finite-dimensional $\F_p$-vector space by the admissibility of the completed cohomology. Hence the action of $G_{\Q_p}$ on $\tilde{H}^i(K^p,\Q_p)^{G-\an,o}/p$  necessarily factors through a finite quotient of $G_{\Q_p}$.

Now we can apply Proposition 3.3.1 of \cite{BC08} (with $G_0=G_{\Q_p}, H_0=G_{\Q_p(\mu_{p^\infty})}$) to obtain the Sen operator on $\tilde{H}^i(K^p,\Q_p)^{G-\an}\widehat\otimes_{\Q_p}C$ except that $\tilde{H}^i(K^p,\Q_p)^{G-\an}$ is not finite-dimensional. To get around this, one can argue in a similar way to \ref{VsCan} by finding a dense subspace of $\tilde{H}^i(K^p,\Q_p)^{G-\an}$ which can be written as  a union of finite-dimensional $G_{\Q_p}$-invariant subspaces. For example, one can take the subspace of $\GL_2(\Z_p)$-algebraic vectors. Indeed, the density is clear when $i$=0. When $i=1$, the proof of Lemma \ref{density} below implies that $\GL_2(\Z_p)$-algebraic vectors are dense in  $\tilde{H}^i(K^p,\Q_p)$, hence also dense in $\tilde{H}^i(K^p,\Q_p)^{G-\an}$. Therefore one gets the desired unique Sen operator on  $\tilde{H}^i(K^p,\Q_p)^{G-\an}\widehat\otimes_{\Q_p}C$. This defines the Sen operator on 
\[\varinjlim_n\left(\tilde{H}^i(K^p,\Q_p)^{G_n-\an}\widehat\otimes_{\Q_p}C\right)\cong \tilde{H}^i(K^p,C)^{\la}.\]
To see this isomorphism as LB-spaces, we claim that $\varinjlim_n \left(W^{G_n-\an}\widehat\otimes_{\Q_p}C\right)\cong (W\widehat\otimes_{\Q_p}C)^{\la}$ for any admissible $\Q_p$-Banach representation $W$ of $G_0$.  This is clear for $W=\sC(G_0,\Q_p)$ because we even have $(W\widehat\otimes_{\Q_p}C)^{G_n-\an}\cong W^{G_n-\an}\widehat\otimes_{\Q_p}C$ in this case. The general case can be proved by embedding $W$ into $\sC(G_0,\Q_p)^{\oplus d}$ for some $d$ and applying the acyclicity result \ref{admLAacyc}. We omit the details here.
\end{rem}

\subsection{\texorpdfstring{$\mathfrak{n}$}{Lg}-cohomology (I)} 
\begin{para}
We start to compute the $\mathfrak{n}$-cohomology of $H^1(\Fl,\cO_{K^p}^{\la,\chi})$. Since $\mathfrak{n}$ is one-dimensional, $H^0(\mathfrak{n},\bullet)$ (resp. $H^1(\mathfrak{n},\bullet)$) is the $\mathfrak{n}$-invariants (resp. $\mathfrak{n}$-coinvariants). Denote by $\cO_{K^p}^{\la,\chi,\kn}$ (resp. $(\cO_{K^p}^{\la,\chi})_\kn$) the $\kn$-invariants (resp. $\kn$-coinvariants) of $\cO_{K^p}^{\la,\chi}$. For the purpose of introduction, we assume $\chi(h)\neq 0$. Then $H^0(\Fl,\cO^{\la,\chi}_{K^p})=0$ and we have  $B$-equivariant maps \footnote{Again, these follow from some standard spectral sequences, but one can avoid them here. Using the cover $\{U_1,U_2\}$, we have the exact sequence $0\to \cO^{\la,\chi}_{K^p}(U_1)\oplus \cO^{\la,\chi}_{K^p}(U_2)\to \cO^{\la,\chi}_{K^p}(U_{12})\to H^1(\Fl,\cO_{K^p}^{\la,\chi})\to 0$. Our claim follows on applying $H^i(\kn,\cdot)$ to this exact sequence.}
\begin{eqnarray} \label{ss1}
0\to H^1(\Fl,\cO_{K^p}^{\la,\chi,\kn})\to H^1(\Fl,\cO_{K^p}^{\la,\chi})^{\kn}&\to& H^0(\Fl,(\cO_{K^p}^{\la,\chi})_\kn)\otimes_{C} \kn^*\to 0,\\ \label{ss2}
H^1(\Fl,\cO_{K^p}^{\la,\chi})_{\kn}&\cong& H^1(\Fl,(\cO_{K^p}^{\la,\chi})_\kn).
\end{eqnarray}
Here $\kn^*=\Hom_C(\kn,C)$.  We will compute $\cO_{K^p}^{\la,\chi,\kn},(\cO_{K^p}^{\la,\chi})_\kn$ in this subsection.
\end{para}

\begin{para}
For a weight $\chi$, we write $\chi(\begin{pmatrix} a & 0\\ 0 & d \end{pmatrix}) =n_1a+n_2d$ for some $n_1,n_2\in C$ and sometimes identify $\chi$ with an ordered pair $(n_1,n_2)\in C^2$.

First we compute $(\cO_{K^p}^{\la,\chi})_\kn$. It turns out that for generic $\chi$, this is essentially the space of overconvergent modular forms. We need some notation here. Let $\infty\in \Fl$ be the point where $e_1$ vanishes.  We can consider the fiber of $\cO_{K^p}^{\la,\chi}$ at $\infty$ (as a sheaf of $\cO_{\Fl}$-modules), i.e. $\cO_{K^p}^{\la,\chi}/\km_\infty\cO_{K^p}^{\la,\chi}$. Here $\km_\infty$ denotes the ideal sheaf defined by $\infty$.
\end{para}

\begin{defn} \label{ovc}
For a weight $\chi$, we define 
\[M_\chi^\dagger(K^p):=H^0(\Fl,\cO_{K^p}^{\la,\chi}/\km_\infty\cO_{K^p}^{\la,\chi}),\]
 the fiber of $\cO_{K^p}^{\la,\chi}$ at $\infty$, and call it the space of overconvergent modular forms of weight $\chi$ of tame level $K^p$. There are natural actions of $G_{\Q_p}$, the Borel subgroup $B$ and Hecke operators away from $p$ on this space.
\end{defn}

\begin{para} \label{comf}
To justify its name, we can compare this definition with other existing definitions of overconvergent modular forms in the literature \cite{Katz73,CM98,Pil13,AIS14,CHJ17}... We will only focus on integral weights to illustrate the main difference here.

Let $k$ be an integer. We first introduce overconvergent modular forms with full level at $p$. Let $\Gamma(p^n)=1+p^nM_2(\Z_p)$. We  define the canonical locus $\mathcal{X}_{K^p\Gamma(p^n),c}\subset \mathcal{X}_{K^p\Gamma(p^n)}$ as follows:  using the integral model and the index of Igusa components in Theorem 13.7.6 of \cite{KM85}, $\mathcal{X}_{K^p\Gamma(p^n),c}$ is the tubular neighborhood of the non-singular points ($=$ non-supersingular points) of irreducible components of indices $(\Z/p^n)^{2}\to\Z/p^n,~(1,0)\mapsto 0$. Equivalently, on the ordinary locus of these irreducible components, the canonical subgroup (of level $n$) corresponds to $(*,0)\subset(\Z/p^n)^2$ under the level structure. Sections of $\omega^k$ on any strict neighborhood of $\mathcal{X}_{K^p\Gamma(p^n),c}$ are called overconvergent modular forms of weight $k$ of level $K^p\Gamma(p^n)$, which we denote by $M_k^\dagger(K^p\Gamma(p^n))$. Here a strict neighborhood of $\mathcal{X}_{K^p\Gamma(p^n),c}$ means an open set containing the closure $\bar{\mathcal{X}}_{K^p\Gamma(p^n),c}$. Clearly $M_k^\dagger(K^p\Gamma(p^n))$ forms a direct system when $n$ varies.
\end{para}

\begin{defn} \label{omfk}
Let $k\in \Z$. We define the space of overconvergent modular forms of weight $k$ of tame level $K^p$ as 
\[M^{\dagger}_k(K^p):=\varinjlim_n M_k^\dagger(K^p\Gamma(p^n)).\]
The Galois group $G_{\Q_p}$, Borel subgroup $B$ and Hecke operators away from $p$ act naturally on it.
\end{defn} 

Recall that classical overconvergent modular forms are defined as follows. Let $\Gamma_1(p^n)=\{\begin{pmatrix} a & b \\ c & d \end{pmatrix}\in\GL_2(\Z_p)~|~a-1,d-1,c\in p^n\Z_p\}$. Then we have the canonical locus $\mathcal{X}_{K^p\Gamma_1(p^n),c}\subset \mathcal{X}_{K^p\Gamma_1(p^n)}$, which is defined as the image of $\mathcal{X}_{K^p\Gamma(p^n),c}$ under the natural map $\mathcal{X}_{K^p\Gamma(p^n)}\to \mathcal{X}_{K^p\Gamma_1(p^n)}$. Using Katz-Mazur's integral model, this can also be defined as tubular neighborhood of non-supersingular points of irreducible components whose ordinary points classify ordinary elliptic curves with level structure at $p$ given by the canonical subgroup (of level $n$). We define $M_k^\dagger(K^p\Gamma_1(p^n))$ similarly as sections defined in a strict neighborhood of $\mathcal{X}_{K^p\Gamma_1(p^n),c}$. Note that this is slightly different from the usual definition as $\mathcal{X}_{K^p\Gamma_1(p^n)}$ is not connected.

We can obtain $M^{\dagger}_k(K^p)$  from $\varinjlim_n M_k^\dagger(K^p\Gamma_1(p^n))$ by inverting the action of $\begin{pmatrix} p^{-1} & 0 \\ 0 & 1 \end{pmatrix}$. We can also recover $\varinjlim_n M_k^\dagger(K^p\Gamma_1(p^n))$ from $M^{\dagger}_k(K^p)$ by taking invariants of $N_0=\begin{pmatrix} 1 & \Z_p\\ 0 & 1 \end{pmatrix}$:
\[\varinjlim_n M_k^\dagger(K^p\Gamma_1(p^n))=M^{\dagger}_k(K^p)^{N_0}.\]

\begin{prop} \label{defnocmf}
Suppose $\chi(\begin{pmatrix} a & 0 \\ 0 & d \end{pmatrix})=n_1a+n_2d$ with $n_1,n_2\in\Z$. Let $k=\chi(h)=n_1-n_2$. There is a canonical isomorphism induced by multiplication by $e_2^{-k}t^{n_1}$:
\[\phi_\chi:M^{\dagger}_k(K^p) \xrightarrow{\sim} M_\chi^\dagger(K^p)\]
satisfying 
\[(g_1,g_2)\cdot\phi_{\chi}(f)=\phi_{\chi}\left((g_1,g_2)\cdot f\right)d^{-k}\varepsilon_p(ad)^{n_1}\varepsilon_p(g_2)^{n_1-k},\] 
$f\in M^{\dagger}_k(K^p),~g_1=\begin{pmatrix} a & b \\ 0 & d \end{pmatrix}\in B,~g_2\in G_{\Q_p}$. Recall that $\varepsilon_p:G_{\Q_p}\to\Z_p^\times$ is the $p$-adic cyclotomic character and regarded as a character $\Q_p^\times\to\Z_p^\times$ sending $x$ to $x|x|$ via local class field theory. 
\end{prop}

\begin{rem} \label{e1e2tact}
For a $B\times G_{\Q_p}$-representation $W$ and integers $i,j,k$, we write $W\cdot e_1^ie_2^jt^k$ to denote the twist of $W$ by the character sending $(\begin{pmatrix} a & b\\ 0 & d \end{pmatrix},g)\in B\times G_{\Q_p}$ to $a^id^j\varepsilon_p(ad)^k\varepsilon_p(g)^{i+j+k}$. Therefore, we can rewrite the isomorphism in Proposition \ref{defnocmf} as 
\[M^{\dagger}_k(K^p)\cdot e_2^{-k}t^{n_1}\cong M_\chi^\dagger(K^p).\]
\end{rem}

\begin{proof}
For any $U\in\mathfrak{B}$, a neighborhood of $\infty$ not containing the zero of $x$, since $e_1$ is not a basis now, as in \ref{xy12}, we can use $e_2$ instead. More precisely, let $y=\frac{1}{x}=\frac{e_1}{e_2}$. We can find $y_n,e_{2,n},t_n$ as in \ref{setupexplicit} and define $\cO^n(U)\{y,e_2,t\}$. Denote by $\cO^n(U)\{y\}\subset \cO^n(U)\{y,e_2,t\}$ the subset on which $\theta_\kh$ acts by zero. Then $\cO^{\la,\chi}_{K^p}(U)=\varinjlim_n \cO^n(U)\{y\}\cdot {t}^{n_1}(\frac{e_2}{e_{2,n}})^{n_2-n_1}$. Equivalently, as in Lemma \ref{expchi}, any element in $\cO^n(U)\{y\}\cdot {t}^{n_1}(\frac{e_2}{e_{2,n}})^{n_2-n_1}$ can be written as
\[{t}^{n_1}(\frac{e_2}{e_{2,n}})^{n_2-n_1}\sum_{i\geq0} c_i (y-y_n)^i, \]
where $c_i\in H^0(V_{G_{r(n)}},\cO_{\mathcal{X}_{K^pG_{r(n)}}})$ with  bound $\|c_i\|\leq C'p^{(n-1)i}$ for a uniform $C'$. There is no need to put $t_n^{n_1}$ here because $n_1$ is an integer.

 We can take $G_{r(n)}$ to be some $\Gamma(p^m)$. We claim
\begin{itemize}
\item $V_{G_{r(n)}}$ is a strict neighborhood of $\mathcal{X}_{K^pG_{r(n)},c}$.
\end{itemize}
Since both are affinoid subsets, it is enough to check that $V_{G_{r(n)}}$ contains the closure of the non-cusp classical points of $\mathcal{X}_{K^pG_{r(n)},c}$. We check this relation for their preimages in $\mathcal{X}_{K^p}$. Note that $U$ is a $G_{r(n)}$-invariant open neighborhood of $\infty$. In particular, $U$ contains $G_{r(n)}\cdot \infty$ and $V_\infty$ contains the closed set $\pi^{-1}_{\HT}(G_{r(n)}\cdot \infty)$. Our claim follows from the following lemma.

\begin{lem} \label{preimrt}
The preimages of non-cusp points of $\mathcal{X}_{K^p\Gamma(p^m),c}$ in $\mathcal{X}_{K^p}$ map to $\Gamma(p^m)\cdot \infty$ under the Hodge-Tate period map for any $m\geq 1$. 
\end{lem}

\begin{proof}
If $m=1$, this follows from Lemma III.3.14 of \cite{Sch15} by noting our canonical locus $\mathcal{X}_{K^p\Gamma(p),c}$ does not intersect the anticanonical locus in the reference. The general case $m\geq 2$ can be reduced to this case by using the action of $\begin{pmatrix} p^{m-1} & 0 \\ 0 & 1 \end{pmatrix}$.
\end{proof}

Now fix such a $U$. Note that $\infty$ is defined by $y=0$. We need to construct an isomorphism 
\[\cO^{\la,\chi}_{K^p}(U)/(y)\xrightarrow{\sim}\varinjlim_n M_k^\dagger(K^p\Gamma(p^n)).\]
Since $\cO^n(U)\{y\}\cong \cO_{V_{G_{r(n)}}}^+(V_{G_{r(n)}})[[p^{-(n-1)}(y-y_n)]]\otimes_{\Z_p}\Q_p$, there is a natural $\cO_{V_{G_{r(n)}}}(V_{G_{r(n)}})$-algebra homomorphism
\[\varphi_n:\cO^n(U)\{y\}/(y)\to\cO_{V_{G_{r(n)}}}(V''_{G_{r(n)}})\]
sending $p^{-(n-1)}(y-y_n)$ to $-p^{-(n-1)}y_n\in \cO_{V_{G_{r(n)}}}(V''_{G_{r(n)}})$, where $V''_{G_{r(n)}}\subset V_{G_{r(n)}}$ is the rational subset defined by $|p^{-n}y_n|\leq 1$. The same argument as above shows that $V''_{G_{r(n)}}$ is a strict neighborhood of $\mathcal{X}_{K^pG_{r(n)},c}$. Let $V'_{G_{r(n)}}\subset V_{G_{r(n)}}$ defined by $|p^{-(n-1)}y_n|\leq 1$. We claim that 
\begin{itemize}
\item the image of $\varphi_n$ contains analytic functions convergent on $V'_{G_{r(n)}}$.
\end{itemize}
Indeed, since $\cO_{V_{G_{r(n)}}}(V'_{G_{r(n)}})\cong \cO_{V_{G_{r(n)}}}(V_{G_{r(n)}})\langle p^{-(n-1)}y_n\rangle$, this claim is clear in view of the definition of $\varphi_n$.

Hence we have a map
\[\cO^n(U)\{y\}\cdot t^{n_1}(\frac{e_2}{e_{2,n}})^{n_2-n_1} /(y) \to \omega^k(V''_{G_{r(n)}})\] 
by sending $ft^{n_1}(\frac{e_2}{e_{2,n}})^{n_2-n_1} $ to $\varphi_n(f)e_{2,n}^{n_1-n_2}$, i.e. multiplication by $t^{-n_1}e_2^{n_1-n_2}$. Clearly this map is compatible when $n$ varies and induces a map $\cO^{\la,\chi}_{K^p}(U)/(y)\to\varinjlim_n M_k^\dagger(K^p\Gamma(p^n))$. 

To prove this is an isomorphism, it suffices to show that for any $n$ and strict neighborhood $W$ of $\mathcal{X}_{K^p\Gamma(p^n),c}$, we can find $m\geq n$ such that the preimage of $W$ in $\mathcal{X}_{K^pG_{r(m)}}$ contains $V'_{G_{r(m)}}$. Let $\tilde{W}$ be the preimage of $W$ in $\mathcal{X}_{K^p}$. Then it is an open neighborhood of $\pi_{\HT}^{-1}(\infty)=\{x\in\mathcal{X}_{K^p}\,|\,|y(x)|=0\}$. Therefore $\{x\in\mathcal{X}_{K^p}\,|\,|p^{-m}y(x)|\geq 1\},m=0,1,\cdots$ and $\tilde{W}$ form an open cover of $V_\infty$.
Note that $V_\infty$ is quasi-compact because it is affinoid, hence $\{z\in V_\infty\,|\,|p^{-(m-1)}y(z)|\leq 1\}\subset \tilde{W}$ for some $m$. Using $\|y-y_m\|\leq p^{-m}$ on $V_\infty$, it is easy to see that this $m$ works here.

The claim for the $B$-action follows from the construction and the fact that 
\[\begin{pmatrix} a & b\\ 0 & d \end{pmatrix}\cdot e_2=be_1+d e_2\]
 whose reduction modulo $\km_\infty$ is $de_2$ and that $\begin{pmatrix} a & b\\ 0 & d \end{pmatrix}\cdot t=ad|ad|$. The claim for $G_{\Q_p}$-action is clear as $G_{\Q_p}$ acts via the cyclotomic character on $e_2,t$.
\end{proof}

We record the following result obtained in this proof.
\begin{lem} \label{twotopcomp}
Given an open subset $U$ of $\Fl$ containing $\infty$, there exists a strict neighborhood of $\mathcal{X}_{K^p\Gamma(p^n),c}$ for some $n$, whose preimage in $\mathcal{X}_{K^p}$ is contained in $\pi_{\HT}^{-1}(U)$. Conversely, for any $n>0$ and any strict neighborhood $W$ of $\mathcal{X}_{K^p\Gamma(p^n),c}$, there exists an open subset of $\Fl$ containing $\infty$ whose preimage in $\mathcal{X}_{K^p}$ is contained in the preimage of $W$.
\end{lem}

\begin{proof}
For the first claim, we may assume $U\in\mathfrak{B}$. Then this was proved around Lemma \ref{preimrt}. The second claim follows from the argument in the second to last paragraph of the proof of Theorem \ref{defnocmf}.
\end{proof}

We are going to compare $M^\dagger_\chi(K^p)$ with the $\kn$-coinvariants $(\cO_{K^p}^{\la,\chi})_\kn$ of $\cO^{\la,\chi}_{K^p}$. Note that there is a natural action of $\kh$ on $(\cO_{K^p}^{\la,\chi})_\kn$ induced from the action of $\mathfrak{b}$.  We will always use this constant $\kh$-action from now on, unless otherwise specified (to distinguish with the horizontal action $\theta_\kh$). Since  $M^\dagger_\chi(K^p)$ is the fiber of $\cO_{K^p}^{\la,\chi}$ at $\infty$,  the action of $\mathfrak{b}$ on it factors through $\kh$. It is easy to see that 
\begin{itemize}
\item $\kh$ acts on  $M^\dagger_\chi(K^p)$ via $\chi$.
\end{itemize}

We denote by $i_\infty$ the natural embedding $\infty\hookrightarrow \Fl$.

\begin{prop} \label{kncoin}
Let $\chi=(n_1,n_2)$ be a weight. 
\begin{enumerate}
\item $(\cO_{K^p}^{\la,\chi})_\kn$ is a sky-scrapper sheaf supported at $\infty$, hence has no $H^1$. 
\item $\kn(\cO_{K^p}^{\la,\chi})\subset\km_\infty\cO_{K^p}^{\la,\chi}$. Thus we get a natural exact sequence
\[0\to (i_{\infty})_*N^{\chi}\to (\cO_{K^p}^{\la,\chi})_\kn\to (i_{\infty})_*M^\dagger_\chi(K^p)\to 0\]
for some $N^\chi$. Moreover $\kh$ acts on $N^\chi$ via $(n_2+1,n_1-1)$. In particular, this exact sequence splits naturally if $\chi(h)\neq 1$, i.e. $\chi|_{\kh_0}\neq \rho|_{\kh_0}$. Recall that $\rho$ denotes the half-sum of positive roots.
\item $(\cO_{K^p}^{\la,\chi})_\kn=(i_{\infty})_*M^\dagger_\chi(K^p)$, i.e. $N^\chi=0$ if $\chi(h)\neq 0,-1,-2,\cdots$ and is $p$-adically non-Liouville (See below for  a definition).
\item $(\cO_{K^p}^{\la,\chi})_\kn\cong (i_{\infty})_*M^\dagger_\chi(K^p)\oplus (i_{\infty})_*M^\dagger_\chi(K^p)\cdot (e_1/e_2)^{1-\chi(h)}$ if $\chi(h)\in\{0,-1,-2,\cdots\}$. \end{enumerate} 
\end{prop}

\begin{rem}
Recall that according to Clark \cite{Clark66}, we say $\alpha\in C$ is  $p$-adically non-Liouville if 
\[\lim\inf_{n\to\infty}|\alpha+n|^{1/n}>0,\]
i.e. it cannot be well approximated by rational integers. It appears naturally in the study of $p$-adic differential equations. Suppose $\alpha\notin \Z$ and consider the inhomogeneous differential equation around $x=0$
\[(x\frac{d}{dx}+\alpha)y=\frac{1}{1-x}.\]
It has a unique formal power series solution $\sum_{n\geq 0}\frac{1}{n+\alpha}x^n$, whose convergence is equivalent with $\alpha$ being non-Liouville. In fact, it will be clear to the reader that a variant of this differential equation will show up in our case.
\end{rem}

\begin{proof}
Let $U\in \mathfrak{B}$. We first show $(\cO_{K^p}^{\la,\chi}(U))_\kn=0$ if $e_1$ is a basis on $U$. Clearly this implies that $(\cO_{K^p}^{\la,\chi})_\kn$ is supported at $\infty$. On such a $U$, using the notation in Lemma \ref{expchi}, any element in $f\in \cO_{K^p}^{\la,\chi}(U)$ can be written as 
\[f=(\frac{t}{t_{N}})^{n_1}(\frac{e_1}{e_{1,N}})^{n_2-n_1}\sum_{i\geq0} c^{(n)}_i(f) (x-x_n)^i .\]
Since $u^+\in\kn$ acts trivially on $e_1,t$ and acts as the usual derivation on $(x-x_n)$, an argument similar to the proof of Lemma \ref{expchi} shows that $\kn$ is a surjection on $\cO_{K^p}^{\la,\chi}(U)$.

Now let $U\in \mathfrak{B}$ containing $\infty$ and on which $e_2$ is a basis. 
As in the proof of Proposition \ref{defnocmf}, any $f\in \cO_{K^p}^{\la,\chi}(U)$ can be written as 
\[f=(\frac{t}{t_{N}})^{n_1}(\frac{e_2}{e_{2,N}})^{n_2-n_1}\sum_{i\geq0} c_i (y-y_n)^i\]
for some $n$ and $c_i$. Using $u^+\cdot e_2=e_1=ye_2$ and $u^+\cdot y=-y^2$, we get 
\begin{eqnarray*}
u^+\cdot f&=&  y(\frac{t}{t_{N}})^{n_1}(\frac{e_2}{e_{2,N}})^{n_2-n_1}(\sum_{i\geq0} (n_2-n_1)c_i(y-y_n)^i -y\sum_{i\geq 1}ic_i (y-y_n)^{i-1})\\ 
&=& y(\frac{t}{t_{N}})^{n_1}(\frac{e_2}{e_{2,N}})^{n_2-n_1}\sum_{i\geq0} ((n_2-n_1-i)c_i-(i+1)c_{i+1}y_n ) (y-y_n)^i\in y\cO_{K^p}^{\la,\chi}(U).
\end{eqnarray*}
Hence $\kn(\cO_{K^p}^{\la,\chi})\subset\km_\infty\cO_{K^p}^{\la,\chi}$ and we can define $N^{\chi}$ as in the lemma. Moreover, one can compute
\begin{eqnarray*}
(h-(n_2-n_1+2))\cdot(yf)&=& y(\frac{t}{t_{N}})^{n_1}(\frac{e_2}{e_{2,N}})^{n_2-n_1}(h+2(n_1-n_2))\cdot(\sum_{i\geq 0}c_i(y-y_n)^i)\\
&=& -2u^+\cdot f.
\end{eqnarray*}
Also it is easy to check that $z$ acts on everything via $\chi(z)=n_1+n_2$. Hence we conclude that $\kh$ acts via $(n_2+1,n_1-1)$ on $N^\chi$.

The last two claims require a bit more work. We denote by $U_n\subset U$ the rational subset defined by $|p^{-n}y|\leq 1$. Then $U_n\in\mathfrak{B}$ and contains $\infty$. Note that $\pi_\HT^{-1}(U_n)$ is also the preimage of $V''_{G_{r(n)}}=\{x\in V_{G_{r(n)}},\,|y_n(x)p^{-n}|\leq 1\}$ because $\|y-y_n\|\leq p^{-n}$. For $n$ sufficiently large so that both  $(\frac{t}{t_n})^{n_1}$, $(\frac{e_2}{e_{2,n}})^{n_2-n_1}$ converge, we denote by $C_n\subset\cO_{K^p}^{\la}(U_n)$ the subset of elements of the form
\[(\frac{t}{t_n})^{n_1}(\frac{e_2}{e_{2,n}})^{n_2-n_1}\sum_{i\geq0} c_i y^i, \]
where $c_i\in H^0(V''_{G_{r(n)}},\cO_{\mathcal{X}_{K^pG_{r(n)}}})$ with  bound $\|c_i\|\leq C'p^{(n-1)i}$ for a uniform $C'$. Note that as $\|y\|\leq p^{-n}$ in $H^0(\pi_{\HT}^{-1}(U_n),\cO_{\mathcal{X}_{K^p}})$, we can choose $y_n=0$ here. As before, $p^{(n-1)i}c_i$ defines an isomorphism $C_n\cong(\prod_{i\geq 0}H^0(V''_{G_{r(n)}},\cO^+_{\mathcal{X}_{K^pG_{r(n)}}}))\otimes_{\Z_p}\Q_p$ and $C_n$ forms a direct system when $n$ varies. It is easy to see that 
\[(\cO_{K^p}^{\la,\chi})_\kn=(i_{\infty})_*(\varinjlim_n (C_n)_\kn).\]

Assume that $n_1-n_2=\chi(h)\neq 0,-1,-2,\cdots$ and is $p$-adically non-Liouville. For any $f\in C_n$ and $n$ sufficiently large, we need to show that $yf\in u^+(C_m)$ for some $m\geq n$. Write 
\[f=(\frac{t}{t_n})^{n_1}(\frac{e_2}{e_{2,n}})^{n_2-n_1}\sum_{i\geq0} c_i y^i,\]
for some $c_i\in H^0(V''_{G_{r(n)}},\cO_{\mathcal{X}_{K^pG_{r(n)}}})$ and $\|c_i\|\leq C'p^{(n-1)i}$ for some $C'$. By our assumption, we can find $m\geq n$ such that 
\[|n_1-n_2+i|\geq p^{-i(m-n)}\]
for $i$ sufficiently large. Then $\|\frac{c_i}{-n_1+n_2-i}\|\leq C'p^{(m-1)i}$ for $i$ large enough and 
\[(\frac{t}{t_n})^{n_1}(\frac{e_2}{e_{2,n}})^{n_2-n_1}\sum_{i\geq0} \frac{c_i}{-n_1+n_2-i} y^i\]
defines an element $g\in C_m$. A simple computation gives that $u^+\cdot g=yf$. Hence $N^\chi=0$ and $(\cO_{K^p}^{\la,\chi})_\kn=(i_{\infty})_*M^\dagger_\chi(K^p)$ in this case.

Now assume $n_1-n_2=\chi(h)\in \{0,-1,-2,\cdots\}$. Recall that for $f\in C_n$ as above, we have
\[u^+\cdot f=(\frac{t}{t_n})^{n_1}(\frac{e_2}{e_{2,n}})^{n_2-n_1}\sum_{i\geq0} (n_2-n_1-i)c_i y^{i+1}.\]
Hence the $y^{n_2-n_1+1}$ term in the summation always has zero coefficient. Conversely, if  $f=(\frac{t}{t_n})^{n_1}(\frac{e_2}{e_{2,n}})^{n_2-n_1}\sum_{i\geq0} c_i y^i\in C_n$ with $c_{n_2-n_1}=0$, we have $|n_2-n_1-i|\geq p^{-i}$ for $i$ sufficiently large and $(\frac{t}{t_n})^{n_1}(\frac{e_2}{e_{2,n}})^{n_2-n_1}\sum_{0\leq i\neq n_2-n_1} \frac{c_i}{n_2-n_1-i} y^i$ converges to an element $g\in C_{n+1}$. Again it is easy to check that $yf=u^+\cdot g$. Thus we can define an isomorphism 
\[M^\dagger_\chi(U^p)\xrightarrow{\sim} N^\chi\]
by sending $f$ to $\tilde fy^{n_2-n_1+1}$, where $\tilde f\in \cO_{K^p}^{\la,\chi}(U)$ is a lifting of $f\in  \cO_{K^p}^{\la,\chi}(U)/(y)$. Since $y=e_1/e_2$, this isomorphism becomes $B$-equivariant if we twist the left hand side by $(e_1/e_2)^{1-n_1+n_2}$.
\end{proof}

Next we compute $\cO^{\la,\chi,\kn}_{K^p}$, the $\kn$-invariants of $\cO^{\la,\chi}_{K^p}$.

\begin{prop} \label{kninv}
\hspace{2em}
\begin{enumerate}
\item Suppose $U$ does not contain $\infty$, i.e. $e_1$ is a basis on $U$. Fix $N$ large enough so that $(\frac{t}{t_{N}})^{n_1},(\frac{e_1}{e_{1,N}})^{n_2-n_1}$ converge. Then any $f\in \cO^{\la,\chi,\kn}_{K^p}(U)$ can be written as 
\[f=(\frac{t}{t_{N}})^{n_1}(\frac{e_1}{e_{1,N}})^{n_2-n_1}c \]
for some $n\geq N$ sufficiently large and $c\in H^0(V_{G_{r(n)}},\cO_{\mathcal{X}_{K^pG_{r(n)}}})$. 
\item Suppose $e_2$ is a basis on $U\in\mathfrak{B}$ and $\chi(h)=n_1-n_2\in \{0,-1,-2,\cdots\}$. Fix $N$ large enough so that $(\frac{t}{t_{N}})^{n_1}$ converges. Then any $f\in \cO^{\la,\chi,\kn}_{K^p}(U)$ can be written as 
\[f=(\frac{t}{t_{N}})^{n_1}(\frac{e_2}{e_{2,N}})^{n_2-n_1}cy^{n_2-n_1}=(\frac{t}{t_{N}})^{n_1}(\frac{e_1}{e_{2,N}})^{n_2-n_1}c \]
for some $n\geq N$ and $c\in H^0(V_{G_{r(n)}},\cO_{\mathcal{X}_{K^pG_{r(n)}}})$.
\item $\kh$ acts on $\cO^{\la,\chi,\kn}_{K^p}(U)$ via $(n_2,n_1)$ for any open subset $U\subset\Fl$.
\item Suppose $\chi(h)\neq 0,-1,-2,\cdots$. Then 
\[\varinjlim_{U\ni\infty} \cO^{\la,\chi,\kn}_{K^p}(U)=0,\]
 i.e. the stalk of $\cO^{\la,\chi,\kn}_{K^p}$ at $\infty$ is zero.
\end{enumerate}
\end{prop}

\begin{rem}
The overconvergent modular forms of weight $\chi$ introduced in \cite{Pil13,AIS14,CHJ17} are essentially the stalk of $\cO^{\la,\chi,\kn}_{K^p}$ at a $\Q_p$-rational point of $\Fl\setminus\{\infty\}$. For example, the $\infty$ in \cite[Theorem 1.1]{CHJ17} corresponds to the locus where $e_2=0$, i.e. $x=0$ in our setup.
\end{rem}

\begin{proof}
For the first part, as in Lemma \ref{expchi}, any $f\in \cO_{K^p}^{\la,\chi}(U)$ can be written as 
\[f=(\frac{t}{t_{N}})^{n_1}(\frac{e_1}{e_{1,N}})^{n_2-n_1}\sum_{i\geq0} c^{(n)}_i(f) (x-x_n)^i \]
for some $n$ and $c^{(n)}_i(f)$ are unique for such a $n$. Since $u^+$ acts as the usual derivation on $(x-x_n)$, we see that $u^+\cdot f=0$ if and only if $c^{(n)}_i(f)=0$, $i>0$. This proves the first part.

Suppose  $\infty\in U$ and $e_2$ is a basis on $U$. As in the proof of Proposition \ref{kncoin}, for any $f\in \cO_{K^p}^{\la,\chi}(U)$, we can find $n$ so that $f|_{U_n}$ can be written as $f=(\frac{t}{t_{N}})^{n_1}(\frac{e_2}{e_{2,N}})^{n_2-n_1}\sum_{i\geq0} c_i y^i$ for some $c_i \in H^0(V''_{G_{r(n)}},\cO_{\mathcal{X}_{K^pG_{r(n)}}})$. Then
\[u^+\cdot f = (\frac{t}{t_{N}})^{n_1}(\frac{e_2}{e_{2,N}})^{n_2-n_1}\sum_{i\geq0} (n_2-n_1-i)c_i y^{i+1}.\]
Hence that $u^+\cdot f=0$ is equivalent with 
\[(n_2-n_1-i)c_i=0,\, i\geq 0.\]
This implies all $c_i=0$, i.e. $f|_{U_n}=0$ if $\chi(h)=n_1-n_2\neq 0, -1,-2,\cdots$. When $\chi(h)\in \{0, -1,-2,\cdots\}$, we see that all $c_i=0$ except $i=n_2-n_1$ and 
\[y|_{U_n}=(\frac{t}{t_{N}})^{n_1}(\frac{e_2}{e_{2,N}})^{n_2-n_1}c_iy^{n_2-n_1}=(\frac{t}{t_{N}})^{n_1}(\frac{e_1}{e_{2,N}})^{n_2-n_1}c_i.\]
We have shown that $f$ on $U^n:=\{z\in U\,|\, |y(z)p^{-n}|\geq 1\}$ has the form $(\frac{t}{t_{N}})^{n_1}(\frac{e_1}{e_{2,N}})^{n_2-n_1}c'_i$ for some $c'_i\in H^0(V'''_{G_{r(n)}},\cO_{\mathcal{X}_{K^pG_{r(n)}}})$, where $V'''_{G_{r(n)}}\subset V_{G_{r(n)}}$ defined by $|y_np^{-n}|\geq 1$. It is clear that $c_i,c'_i$ can glue to a section $c\in H^0(V_{G_{r(n)}},\cO_{\mathcal{X}_{K^pG_{r(n)}}})$ and $f=(\frac{t}{t_{N}})^{n_1}(\frac{e_1}{e_{2,N}})^{n_2-n_1}c$. This proves the second and fourth part of the proposition.

The third part follows directly from the first two parts. 
\end{proof}

\begin{rem}
One can reinterpret these computations from the point of view of spectral theory. Consider the stalk of $\cO^{\la,\chi}_{K^p}$ at $\infty$:
\[A^{\chi}:=\varinjlim_{U\ni\infty} \cO^{\la,\chi}_{K^p}(U),\]
where the direct limit runs through all open subsets $U$ containing $\infty$. We equip $A^\chi$ with the direct limit topology. Note that if fix $U\in\mathfrak{B}$ containing $\infty$ and $(\frac{t}{t_n})^{n_1}(\frac{e_2}{e_{2,n}})^{-\chi(h)}\in \cO^{\la,\chi}_{K^p}(U)$, then multiplication by $(\frac{t}{t_n})^{-n_1}(\frac{e_2}{e_{2,n}})^{\chi(h)}$ induces an isomorphism $A^\chi\xrightarrow{\sim}A^0$. Here $0$ denotes the weight $0$. Clearly, this is not $\mathfrak{n}$-equivariant and a simple computation shows that the action of $u^+$ on $A^\chi$ becomes 
\[u^+-\chi(h)y \mbox{ on }A^0.\]
Now Proposition \ref{kncoin} and  Proposition \ref{kninv} can be rephrased as follows.
\begin{enumerate}
\item  $u^+/y$ is a well-defined operator on $A^0$.
\item  (outside of spectrum) $u^+/y-\chi(h)$ is invertible if $\chi(h)\neq 0,-1,-2,\cdots$ and is $p$-adically non-Liouville. 
\item  (continuous spectrum) $u^+/y-\chi(h)$ is injective and has dense image if $\chi(h)$ is $p$-adically Liouville.
\item  (point spectrum) $u^+/y-\chi(h)$ has kernel the stalk of $\cO^{\la,\chi,\kn}_{K^p}$ at $\infty$ if $-\chi(h)\in\N$.
\end{enumerate}
\end{rem}

Summarizing the results we have obtained so far.
\begin{thm} \label{wdec}
Let $\chi=(n_1,n_2)$ be a weight. Then
\begin{enumerate}
\item $(H^1(\Fl,\cO^{\la}_{K^p})^{\chi})_\kn=H^1(\Fl,\cO^{\la,\chi}_{K^p})_\kn=0$ if $\chi(h)\neq 0$.
\item Suppose $\chi(h)\neq0 ,1$. There is a natural Hecke-equivariant weight decomposition of $C[B]$-modules
\[H^1(\Fl,\cO^{\la}_{K^p})^{\chi,\kn}=H^1(\Fl,\cO^{\la,\chi}_{K^p})^{\kn}=M^\dagger_{\chi}(K^p)\otimes_{C}\kn^*\oplus W_2,\]
where  $W_2$ has weight $(n_2,n_1)$ and $M^\dagger_{\chi}(K^p)\otimes_C \kn^*$ has weight $(n_1-1,n_2+1)$.
\end{enumerate}
\end{thm}

\begin{proof}
Assume $\chi(h)\neq 0$. Then by our discussion in the beginning of this subsection \eqref{ss2}, $H^1(\Fl,\cO_{K^p}^{\la,\chi})_{\kn}\cong H^1(\Fl,(\cO_{K^p}^{\la,\chi})_\kn)$, which is zero as pointed out in Proposition \ref{kncoin}.

Suppose $\chi(h)\neq 1$, then $-\chi(h)\neq w\cdot (-\chi)(h)$. It follows from Corollary \ref{khact} that $(h+\chi(h))(h-w\cdot (-\chi)(h))=0$ on $H^1(\Fl,\cO^{\la,\chi}_{K^p})^{\kn}$. Therefore we have a natural weight decomposition $H^1(\Fl,\cO^{\la,\chi}_{K^p})^{\kn}=W_1\oplus W_2$ such that $\kh$ acts on $W_1$ (resp. $W_2$) via$(n_1-1,n_2+1)$ (resp. $(n_2,n_1)$). Since $H^1(\Fl,\cO^{\la,\chi,\kn}_{K^p})$ has no weight-$(n_1-1,n_2+1)$ part by Proposition \ref{kninv} and the  weight-$(n_1-1,n_2+1)$ part of $H^0(\Fl,(\cO^{\la,\chi}_{K^p})_{\kn})\otimes_C \kn^*$ is $M^\dagger_\chi(K^p)\otimes_C \kn^*$ by Proposition \ref{kncoin}, our claim for $W_1$ now follows from \eqref{ss1}.
\end{proof}

The case $\chi(h)=0$ or $1$ will be treated in the next subsection.

\subsection{\texorpdfstring{$\mathfrak{n}$}{Lg}-cohomology (II)}
\begin{para}
In this subsection, we completely determine the $\kn$-cohomology of $H^1(\Fl,\cO^{\la,\chi}_{K^p})$ for  integral weight, i.e. $\chi(h)\in\Z$. Write $k=\chi(h)$. We will distinguish $4$ cases: 
\begin{enumerate}
\item $k\geq 2$;
\item $k=1$;
\item $k=0$;
\item $k\leq -1$.
\end{enumerate}
One key step is to understand $H^1(\Fl,\cO^{\la,\chi,\kn}_{K^p})$. To do this, we need to introduce some auxiliary sheaves. Recall that $\omega_{K^p}$ is defined as the pull-back of $\omega$ (as a coherent sheaf)  from some finite level to the infinite level. For an integer $k$, we write $\omega_{K^p}^{k}:=\omega_{K^p}^{\otimes k}$.
\end{para}

\begin{defn}
We define $\omega^{k,\sm}_{K^p}\subset {\pi_{\HT}}_{*} \omega_{K^p}^{k}$ as the subsheaf of $\GL_2(\Q_p)$-smooth sections. More precisely, for any quasi-compact open subset $U\subset \Fl$, we can find an open subgroup $K_p$ of $\GL_2(\Q_p)$ stabilizing $U$. Then $\omega^{k,\sm}_{K^p}(U)\subset \omega_{K^p}^{k}(\pi_{\HT}^{-1}(U))$ is the subspace of $K_p$-smooth vectors, i.e. vectors fixed by some open subgroup of $K_p$. It is easy to see that this definition is independent of the choice of $K_p$.
\end{defn}

\begin{rem}
For $U\in\mathfrak{B}$, using the notation in Theorem \ref{piHT}, we have
\[\omega^{k,\sm}_{K^p}(U)=\varinjlim_{K_p}\omega^{k}(V_{K_p}).\]
Hence compared with Proposition \ref{kninv}, it is clear that 
\[\omega^{0,\sm}_{K^p}=\cO^{\la,0,\kn}_{K^p}.\]
From the point of view of classical Riemann-Hilbert correspondence, it seems better to think of $\omega^{k,\sm}_{K^p}$ as ``local system on the analytic site of $\Fl$''.
\end{rem}

\begin{rem}
An equivalent definition is $\omega^{k,\sm}_{K^p}={\pi_{\HT}}_{*} (\varinjlim_{K_p\subset \GL_2(\Q_p)}(\pi_{K_p})^{-1} \omega^{ k})$. Here $\pi_{K_p}:\mathcal{X}_{K^p}\to\mathcal{X}_{K^pK_p}$ denotes the natural projection and $\pi_{K_p}^{-1}$ denotes pull-back as sheaf of abelian groups. In fact, the cohomology of $\omega^{k,\sm}_{K^p}$ is closely related to the cohomology of $\omega^{k}$.
\end{rem}

\begin{lem} \label{okcc}
There is a natural isomorphism 
\[H^i(\Fl,\omega^{k,\sm}_{K^p})\cong \varinjlim_{K_p\subset\GL_2(\Q_p)}H^i(\mathcal{X}_{K^pK_p},\omega^k).\]
\end{lem}

\begin{proof}
The rough idea is that  the Hodge-Tate period map $\pi_\HT$ behaves as a finite map. Let $\mathfrak{U}=\{U_1,\cdots,U_r\}$ be a subset of $\mathfrak{B}$ (not necessarily a cover of $\Fl$), then for sufficiently small $K_p$ and any $i=1,\cdots,r$, we know that $\pi_{\HT}^{-1}(U_i)$ is the preimage of an affinoid open subset $V_{i,K_p}\subset \mathcal{X}_{K^pK_p}$ . Denote by $\mathfrak{V}_{K_p}=\{V_{1,K_p},\cdots,V_{r,K_p}\}$. Let $C^{\bullet}(\mathfrak{U},\omega^{k,\sm}_{K^p})$ be the \v{C}ech complex of $\omega^{k,\sm}_{K^p}$ with respect to $\mathfrak{U}$ and define $C^{\bullet}(\mathfrak{V}_{K_p},\omega^k)$ in a similar way. Clearly, 
\[C^{\bullet}(\mathfrak{U},\omega^{k,\sm}_{K^p})=\varinjlim_{K_p} C^{\bullet}(\mathfrak{V}_{K_p},\omega^k).\]
In particular, if $\mathfrak{U}$ is an open cover of some $U\in\mathfrak{B}$, we conclude from the usual Tate acyclicity result that $H^i(C^{\bullet}(\mathfrak{U},\omega^{k,\sm}_{K^p}))=0,i\geq 1$. Hence by Corollaire 4., p.176 of \cite{Gro57}, we have $H^i(U,\omega^{k,\sm}_{K^p})=0,i\geq 1$ and $H^i(\Fl,\omega^{k,\sm}_{K^p})$ can be computed by \v{C}ech complex. Now taking $\mathfrak{U}$ as a cover of $\Fl$, we get our claim.
\end{proof}

\begin{cor} \label{okv}
 $H^1(\Fl,\omega^{k,\sm}_{K^p})=0$ if $k\geq 2$.
\end{cor}

\begin{proof}
By our previous lemma, it suffices to prove $H^1(\mathcal{X}_{K^pK_p},\omega^k)=0$ when $k\geq 2$. This follows from the positivity of $\omega$ and the Kodaira-Spencer isomorphism $\omega^2\cong \Omega^1_{\mathcal{X}_{K^pK_p}}(\mathcal{C})$, where $\mathcal{C}$ denotes the cusps in $\mathcal{X}_{K^pK_p}$ and is non-empty.
\end{proof}

The following lemma implies that the stalk of $\omega^{k,\sm}_{K^p}$ at $\infty$ is the space of overconvergent forms of weight $k$ introduced in \ref{omfk}.

\begin{lem}
$\varinjlim_{U\ni\infty}\omega^{k,\sm}_{K^p}(U)=M^\dagger_k(K^p)$.
\end{lem}

\begin{proof}
This follows from Lemma \ref{twotopcomp} directly.
\end{proof}

\begin{defn}
We denote by 
\[M_k(K^p):=H^0(\Fl,\omega^{k,\sm}_{K^p})= \varinjlim_{K_p\subset\GL_2(\Q_p)}H^0(\mathcal{X}_{K^pK_p},\omega^k),\] 
the space of classical weight $k$ modular forms of tame level $K^p$.
\end{defn}

\begin{para}[non-integral powers of $t$] \label{twisttn1}
Recall that in \ref{ttn}, we defined $t$ as the coordinate function on $\mathrm{Isom}(\Z_p,\Z_p(1))\cong \Z_p^\times$ by choosing a basis of $\Z_p(1)$. For any $n_1\in C$, we can find a continuous character $\Z_p^\times\to C^\times$ whose derivative sends $1\in\Q_p=\Lie(\Z_p^\times)$ to $n_1\in C=\Lie(C^\times)$. When $n_1$ is an integer, we can simply take $t^{n_1}$. For non-integer $n_1$, we \textbf{fix} one choice from now on and denote it by $t^{n_1}$ and  view it as an invertible element of $H^0(\mathcal{X}_{K^p},\cO_{\mathcal{X}_{K^p}})^{\la}$ by abuse of notation. It follows from the discussion in \ref{ttn} that $\GL_2(\A_f)$ acts on $t^{n_1}$ via $t^{n_1}\circ\varepsilon\circ\det$ and $G_{\Q_p}$ acts on it via $t^{n_1}\circ\varepsilon_p$. Hence we can allow non-integral power of $t$ in \ref{e1e2tact}.

Let $\chi=(n_1,n_2)$. Then multiplication by $t^{-n_1}$ induces an isomorphism $\cO^{\la,\chi}_{K^p}\xrightarrow{\sim} \cO^{\la,\chi'}_{K^p}$, where $\chi'=(0,n_2-n_1)$. It is easy to see how the actions of $\GL_2(\Q_p)\times G_{\Q_p}$ and Hecke operators away from $p$ change accordingly. From this, we can reduce computations to the case $n_1=0$. 
\end{para}
 
\begin{para}[\texorpdfstring{$k\geq 2$}{Lg}] \label{kgeq2}
Now we are ready to treat the case $n_1-n_2=k\geq 2$. By Proposition \ref{kncoin}, we have $(\cO_{K^p}^{\la,\chi})_\kn\cong (i_{\infty})_*M^\dagger_\chi(K^p)$ in this case and it suffices to determine $H^1(\Fl,\cO^{\la,\chi,\kn}_{K^p})$. Note that multiplication by ${t}^{-n_1}e_1^{k}$ induces an injection:
\[\cO^{\la,\chi,\kn}_{K^p}\to\omega^{k,\sm}_{K^p}.\]
More precisely, since the stalk of $\cO^{\la,\chi,\kn}_{K^p}$ at $\infty$ is zero, it is enough to define this map outside of $\infty$. So let $U\in\mathfrak{B}$ be an open subset not containing $\infty$. It follows from the first part of Proposition \ref{kninv} that this map is nothing but sending $f={t}^{n_1}(\frac{e_1}{e_{1,N}})^{n_2-n_1}c\in \cO^{\la,\chi,\kn}_{K^p}(U)$ to $e_{1,N}^{k}c\in\omega^{k,\sm}_{K^p}(U)$. This is in fact an isomorphism $\cO^{\la,\chi,\kn}_{K^p}(U)\xrightarrow{\sim}\omega^{k,\sm}_{K^p}(U)$. Hence the quotient of $\cO^{\la,\chi,\kn}_{K^p}\to\omega^{k,\sm}_{K^p}$ is simply the stalk of $\omega^{k,\sm}_{K^p}$ at $\infty$, i.e.
\[0\to \cO^{\la,\chi,\kn}_{K^p}\xrightarrow{\times{t}^{-n_1}e_1^{k}}\omega^{k,\sm}_{K^p}\to (i_\infty)_*M^\dagger_k(K^p)\to 0.\]
Take the cohomology of this exact sequence. It follows from Corollary \ref{okv} and Lemma \ref{okcc} that
\[0\to M_k(K^p)\to M^\dagger_k(K^p)\to H^1(\Fl,\cO^{\la,\chi,\kn}_{K^p})\to 0. \]
Hence $H^1(\Fl,\cO^{\la,\chi,\kn}_{K^p})\cong M^\dagger_k(K^p)/ M_k(K^p)$, i.e. the  quotient of overconvergent modular forms of weight $k$ by classical forms. 
\end{para}

\begin{thm} \label{thmkgeq2}
Suppose $\chi(h)=k\in\{2,3,\cdots\}$. There is a Hecke and $B\times G_{\Q_p}$-equivariant weight decomposition (into weight $(n_1-1,n_2+1)$ and $(n_2,n_1)$ components):
\[H^1(\Fl,\cO^{\la}_{K^p})^{\chi,\kn}=H^1(\Fl,\cO^{\la,\chi}_{K^p})^{\kn}=M^\dagger_k(K^p)\cdot e_1^{-1}e_2^{-k+1}t^{n_1} \oplus  (M^\dagger_k(K^p)/ M_k(K^p))\cdot e_1^{-k}t^{n_1}.\]
\end{thm}

\begin{proof}
We can multiply $t^{-n_1}$ to reduce to the case $n_1=0$. Assume $n_1=0$. Note that $\kn^*=e_1^{-1}e_2$ as $B$-representations. Hence by Proposition \ref{defnocmf}, the weight $(-1,-k+1)$-part is isomorphic to $M^\dagger_k(K^p)\cdot e_1^{-1}e_2^{-k+1}$. Since the isomorphism $H^1(\Fl,\cO^{\la,\chi,\kn}_{K^p})\cong M^\dagger_k(K^p)/ M_k(K^p)$ is essentially induced by $\times e_1^k$, hence we can twist by $e_1^{-k}$ to make it $B\times G_{\Q_p}$-equivariant.
\end{proof}

\begin{para}[\texorpdfstring{$k\leq -1$}{Lg}] 
Another relatively simple case is when $\chi(h)=k\leq -1$. In this case, Proposition \ref{kninv} implies that multiplication by $t^{-n_1}e_1^{k}$ induces an isomorphism:
\[\cO^{\la,\chi,\kn}_{K^p}\xrightarrow{\sim}\omega^{k,\sm}_{K^p}.\]
Therefore by Lemma \ref{okcc}, $H^1(\Fl,\cO^{\la,\chi,\kn}_{K^p})\cong \varinjlim_{K_p} H^1(\mathcal{X}_{K^pK_p},\omega^{k})$. Combining this with  Proposition \ref{kncoin}, we get
\end{para}

\begin{thm} \label{thmkleq-1}
Suppose $\chi(h)=k\in\{-1,-2,\cdots\}$. There is a weight decomposition (into weight $(n_1-1,n_2+1)$ and $(n_2,n_1)$ components):
\[H^1(\Fl,\cO^{\la,\chi}_{K^p})^{\kn}=M^\dagger_k(K^p)\cdot e_1^{-1}e_2^{-k+1}t^{n_1} \oplus N_{k,1}\cdot e_1^{-k}t^{n_1},\]
where  $N_{k,1}$ sits inside the exact sequence 
\[0\to\varinjlim_{K_p} H^1(\mathcal{X}_{K^pK_p},\omega^{k})\to N_{k,1}\to M^\dagger_{k}(K^p)\to 0.\]
All the maps are Hecke and $B\times G_{\Q_p}$-equivariant.
\end{thm}

\begin{para}[\texorpdfstring{$k=0$}{Lg}] 
Now we consider the case $\chi(h)=0$. As before, we can multiply by $t^{-n_1}$ and assume $n_1=0$, i.e. $\chi=0$. First we determine $H^1(\Fl,\cO^{\la,0}_{K^p})^{\kn}$. Let $\mathcal{F}$ be $\kn(\cO^{\la,0}_{K^p})\subset \cO^{\la,0}_{K^p}$.  Hence the composite 
\[H^1(\Fl,\cO^{\la,0}_{K^p})\xrightarrow{u^+}H^1(\Fl,\mathcal{F})\otimes\kn^*\to H^1(\Fl,\cO^{\la,0}_{K^p})\otimes\kn^*\] 
is the endomorphism $u^+$ on $H^1(\Fl,\cO^{\la,0}_{K^p})$ and it suffices to determine the kernels of both maps. Consider 
\[0\to \cO^{\la,0,\kn}_{K^p}\to \cO^{\la,0}_{K^p}\xrightarrow{u^+}\mathcal{F}\otimes \kn^*\to 0.\]
By Proposition \ref{kncoin}, we know that $\mathcal{F}\subset\km_\infty\cO^{\la,0}_{K^p}$. Hence $H^0(\Fl,\mathcal{F})=0$ because the global sections of $\cO^{\la,0}_{K^p}$ are $\varinjlim_{K_p\subset\GL_2(\Q_p)}H^0(\mathcal{X}_{K^pK_p},\cO_{\mathcal{X}_{K^pK_p}})$. Therefore
\[H^1(\Fl,\cO^{\la,0,\kn}_{K^p})=\ker(H^1(\Fl,\cO^{\la,0}_{K^p})\xrightarrow{u^+}H^1(\Fl,\mathcal{F})\otimes \kn^*)\]
and $H^1(\Fl,\cO^{\la,0}_{K^p})\xrightarrow{u^+}H^1(\Fl,\mathcal{F})\otimes \kn^*$ is surjective because $\Fl$ is one-dimensional.
Note that $\cO^{\la,0,\kn}_{K^p}=\omega^{0,\sm}_{K^p}$. Hence its $H^1$ is $\varinjlim_{K_p}H^1(\mathcal{X}_{K^pK_p},\cO_{\mathcal{X}_{K^pK_p}})$. On the other hand, consider 
\[0\to\mathcal{F}\to \cO^{\la,0}_{K^p}\to(\cO^{\la,0}_{K^p})_\kn\to 0.\]
By Proposition \ref{kncoin}, $H^0(\Fl,(\cO^{\la,0}_{K^p})_\kn)\cong M^{\dagger}_0(K^p)\oplus  M^{\dagger}_0(K^p)\otimes_C\kn$.  It is easy to see that $H^0(\Fl,\cO^{\la,0}_{K^p})=M_0(K^p)$ maps to the first factor. Hence 
\[\ker(H^1(\Fl,\mathcal{F})\to H^1(\Fl,\cO^{\la,0}_{K^p}))\cong M^{\dagger}_0(K^p)/M_0(K^p)\oplus  M^{\dagger}_0(K^p)\otimes_C\kn.\]
Also since $(\cO^{\la,0}_{K^p})_\kn$ has no $H^1$, we have $H^1(\Fl,\mathcal{F})\to H^1(\Fl,\cO^{\la,0}_{K^p})$ is surjective. Summarizing our discussions, we have the following result.
\end{para}

\begin{prop}
Suppose $\chi=0$. Then
\begin{enumerate}
\item $H^1(\Fl,\cO^{\la,0}_{K^p})_\kn=0$;
\item There is a Hecke and $B\times G_{\Q_p}$-equivariant exact sequence:
\[0\to \varinjlim_{K_p}H^1(\mathcal{X}_{K^pK_p},\cO_{\mathcal{X}_{K^pK_p}})\to H^1(\Fl,\cO^{\la,0}_{K^p})^{\kn}\to \kn^*\otimes M^{\dagger}_0(K^p)/M_0(K^p)\oplus   M^{\dagger}_0(K^p)\to 0.\]
\end{enumerate}
\end{prop}

By Corollary \ref{chicomp}, there is an exact sequence
\begin{eqnarray} \label{chi0}
0\to M_0(K^p)\to H^1(\Fl,\cO^{\la,0}_{K^p}) \to H^1(\Fl,\cO^{\la}_{K^p})^{0}\to 0.
\end{eqnarray}
Note that $M_0(K^p)$ has trivial $\kn$-action. It is crucial to know its image in $H^1(\Fl,\cO^{\la,0}_{K^p})^{\kn}$. A direct computation shows that the composite of
\[M_0(K^p)\to H^1(\Fl,\cO^{\la,0}_{K^p})^{\kn}\to \kn^*\otimes M^{\dagger}_0(K^p)/M_0(K^p)\oplus   M^{\dagger}_0(K^p)\]
is the natural inclusion  $M_0(K^p)\to M^{\dagger}_0(K^p)$ (up to a sign). Here the second map comes from the previous proposition.
From this, taking the $\kn$-cohomology of \eqref{chi0}, we obtain the following description in the case $\chi(h)=0$.

\begin{thm} \label{thmk=0}
Suppose $\chi=(n_1,n_1)$, i.e. $k=0$. Then
\begin{enumerate}
\item $(H^1(\Fl,\cO^{\la}_{K^p})^\chi)_\kn=0$;
\item There is a weight decomposition (into weight $(n_1-1,n_1+1)$ and $(n_1,n_1)$ components):
\[H^1(\Fl,\cO^{\la}_{K^p})^{\chi,\kn}=N_{2,w}\cdot e_1^{-1}e_2t^{n_1}\oplus N_{0,1}\cdot t^{n_1},\]
where $N_{2,w}$ and $N_{0,1}$ sit inside exact sequences
\[0\to  M^{\dagger}_0(K^p)/M_0(K^p)  \to N_{2,w} \to M_0(K^p)\to 0,\]
\[0\to \varinjlim_{K_p}H^1(\mathcal{X}_{K^pK_p},\cO_{\mathcal{X}_{K^pK_p}})\to N_{0,1}\to  M^{\dagger}_0(K^p)/M_0(K^p)\to 0.\]
All the maps are Hecke and $B\times G_{\Q_p}$-equivariant. 
\end{enumerate}
\end{thm}

\begin{para}[\texorpdfstring{$k=1$}{Lg}]  \label{weightk1}
Finally, we treat the case of singular weight, i.e. $\chi(h)=1$. Again we may assume $n_1=0$ by a twist. Note that by Proposition \ref{kncoin}, we have  $(\cO_{K^p}^{\la,\chi})_\kn=(i_{\infty})_*M^\dagger_\chi(K^p)$.  Similar to the case $k\geq 2$ in \ref{kgeq2}, there is an exact sequence 
\[0\to \cO^{\la,\chi,\kn}_{K^p}\xrightarrow{\times e_1}\omega^{1,\sm}_{K^p}\to (i_\infty)_*M^\dagger_1(K^p)\to 0\]
induced by multiplication by $e_1$. The difference here is that $\omega^{1,\sm}_{K^p}$ have both non-trivial $H^0$ and $H^1$. More precisely, taking the cohomology of this exact sequence, we get
\[0\to M_1(K^p)\to M^{\dagger}_1(K^p)\to H^1(\Fl,\cO^{\la,\chi,\kn}_{K^p})\to \varinjlim_{K_p}H^1(\mathcal{X}_{K^pK_p}, \omega^{1})\to 0.\]
Then by our discussion in \ref{ss1}, we have the following result.
\end{para}

\begin{thm} \label{thetaweight1}
Suppose $\chi=(n_1,n_1-1)$, i.e. $\chi(h)=1$. All the maps below are Hecke and $B\times G_{\Q_p}$-equivariant.
\begin{enumerate}
\item There is a short exact sequence:
\begin{eqnarray}\label{weight11}
0\to   N_1\cdot e_1^{-1}t^{n_1}\to H^1(\Fl,\cO^{\la}_{K^p})^{\chi,\kn}\to M^\dagger_1(K^p)\cdot e_1^{-1}t^{n_1}\to 0,
\end{eqnarray}
where $N_1$ sits inside the exact sequence:
\[0\to  M^{\dagger}_1(K^p)/M_1(K^p)\to N_1\to \varinjlim_{K_p}H^1(\mathcal{X}_{K^pK_p}, \omega^{1})\to 0.\]
\item Under \eqref{weight11}, the weight-$(n_1-1,n_1)$ part $H^1(\Fl,\cO^{\la}_{K^p})^{\chi,\kn}_{(n_1-1,n_1)}$ of $H^1(\Fl,\cO^{\la}_{K^p})^{\chi,\kn}$ is the pull-back of $M_1(K^p)\cdot e_1^{-1}t^{n_1}\subset M^\dagger_1(K^p)\cdot e_1^{-1}t^{n_1}$, i.e. it sits inside the exact sequence
\[0\to   N_1\cdot e_1^{-1}t^{n_1}\to H^1(\Fl,\cO^{\la}_{K^p})^{\chi,\kn}_{(n_1-1,n_1)}\to M_1(K^p)\cdot e_1^{-1}t^{n_1}\to 0.\]
\end{enumerate}
\end{thm}

\begin{proof}
Only the second part requires extra explanation. To see this, we tensor \eqref{weight11} with $e_1t^{-n_1}$ and take $\kh_0$-cohomology. Recall that $\kh_0\subset\kh$ is the subalgebra of elements with trace zero. Again we may assume $n_1=0$. We need to understand the kernel of the  connecting homomorphism
\[\delta_1:M^\dagger_1(K^p)\to (N_1)_{\kh_0}\cong N_1.\]
Let $f\in M^\dagger_1(K^p)$. Since $H^1(\Fl,\cO^{\la}_{K^p})$ can be computed by \v{C}ech complex,  we can take an cover $\{U'_0,U'_1\}\subset\mathfrak{B}$ of $\Fl$ such that only $U'_1$ contains  $\infty$ and the pull-back of $f$ to $\mathcal{X}_{K^p}$ is defined on $\pi_{\HT}^{-1}(U'_1)$. On $U'_{01}:=U'_0\cap U'_1$, we can find $g\in \cO^{\la,\chi}_{K^p}(U'_{01})$ such that $u^+\cdot g=fe_2^{-1}$. Then $-g$ can be viewed as a $1$-cocycle of  the \v{C}ech complex of $\cO^{\la,\chi}_{K^p}$ with respect to the cover $\{U'_0,U'_1\}$, and it maps to $f\otimes 1\in M^\dagger_1(K^p)\otimes e_1^{-1}=M^\dagger_1(K^p)\cdot e_1^{-1}$ in \eqref{weight11}. (Here we view $e_1=C$ as a character of $B$.) Using the notation in Lemma \ref{expchi}, we can take $g$ of the form
\[g=\frac{f}{e_1}\log(\frac{x}{x_n}).\]
If we view $g\otimes 1$ as an element of $\cO^{\la,\chi}_{K^p}(U'_{01})\otimes e_1$, then
\[h\cdot (g\otimes 1)= h\cdot g\otimes 1+g\otimes 1=(-g\otimes 1-\frac{2f}{e_1}\otimes 1)+g\otimes 1=-\frac{2f}{e_1}\otimes 1.\]
In view of discussion in \ref{weightk1}, this means $\delta_1(f)=-2f\in M^{\dagger}_1(K^p)/M_1(K^p)\subset N_1$. Hence $\delta_1(f)=0$ if and only if $f\in M_1(K^p)$. From this, we easily deduce our claim in the Theorem.
\end{proof}
 
\subsection{\texorpdfstring{$\mu$}{Lg}-isotypic part}
\begin{para}
Let $\mu\in\kh^*$ be a weight, viewed as a character of $\mathfrak{b}$. The goal of this subsection is to determine the $\mu$-isotypic part of $\tilde{H}^1(K^p,C)^{\la}\cong H^1(\Fl,\cO^{\la}_{K^p})$ which we denote by
\[\tilde{H}^1(K^p,C)^{\la}_\mu\cong H^1(\Fl,\cO^{\la}_{K^p})_{\mu}. \]
Write $\mu=(k_1,k_2)$. Then by Corollary \ref{khact} and Harish-Chandra's theory\footnote{Note $w\cdot\mu=(k_2-1,k_1+1)$. Hence we are looking for $(a,b)\in C$ such that $a+b=k_1+k_2$ and $-(a-b)=k_1-k_2$ or $(k_2-1)-(k_1+1)$.}, we have
\[H^1(\Fl,\cO^{\la}_{K^p})_\mu=H^1(\Fl,\cO^{\la}_{K^p})^{(k_2,k_1)}_\mu\oplus H^1(\Fl,\cO^{\la}_{K^p})^{(k_1+1,k_2-1)}_\mu\]
if $(k_2,k_1)\neq (k_1+1,k_2-1)$, i.e. $\mu(h)\neq -1$. (This is opposite to the singular weight in the previous section as there is a sign in Corollary \ref{khact}.) Since the right hand side is computed in Theorem \ref{wdec}, \ref{thmkgeq2}, \ref{thmkleq-1}, \ref{thmk=0} (which is complete in the integral weight case), we obtain the following theorem by writing 
\[M_{\mu,1}:=H^1(\Fl,\cO^{\la}_{K^p})^{(k_2,k_1)}_\mu, M_{\mu,w}:=H^1(\Fl,\cO^{\la}_{K^p})^{(k_1+1,k_2-1)}_\mu.\] 
Note that by Theorem \ref{Senkh}, the horizontal action $\theta_\kh$ of $\kh$ essentially agrees with the Sen operator. Hence $M_{\mu,1}$ has pure Hodge-Tate-Sen weight $-k_1$ and $M_{\mu,w}$ has pure Hodge-Tate-Sen weight $1-k_2$ in the sense of \ref{Senop}. Our convention is that cyclotomic character has Hodge-Tate weight $-1$. 
\end{para}

\begin{thm} \label{hwvg}
Suppose $\mu=(k_1,k_2)\in\kh^*$ and $k:=-\mu(h)\neq 1$. There is a natural (Hodge)-decomposition
\[\tilde{H}^1(K^p,C)^{\la}_\mu=M_{\mu,1}\oplus M_{\mu,w},\]
into Hodge-Tate-Sen weight $-k_1$ and $1-k_2$ components. We have  $M_{\mu,w}= M^{\dagger}_{\mu+2\rho}(K^p)\otimes_C \kn^*$ unless $k=2$. Moreover, if $k\in\Z$ and 
\[M_{\mu,1}=N_{k,1}\cdot e_1^{-k}t^{k_2},\]
\[M_{\mu,w}=N_{k,w}\cdot e_1^{-1}e_2^{k-1}t^{k_1+1},\]
we have the following description of $N_{k,1},N_{k,w}$. All the isomorphisms and maps below are Hecke and $B\times G_{\Q_p}$-equivariant.
\begin{enumerate}
\item $N_{k,w}\cong M^{\dagger}_{2-k}(K^p)$ unless $k\neq 2$. When $k=2$, we have 
\[0\to  M^{\dagger}_0(K^p)/M_0(K^p)  \to  N_{2,w} \to M_0(K^p)\to 0.\]
\item If  $k\leq -1$, then there is an exact sequence
\[0\to\varinjlim_{K_p} H^1(\mathcal{X}_{K^pK_p},\omega^{k})\to N_{k,1}\to M^\dagger_{k}(K^p)\to 0.\]
\item If $k=0$, then  there is an exact sequence
\[0\to \varinjlim_{K_p}H^1(\mathcal{X}_{K^pK_p},\cO_{\mathcal{X}_{K^pK_p}})\to N_{0,1}\to  M^{\dagger}_0(K^p)/M_0(K^p)\to 0.\]
\item If $k\geq 2$, then $N_{k,1}\cong M^\dagger_k(K^p)/ M_k(K^p)$.
\end{enumerate}
\end{thm}

\begin{para}
Now assume $\mu(h)=-1$. Write $\mu=(k_1,k_1+1)$. In this case, it follows from Corollary \ref{khact} that
\[H^1(\Fl,\cO^{\la}_{K^p})_\mu=H^1(\Fl,\cO^{\la,(k_1+1,k_1)^2}_{K^p})_\mu\]
by considering the action of the centre of $U(\mathfrak{g})$. Here $\cO^{\la,(k_1+1,k_1)^2}_{K^p}\subset \cO^{\la}_{K^p}$ is the subsheaf of sections annihilated by $(\theta_\kh(h)-1)^2$ and $\theta_\kh(z)-\mu(z)=z-(2k_1+1)$. We compute the $\kn$-invariants of $H^1(\Fl,\cO^{\la,(k_1+1,k_1)^2}_{K^p})$ first and then determine the $\mu$-component. Note that $\theta_\kh(h)-1$ induces an exact sequence
\[0\to\cO^{\la,(k_1+1,k_1)}_{K^p}\to\cO^{\la,(k_1+1,k_1)^2}_{K^p}\xrightarrow{\theta_\kh(h)-1} \cO^{\la,(k_1+1,k_1)}_{K^p} \to 0,\]
where the third map is surjective by the first part of Lemma \ref{expchi}. For simplicity, we write $\mathcal{F}_1=\cO^{\la,(k_1+1,k_1)}_{K^p}$ and $\mathcal{F}_2=\cO^{\la,(k_1+1,k_1)^2}_{K^p}$. Then $\cF_2$ has no $H^0$ because $\mathcal{F}_1$ has no global section (Corollary \ref{chicomp}). Hence 
\[0\to H^1(\Fl,(\cF_2)^{\kn})\to H^1(\Fl,\cF_2)^{\kn}\to H^0(\Fl,(\cF_2)_\kn)\otimes_{C} \kn^*\to 0.\]
By Proposition \ref{kncoin} and \ref{kninv}, we know that the support of $(\cF_1)^{\kn}$ and $(\cF_1)_{\kn}$ do not intersect. Therefore $\theta_\kh(h)-1$ induces  exact sequences:
\[0\to (\cF_1)^{\kn}\to (\cF_2)^{\kn}\xrightarrow{\theta_\kh(h)-1}(\cF_1)^{\kn}\to 0;\]
\[0\to (\cF_1)_{\kn}\to (\cF_2)_{\kn}\xrightarrow{\theta_\kh(h)-1}(\cF_1)_{\kn}\to 0.\]
Taking $H^1$ of the first sequence, we obtain
\[0\to H^1(\Fl,(\cF_1)^{\kn})\to H^1(\Fl,(\cF_2)^{\kn})\xrightarrow{\theta_\kh(h)-1}H^1(\Fl,(\cF_1)^{\kn})\to 0\]
as $(\cF_1)^{\kn}\subset \cF_1$ has no $H^0$. Note that $h$ acts as $-1$ on $H^1(\Fl,(\cF_1)^{\kn})$ by Proposition \ref{kninv}. 
\end{para}

\begin{lem} \label{h=-thetah}
$\theta_{\kh}(h)-1=-(h+1)$ on $(\cF_2)^{\kn}$. In particular, the kernel of $h+1$ on $H^1(\Fl,(\cF_2)^{\kn})$ is $H^1(\Fl,(\cF_1)^{\kn})$ and $h+1$ induces an isomorphism 
\[h+1:H^1(\Fl,(\cF_2)^{\kn})/H^1(\Fl,(\cF_1)^{\kn})\xrightarrow{\sim}H^1(\Fl,(\cF_1)^{\kn}).\]
\end{lem}
\begin{proof} 
This is proved  by explicit computation. We only need to check $U\in\mathfrak{B}$ not containing $\infty$ since the stalk of $(\cF_1)^{\kn}$ hence also $(\cF_2)^{\kn}$ at $\infty$ is zero. Using the notation in \ref{Exph}, we have:
\[(\cF_2)^{\kn}(U)=(\cF_1)^{\kn}(U)+(\cF_1)^{\kn}(U)\cdot\log(\frac{e_1}{e_{1,N}})\]
for some $e_{1,N}$. Since $\theta_\kh(h)$ acts as $\begin{pmatrix} -1 & -2x \\ 0& 1 \end{pmatrix}$ on $U$, cf. \ref{expkh}, for $f\in(\cF_1)^{\kn}(U)$, 
\[(\theta_\kh(h)-1)\cdot (f\log(\frac{e_1}{e_{1,N}}))=\log(\frac{e_1}{e_{1,N}})(\theta_\kh(h)-1)\cdot f+f\theta_\kh(h)\cdot \log(\frac{e_1}{e_{1,N}})=-f.\]
On the other hand, by Proposition \ref{kninv}, we know that $(h+1)\cdot f=0$. Therefore 
\[(h+1)\cdot (f\log(\frac{e_1}{e_{1,N}}))=f h\cdot \log(\frac{e_1}{e_{1,N}})=f.\]
Hence $\theta_{\kh}(h)-1=-(h+1)$ on $(\cF_2)^{\kn}(U)$. 
\end{proof}

\begin{para}
Similarly, we consider 
\[0\to H^0(\Fl,(\cF_1)_{\kn})\to H^0(\Fl,(\cF_2)_{\kn})\xrightarrow{\theta_\kh(h)-1}H^0(\Fl,(\cF_1)_{\kn})\to 0.\]
Recall that $(\cF_1)_{\kn}\cong (i_{\infty})_*M^\dagger_{(k_1+1,k_1)}(K^p)$ in this case and we may identify 
\[(\cF_2)_{\kn}=(i_{\infty})_*(M^\dagger_{(k_1+1,k_1)}(K^p)\oplus M^\dagger_{(k_1+1,k_1)}(K^p)\cdot \log(\frac{e_2}{e_{2,N}}))\]
for some $e_{2,N}$. Then a similar computation shows that $\theta_\kh(h)-1=h-1$ on $(\cF_2)_{\kn}$. Hence the kernel of $h+1$ on $H^0(\Fl,(\cF_2)_\kn)\otimes_{C} \kn^*$ is $H^0(\Fl,(\cF_1)_\kn)\otimes_{C} \kn^*\subset H^0(\Fl,(\cF_2)_\kn)\otimes_{C} \kn^*$. 

Now to determine the kernel of $h+1$ on $H^1(\Fl,\cF_2)^{\kn}$, we claim that the surjection
\[H^1(\Fl,\cF_2)^{\kn}\to H^0(\Fl,(\cF_2)_\kn)\otimes_{C} \kn^*\]
remains surjective when passing to the kernel of $h+1$, i.e. there is an exact sequence 
\[0\to H^1(\Fl,(\cF_2)^{\kn})_\mu \to H^1(\Fl,\cF_2)^{\kn}_{\mu}\to (H^0(\Fl,(\cF_2)_\kn)\otimes_{C} \kn^*)_\mu\to 0. \]
In fact, this is really a formal consequence of the previous lemma. Consider the following commutative diagram:
\[\begin{tikzcd}
& 0  \arrow[d] & 0  \arrow[d] &0 \arrow[d]& \\
0 \arrow[r]  & H^1(\Fl,(\cF_1)^{\kn})  \arrow [d] \arrow[r] & H^1(\Fl,\cF_1)^{\kn} \arrow[d] \arrow[r] & H^0(\Fl,(\cF_1)_\kn)\otimes_{C} \kn^* \arrow[d] \arrow[r] & 0 \\
0 \arrow[r] & H^1(\Fl,(\cF_2)^{\kn})  \arrow [d,"\theta_\kh(h)-1=-(h+1)"] \arrow[r] & H^1(\Fl,\cF_2)^{\kn} \arrow[d,"\theta_\kh(h)-1"] \arrow[r] &  H^0(\Fl,(\cF_2)_\kn)\otimes_{C} \kn^*  \arrow[d,"\theta_\kh(h)-1=h+1"] \arrow[r] & 0\\
0 \arrow[r] & H^1(\Fl,(\cF_1)^{\kn}) \arrow[d] \arrow[r] & H^1(\Fl,\cF_1)^{\kn}\arrow[d]\arrow[r] & H^0(\Fl,(\cF_1)_\kn)\otimes_{C} \kn^*\arrow[d]\arrow[r] & 0 \\
& 0   & 0 &0& 
\end{tikzcd}.\]
Let $f\in H^0(\Fl,(\cF_1)_\kn)\otimes_{C} \kn^*\subset H^0(\Fl,(\cF_2)_\kn)\otimes_{C} \kn^*$. We can find a preimage $\tilde{f}\in H^1(\Fl,\cF_1)^{\kn}$. Then $(h+1)\cdot \tilde f\in H^1(\Fl,(\cF_1)^{\kn})$. By Lemma \ref{h=-thetah}, there exists $\tilde{g}\in H^1(\Fl,(\cF_2)^{\kn})$ such that $(h+1)\cdot \tilde{g}=(h+1)\cdot \tilde f$. Now $g:=\tilde{f}-\tilde{g}$ defines an element in $H^1(\Fl,\cF_2)^{\kn}$ lifting $f$ and annihilated by $h+1$. 

Since $H^1(\Fl,\cF_1)^{\kn}$ is computed in \ref{weightk1}, we obtain the following result. The last part is a restatement of Theorem \ref{thetaweight1}.
\end{para}

\begin{thm} \label{weight1}
Suppose $\mu=(k_1,k_1+1)$, i.e. $\mu(h)=-1$. There exists an exact sequence 
\[0\to N_{1}\cdot e_1^{-1}t^{k_1+1} \to \tilde{H}^1(K^p,C)^{\la}_\mu \to M^{\dagger}_{1}(K^p)\cdot e_1^{-1}t^{k_1+1}\to 0,\]
where $N_{1}$ sits inside the exact sequence:
\[0\to  M^{\dagger}_1(K^p)/M_1(K^p)\to N_{1}\to \varinjlim_{K_p}H^1(\mathcal{X}_{K^pK_p}, \omega^{1})\to 0.\]
Moreover, $\tilde{H}^1(K^p,C)^{\la,(k_1+1,k_1)}_\mu\subset  \tilde{H}^1(K^p,C)^{\la}_\mu$ is identified with the pull-back of $M_{1}(K^p)\cdot e_1^{-1}t^{k_1+1}\subset M^{\dagger}_{1}(K^p)\cdot e_1^{-1}t^{k_1+1}$, i.e. there is an exact sequence 
\[0\to N_{1}\cdot e_1^{-1}t^{k_1+1} \to \tilde{H}^1(K^p,C)^{\la,(k_1+1,k_1)}_\mu \to M_{1}(K^p)\cdot e_1^{-1}t^{k_1+1}\to 0.\]
All the maps are  Hecke and $B\times G_{\Q_p}$-equivariant.
\end{thm}

\section{Applications} \label{app}
In this section, we present several applications of the main results of previous section to the study of overconvergent modular forms and the  Fontaine-Mazur conjecture in the irregular case. One main ingredient is Emerton's local-global compatibility result which allows us to study completed cohomology using the ($p$-adic) representation theory of $\GL_2(\Q_p)$.

\subsection{Hecke algebra}
\begin{para} \label{HA}
First we recall the definition of  the (big) Hecke algebra associated to completed cohomology. Let $K^p=\prod_{l\neq p} K_l\subset\GL_2(\A_f^p)$ be a tame level and $S$  a finite set of rational primes containing $p$ and all places $l$ for which $K_l$ is not maximal. As usual, we denote the double coset action of
\[[K_l\begin{pmatrix}l&0\\0&1\end{pmatrix}K_l]\]
by $T_l$ and the double coset action of 
\[[K_l\begin{pmatrix}l&0\\0&l\end{pmatrix}K_l]\]
by $S_l$. For any open compact subgroup $K_p$ of $\GL_2(\Q_p)$, we define
\[\T(K^pK_p)\subset \End_{\Z_p} (H^1_{\et}(\mathcal{Y}_{K^pK_p},\Z_p))\]
as the $\Z_p$-subalgebra generated by Hecke operators $T_l,S_l^{\pm 1}$ at places $l\notin S$. As the notation suggests, this does not depend on the choice of $S$ because of the existence of continuous $2$-dimensional determinants of $G_{\Q,S}$ over $\T(K^pK_p)$ (see below) and Chebatorev's density theorem.  Now we define the Hecke algebra of tame level $K^p$ as 
\[\T(K^p):=\varprojlim_{K_p\subset\GL_2(\Q_p)}\T(K^pK_p).\]
This is a complete semi-local $\Z_p$-algebra and $\T(K^p)/\km$ is a finite field for any maximal ideal $\mathfrak{m}$. It acts faithfully on $\tilde{H}^1(K^p,\Z_p)$ and commutes with the action of $\GL_2(\Q_p)\times G_{\Q_p}$. Moreover there is a continuous $2$-dimensional determinant $D_S$  of $G_{\Q,S}$ valued in $\T(K^p)$ in the sense of Chenevier \cite{Che14} satisfying the following property:\footnote{The existence of $D_S$ implies that $\T(K^p)$ is noetherian. Note that $T_l,S_l^{\pm1},l\notin S$ generate $\T(K^p)$ topologically, hence $\T(K^p)$ receives a surjective map from a finite product of some universal deformation rings of $2$-dimensional determinants of $G_{\Q,S}$, which are noetherian by the work of Chenevier.} for any $l\notin S$, the characteristic polynomial of $D_S(\Frob_l)$ is
\[X^2-l^{-1}T_lX+l^{-1}S_l.\]
Note that its twist by inverse of the cyclotomic character is also the Eichler-Shimura congruence relation (\cite[Th\'eor\`eme 4.9]{De71}), i.e. 
\[\Frob_l^2-T_l\Frob_l+lS_l=0\]
on $\tilde{H}^1(K^p,\Z_p)$. This can be checked first on finite levels and then by passing to the limit over $K_p$. Note that $T_l,S_l^{\pm1},l\notin S$ generate $\T(K^p)$ topologically. 

Let $\lambda:\T(K^p)\to \overbar{\Q}_p$ be a $\Z_p$-algebra homomorphism. We can associate an odd semi-simple Galois representation (unique up to conjugation)
\[\rho_\lambda:G_{\Q,S}\to \GL_2(\overbar{\Q}_p)\]
whose determinant is $\lambda\circ D_S$. Here odd means $\det(\rho_\lambda(c))=-1$ for any complex conjugation $c\in G_{\Q,S}$. Let $\tilde{H}^1(K^p,\overbar{\Q}_p)=\tilde{H}^1(K^p,\Q_p)\otimes\overbar{\Q}_p$ and let $\kp_\lambda:=\ker(\lambda\otimes\overbar\Q_p)\in\Spec\T(K^p)\otimes_{\Q_p}\overbar\Q_p$. If $\rho_\lambda$ is irreducible, by  the Eichler-Shimura relation, we have 
\[\Hom_{G_{\Q}}(\rho_\lambda(-1),\tilde{H}^1(K^p,\overbar{\Q}_p))=\Hom_{G_{\Q}}(\rho_\lambda(-1),\tilde{H}^1(K^p,\overbar{\Q}_p)[\kp_\lambda]),\]
where $\rho_\lambda(-1)=\rho_\lambda\otimes \varepsilon^{-1}$ is the twist of $\rho_\lambda$ by  inverse of the $p$-adic cyclotomic character and  $\tilde{H}^1(K^p,\overbar{\Q}_p)[\kp_\lambda]$ denotes the $\lambda$-isotypic subspace. By the main result of \cite{BLR91}, $\tilde{H}^1(K^p,\overbar{\Q}_p)[\kp_\lambda]$ is $\rho_\lambda(-1)$-isotypic in the sense that 
\[\tilde{H}^1(K^p,\overbar{\Q}_p)[\kp_\lambda]=\rho_\lambda(-1)\otimes_{\overbar\Q_p} \Hom_{G_{\Q}}(\rho_\lambda(-1),\tilde{H}^1(K^p,\overbar{\Q}_p)).\]
We remark that the centre $\A_f^\times$ of $\GL_2(\A_f)$ acts via $\det(\rho_\lambda)\varepsilon^{-1}$ on $\tilde{H}^1(K^p,\overbar{\Q}_p)[\kp_\lambda]$ via global class field theory.

\end{para}

\begin{defn}
Let
\[\rho:G_\Q\to\GL_2(\overline\Q_p)\]
be a continuous two-dimensional $p$-adic Galois representation. We say $\rho$ is 
\begin{itemize}
\item pro-modular if there exists a tame level $K^p$ and $\lambda:\T(K^p)\to \overbar{\Q}_p$ such that $\rho\cong\rho_\lambda$;
\item pro-cohomological if $\Hom_{G_{\Q}}(\rho(-1),\tilde{H}^1(K^p,\overbar{\Q}_p))\neq 0$ for some tame level $K^p$. 
\end{itemize}
\end{defn}

Clearly $\rho$ is pro-modular if it is pro-cohomological and irreducible by our previous discussion. Conversely,  we have the following result.

\begin{lem}
Let $\rho=\rho_\lambda$ be an irreducible pro-modular representation for some $K^p$ and $\lambda$. The following statements are equivalent.
\begin{enumerate}
\item $\rho$ is pro-cohomological.
\item $\tilde{H}^1(K^p,\overbar{\Q}_p)[\kp_\lambda]\neq 0$.
\item $\tilde{H}^1(K^p,C)[\kp_\lambda]\neq 0$.
\end{enumerate}
\end{lem}

\begin{proof}
The first two are equivalent by the Eichler-Shimura relation. The equivalence between the last two is a consequence of the following lemma.  
\end{proof}

\begin{lem} \label{kpC}
Let $\rho=\rho_\lambda$ be a pro-modular representation for some $K^p$ and $\lambda$. Then
\[\tilde{H}^1(K^p,\overbar{\Q}_p)[\kp_\lambda]\widehat\otimes_{\overbar\Q_p}C\cong \tilde{H}^1(K^p,C)[\kp_\lambda],\]
where $\tilde{H}^1(K^p,\overbar{\Q}_p)[\kp_\lambda]$ is endowed with a norm with unit ball $(\tilde{H}^1(K^p,\Z_p)\otimes\overbar\Z_p)[\kp_\lambda]$.
\end{lem}

\begin{proof}
Choose generators $g_1,\cdots,g_s$ of $\kp_\lambda$ as a $\T(K^p)\otimes\overbar\Q_p$-module. We have an exact sequence
\[0\to\tilde{H}^1(K^p,\overbar{\Q}_p)[\kp_\lambda]\to \tilde{H}^1(K^p,\overbar\Q_p)\xrightarrow{(g_1,\cdots,g_s)} \bigoplus_{i=1}^s\tilde{H}^1(K^p,\overbar\Q_p)\]
which is continuous with respect to the $p$-adic topology, i.e. defined by the unit open ball $\tilde{H}^1(K^p,\Z_p)\otimes\overbar\Z_p$. Note that the second map is strict. This sequence remains exact after taking $p$-adic completion, which is exactly what we want.
\end{proof}

We conclude this subsection by the following result on infinitesimal characters.

\begin{prop} \label{icht}
Suppose $\rho=\rho_\lambda$ is pro-cohomological and irreducible and has Hodge-Tate-Sen weights $(a,b)$. Then as a representation of $\GL_2(\Q_p)$, the locally analytic vectors $\tilde{H}^1(K^p,C)^{\la}[\kp_\lambda]\neq 0$ and has infinitesimal character $\{(-a-1,-b),(-b-1,-a)\}\subset \kh^*$.
\end{prop}

\begin{proof}
Note that $\rho$ can be defined over a finite extension $E$ of $\Q_p$. Hence $\tilde{H}^1(K^p,\overbar{\Q}_p)[\kp_\lambda]\cap \tilde{H}^1(K^p,\Q_p)\otimes E\neq 0$ and is an admissible representation of $\GL_2(\Q_p)$. It has non-zero locally analytic vectors by the main result of \cite{ST03}. Since $\det\rho$ has Hodge-Tate-Sen weight $a+b$, the centre $Cz$ of $\mathfrak{g}$ acts on $\tilde{H}^1(K^p,C)^{\la}[\kp_\lambda]$ via $z\mapsto -(a+b)-1$. Now the claim for infinitesimal character follows directly from Corollary \ref{khact} and Theorem \ref{Senkh}.
\end{proof}

\begin{rem}
A higher dimensional generalization of this result was obtained  by Dospinescu-Pa\v{s}k\={u}nas-Schraen in  \cite[Theorem 1.4]{DPS20}. 
\end{rem}

\subsection{A classicality result for overconvergent weight one forms}

\begin{para}
Suppose $\rho=\rho_\lambda$ has equal Hodge-Tate-Sen weights $0,0$. (It's easy to reduce to this case by twisting by a character if weights are $a,a$.) There are two possibilities for $\rho$: either $\rho|_{G_{\Q_p}}$ is Hodge-Tate, or not. In other words, $(\rho\otimes_{\Q_p} C)^{G_{\Q_p}}$ is a $\overbar{\Q}_p$-vector space of dimension $2$ or $1$. Recall that by Deligne-Serre \cite{DS74}, the $2$-dimensional Galois representation associated to a classical weight $1$ eigenform is Hodge-Tate at $p$. The main result of this subsection gives a converse.

Recall that there is a natural action of $B$ on $M^{\dagger}_{1}(K^p)$. Let $N_0\subset B$ be $\begin{pmatrix} 1 & \Z_p\\ 0 &  1 \end{pmatrix}$. Then we have the usual operator 
\[U_p:=\sum_{i=0}^{p-1}\begin{pmatrix} p & i \\ 0 & 1 \end{pmatrix}\]
acting on $M^{\dagger}_{1}(K^p)^{N_0}$. Note that Theorem \ref{hwvg} implies that Hecke algebra $\T(K^p)$ also acts on the space of overconvergent modular forms. 
\end{para}

\begin{thm} \label{cl1}
Suppose $\rho=\rho_\lambda$ is pro-modular for some tame level $K^p$.  If $M^{\dagger}_{1}(K^p)^{N_0}[\kp_\lambda]$ has a non-zero $U_p$-eigenvector and $\rho|_{G_{\Q_p}}$ is Hodge-Tate of weights $0,0$, then $\rho$ comes from a classical weight $1$ eigenform, i.e. $M_1(K^p)[\kp_{\lambda}]\neq 0$.
\end{thm}

\begin{rem}
Theorem \ref{cl1} implies that an overconvergent weight one modular form of finite slope is classical if its associated Galois representation is Hodge-Tate. In particular, this gives a different proof of the main result of Buzzard-Taylor \cite{BT99} in the ordinary case.

Note that we don't assume the eigenvalue of $U_p$ is non-zero. In fact, using Colmez's Kirillov model, we will see that  the kernel of $U_p$ is always non-zero if we know the local-global compatibility at $p$ (in the sense of Emerton).
\end{rem}

\begin{para}
The rest of this subsection is devoted to the proof of Theorem \ref{cl1}. As mentioned in the introduction, we will need Hecke operators at bad places in the proof. After making a right translation by some element in $\GL_2(\A^p_f)$ and shrinking the level, we can find a finite set of rational primes $S$ containing $p$, an integer $m>0$ and assume the following: $K^p=\prod_{l\neq p} K_l\subset\prod_{l\neq p}\GL_2(\Z_l)$, where $K_l=\GL_2(\Z_l)$ for $l\notin S$, and 
\[K_l=\Gamma_1(l^m)=\{\begin{pmatrix} a & b \\ c & d \end{pmatrix},~a-1,c,d-1\in l^m\Z_l\},\,l\in S\setminus \{p\}.\]
Moreover, $M^{\dagger}_{1}(K^p\Gamma_1(p^m))[\kp_\lambda]$ has a non-zero eigenvector of $U_p$. See \ref{comf} for the notation here. Let $K_p=\Gamma_1(p^n)$ for some $n\geq m$ and $K=K^pK_p$. Let $A_S=\prod_{l\in S}\Z_l^\times\subset\prod_{l\in S}\GL_2(\Z_l)$ be the subgroup of the form $\begin{pmatrix} * & 0 \\ 0 & 1 \end{pmatrix}$. Then $A_S$ is in the normalizer of $\prod_{l\in S}K_l$, hence $A_S$ acts on $\mathcal{X}_{K}$ by right translation, which induces an action of $A_S$ on $M^{\dagger}_{1}(K)$. Clearly this action factors through a finite group and we can find a (finite-order) character $\psi:A_S\to \overbar\Q_p^\times$ such that the $\psi$-isotypic component 
\[M^{\dagger}_{1}(K)^{\psi}[\kp_\lambda]\]
contains a non-zero eigenvector of $U_p$. For $l\in S$, we denote by $U_l$ the abstract double coset action of 
\[[K_l\begin{pmatrix}l&0\\0&1\end{pmatrix}K_l].\]
These operators act naturally on $M^{\dagger}_{1}(K)$. If moreover $l\neq p$, then $U_l$ acts naturally on $\tilde{H}^1(K^p,C)$. We denote by 
\[\tilde\T(K^p)\subset\End(\tilde{H}^1(K^p,C))\] 
the $\T(K^p)$-subalgebra generated by $U_l,l\in S\setminus\{p\}$.
\end{para}

\begin{lem}
$\tilde\T(K^p)$ is a finite $\T(K^p)$-module, i.e. each $U_l$ is integral over $\T(K^p)$, $l\in S\setminus\{p\}$.
\end{lem}

\begin{proof}
Fix $l\in S\setminus\{p\}$. Since $\varinjlim_{K'_p}H^1_{\et}(\mathcal{X}_{K^pK'_p},\Q_p)$ is dense in $\tilde{H}^1(K^p,\Q_p)$ (Theorem 2.2.16 (iv) of \cite{Eme06}), it suffices to find a monic polynomial $P_l(X)\in \T(K^p)[X]$ such that $P_l(U_l)$ acts as zero on $H^1_{\et}(\mathcal{X}_{K^pK'_p},\Q_p)$ for any open subgroup $K'_p$ of $\GL_2(\Q_p)$. Fix a lift of geometric Frobenius $\Frob_l\in G_{\Q_l}$ whose image in $G_{\Q_l}^{\mathrm{ab}}$ corresponds to $l\in\Q^\times_l$ via the local Artin map. Recall that there is a determinant $D_S$ of $G_{\Q,S}$ valued in $\T(K^p)$. We denote by $Q_l(X)\in \T(K^p)[X]$ the characteristic polynomial of $l\Frob_l\in\T(K^p)[G_{\Q,S}]$. We claim that
\[P_l(X):=X^{m+1}Q_l(X)\]
works. To see this, since $H^1_{\et}(\mathcal{X}_{K^pK'_p},\overbar\Q_p)$ can be decomposed as a direct sum of $(\pi_f)^{K^pK'_p}$ for some cuspidal automorphic representations $\pi=\pi_f\otimes\pi_\infty$ on $\GL_2(\A)$, it is enough to show $P_l(U_l)=0$ on each non-zero $(\pi_f)^{K^pK'_p}$. By the theory of newforms, $(\pi_l)^{K_l}$ has dimension at most $m+1$ and $U_l$ is not nilpotent only when $\pi_l$ is special or a principal series. Moreover,  $U_l$ acts semi-simply on generalized eigenspaces associated to non-zero eigenvalues. See for example Corollary 2.2 of \cite{Hid89}. The local-global compatibility then implies that a non-zero eigenvalue of $U_l$ must be a root of $Q_l(X)$. Hence $P_l(U_l)=0$ on $(\pi_f)^{K^pK'_p}$.
\end{proof}

Now we consider the action of $\tilde\T(K^p)[U_p]$ on $M^{\dagger}_{1}(K)^{\psi}$.

\begin{lem} \label{keymult1}
$\lambda:\T(K^p)\to\overbar\Q_p$ can be extended to $\lambda':\tilde\T(K^p)[U_p]\to C$ such that 
\[M^{\dagger}_{1}(K^p\Gamma_1(p^n))^{\psi}[\kp_{\lambda'}]\neq 0,\]
where $\kp_{\lambda'}$ denotes the kernel of $\lambda'$. Moreover, for any such $\lambda'$, we have
\[\dim_C M^{\dagger}_{1}(K^p\Gamma_1(p^n))^{\psi}[\kp_{\lambda'}]=1.\]
In particular, we obtain the following after taking direct limit over all $n$, 
\[\dim_C \varinjlim_n M^{\dagger}_{1}(K^p\Gamma_1(p^n))^{\psi}[\kp_{\lambda'}]=\dim_C M^{\dagger}_{1}(K^p)^{\psi,N_0}[\kp_{\lambda'}]=1.\]
\end{lem}

\begin{proof}
Let $\alpha\in C$ be an eigenvalue of $U_p$ on $M^{\dagger}_{1}(K)^{\psi}[\kp_\lambda]$. Then $\tilde{\T}(K^p)/(\kp_\lambda)$ acts on the $\alpha$-eigenspace of $M^{\dagger}_{1}(K)^{\psi}[\kp_\lambda]$. Since $\tilde{\T}(K^p)/(\kp_\lambda)$ is Artinian by our previous lemma, we can find $\lambda':\tilde\T(K^p)[U_p]\to C$ which sends $U_p$ to $\alpha$ and extends $\lambda$ and  $M^{\dagger}_{1}(K)^{\psi}[\kp_{\lambda'}]\neq 0$. To see the multiplicity one claim, we remark that $A_S$ acts transitively on all connected components of $\mathcal{X}_K$. Hence any element of $M^{\dagger}_{1}(K)^{\psi}$ is determined by its value on one component. Since elements in $M^{\dagger}_{1}(K)^{\psi}[\kp_{\lambda'}]$ are eigenvectors of $T_l,S_l,l\notin S$ and $U_l,l\in S$, we conclude from usual $q$-expansion principle that this space has dimension one.
\end{proof}

\begin{proof}[Proof of Theorem \ref{cl1}] \label{proofcl1}
If $\rho$ is reducible, then it is a sum of two finite order characters. The theory of Eisenstein series shows that it comes from a classical modular form. Hence we may assume $\rho$ is irreducible. 

Assume $\rho$ is not classical, i.e. $M^1(K^p)[\kp_\lambda]=0$. Since $M^1(K^p)=\varinjlim_{K'_p}H^0(\mathcal{X}_{K^pK'_p},\omega)$ is a direct sum of $(\pi_f)^{K^p}$ for some automorphic representations $\pi$ of $\GL_2(\A)$, the localization 
\[M^1(K^p)_{\kp_\lambda}=0.\]
Similarly, by Serre duality and the Kodaira-Spencer isomorphism, we also have 
\[(\varinjlim_{K'_p}H^1(\mathcal{X}_{K^pK'_p}, \omega^{1}))_{\kp_\lambda}=0.\] 
Hence by the second part of Theorem \ref{weight1}, $(\tilde{H}^1(K^p,C)^{\la,(1,0)}_{(0,1)})_{\kp_{\lambda\cdot t}}=M^{\dagger}_{1}(K^p)_{\kp_\lambda}\cdot e_1^{-1}t$ and therefore
\[\tilde{H}^1(K^p,C)^{\la,(1,0)}_{(0,1)}[\kp_{\lambda\cdot t}]=M^{\dagger}_{1}(K^p)[{\kp_\lambda}]\cdot e_1^{-1}t,\]
where $\lambda\cdot t:\T(K^p)\to \overbar\Q_p$ denotes ``the twist of $\lambda$ by $t$'', i.e. sends $T_l$ to $\lambda(T_l)l^{-1}$. It can be checked that $\rho_{\lambda\cdot t}(-1)=\rho_\lambda$. Hence $\tilde{H}^1(K^p,\overbar\Q_p)[\kp_{\lambda\cdot t}]$ is $\rho$-isotypic.
On the other hand, by our assumption that $\rho$ is Hodge-Tate and Theorem \ref{Senkh}, we have 
\[\tilde{H}^1(K^p,C)^{\la}_{(0,1)}[\kp_{\lambda\cdot t}]=\tilde{H}^1(K^p,C)^{\la,(1,0)}_{(0,1)}[\kp_{\lambda\cdot t}].\]
If we write $W=\Hom_{G_\Q}(\rho,\tilde{H}^1(K^p,\overbar\Q_p))$, then by Lemma \ref{kpC} and our discussion in \ref{HA},
\[\tilde{H}^1(K^p,C)_{(0,1)}^{\la}[\kp_{\lambda\cdot t}]=\rho\otimes_{\overbar\Q_p}(W\widehat\otimes_{\overbar\Q_p}C)_{(0,1)}^{\la},\]
where $W$ is endowed a norm with unit ball $\Hom_{G_\Q}(\rho^o,\tilde{H}^1(K^p,\Z_p)\otimes \overbar\Z_p)$ and $\rho^o\subset \rho$ is any $G_{\Q}$-stable lattice. Note that $A_S\times\tilde\T(K^p)$ acts naturally on $W$ and $U_p$ acts on its $N_0$-fixed vectors. Fix $\lambda'$ as in Lemma \ref{keymult1} and consider its ``twist by $e_1^{-1}t$'' 
\[\lambda'':\tilde{T}(K^p)[U^p]\to C\] 
which sends $T_l$ to $\lambda'(T_l)l^{-1}$ for $l\notin S$ and sends $U_l$ to $\lambda'(U_l)l^{-1}$ for $l\in S$. Then we have
\[M^{\dagger}_{1}(K^p)^{\psi,N_0}[{\kp_{\lambda'}}]\cdot e_1^{-1}t=\tilde{H}^1(K^p,C)_{(0,1)}^{\la,\psi,N_0}[\kp_{\lambda''}]=\rho\otimes_{\overbar\Q_p}(W\widehat\otimes_{\overbar\Q_p}C)_{(0,1)}^{\la,\psi,N_0}[\kp_{\lambda''}].\]
The left hand side has dimension $1$ over $C$ by Lemma \ref{keymult1}. However the last term in the equality has dimension at least two because $\rho$ is $2$-dimensional! Thus we get a contradiction. This proves that $\rho$ has to be classical.
\end{proof}

\subsection{Local-global compatibility}
\begin{para}
In order to better understand $\tilde{H}^1(K^p,\overbar\Q_p)[\kp_\lambda]$ for a pro-modular $\lambda$, we need Emerton's local-global compatibility conjecture, which gives a description of $\tilde{H}^1(K^p,\overbar\Q_p)[\kp_\lambda]$ in terms of $p$-adic local Langlands for $\GL_2(\Q_p)$ as established by Breuil, Colmez, Berger, Kisin, Pa\v{s}k\={u}nas. In \cite{Eme1,Eme06C}, Emerton formulated and proved his conjecture assuming the residual representation is irreducible and generic. Building upon the work of Pa\v{s}k\={u}nas, we gave a proof (which is complete when $p\geq 5$) in \cite{Pan19} in the case of definite quaternion algebras. We will mostly follow the argument in \S3 of \cite{Pan19} and also \S5 of \cite{Pas18}. One key step is to prove the density of algebraic vectors when we are localizing at an Eisenstein maximal ideal. This part might be of some independent interest.
\end{para}

\begin{para}
We begin by recalling work of Pa\v{s}k\={u}nas on representations of $\GL_2(\Q_p)$. See \cite{Pas13} and also the introduction of \cite{Pas16}. Fix $E$ a finite extension of $\Q_p$ with ring of integers $\cO$, residue field $\F$ and fix a uniformizer $\varpi$. Write $G=\GL_2(\Q_p) $ and let $\zeta:\Q_p^\times\to\cO^\times$ be a continuous character. Following Pa\v{s}k\={u}nas, we denote by $\mathrm{Mod}^{\mathrm{l\, adm}}_{G,\zeta}(\cO)$ the category of smooth, locally admissible $G$-representations on $\cO$-torsion modules with central character $\zeta$. 

Let $\mathrm{Irr}^{\mathrm{adm}}_G$ be the set of $G$-irreducible representations in $\mathrm{Mod}^{\mathrm{l\, adm}}_{G,\zeta}(\cO)$. There is a natural equivalence relation $\sim$ on $\mathrm{Irr}^{\mathrm{adm}}_G$ defined as follows: $\pi\sim\tau$ if $\pi=\tau$ or there exists a sequence $\pi_0=\pi,\cdots,\pi_n=\tau\in \mathrm{Irr}^{\mathrm{adm}}_G$ such that $\Ext^1(\pi_i,\pi_{i+1})\neq 0$ or $\Ext^1(\pi_{i+1},\pi_{i})\neq 0$. An equivalence class $\mathfrak{B}\in \mathrm{Irr}^{\mathrm{adm}}_G/\sim$ is called a block of $\mathrm{Mod}^{\mathrm{l\, adm}}_{G,\zeta}(\cO)$. There is a natural decomposition 
\[\mathrm{Mod}^{\mathrm{l\, adm}}_{G,\zeta}(\cO)\cong\prod_{\mathfrak{B}\in \mathrm{Irr}^{\mathrm{adm}}_G/\sim}\mathrm{Mod}^{\mathrm{l\, adm}}_{G,\zeta}(\cO)[\mathfrak{B}],\]
where $\mathrm{Mod}^{\mathrm{l\, adm}}_{G,\zeta}(\cO)[\mathfrak{B}]$ denotes the full subcategory of representations with irreducible constituents in $\mathfrak{B}$. A complete list of blocks containing an absolutely irreducible representation can be found on the beginning of page 3 of \cite{Pas16}.

The semi-simple mod $p$ correspondence gives a bijection between isomorphism classes of two-dimensional absolutely semi-simple $\F$-representations $\bar\rho_{\mathfrak{B}}$ of $G_{\Q_p}$ and blocks containing an absolutely irreducible representation.
It should be mentioned that $\det\bar\rho_{\mathfrak{B}}\equiv\zeta\varepsilon_p\mod \varpi$. Let $R^{\ps,\zeta\varepsilon_p}_{\bar\rho_\mathfrak{B}}$ be the deformation ring parameterizing all the $2$-dimensional continuous determinants (in the sense of \cite{Che14}) of $G_{\Q_p}$ lifting $(\tr\bar\rho_{\mathfrak{B}},\det\bar\rho_{\mathfrak{B}})$ with (usual) determinant $\zeta\varepsilon_p$. Here is a summary of results of Pa\v{s}k\={u}nas and Pa\v{s}k\={u}nas-Tung that we will use later.
\end{para}

\begin{thm} \label{Pasr}
Suppose $\mathfrak{B}$ contains an absolutely irreducible representation. Let $Z_{\mathfrak{B}}$ be the Bernstein centre of $\mathrm{Mod}^{\mathrm{l\, adm}}_{G,\zeta}(\cO)[\mathfrak{B}]$.
\begin{enumerate}
\item There is a finite $\cO$-algebra homomorphism 
\[R^{\ps,\zeta\varepsilon_p}_{\bar\rho_\mathfrak{B}}\to Z_{\mathfrak{B}}\]
compatible with Colmez's functor (See \cite[Theorem 1.2]{PT21} for the precise statement).
Hence $\mathrm{Mod}^{\mathrm{l\, adm}}_{G,\zeta}(\cO)[\mathfrak{B}]$ admits a $R^{\ps,\zeta\varepsilon_p}_{\bar\rho_\mathfrak{B}}$-module structure from this map.
\item There is a faithful $R^{\ps,\zeta\varepsilon_p}_{\bar\rho_\mathfrak{B}}$-linear contravariant exact functor $\mathrm{m}$ from $\mathrm{Mod}^{\mathrm{l\, adm}}_{G,\zeta}(\cO)[\mathfrak{B}]$ to the category of $R^{\ps,\zeta\varepsilon_p}_{\bar\rho_\mathfrak{B}}$-modules sending admissible representations to finitely generated $R^{\ps,\zeta\varepsilon_p}_{\bar\rho_\mathfrak{B}}$-modules.
\end{enumerate}

\end{thm}

\begin{proof}
Both claims follow from \cite[Theorem 1.2, 1.3]{PT21}. In fact,  Pa\v{s}k\={u}nas and Tung show that $\mathrm{Mod}^{\mathrm{l\, adm}}_{G,\zeta}(\cO)[\mathfrak{B}]$ is anti-equivalent to the category of right pseudo-compact $E_{\mathfrak{B}}$-modules for some pseudo-compact ring $E_{\mathfrak{B}}$ and construct a map $R^{\ps,\zeta\varepsilon_p}_{\bar\rho_\mathfrak{B}}\to Z_{\mathfrak{B}}\subseteq E_{\mathfrak{B}}$ which makes $E_{\mathfrak{B}}$ into a finitely-generated $R^{\ps,\zeta\varepsilon_p}_{\bar\rho_\mathfrak{B}}$-module. The functor $\mathrm{m}$ is defined by passing the action of $E_\mathfrak{B}$ to $R^{\ps,\zeta\varepsilon_p}_{\bar\rho_\mathfrak{B}}$.
\end{proof}

\begin{para}
Next we turn to the global side. Fix a tame level $K^p$ from now on. Let $\km$ be a maximal ideal of $\T(K^p)\otimes_{\Z_p}\cO$. We may assume the residue field of $\km$ is $\F$ by enlarging $E$. Recall that there is a $2$-dimensional determinant of $G_{\Q}$ valued in $\T(K^p)$ and we extend it to a determinant valued in $\T(K^p)\otimes_{\Z_p}\cO$. Denote by $D_p$ its restriction to $G_{\Q_p}$ and by $\bar{D}_p$ its reduction modulo $\km$. Again enlarging $E$ if necessary, we may assume that $\bar{D}_p$ arises from a $2$-dimensional semi-simple representation $\bar\rho_{\km,p}:G_{\Q_p}\to\GL_2(\F)$. Fix a continuous character $\zeta':\A_f^\times/\det(K^p)\Q^{\times}_{>0}\to \cO^\times$ such that  $\zeta=\zeta'|_{\Q_p^\times}:\Q_p^\times\to \cO^\times$ is congruent to $(\det\bar\rho_{\km,p})\omega^{-1}$ modulo $\varpi$ via local class field theory.  For $?=E,\cO,E/\cO$, let $\tilde{H}^1(K^p,?)_{\zeta',\km}$ be the subspace of $(\tilde{H}^1(K^p,\Z_p)\otimes_{\Z_p} ?)_\km$ on which the centre $\A_f^\times$ acts via $\zeta'$ and let $\T$ be the image of $(\T(K^p)\otimes\cO)_\km$ inside $\End(\tilde{H}^1(K^p,E/\cO)_{\zeta',\km})$. Then $D_p$ induces a natural map
\[R^{\ps,\zeta\varepsilon_p}_{\bar\rho_{\km,p}}\to\T.\]
Hence the faithful Hecke action of $\T$ induces an action of $R^{\ps,\zeta\varepsilon_p}_{\bar\rho_{\km,p}}$ on $\tilde{H}^1(K^p,E/\cO)_{\zeta',\km}$ by our discussion. We denote this action by $\tau^{p}$.

On the other hand, $\bar\rho_{\km,p}$ determines a block $\mathfrak{B}_\km$ of $\mathrm{Mod}^{\mathrm{l\, adm}}_{G,\zeta}(\cO)$. Now we can state the main result of this subsection.
\end{para}

\begin{thm} \label{LGC}
Let $\km$ be a maximal ideal of $\T(K^p)\otimes_{\Z_p}\cO$. 
\begin{enumerate}
\item $\tilde{H}^1(K^p,E/\cO)_{\zeta',\km}\in \mathrm{Mod}^{\mathrm{l\, adm}}_{G,\zeta}(\cO)[\mathfrak{B}_\km]$. Hence by Theorem \ref{Pasr}, there is a natural action  $\tau_p$ of $R^{\ps,\zeta\varepsilon_p}_{\bar\rho_\mathfrak{B_\km}}=R^{\ps,\zeta\varepsilon_p}_{\bar\rho_{\km,p}}$ on $\tilde{H}^1(K^p,E/\cO)_{\zeta',\km}$.
\item Two actions $\tau^p$ and $\tau_p$ of $R^{\ps,\zeta\varepsilon_p}_{\bar\rho_{\km,p}}$ are the same.
\end{enumerate}
\end{thm}

Before giving a proof, we show that this result implies a (weak) form of local-global compatibility which will be enough for our applications. Recall that by \cite[Theorem 1.1]{CDP14}, there is a bijection between the isomorphism classes of the following two sets:
\begin{itemize}
\item $2$-dimensional absolutely irreducible representation  of $G_{\Q_p}$ over $E$;
\item non-ordinary admissible unitary $E$-Banach representations  of $\GL_2(\Q_p)$. 
\end{itemize}
Here, a Banach representation is non-ordinary if it is not a subquotient of a parabolic induction of a unitary character. This bijection is compatible with taking finite extensions of $E$. Hence we can extend this bijection $\overbar\Q_p$-linearly to a map (not bijection) from the isomorphism classes of $2$-dimensional irreducible representation of $G_{\Q_p}$ over $\overbar\Q_p$ to the isomorphism classes of $\overbar\Q_p$-representations of $\GL_2(\Q_p)$. We denote this map by $\rho\mapsto \Pi(\rho)$.

\begin{cor} \label{promodprocoh}
Suppose $\rho=\rho_\lambda$ is pro-modular and irreducible for some tame level $K^p$. 
Then $\tilde{H}^1(K^p,\overbar\Q_p)[\kp_\lambda]\neq 0$, i.e. $\rho$ is pro-cohomological. Moreover, we have the following description of $\tilde{H}^1(K^p,\overbar\Q_p)[\kp_\lambda]$ as a representation of $\GL_2(\Q_p)$: 
\begin{itemize}
\item If $\rho|_{G_{\Q_p}}$ is irreducible, then $\tilde{H}^1(K^p,\overbar\Q_p)[\kp_\lambda]\cong \Pi(\rho|_{G_{\Q_p}})^{\oplus d}$ for some $d$.
\item If $\rho|_{G_{\Q_p}}$ is reducible, then $\tilde{H}^1(K^p,\overbar\Q_p)[\kp_\lambda]$ has an ordinary subrepresentation. 
\end{itemize}
\end{cor}

\begin{rem}
For later application (Theorem \ref{classicality1}), all we need is that $\tilde{H}^1(K^p,\overbar\Q_p)[\kp_\lambda]$ is non-zero and contains an irreducible admissible representation.
\end{rem}

\begin{proof}
We denote by $\mathrm{m}_0$ the $R^{\ps,\zeta\varepsilon}_{\bar\rho_{\km,p}}$-module obtained by applying the faithful functor $\mathrm{m}$ in Theorem \ref{Pasr} to $\tilde{H}^1(K^p,E/\cO)_{\zeta',\km}$. Then the admissibility of $\tilde{H}^1(K^p,E/\cO)_{\zeta',\km}$ implies that $\mathrm{m}_0$ is a finitely generated $R^{\ps,\zeta\varepsilon}_{\bar\rho_{\km,p}}$-module. Since the action $\tau_p$ agrees with $\tau^p$ which factors through $\T$ by Theorem \ref{LGC},  we conclude that $\mathrm{m}_0$ is a faithful finitely generated $\T$-module. In particular, $\mathrm{m}_0\otimes_{\T} k(\kp)\neq0$ for any $\kp\in\Spec\T$, where $k(\kp)$ denotes the residue field of $\kp$. From this, we easily deduce $\tilde{H}^1(K^p,\overbar\Q_p)[\kp_\lambda]\neq 0$ in the corollary. Using $\tau^p=\tau_p$, our description of $\tilde{H}^1(K^p,\overbar\Q_p)[\kp_\lambda]$ follows from Pa\v{s}k\={u}nas's work on Banach representations \cite[Theorem 1.10, 1.11]{Pas13}, \cite[Corollary 2.24]{Pas16}.
\end{proof}

\begin{proof}[Proof of Theorem \ref{LGC}]
We can replace $K^p$ by an open subgroup and assume $K^p\GL_2(\Z_p)$ is sufficiently small. After twisting a character of $\A_f^\times/\Q^\times_{>0}$, we may assume $\zeta(x)=x^k,x\in\Z_p^\times$ for some integer $k$. We denote by $\tilde{H}^1(K^p,\Q_p)_k\subset\tilde{H}^1(K^p,\Q_p)$ the subspace on which the centre $\Z_p^\times\subset\GL_2(\Q_p)$ acts via $k$-th power. It has a norm induced from $\tilde{H}^1(K^p,\Z_p)$. Then after enlarging $E$ if necessary, there is a natural decomposition 
\[\tilde{H}^1(K^p,\Q_p)_k\otimes_{\Q_p} E=\bigoplus_{(\km',\zeta'')}\tilde{H}^1(K^p,E)_{\zeta'',\km'},\]
where $\km'$ runs through all maximal ideals of $\T(K^p)\otimes\cO$ and $\zeta'':\A_f^\times/\det(K^p)\Q^{\times}_{>0}$ runs through all continuous characters satisfying $\zeta''(x)=x^k,x\in\Z_p^{\times}$. In particular, $(\km,\zeta')$ appears inside this decomposition. Note that there are only finitely many $(\km',\zeta'')$.

\begin{lem} \label{density}
The $\GL_2(\Z_p)$-algebraic vectors in $\tilde{H}^1(K^p,\Q_p)_k$ are dense in $\tilde{H}^1(K^p,\Q_p)_k$.
\end{lem}

\begin{lem} \label{ssH}
The Hecke action of $\T(K^p)$ on the $\GL_2(\Z_p)$-algebraic vectors in $\tilde{H}^1(K^p,\Q_p)$ is semi-simple.
\end{lem}

We will give proofs of both lemmas at the end. Assuming these results, we deduce from the natural decomposition above that the $\GL_2(\Z_p)$-algebraic vectors in $\tilde{H}^1(K^p,E)_{\zeta',\km}$ are dense. Now one can argue in exactly the same way as proof of Proposition 5.5 of \cite{Pas18}. We only give a sketch here. It should be mentioned that these algebraic vectors correspond to cohomology of certain standard local systems on $\mathcal{Y}_{K^p\GL_2(\Z_p)}$. Let $\mathcal{S}\subset\Spec\T[\frac{1}{p}]$ be the subset of maximal ideals $\kp$ for which $\tilde{H}^1(K^p,E)_{\zeta',\km}[\kp]$ has non-zero $\GL_2(\Z_p)$-algebraic vectors. We write $\Pi(\kp)$ as the smallest $E[G]$-stable closed subspace containing the $\GL_2(\Z_p)$-algebraic vectors in $\tilde{H}^1(K^p,E)_{\zeta',\km}[\kp]$. Let $\Pi(\kp)^o\subset\Pi(\kp)$ be its open unit ball. By our density result and the semi-simplicity of Hecke action on the $\GL_2(\Z_p)$-algebraic vectors in $\tilde{H}^1(K^p,E)_{\zeta',\km}$, it suffices to prove that for any $\kp\in\mathcal{S}$,
\[\Pi(\kp)/\Pi(\kp)^o\in \mathrm{Mod}^{\mathrm{l\, adm}}_{G,\zeta}(\cO)[\mathfrak{B}_\km]\]
and both actions $\tau^p,\tau_p$ are the same on it. Both claims are really formal consequences of classical local-global compatibility and uniqueness of the universal unitary completion in this case (\cite[Corollarie 5.3.4]{BB10}, \cite[Proposition 2.2.1]{BE10}) and the compatibility between $p$-adic and classical local Langlands correspondence \cite[Theorem 1.3]{CDP14}. This finishes the proof of  Theorem \ref{LGC}.
\end{proof}

\begin{proof}[Proof of Lemma \ref{density}]
As pointed out by a referee, the following argument was sketched by Matthew Emerton in \cite[Remark 5.4.2]{Eme1}.

We begin the proof by recalling a complex computing $\tilde{H}^{i}(K^p,\Q_p)$. Let $\tilde{H}_c^i(K^p,\Z_p)$ be the completed cohomology with compact support, which is defined by 
\[\tilde{H}_c^i(K^p,\Z_p):=\varprojlim_n\varinjlim_{K_p} H^i_c(Y_{K^pK_p}(\mathbb{C}),\Z/p^n).\]
Since $Y_{K^pK_p}$ is a non-compact Riemann surface, it is easy to see that $\tilde{H}_c^i(K^p,\Z_p)=0,i\neq 1$.  There is a spectral sequence (coming from the Poincar\'e duality) relating $\tilde{H}^\bullet$ and $\tilde{H}^1_c$ (\S1.3 of \cite{CE12}) as follows.  Let $\Lambda$ be the completed group ring $\Z_p[[\GL_2(\Z_p)]]$. Then $\tilde{H}_c^1(K^p,\Z_p)^d:=\Hom_{\Z_p}(\tilde{H}_c^1(K^p,\Z_p),\Z_p)$ admits a resolution by finite free $\Lambda$-modules:
\[\cdots\to \Lambda^{\oplus d_{-1}}\to \Lambda^{\oplus d_0}\to \Lambda^{\oplus d_1}\to \tilde{H}_c^1(K^p,\Z_p)^d\to 0.\]
Taking $\Hom_{\Lambda}(\bullet,\Lambda)$ of this resolution, where $\Lambda$ is viewed as a left module of itself, we obtain a complex of right $\Lambda$-modules (using the right $\Lambda$-module structure of $\Lambda$)
\[\cdots\leftarrow \Lambda^{\oplus d_{-1}}\leftarrow \Lambda^{\oplus d_0}\leftarrow \Lambda^{\oplus d_1}.\]
We can view this sequence as a complex of left $\Lambda$-modules by applying the main involution $g\mapsto g^{-1}$ of $\Lambda$. Now taking $\Hom^{\cont}_{\Z_p}(\bullet,\Q_p)$, we get a $\GL_2(\Z_p)$-equivariant complex
\[C^{\bullet}=\cdots\to \sC(\GL_2(\Z_p),\Q_p)^{\oplus d_{-1}}\to  \sC(\GL_2(\Z_p),\Q_p)^{\oplus d_0}\to  \sC(\GL_2(\Z_p),\Q_p)^{\oplus d_1},\]
where $\sC(\GL_2(\Z_p),\Q_p)$ denotes the space of $\Q_p$-valued continuous functions on $\GL_2(\Z_p)$. The result in \S1.3 of \cite{CE12} says that this complex computes $\tilde{H}^{i}(K^p,\Q_p)$ as a $\GL_2(\Z_p)$-representation. For a $\Z_p[\GL_2(\Z_p)]$-module $M$, we denote by $M_k$ the subspace on which the centre $\Z_p^\times$ acts via $k$-th power. We claim that 
\[C^{\bullet}_k=\cdots\to \sC(\GL_2(\Z_p),\Q_p)_k^{\oplus d_{-1}}\to  \sC(\GL_2(\Z_p),\Q_p)_k^{\oplus d_0}\to  \sC(\GL_2(\Z_p),\Q_p)_k^{\oplus d_1}\]
computes $\tilde{H}^{i}(K^p,\Q_p)_k$. Assuming this, $\tilde{H}^{1}(K^p,\Q_p)_k$ is a $\GL_2(\Z_p)$-equivariant quotient of $\sC(\GL_2(\Z_p),\Q_p)_k^{\oplus d_1}$. Hence we get the density statement in Lemma \ref{density} because $\GL_2(\Z_p)$-algebraic vectors are dense in $\sC(\GL_2(\Z_p),\Q_p)_k$ (\cite[Proposition 3.2.9]{Pan19}).

To see that $C^{\bullet}_k$ computes $\tilde{H}^{i}(K^p,\Q_p)_k$, it is enough to show that 
\[H_\cont^j(\Z_p^\times,\sC(\GL_2(\Z_p),\Q_p)\otimes z^{-k})=H_\cont^j(\Z_p^\times,\tilde{H}^{0}(K^p,\Q_p)\otimes z^{-k})=0,\,j\geq1,\]
where $z^{-k}=\Q_p$ with $\Z_p^\times$ acting via ($-k$)-th power. Since everything is over a characteristic zero field, it suffices to prove this after replacing $\Z_p^\times$ by an open subgroup. Let $H=1+p^2\Z_p\subset\Z_p^\times$. Using the description of $\tilde{H}^{0}(K^p,\Q_p)$ in \cite[4.2]{Eme06}, it is easy to see that as Banach representations of $H$, both $\sC(\GL_2(\Z_p),\Q_p)$ and $\tilde{H}^{0}(K^p,\Q_p)$ are of the form $\sC(H\times R,\Q_p)$ for some profinite set $R$ with trivial $H$-action. Hence, their tensor products with $z^{-k}$ have no higher cohomology.
\end{proof}

\begin{proof}[Proof of Lemma \ref{ssH}]
By Emerton's description \cite[4.3.4]{Eme06} of  the $\GL_2(\Z_p)$-algebraic vectors  in  $\tilde{H}^1(K^p,\Q_p)$, it suffices to show that the Hecke action on $H^1(Y_{K^p\GL_2(\Z_p)}(\bC),\mathcal{V_W})$ is semi-simple for any irreducible algebraic representation $W$ of $\GL_2$ over $\Q_p$, where $\mathcal{V_W}$ denotes the local system corresponding to $W$. This result is presumably well-known for experts. We sketch a proof here. We may change the coefficient field of the local system to $\bC$. Then using the Shimura isomorphism, it suffices to prove the semi-simplicity of the Hecke action on the space of classical modular forms. This is clear for cusp forms because  $T_l,\,l\notin S$ are normal operators with respect to the Petersson inner product. The theory of Eisenstein series then provides an eigenbasis for the orthogonal complement of cusp forms. See for example Weisinger's thesis \cite[Chap. 1]{Wei77}. 

Here is a more conceptual argument suggested to me by Matthew Emerton using polarizations. For simplicity, we assume $W$ is the trivial local system. But the argument can be generalized for any $W$. Let $Y=Y_{K^p\GL_2(\Z_p)}(\bC)$ and $X=X_{K^p\GL_2(\Z_p)}(\bC)$. By the Poincar\'e duality, it is enough to consider $H^1_c(Y,\bC)$. There is a natural exact sequence 
\[\bC^{\oplus |\mathcal{C}|}\to H^1_c(Y,\bC)\to H^1(X,\bC)\to 0.\]
(Recall that $\mathcal{C}$ denotes the set of cusps in $X$.) The Hecke operators $T_l,l\notin S$ are normal on $H^1(X,\bC)$  (resp. $\bC^{\oplus |\mathcal{C}|}$) with respect to the Riemann form (resp. standard Hermitian form on $\bC^{\oplus |\mathcal{C}|}$). Hence the Hecke actions of $\T(K^p)$ on both spaces are semi-simple. Now our claim follows from the Manin-Drinfeld theorem as it implies that  $H^1_c(Y,\bC)\to H^1(X,\bC)$ has a natural splitting, cf. \cite{El90}. 
\end{proof}

\subsection{Colmez's Kirillov model}
\begin{para}
Let $\rho=\rho_\lambda$ be as in Corollary \ref{promodprocoh}. Suppose $L$ is a finite extension of $\Q_p$ containing $\lambda(\T(K^p))$. Then Corollary \ref{promodprocoh} implies that $\tilde{H}^1(K^p,L)[\kp_\lambda]=\tilde{H}^1(K^p,\overbar\Q_p)[\kp_\lambda]\cap \tilde{H}^1(K^p,L)$ contains an irreducible admissible Banach representation $\Pi$ of $\GL_2(\Q_p)$ over $L$. To apply our previous results (Theorem \ref{hwvg}, Theorem \ref{cl1}), it is crucial to know whether its locally analytic vectors $\Pi^{\la}$ admit non-zero $\kn$-invariants and whether its $N_0$-invariants have a $U_p$-eigenvector. Thanks to the beautiful work of Colmez, both questions have affirmative answers by his Kirillov model \cite[Chap. VI]{Col10}. We first give a brief review of his work. A nice summary can be found in \S7.3, 7.4 of \cite{DLB17}. I would like to thank Matthew Emerton for sharing his unpublished notes \cite{EPW} with Robert Pollack and Tom Weston on this subject.
\end{para}

\begin{para}
Let $E$ be a finite extension of $\Q_p$ and $V$ a two-dimensional  representation of $G_{\Q_p}$ over $E$. 
By Fontaine's work \cite{Fo04}, we can study $V\otimes_{\Q_p}B_{\dR}^+$ by Sen's method. As a result, we obtain a free $E[[t]]\otimes_{\Q_p}\Q_p(\mu_{p^\infty})$-module $D^+_{\mathrm{dif}}(V)$  of rank $2$ equipped with a semi-linear action of $\Gamma=\Gal(\Q_p(\mu_{p^\infty})/\Q_p)$. Let $D_{\mathrm{dif}}(V):=D^+_{\mathrm{dif}}(V)[\frac{1}{t}]$ and 
\[D^-_{\mathrm{dif}}(V):=D_{\mathrm{dif}}(V)/D^+_{\mathrm{dif}}(V),\]
a torsion $\Q_p(\mu_{p^\infty})[t]$-module. We denote by $\sC_c(\Q_p^\times,D^-_{\mathrm{dif}}(V))^\Gamma$ the space of functions $\phi:\Q_p^\times\to D^-_{\mathrm{dif}}(V)$ with compact support satisfying
\[\phi(ax)=\sigma_a(\phi(x)),a\in\Z_p^\times,\,x\in\Q_p^\times,\]
where $\sigma_a\in\Gamma$ is inverse image of $a$ under $\varepsilon_p:\Gamma\xrightarrow{\sim}\Z_p^\times$. In particular, $\phi$ is completely determined by $\{\phi(p^n)\}_{n\in\Z}$. Let $P=\begin{pmatrix} \Q_p^\times & \Q_p \\ 0 & 1\end{pmatrix}$, the mirabolic subgroup of $\GL_2(\Q_p)$. One can define an action of $P$ on  $\sC_c(\Q_p^\times,D^-_{\mathrm{dif}}(V))^\Gamma$ by the formula:
\[(\begin{pmatrix} a & b \\ 0 & 1\end{pmatrix}\phi)(x)=\varepsilon'(bx)e^{tbx}\phi(ax),\]
where $\varepsilon':\Q_p\to \mu_{p^\infty}$ is an additive character with kernel $\Z_p$, and $e^{tbx}$ is understood as $\sum_{n\geq0}\frac{(bx)^n}{n!}t^n$. (One can view $\Q_p\to (\Q_p[[t]]\otimes_{\Q_p}\Q_p(\mu_{p^\infty}))^\times,\,x\mapsto \varepsilon'(x)e^{tx}$ as an additive character.) This is Colmez's Kirillov model. 

Now we assume 
\begin{itemize}
\item $V$ is absolutely irreducible.
\end{itemize} 
We denote by $\Pi(V)_E$ the irreducible unitary $E$-Banach representation of $\GL_2(\Q_p)$ associated to $V$ via the $p$-adic local Langlands correspondence (\cite[Theorem 1.1]{CDP14}) and by $\Pi(V)_E^{\la}$ its locally analytic vectors. 
\end{para}

\begin{thm}
There is a $P$-equivariant embedding $\sC_c(\Q_p^\times,D^-_{\mathrm{dif}}(V))^\Gamma\to \Pi(V)_E^{\la}$. 
\end{thm}

\begin{proof}
See Proposition 7.6 of \cite{DLB17}.
\end{proof}

The following corollary and remark were first pointed out to me by Matthew Emerton. 

\begin{cor} \label{ColKir}
$\Pi(V)_E^{\la,N_0}\neq 0$ and has a non-zero vector annihilated by $U_p$. 
\end{cor}

\begin{proof}
This is a standard formal consequence of the Kirillov model. More precisely, let $v\in D^-_{\mathrm{dif}}(V)$ be a non-zero element annihilated by $t$. Consider $\phi\in\sC_c(\Q_p^\times,D^-_{\mathrm{dif}}(V))^\Gamma$ defined by $\phi(1)=v$ and $\phi(x)=0,x\notin\Z_p^\times$. Clearly $\phi$ is fixed by $\begin{pmatrix} 1 & 1 \\ 0 & 1 \end{pmatrix}$. Moreover, 
\[(U_p\cdot \phi)(x)=\sum_{i=0}^{p-1}(\begin{pmatrix} p & i \\ 0 & 1\end{pmatrix}\phi)(x)=\sum_{i=0}^{p-1}\varepsilon'(ix)\phi(px).\]
If $px\notin \Z_p^\times$, this is zero  for trivial reason. If $x=\frac{1}{p}$, this is essentially the sum over all $p$-th roots of unity, which is again zero. Hence $\phi$ is annihilated by $U_p$.
\end{proof}

\begin{rem} \label{csColKir}
It's easy to see that all the vectors in $\sC_c(\Q_p^\times,D^-_{\mathrm{dif}}(V))^{\Gamma,N_0}$ annihilated by $U_p$ (which we denote by $\sC_c(\Q_p^\times,D^-_{\mathrm{dif}}(V))^{\Gamma,N_0,U_p=0}$) are of this form, i.e. 
\[\sC_c(\Q_p^\times,D^-_{\mathrm{dif}}(V))^{\Gamma,N_0,U_p=0}=D^+_{\mathrm{dif}}(V)\cdot\frac{1}{t}/D^+_{\mathrm{dif}}(V)=D_{\mathrm{Sen}}(V(-1)).\] 
Here $D_{\mathrm{Sen}}(V(-1))$ is a free $E\otimes_{\Q_p}\Q_p(\mu_{p^\infty})$-module of rank $2$ equipped with a semi-linear action of $\Gamma$ associated to $V(-1)$ in Sen's theory.

In Colmez's work, it's also possible to describe the image of $\sC_c(\Q_p^\times,D^-_{\mathrm{dif}}(V))^\Gamma$ inside $\Pi(V)_E^{\la}$. In fact, one can show that this image contains $(\Pi(V)_E^{\la})^{N_0,U_p=0}$. Hence 
\[\Pi(V)_E^{\la,N_0,U_p=0}\cong D_{\mathrm{Sen}}(V(-1)),\]
which identifies the action of $\begin{pmatrix} \Z_p^\times & 0 \\ 0 & 1 \end{pmatrix}$ on the left hand side with the action of $\Gamma$ on the right hand side. Therefore, let $a,b$ be the Hodge-Tate-Sen weights of $V$. Since $\Pi(V)$ has central character $\det(V)\varepsilon_p^{-1}$ (\cite[Corollary 1.2]{CDP14}), we conclude that both
\[(\Pi(V)_E^{\la})_{(-a-1,-b)},(\Pi(V)_E^{\la})_{(-b-1,-a)}\]
are non-zero. Moreover, let $P_0=\begin{pmatrix} \Z_p^\times & \Z_p \\ 0 & 1 \end{pmatrix}$. We have
\[\dim_E \Pi(V)_E^{\la,P_0,U_p=0}=\dim_E (V(-1)\otimes_{\Q_p}C)^{G_{\Q_p}},\]
which gives a representation-theoretic way to determine whether $V(-1)$ is Hodge-Tate of weight $0,0$ or not. We will use this to give another proof of Corollary \ref{modwt1} below.
\end{rem}

\begin{cor} \label{U_pev}
Suppose $\Pi$ is an irreducible admissible unitary Banach representation of $\GL_2(\Q_p)$ over some finite extension of $\Q_p$. Then $\Pi^{\la,N_0}\neq 0$ and has a non-zero $U_p$-eigenvector.
\end{cor}

\begin{proof}
We may enlarge the field of coefficients $E$ and assume that $\Pi$ is absolutely irreducible by \cite[Theorem 1.1]{DS13}. If $\Pi$ is non-ordinary, then this follows Corollary \ref{ColKir} as $\Pi=\Pi(V)_E$ for some absolutely irreducible Galois representation $V$ by \cite[Theorem 1.1]{CDP14}. If $\Pi$ is ordinary, i.e. a subquotient of a parabolic induction of a unitary character, then up to twist by a character $\eta\circ\det$, the representation $\Pi$ is either trivial representation $\mathbf{1}$, unitary Steinberg representation $\widehat{\mathrm{St}}$ or $\Pi$ is a unitary parabolic induction of a unitary character. Recall that $\widehat{\mathrm{St}}=(\mathrm{Ind}^{\GL_2(\Q_p)}_B\mathbf{1})/\mathbf{1}$, the universal unitary completion of the usual Steinberg representation. In all these cases, it is easy to check that $\Pi^{\la,N_0}\neq 0$ and has a non-zero $U_p$-eigenvector.
\end{proof}

We obtain the following classicality result by combining Corollary \ref{U_pev}, Corollary \ref{promodprocoh} and Theorem \ref{cl1}.

\begin{thm} \label{classicality1}
Suppose $\rho=\rho_{\lambda}$ is pro-modular and irreducible. Assume  
\begin{itemize}
\item $\rho|_{G_{\Q_p}}$ is Hodge-Tate of weight $0,0$. 
\end{itemize}
Then $\rho$ is classical, i.e. $M_1(K^p)[\kp_\lambda]\neq 0$ for some tame level $K^p$.
\end{thm}

\begin{proof}
Suppose $\rho$ is pro-modular for some tame level $K^p$. We have 
\[\tilde{H}^1(K^p,\overbar\Q_p)[\kp_{\lambda\cdot t}]=\tilde{H}^1(K^p,\overbar\Q_p)[\kp_\lambda]\cdot t\neq 0\]
by Corollary \ref{promodprocoh}. Recall that $t\in \tilde{H}^0(K^p,\Z_p)$ is an invertible function, so $\tilde{H}^1(K^p,\overbar\Q_p)[\kp_\lambda]\cdot t$ is understood as a subspace of $\tilde{H}^1(K^p,\overbar\Q_p)$. Here $\lambda\cdot t$ was defined in proof of Theorem \ref{cl1}. Corollary \ref{U_pev} then implies that $\tilde{H}^1(K^p,C)^{\la}[\kp_{\lambda\cdot t}]^{N_0,U_p=\alpha}\neq0$ for some $\alpha\in \overbar\Q_p$. By our result on the infinitesimal character (Proposition \ref{icht}), $(h+1)^2=0$ on this space. Hence the weight-$(0,1)$ subspace
\[\tilde{H}^1(K^p,C)^{\la}_{(0,1)}[\kp_{\lambda\cdot t}]^{N_0,U_p=\alpha}\neq 0.\]
Now assume $\rho$ is not classical. As in the proof of Theorem \ref{cl1}, we have 
\[\tilde{H}^1(K^p,C)^{\la}_{(0,1)}[\kp_{\lambda\cdot t}]=M^{\dagger}_{1}(K^p)[{\kp_\lambda}]\cdot e_1^{-1}t.\]
However, this means that $M^{\dagger}_{1}(K^p)[{\kp_\lambda}]^{N_0,U_p=p\alpha}\neq 0$, which contradicts Theorem \ref{cl1}. Hence $\rho$ has to be classical.
\end{proof}

We quote the following result (\cite[Theorem 1.2.3]{Eme1}, which is based on previous work of B\"ockle, Diamond-Flach-Guo, Khare-Wintenberger, Kisin) on promodularity of a Galois representation. For simplicity, we exclude the case $\bar\rho|_{G_{\Q_p}}\cong\eta\otimes\begin{pmatrix} \omega & * \\ 0 & 1 \end{pmatrix}$, a non-split extension. As a corollary, we give a new proof of the Fontaine-Mazur conjecture in the irregular case under conditions below.

\begin{thm} \label{promodularity}
Let $p>2$ be a prime number and $\rho:G_{\Q}\to\GL_2(\overbar\Q_p)$ be a continuous two-dimensional irreducible odd representation which is unramified outside of finitely many primes. Assume 
\begin{enumerate}
\item $\bar\rho|_{G_{\Q(\mu_p)}}$ is irreducible;
\item $(\bar\rho|_{G_{\Q_p}})^{ss}$ is either irreducible or of the form $\eta_1\oplus \eta_2$ for some characters $\eta_1,\eta_2$ satisfying $\eta_1/\eta_2\neq 1,\omega^{\pm1}$.
\end{enumerate}
Then $\rho$ is pro-modular.
\end{thm}

\begin{cor} \label{modwt1}
Let $\rho$ be as in Theorem \ref{promodularity}. If $\rho$ is Hodge-Tate of weight $0,0$, then $\rho$ comes from a cuspidal eigenform of weight one.
\end{cor}

\begin{proof}
A combination of previous two results.
\end{proof}

\begin{rem}
The assumptions in this result could possibly be removed in the following way. Let $\rho$ be  a continuous two-dimensional irreducible odd representation $\rho:G_{\Q,S}\to \GL_2(\overbar\Q_p)$ which is unramified outside of finitely many primes. In \cite{Pan19}, it is proved that $\rho|_{G_F}$ is irreducible and  pro-modular for a solvable totally real field extension $F/\Q$ in which $p$ completely splits, under the assumption
\begin{itemize}
\item $p>2$;
\item If $p=3$, then $(\bar\rho|_{G_{\Q_p}})^{ss}$ is not of the form $\eta\oplus\eta\omega$. 
\end{itemize}
Hence one can prove Corollary \ref{modwt1} under these assumptions, if Theorem \ref{classicality1} can be extended to the case of Hilbert modular variety over a totally real field $F$ in which $p$ completely splits.  It seems (at least to me) that it's reasonable to expect such a generalization.
\end{rem}

We can also determine the possible systems of Hecke eigenvalues appearing in spaces of overconvergent modular forms.

\begin{thm} \label{ocmfwt}
Suppose $\rho=\rho_\lambda$ is irreducible and pro-modular for some tame level $K^p$. Let $a,b$ be the Hodge-Tate-Sen weights of $\rho|_{G_{\Q_p}}$.
\begin{enumerate}
\item If $M^{\dagger}_{(n_1,n_2)}(K^p)[\kp_\lambda]\neq 0$, then $(n_1,n_2)=(-a,-b-1)$ or $(-b,-a-1)$. In particular, if $M^\dagger_k(K^p)[\kp_\lambda]\neq 0$ for some $k\in\Z$, then $\{a,b\}=\{0,k-1\}$.
\item Conversely, if $\rho|_{G_{\Q_p}}$ is irreducible, 
then $M^{\dagger}_{(n_1,n_2)}(K^p)[\kp_\lambda]\neq 0$ for $(n_1,n_2)=(-a,-b-1)$ and $(-b,-a-1)$.
\end{enumerate}
\end{thm}

\begin{proof}
To see the first part, if $M^{\dagger}_{(n_1,n_2)}(K^p)[\kp_\lambda]\neq 0$ and $n_1\neq n_2+1$, then by Theorem \ref{hwvg}, we have $\tilde{H}^1(K^p,C)^{\la}_{(n_1-1,n_2+1)}[\kp_\lambda]\neq 0$. Hence Proposition \ref{icht} implies that $\rho_\lambda$ has Hodge-Tate-Sen weights $-n_1,-1-n_2$. The case $n_1=n_2+1$ can be proved in a similar way. We omit the details here.

For the second part, if $a=b=0$, the proof of Theorem \ref{classicality1} shows that $M^{\dagger}_{(0,-1)}(K^p)[\kp_\lambda]\neq 0$. For the general case $a=b$, one can twist by some $t^a$ as in \ref{twisttn1} to reduce to the case $a=0$.

Now assume $a\neq b$. By Theorem \ref{promodprocoh} and Remark \ref{csColKir}, both
\[\tilde{H}^1(K^p,C)^{\la}_{(-a-1,-b)}[\kp_\lambda]\;\mbox{and}\;\tilde{H}^1(K^p,C)^{\la}_{(-b-1,-a)}[\kp_\lambda]\]
are non-zero.  By symmetry, it's enough to consider $\tilde{H}^1(K^p,C)^{\la}_{(-a-1,-b)}[\kp_\lambda]$. Theorem \ref{hwvg} provides a natural decomposition of this space according to the Hodge-Tate-Sen weights. Since $\tilde{H}^1(K^p,\overbar{\Q}_p)[\kp_\lambda]$ is $\rho_\lambda(-1)$-isotypic, the Hodge-Tate-Sen weight-$(b+1)$ component of $\tilde{H}^1(K^p,C)^{\la}_{(-a-1,-b)}[\kp_\lambda]$ is  non-zero. Hence $M^{\dagger}_{(-a,-b-1)}(K^p)[\kp_\lambda]\neq 0$ if $-a\neq -b-1$ by Theorem \ref{hwvg} (i.e. $k\neq 2$). If $-a=-b-1$, then we claim we still have $M_{(-a-1,-b),w}[\kp_\lambda]\cong M^{\dagger}_{(-a,-b-1)}(K^p)[\kp_\lambda]$ (as $\T(K^p)$-modules). This is because the difference between $M_{(-a-1,-b),w}$ and $M^{\dagger}_{(-a,-b-1)}(K^p)$ is essentially $M_0(K^p)\cdot t^{-a}$, which can be decomposed as 
\[\bigoplus_{\kq\in\mathrm{Max}(\T(K^p)[\frac{1}{p}])} M_0(K^p)\cdot t^{-a}[\kq]\]
as a Hecke-module, and $M_0(K^p)\cdot t^{-a}[\kp_\lambda]=0$ since  $M_0(K^p)\cdot t^{-a}[\kp_{\lambda'}]\neq 0$ only when $\rho_{\lambda'}$ is reducible. This proves the second part.
\end{proof}

\begin{rem}
We sketch another proof of Theorem \ref{classicality1} in the non-ordinary case (i.e. $\rho|_{G_{\Q_p}}$ is irreducible) without using our work on the Sen operator (Theorem \ref{Senkh}). Suppose $\rho$ is not classical. Then as in the proof of Theorem \ref{cl1}, we have the following exact sequence by Theorem \ref{weight1}
\[0\to M^\dagger_1(K^p)[\kp_\lambda]\cdot e_1^{-1}t\to \tilde{H}^1(K^p,C)_{(0,1)}^{\la}[\kp_{\lambda\cdot t}]\to M^\dagger_1(K^p)[\kp_\lambda]\cdot e_1^{-1}t.\]
For simplicity, let's assume $K^p=\prod_{l\neq p}\GL_2(\Z_l)$. In general, we can use Hecke operators at ramified places as in Lemma \ref{keymult1}. Now the key point is that $\tilde{H}^1(K^p,C)[\kp_{\lambda\cdot t}]$ contains a copy of $\rho\otimes \Pi(\rho(1))$. Hence by Remark \ref{csColKir} (with $V(-1)=\rho$ here), 
\[\tilde{H}^1(K^p,C)_{(0,1)}^{\la}[\kp_\lambda]^{P_0,U_p=0}=\tilde{H}^1(K^p,C)^{\la}[\kp_\lambda]^{P_0,U_p=0}\]
has dimension (over $C$) at least $(\dim_{\overbar\Q_p}\rho)\times (\dim_{\overbar\Q_p} (\rho\otimes C)^{G_{\Q_p}})=4$.  But this contradicts the exact sequence above as $\dim_C (M^\dagger_1(K^p)[\kp_\lambda]\cdot e_1^{-1}t)^{P_0,U_p=0}=1$ by $q$-expansion principle and taking $P_0$-invariants and kernel of $U_p$ are left-exact. Hence $\rho$ has to be classical. 
\end{rem}

\bibliographystyle{amsalpha}

\bibliography{bib}

\end{document}